\documentclass[11pt]{article}
\usepackage[utf8]{inputenc}
\usepackage{amsmath,amsthm,amssymb}
\usepackage{tikz}
\usepackage{tikz-cd} % commutative diagrams
\usepackage{mathtools} % enhance amsmath
\usepackage{textcomp} % text symbols suchas bullet copyright musical notes
\usepackage{bbm} % mathbbm for 2
\usepackage{bbold} % all chars with bbm
\usepackage{mathpartir}
\usepackage{proof} % proof trees
\usepackage[tiny]{titlesec} % modify the size of the titles in ssections
\usepackage{mathrsfs} % allow bicategories notation using script with \mathscr
\usepackage{enumerate}

\usepackage{hyperref}
\hypersetup{
	colorlinks=true,
	linkcolor=blue,
      }

\usepackage{cleveref} % Use clever ref to carry the name of the referenced object. Must be loaded after any ref package.

\newtheorem{theorem}{Theorem}[section] % counter of results by section

\newtheorem{lemma}[theorem]{Lemma}
\newtheorem{proposition}[theorem]{Proposition}
\newtheorem{corollary}[theorem]{Corollary}

\theoremstyle{definition}
\newtheorem{definition}[theorem]{Definition}
\newtheorem{example}[theorem]{Example}
\newtheorem{construction}[theorem]{Construction}

\theoremstyle{remark}
\newtheorem{remark}[theorem]{Remark}
\newtheorem{observation}[theorem]{Observation}
\newtheorem{notation}[theorem]{Notation}

\newtheoremstyle{named}{}{}{\itshape}{}{\bfseries}{.}{.5em}{\thmnote{#3 }#1}
\theoremstyle{named}
\newtheorem*{introtheorem}{Theorem}

\newcommand{\N}{\mathbb{N}} % Naturals
\newcommand{\Z}{\mathbb{Z}} % Integers
\newcommand{\C}{\mathbb{C}} % Complex numbers

\newcommand{\cfam}{\mathbf{Fam}} % category FAM
\newcommand{\set}{\mathbf{Set}} % category of sets
\newcommand{\topo}{\mathbf{Top}} % category of topological spaces
\newcommand{\cat}{\mathbf{Cat}} % category of categories
\newcommand{\twocat}{\mathbf{2\text{-}Cat}} % category of 2 two-categories and 2 functors
\newcommand{\bicat}{\mathbf{Bicat_s}} % category of bicategories and strict functors
\newcommand{\sset}{\mathbf{sSet}} % category of simplicial sets
\newcommand{\bisim}{\mathbf{ssSet}} % category of bisimplicial sets
\newcommand{\bool}{\mathbf{Bool}} % category of simplicial categories

\newcommand{\catB}{\mathcal{B}} % abstract category B
\newcommand{\catC}{\mathcal{C}} % abstract category C
\newcommand{\catD}{\mathcal{D}} % abstract category D
\newcommand{\catE}{\mathcal{E}} % abstract category E
\newcommand{\catJ}{\mathcal{J}} % abstract category J
\newcommand{\catM}{\mathcal{M}} % abstract category M
\newcommand{\catN}{\mathcal{N}} % abstract category N
\newcommand{\catU}{\mathcal{U}} % abstract category F
\newcommand{\pointtyp}{\textup{\textbf{1}}} % single point type

\newcommand{\iscontr}{\textup{\textsf{isContr}}} % a type is contractible
\newcommand{\isinitial}{\textup{\textsf{isInitial}}} % a type is small

\newcommand{\idx}{\textup{\textsf{id}}}

\newcommand{\type}{\textup{\textsf{Type}}}
\newcommand{\context}{\textup{\textsf{Ctxt}}}
\newcommand{\spc}{\textup{\textsf{Set}}} %types for bisimplicial sets

\newcommand{\kcon}{\kappa\text{-CON}} % category of contextual categories
\newcommand{\kgat}{\kappa\text{-GAT}} % category of generalized algebraic theories

\newcommand{\display}{\twoheadrightarrow}
\newcommand{\fibration}{\twoheadrightarrow}
\newcommand{\cofibration}{\hookrightarrow}
\newcommand{\trivialcof}{\overset{\sim}{\hookrightarrow}}
\newcommand{\trivialfib}{\overset{\sim}{\twoheadrightarrow}}

\newcommand{\lefttrivialcofib}{\overset{\sim}{\hookleftarrow}}

\newcommand{\Lcal}{\mathcal{L}}
\newcommand{\Pcal}{\mathcal{P}}

\newcommand{\Jcal}{\mathcal{J}}

\newcommand{\Lb}{\mathbb{L}}

\newcommand{\obtyp}{\textup{\textsf{Ob}}}
\newcommand{\homtyp}{\textup{\textsf{Hom}}}
\newcommand{\eqtyp}{\textup{\textsf{Eq}}} 
\newcommand{\simp}{\textup{\textsf{-simplex}}}

\newcommand{\cof}{\textsc{Cof}}
\newcommand{\fib}{\textsc{Fib}}
\newcommand{\bif}{\textsc{Bif}}

\newcommand{\ty}{\textsf{Ty}}

\DeclareMathOperator{\Colim}{\texttt{Colim}}
\DeclareMathOperator{\Lim}{\texttt{Lim}}
\newcommand{\op}{\texttt{op}}

\DeclareMathOperator{\Mod}{Mod} % For the category of models of a clan or a cartmel theory
\DeclareMathOperator{\Hom}{Hom}
\DeclareMathOperator{\Ho}{Ho} % Homotophy category
\newcommand{\yoneda}{\text{\usefont{U}{min}{m}{n}\symbol{'210}}}
\newcommand{\ie}{\textit{i.e., }}

\DeclareFontFamily{U}{min}{}
\DeclareFontShape{U}{min}{m}{n}{<-> udmj30}{}

\newcommand{\first}{$1^{st}$ invariance theorem}
\newcommand{\First}{$1^{st}$ invariance theorem}

\newcommand{\Second}{$2^{nd}$ invariance theorem}

\newcommand{\third}{$3^{rd}$ invariance theorem}
\newcommand{\Third}{$3^{rd}$ invariance theorem}

\newcommand{\fourth}{$4^{th}$ invariance theorem}
\newcommand{\Fourth}{$4^{th}$ invariance theorem}

\usepackage{todonotes}

\title{ Homotopy Languages}
\author{C\'esar Bardomiano Mart\'inez and Simon Henry}

\begin{document}

	\maketitle
	
	\begin{abstract} We attach to each weak model category $\mathcal{M}$ a class of first order formulas about the fibrant objects of $\mathcal{M}$ whose validity is invariant under homotopies and weak equivalences. This is a generalization of the classical Blanc-Freyd Language of categories---which involves formula avoiding equality on objects and which are invariant under isomorphism and equivalences of categories. In particular, we obtain similar homotopy invariant languages for $2$-categories, bicategories, chain complexes, Kan complexes, quasi-categories, Segal spaces, and so on...

	\end{abstract}

\renewcommand{\thefootnote}{\fnsymbol{footnote}} 
\footnotetext{\emph{emails:} cbard035@uottawa.ca, shenry2@uottawa.ca}
\renewcommand{\thefootnote}{\arabic{footnote}}

% 18A15 Foundations, relations to logic and deductive system
% 18C10 Theories (e.g., algebraic theories), structure, and semantics
% 18N45 Categories of fibrations, relations to K-theory, relations to type theory
% 18N40 Homotopical algebra, Quillen model categories, derivators
% 55U35 Abstract and axiomatic homotopy theory in algebraic topology
% 55U40 Topological categories, foundations of homotopy theory

        % table of contents
        \setcounter{tocdepth}{2}
        \tableofcontents
        
	%Introduction
	
	\section{Introduction}

It is a well-known result in category theory (see for example \cite{freyd76}, \cite{blanc78}) that any property of a category, or of objects and morphisms of this category, that does not use equality between objects is automatically invariant both under equivalence of categories, and under substitution of all the objects and morphisms involved by isomorphic ones consistently.

For example, because the notion of limit in a category is naturally formulated without using equality between objects we automatically know that equivalences of categories preserve limits, or that if two diagrams are naturally isomorphic then a limit for one is also a limit for the other.

To be a little more precise, the above-mentioned results are about first-order formulas in which we can have quantifiers over all objects of the category, or over all morphisms in a given hom-set ``$\hom(X,Y)$''. We can use equality between two terms taken from the same $\hom(X,Y)$, but not between two terms of type ``objects'', or two terms that are in different hom-sets.

For example, the property of an object $X$ to be a terminal object, which can be written as
\[\texttt{isTerminal}(X) \coloneqq \forall Y \in \text{Ob}, \left( \exists v \in \Hom(Y,X) \text{ and } \forall u,w \in \Hom(Y,X), u = w \right)\]
is an instance of such a formula, but the following formula
\begin{multline*} \forall X,Y \in \text{Ob}, \forall f \in \Hom(X,Y), \forall g \in \Hom(Y,X), \\ \left( f \circ g = \idx_Y \text{ and } g \circ f = \idx_X \Rightarrow X = Y \right)
 \end{multline*}
 which states that the category we are working with is skeletal, or the formula
 \begin{multline*}
   \forall X,Y \in \text{Ob}, \forall f \in \Hom(X,Y), \forall g \in \Hom(Y,X), \\ \left( f \circ g = \idx_Y \text{ and } g \circ f = \idx_X \Rightarrow f = \idx_X \right)
 \end{multline*}
 which expresses that identities are the only isomorphisms, are not of this form: the first one involves the equality $X = Y$, and the second one involves an equality $f = \idx_X$ that is not correctly typed as $f \in \Hom(X,Y)$. And these two formulas are indeed not invariant under equivalence of categories\footnote{As they are formulas with no free parameters, invariance under substitution by isomorphic objects does not really make sense.}.

 \bigskip
 
    Note that in order for this to make sense, it is key to use a notion of ``dependent types''. Indeed, we need to be able to formulate the idea that a morphism $f$ is in $\Hom(X,Y)$, without being able to say that $s(f) =X $ and $t(f) = Y$ as this would involve using equality between objects. So, given two objects $X$ and $Y$, we need to be able to consider the type of arrows from $X$ to $Y$ as a primitive notion.

    \bigskip

    Now, it is natural to expect that similar results can be generalized to higher categories. For example, we expect (and it can be shown) that a property of $2$-categories or bicategories that does not use equality between objects or between $1$-arrows will also be invariant under biequivalences. One can also expect it can be generalized to other sorts of higher structures, for example a result about multicategories not using equality between objects should also have similar invariance properties.

    \bigskip
    
    The main goal of this paper is, informally, to establish a version of this result for essentially any kind of higher structure independently of the type of structure or the ``categoricity level''. The only requirement is that the sort of higher structure we are considering must be organized as the fibrant objects of a model category (or semi-model category, or weak model category).

    That is, we will attach to every (semi/weak) model category a ``first-order language'', whose formulas are statements about objects of the category (possibly with parameters) such that

    \begin{itemize}
    \item Replacing the value of the parameters by homotopically equivalent parameters does not change the validity of a formula.
    \item Two weakly equivalent fibrant objects satisfy the same formulas.
    \end{itemize}

    We call these two results respectively the $1^{st}$ and $2^{nd}$ invariance theorem, and their precise statement is given as \cref{invariance-theorems}. We will now go into a little more detail about how this language is defined, and explain the role of the different sections of the paper.
    
    \bigskip

    As mentioned above, our language is based first on dependent types. More precisely, we use the formalism of ``Generalised algebraic theory'' in the sense of Cartmell (\cite{cartmell1978}) as our basis, which are algebraic theories with dependent types.  If we compare our approach to traditional model theory, our choice of a generalized algebraic theory $T$ plays a role similar to the choice of a signature. However, contrary to traditional model theory, it is crucial for us that the theory $T$ (\ie our signature) can be any generalized algebraic theory, in particular the theory $T$ can include equality axioms. This is in part because the first-order logic we will introduce on top of it will not have equality, so algebraic equations cannot be treated as axioms like any other.

\subsection*{Overview}
    
    Starting from a generalized algebraic theory $T$, we build in \cref{sec:syntactic-approach} the first-order language $\Lcal^T$, as well as its quotient $\Lb^T$ where ``provably equivalent formulas'' (for a relatively weak notion of proof) are identified.

    The idea is that for each formula, the (free) variables are taken from a context of the theory $T$, and there can be no equality at all. In particular, the theory $T$ itself can have axioms that are not part of this first order language $\Lcal^T$. We will see through examples how in some cases, some notion of equality, for example the case of equality between morphisms in the same $\Hom(X,Y)$ in the example of categories we started from, can be recovered indirectly using certain equality axioms in the theory $T$ itself.

    \bigskip
    
    Since we want to be able to do infinitary logic, we use everywhere an infinitary generalization of the notion of generalized algebraic theory that is introduced in \cref{appendix-a}. However, a reader familiar with generalized algebraic theories can probably guess how it works. The logic $\Lcal^T$ we introduce can include arbitrary disjunction and conjunction, as well as quantifiers ranging on infinitely many variables. We will denote by $\Lcal^T_{\lambda,\kappa}$ the language where the formulas only include disjunction and conjunction on less than $\lambda$ subformulas and where a quantifier quantifies on less than $\kappa$ many variables at the same time. The $\kappa$ is very often omitted from the notation for technical reasons, but see \cref{cstr:L_lambda-kappa}.

    In \cref{sec:models_of_clans} we review quickly some important properties of the category of models of a generalized algebraic theory, most notably their canonical weak factorization system.
    In \cref{sec:category_approach} we explain how the language defined in \cref{sec:syntactic-approach} can be given an alternative categorical definition that can be applied to any ``clan'' --- clans are a notion of a categories with a class of fibrations --- for any generalized algebraic theory $T$, the syntactic, or contextual, category $\C_T$ is a clan. And we show that the category-theoretic definition of the language of the clan is equivalent to the syntactic definition of the language of any such generalized algebraic theory. Note that every clan can be shown to be the syntactic category of a generalized algebraic theory (and we prove more generally that in our infinitary setting any ``$\kappa$-clan'' is the syntactic category of a generalized $\kappa$-algebraic theory, this is in \cref{appendix-b}) so that the language can still be seen syntactically as the ``first-order language'' of some generalized algebraic theory, but we now also have the option of working ``categorically'' with it without relying on a choice of a syntax.

    This reinterpretation in terms of clans is the key to associate a language to any model category: Given a (weak) model category $\catM$ we take the category $\catM^\cof$ of cofibrant objects and cofibration between them. This category constitutes a co-clan (the opposite of a clan) and we can take the language associated to it. This is what we call the language of the model category $\catM$. We review briefly the general theory of weak model categories in \cref{sec:wms_intro} and in \cref{sec:wms_language} we explain in detail how this language of $\catM$ actually talks about the objects of $\catM$ and prove the first two invariance theorems mentioned above.

    To give a general picture of how this language works, if $\catM$ is our model category, each formula in the language has a ``context'' $C$, which informally can be thought of as the list of free variables that can appear in the formula as well as their types. This ``context'' $C$ is concretely just a cofibrant object of $\catM$. An interpretation of the context $C$ into an object $X \in \catM$ is just a map $v:C \to X$. And given $\phi$ a formula in context $C$ and $v:C \to X$ a map, $\phi(v)$ can be either true or false. We write
    \[M \vdash \phi(v) \]
    if $\phi(v)$ is true.

    \Cref{fol-gat} ends with our first two invariance theorems, stated as \cref{invariance-theorems}:

    \begin{introtheorem}[$1^{st}$ Invariance]
      If $X$ is fibrant and $v:C \to X$ is homotopic to $v':C \to X$ then $M \vdash \phi(v) \Leftrightarrow M \vdash \phi(v')$.
    \end{introtheorem}

    \begin{introtheorem}[$2^{nd}$ Invariance]
      If $F:X \to Y$ is a weak equivalence between fibrant objects, then $X \vdash \phi(v) \Leftrightarrow Y \vdash \phi(f(v))$.
    \end{introtheorem}

    To give a more concrete example of all this, when $\catM$ is the canonical or folk model structure on categories, our construction recovers the language of categories as in  \cite{freyd76} or \cite{blanc78}. Now, the formula
    \[\forall Z \in \text{Ob}, \forall g,h \in \Hom(Y, Z), g \circ f = h \circ f \Rightarrow g =h \]
    is a formula in context $X,Y \in \text{Ob}, f \in \Hom(X,Y)$ which corresponds to the (cofibrant) category $\catC$ which has two objects $X$ and $Y$ and a unique non-identity arrow $f:X \to Y$. A map from $\catC$ to another category $\catD$ is the choice of an arrow $f$ in $\catD$ and $\phi(f)$ is true if and only if  $f$ is an epimorphism. The second invariance theorem says (in this special case) that equivalence of categories preserves epimorphisms, and the first invariance theorem that if $f$ is isomorphic to another arrow then one is an epimorphism if and only if the other is.

    \bigskip

    In \cref{sec:examples} we show how our notion of language specializes to many classical model structures. We also discuss briefly some general (but informal) tools to construct this language explicitly for any model structure.

    \bigskip
    
    Finally, in \cref{sec:invariance} we prove two more invariance theorems (\cref{invariance-theorems-34}), that are this time about the expressive power of the language and can be stated informally as:

    \begin{introtheorem}[$3^{rd}$ Invariance]
      If $A$ and $B$ are two cofibrant objects of $\catM$, then each formula in context $A$ can be translated into a formula in context $B$ that is ``equivalent'' in the sense that its interpretation in any fibrant object is the same.
    \end{introtheorem}

    \begin{introtheorem}[$4^{th}$ Invariance]
      If $\catM$ and $\catN$ are two Quillen equivalent weak model categories, then any formula in the language of $\catM$ can be similarly translated into an equivalent formula in the language of $\catN$.
    \end{introtheorem}

     More details on these will be given in the introduction to \cref{sec:invariance}.

     \bigskip

    One should also mention that, despite the paper being stated in the language of ``weak'' model categories, all our examples are actual Quillen model categories, and the reader can replace weak model categories by Quillen model categories almost everywhere. The only reason for which we consider weak model categories is because the extra generality doesn't affect any of our results, and also because at some point in the proof of the second half of \cref{invariance-theorems-34} we need to use our construction of a language to something that in general will not be a full Quillen model category (even if we only try to prove \cref{invariance-theorems-34} for Quillen model categories). The main difference between weak model categories and Quillen model categories is that many results (and axioms) of a Quillen model category can only be applied to arrows from cofibrant to fibrant objects in a weak model category. A review of the notion of weak model category is in \cref{sec:wms_intro}.

    Notably, we will use the terminology ``\emph{core cofibration}'' to mean cofibration between cofibrant objects and ``\emph{core fibration}'' to mean fibration between fibrant objects.
     
     \bigskip

    The paper has three appendices that serve to review or introduce basic material. They can either be read first, or skipped entirely: \Cref{appendix-a} reviews Cartmell's notion of generalized algebraic theory, and generalizes it to the infinitary case. The goal of \cref{appendix-b} is to establish the link between generalized $\kappa$-algebraic theory and a notion of $\kappa$-clan, with a notion of $\kappa$-contextual category as an intermediate. This result is absolutely crucial for the paper, but is a very expected generalization of what happens in the finitary case. Finally, \cref{appendix-c} reviews some material on weak model categories and introduces a notion of Reedy model categories in that context, which is only used in \cref{sec:invariance}.

    \subsection*{Further remarks}

    We finish by mentioning that this work is closely related to Makkai's notion of ``First-order logic with dependent sorts'' or FOLDS from \cite{makkai95}. In a sense, Makkai's FOLDS corresponds to the special case where $T$ is the theory of presheaves on a direct category $I$, encoded using dependent type axioms only, with an additional equality predicate for the types corresponding to maximal objects of $I$. Because Makkai does not make assumptions about the existence of a model structure he only establishes an invariance theorem for what he calls ``very surjective maps'' (our ``anodyne fibrations''), that is the analogous to our \cref{cor:Invariance_triv_fib_clan}, more general notions of equivalence and homotopy are not clearly available in his setting.
    
    In conclusion, the present work is at the same time considering a more general algebraic setting (by allowing terms and type in $T$), but also is restricting the setting by assuming the presence of a model structure that gives a good homotopy theory to be invariant under, and allows obtaining much more interesting results. This seems to make our approach considerably more usable in practice, given the richness of examples it potentially covers.

    It should be noted however that there are some results in \cite{makkai95} that we have not yet been able to generalize to this new setting: Makkai established several results that essentially say that any formula that has the desired invariance properties is equivalent to one in the language introduced. Similar results are also given in \cite{freyd76} and \cite{blanc78}, and this paper contains no analogue to these results.

    \subsection*{Acknowledgment}
This work was supported by the Natural Sciences and Engineering Research Council of Canada (NSERC), funding reference number RGPIN-2020-067 awarded to Simon Henry.

%%% Local Variables:
%%% mode: latex
%%% TeX-master: "main"
%%% End:

	% Syntactic language
	
	\section{The homotopy invariant language} \label{fol-gat}

\subsection{Syntactic approach: The first-order language of a generalized algebraic theory}
\label{sec:syntactic-approach}
In this section, we give a very classic syntactical approach to the language we consider in this paper. We start from a generalized algebraic theory, and we build its first-order language on top of it.

Since we aim to do infinitary logic, we enhance Cartmell's notion of generalized algebraic theory to what we call \emph{generalized $\kappa$-algebraic theory} for $\kappa$ a regular cardinal, which we develop in detail in \cref{appendix-a}. Nevertheless, this generalization is straightforward and a reader familiar with Cartmell's formalism should be able to guess how it works and read this section directly. The main difference to keep in mind is that our contexts are sequences of typed variables indexed by ordinals less than $\kappa$ instead of finite sequences. A consequence of this is that we need to use more heavily the ``generalized display maps'' that correspond to ``projections'' from a context $(x_i:X_i)_{i < \gamma}$ to $(x_i:X_i)_{i < \beta}$ for arbitrary $\beta < \gamma < \kappa$, where the classical theory uses the display maps that corresponds to projections that only forget the last variable.

\bigskip

In what follows, we fix $\kappa$, $\lambda$ two regular cardinals and $T$ a generalized $\kappa$-algebraic theory. We will define the first-order language of $T$ with $\lambda$-small conjunction and disjunction, denoted $\Lcal^T_{\lambda}$ or $\Lcal^T_{\lambda,\kappa}$.

\bigskip

More precisely, for each context $\Gamma$ of $T$, we will define a set $\Lcal^T_{\lambda}(\Gamma)$ of ``$T$-formulas in context $\Gamma$". Essentially, these are first-order formulas with $\lambda$-small conjunctions and disjunctions whose free variables are the variables of the context $\Gamma$, in particular, they have less than $\kappa$-variables.

\begin{definition}  \label{folgat}
	The sets $\Lcal^T_{\lambda}(\Gamma)$ of $T$-formulas in context $\Gamma$ are defined inductively using the following rules:
	
	\begin{enumerate}
        \item For each context $\Gamma$, the true formula $\top$ and false formula $\bot$ are in  $\Lcal^T_{\lambda}(\Gamma)$.
        \item If $\Phi \in \Lcal^T_{\lambda}(\Gamma)$ then $\neg \Phi \in \Lcal^T_{\lambda}(\Gamma)$.
	\item For each collection of formulas $\Phi_i \in \Lcal^T_{\lambda}(\Gamma)$, indexed by a $\lambda$-small set $I$, the conjunction and disjunction
		\[ \bigvee_{i \in I} \Phi_i \qquad \bigwedge_{i \in I} \Phi_i \]
		are in $\Lcal^T_{\lambda}(\Gamma)$.
		\item Given two ordinals $\gamma < \alpha < \kappa$: If $\Gamma' \equiv \{x_\beta : \Gamma_\beta\}_{\beta < \alpha}$ is a context of length $\alpha$, and $\Gamma \equiv \{x_\beta : \Gamma_\beta\}_{\beta < \gamma}$ is the subcontext of length $\gamma$, then for any formula $\Phi \in \Lcal^T_\lambda(\Gamma')$ we have formulas
		\[\exists \{x_\beta :\Gamma_\beta\}_{\gamma \leqslant \beta < \alpha } \Phi \qquad \forall \{x_\beta :\Gamma_\beta\}_{\gamma \leqslant \beta < \alpha } \Phi \]
		in $\Lcal^T_\lambda(\Gamma)$.
              \end{enumerate}

              The collection of all formulas $\{\Lcal^T_{\lambda}(\Gamma)\}_{\Gamma \in T}$ is what we call \emph{the language of } $T$. Often, we will simply refer to it by $\Lcal^T_{\lambda}$.
              
\end{definition}

\begin{remark}
	The key point in \cref{folgat} is that we are not including atomic formulas other than $\top$ and $\bot$. In particular, the language \emph{does not include any equality}. At this point it might be unclear how we get non-trivial formulae in this language as it seems that applying quantifiers, conjunction or disjunction to formulas that are either $\bot$  or $\top$ will never produce any formulas that are not immediately interpreted as $\bot$ or $\top$.  Or even, on how we might obtain formulas with free variables. The central idea is that free variables appear thanks to the fact we quantify over dependent types, that is, types in which free variables can appear. The following examples will demonstrate these phenomena.
\end{remark}

\begin{example}
  Let $Cat$ be the generalized $\omega$-algebraic theory of categories as introduced in \cref{gatcategories}. Then, in the context $(x : \obtyp)$ we can write the formula
  \[\phi(x) \coloneqq ( \forall y : \obtyp,\, \exists f: \homtyp(x,y), \top )\]
	which expresses that for any object $y$ there is an arrow from $x$ to $y$. This simply means that $x$ is a weakly initial object. Indeed, $\top$ is a formula in context $(x: \obtyp,\, y:\obtyp, f: \homtyp(x,y))$, so that $\exists f : \homtyp(x,y), \top$ is a formula in context $(x:\obtyp,\, y :\obtyp)$, and $\forall y : \obtyp, \exists f : \homtyp(x,y), \top$ is a formula in context $(x :\obtyp)$. 
\end{example}

The logic is still not strong enough to express many of the interesting category theoretic notions. For example, without any kind of equality predicate on morphisms there is no way to write down a formula for an initial object, or a limit. In the next example, we show how modifying the theory $Cat$ allows the reconvery of equality on morphisms:

\begin{example}\label{ex:cat=}
	We consider the theory $Cat_=$ obtained by adding to the theory $Cat$ the following:
	\[ x,y:\obtyp, f,g : \homtyp(x,y) \vdash \eqtyp(f,g) \type \]
	\[ x,y:\obtyp, f: \homtyp(x,y) \vdash r_f : \eqtyp(f,f) \]
	\[ x,y:\obtyp, f,g : \homtyp(x,y), a: \eqtyp(f,g) \vdash f \equiv g \]
	\[ x,y:\obtyp, f,g : \homtyp(x,y), a: \eqtyp(f,g) \vdash a \equiv r_f \]    
	One can easily see that a model of $Cat_=$ is just a category, with the type $\eqtyp(f,g)$ being empty if $f \neq g$ and $\{r_f\}$ if $f=g$.
	In this new theory, we can now form a formula ``$f =g$'' in context $(x,y : \obtyp, f,g : \homtyp(x,y))$ which is defined as	
	\[ (f = g) \coloneqq (\exists v : \eqtyp(f,g), \top). \]	
Therefore, in the language $\Lcal_\omega^{Cat_=}$ we can form formulas involving equality between parallel morphisms. Then, we recover the ``language of categories'' as studied in \cite{blanc78} and \cite{freyd76}. 
	For example, we can form the formula ``$x$ is initial'' in context $(x : \obtyp)$ as
	\[ \isinitial(x) \coloneqq \forall y : \obtyp, (\exists f: \homtyp(x,y)) \wedge (\forall f,g : \homtyp(x,y), f = g).\]		
      \end{example}

      \begin{construction}\label{cstr:f*} If $f: \Delta \to \Gamma$ is a context morphism and $\phi \in \Lcal^T_\lambda(\Gamma)$, then we can define its pullback $f^* \phi$. This pullback is obtained by substituting the free variables of $\phi$ by the components of $f$. Formally, this is defined inductively as:
        \begin{enumerate}
        \item $f^* \top \coloneqq \top$ and $f^* \bot \coloneqq \bot$.
        \item $f^*(\neg \Phi) \coloneqq \neg f^* \Phi$.
        \item $f^*\left(\bigvee_{i \in I} \Phi_i \right) \coloneqq \bigvee_{i \in I} f^*\Phi_i$ and $f^*\left(\bigwedge_{i \in I} \Phi_i \right) \coloneqq \bigwedge_{i \in I} f^*\Phi_i$.
        \item If $\Gamma' \equiv (\Gamma, x_1 \in X_1,\dots, x_\alpha \in X_ \alpha)$ then
          \[f^* \left( \exists (x_1 \in X_1,\dots,x_\alpha \in X_\alpha) \Phi  \right) \coloneqq \exists (x_1 \in f^*X_1,\dots,x_\alpha \in f^*X_\alpha) f^*\Phi, \]
          \[f^* \left( \forall (x_1 \in X_1,\dots,x_\alpha \in X_\alpha) \Phi  \right) \coloneqq  \forall (x_1 \in f^*X_1,\dots,x_\alpha \in f^*X_\alpha) f^*\Phi, \]
          where $f^* X_i$ denotes the pullback of types, obtained by substitution, that is, the types appearing in the canonical pullback of the generalized display map:
         \[ \begin{tikzcd}
            (\Delta, f^*X_1, \dots, f^* X_\alpha) \ar[r] \ar[d,->>] &  (\Gamma, X_1, \dots, X_\alpha) \ar[d,->>] \\
            \Delta \ar[r] & \Gamma.
          \end{tikzcd}\]
          
        \end{enumerate}
        
      \end{construction}

      \begin{definition}\label{def:vdash_relation} For each context $\Gamma$ in $T$ we define the relation $\vdash_\Gamma$ on $\Lcal^T_\lambda(\Gamma)$  as the smallest family of relations such that:

   \begin{enumerate}

  \item \label{def:vdash_relation:refl} $\vdash_\Gamma$ is a transitive and reflexive relation on $\Lcal^T_\lambda(\Gamma)$.

   \item\label{def:vdash_relation:bot_top} $\forall \Phi \in \Lcal^T_\lambda(\Gamma)$, $\Phi \vdash_{\Gamma} \top$ and $ \bot \vdash_\Gamma \Phi$.

   \item\label{def:vdash_relation:negation}  $\forall \Phi \in \Lcal^T_\lambda(\Gamma),\Phi \wedge \neg \Phi \vdash \bot$ and $\top \vdash \Phi \vee \neg \Phi$.

   \item\label{def:vdash_relation:vee_wedge} For any $\lambda$-small family $(\Phi_i)_{i \in I} \in \Lcal^T_\lambda(\Gamma)$ we have
\[\bigvee_{i \in I} \Phi_i \vdash_\Gamma \Psi \Leftrightarrow \forall i, \left( \Phi_i \vdash_\Gamma \Psi \right)\]
\[\Psi \vdash \bigwedge_{i \in I} \Phi_i \Leftrightarrow \forall i, \left( \Psi \vdash_\Gamma \Phi_i \right)\]

\item\label{def:vdash_relation:Quantifier} For $\Gamma'\equiv \left(\Gamma,\left\lbrace x_\beta: \Gamma'_\beta \right\rbrace_{\gamma \leqslant \beta < \alpha}\right)$ a context extension, with $p: \Gamma' \twoheadrightarrow \Gamma$ the corresponding generalized display map, $\Psi \in \Lcal^T_\lambda(\Gamma')$ and $\Phi \in \Lcal^T_\lambda(\Gamma)$ we have
\[  \exists \{x_\beta :\Gamma_\beta\}_{\gamma \leqslant \beta < \alpha } \Psi \vdash_\Gamma \Phi \Leftrightarrow \Psi \vdash_{\Gamma'} p^* \Phi,\]

\[ \Phi  \vdash_\Gamma \forall  \{x_\beta :\Gamma_\beta\}_{\gamma \leqslant \beta < \alpha } \Psi \Leftrightarrow  p^* \Phi  \vdash_{\Gamma'}  \Psi. \]

   \end{enumerate}

 \end{definition}

While we have not included the following in the definition, we can show that:

   \begin{proposition} \label{prop:vdash_structural_prop}
  If $f: \Delta \to \Gamma$ is a context morphism in $T$, and $\Phi \vdash_\Gamma \Psi$ then $f^* \Phi \vdash_\Delta f^* \Psi$.
   \end{proposition}
   \begin{proof}
We can show that if we define the relation $\Phi \vdash_\Gamma' \Delta$ to be ``For all $f:\Delta \to \Gamma$, we have $f^* \Phi \vdash_{\Delta} f^* \Psi$'' then it satisfies all the conditions from \cref{def:vdash_relation}. Which shows that $\vdash \Rightarrow \vdash'$ and hence concludes the proof. \end{proof}

% In \cref{appendix-b4-models} we define what we mean by a model for a generalized $\kappa$-algebraic theory $T$, we recall that definition here. We fix a set of sets $\catU$ and construct the $\kappa$-contextual category $\cfam_\kappa$ whose objects are functors $(\alpha+1)^\op \to \catU$ (for each $\alpha<\kappa$) subject to some conditions.

% \begin{definition}
%   A \emph{model} of a generalized $\kappa$-algebraic theory $T$ is simply a contextual functor $X:\C_T \to \cfam_\kappa$. We will usually write $X:T \to \cfam_\kappa$. 
% \end{definition}

In \cref{appendix-b4-models} we define a model for a generalized $\kappa$-algebraic theory $T$ is as a morphism of contextual categories $X:\C_T \to \cfam_\kappa$ where $\cfam_\kappa$ is a contextual categories of ``families of sets''. By \cref{rk:models} this turns out to be equivalent to the naive definition of models where for each dependent type we have a family of sets, for each term a function, and equation axioms give us equations. Importantly for us, for each model $X$ and context $\Gamma$, there is a set $X(\Gamma)$, an element of which is a choice of an interpretation of each variable of $\Gamma$ as an element of the corresponding set in $X$. These $X(\Gamma)$ forms a functor on the category of contexts of $T$.

In what follows, we will use notation as explained in \cref{rk:notation-models}.

\begin{construction}\label{cstr:validity_of_formula_syntactic}
Given a model $X$ of our theory $T$, $\Gamma$ a context, $x \in X(\Gamma)$ and $\Phi \in \Lcal^T_\lambda(\Gamma)$, we can interpret $\Phi(x)$ as a proposition \ie true or false in the obvious way by substituting the components of $x$ into $\phi$ and interpreting all the logic symbols in the usual way. Formally we have:

\begin{enumerate}
\item If $\Phi = \top$, then $\Phi(x)$ is true and if $\Phi = \bot$ then $\Phi(x)$ is false,
\item If $\Phi = \neg \Psi$, then $\Phi(x)$ is true if and only if $\Psi(x)$ is false,
\item If $\Phi = \bigvee \Phi_i$, then $\Phi(x)$ is true if and only if $\Phi_i(x)$ is true for some $i$,
\item If $\Phi = \bigwedge \Phi_i$, then $\Phi(x)$ is true if and only $\Phi_i(x)$ is true for all $i$,
\item If $\Phi = \exists \{x_\beta :\Gamma_\beta\}_{\gamma \leqslant \beta < \alpha } \Psi$ for $\Gamma'=\left(\Gamma,\left\lbrace x_\beta: \Gamma'_\beta \right\rbrace_{\gamma \leqslant \beta < \alpha}\right)$ a context extension, with $p: \Gamma' \twoheadrightarrow \Gamma$ the corresponding generalized display map, then $\Phi(x)$ is true if there exists a $y \in X(\Gamma')$ such that $p(y) = x$ and $\Psi(y)$,
\item If $\Phi = \forall \{x_\beta :\Gamma_\beta\}_{\gamma \leqslant \beta < \alpha } \Psi$ in the same situation as above, then $\Phi(x)$ is true if for any $y \in X(\Gamma')$ such that $p(y)= x$ we have $\Psi(y)$.
\end{enumerate}
\end{construction}

The following lemma is immediate by induction, the proof is left to the reader.

\begin{lemma}Let $X$ be a model of a generalized $\kappa$-algebraic theory $T$.
  \begin{enumerate}
  \item For $\Phi,\Psi \in \Lcal^T_\lambda(\Gamma)$ and $x \in X(\Gamma)$, then if $\Psi \vdash_\Gamma \Phi$ and $\Psi(x)$ then $\Phi(x)$.
  \item If $f: \Gamma\to \Delta$ is any context morphism and $\Phi = f^* \Psi$ and $x \in X(\Gamma)$ then $\Phi(x) \Leftrightarrow \Psi(f(x))$.
  \end{enumerate}  
\end{lemma}

\begin{definition}\label{def:Lb}
  We write $\Psi \dashv\vdash_\Gamma \Phi$ to mean both $\Psi \vdash_\Gamma \Phi$ and $\Phi \vdash_\Gamma \Psi$.  We denote by
\[ \Lb_{\lambda}^T(\Gamma) \coloneqq \Lcal_{\lambda}^T(\Gamma)/(\dashv\vdash_\Gamma) \]
  the quotient.
\end{definition}

Note that $(\dashv\vdash_\Gamma)$ is indeed an equivalence relation, as $\vdash_\Gamma$ is transitive and reflexive.

\begin{remark}\label{rk:logical_operation_compatible}
  It follows from \cref{prop:vdash_structural_prop} that for a context morphism $f:\Delta \to \Gamma$ the $f^*$ operation from $\Lcal^T_{\lambda}(\Gamma) \to \Lcal^T_{\lambda}(\Delta)$ is compatible with the relation $\dashv\vdash$, and hence it descends to an operation 
\[f^*: \Lb^T_{\lambda}(\Gamma) \to \Lb^T_{\lambda}(\Delta). \]
It is also easy to see from \cref{def:vdash_relation} that the relation $\vdash$ is compatible with all the logical operations on $\Lcal^T_{\lambda}$, that is $\neg,\bigvee,\bigwedge,\exists,\forall$ in the sense that for example, if $\Phi_i \vdash \Psi_i$ for all $i \in I$ then $\bigvee_{i \in I} \Phi_i \vdash \bigvee_{i \in I} \Psi_i$ and hence they all descend into operations on $\Lb^T_\lambda$.
\end{remark}

\begin{construction}\label{cstr:L_lambda-kappa} At the beginning of the section, we have briefly called the language $\Lcal^T_{\lambda,\kappa}$ before dropping the $\kappa$ from the notation, as it can be read from the fact that $T$ is a generalized $\kappa$-algebraic theory. However, we can consider $\Lcal^T_{\lambda,\kappa'}$ for any $\kappa' \geqslant \kappa$. Indeed, given $T$ a generalized $\kappa$-algebraic theory, we can define a generalized $\kappa'$-algebraic theory $T_{\kappa'}$ by taking a set of axioms for $T$ and seeing them as axioms for a generalized $\kappa'$-algebraic theory. A model of $T_{\kappa'}$ is the same as a model of $T$. We then define
\[\Lcal^T_{\lambda,\kappa'} \coloneqq \Lcal^{T_{\kappa'}}_{\lambda,\kappa'} = \Lcal^{T_{\kappa'}}_{\lambda},  \]
as well as its quotient
\[\Lb^T_{\lambda,\kappa'} \coloneqq \Lb^{T_{\kappa'}}_{\lambda,\kappa'} = \Lb^{T_{\kappa'}}_{\lambda}.  \]
\end{construction}

\begin{example} Let $\Sigma$ be a signature in the sense of traditional model theory, that is a set of formal symbols for types, functions and relations. Then we can consider the generalized algebraic theory $T_{\Sigma,=}$, which has one type in the empty context of each sort symbol $X$ in the signature. Each of these types have an equality predicate as the one constructed in \cref{ex:cat=}, a term for each function symbol, and for each relation symbol $R \subset X_1,\dots,X_n$ a type axiom
\[ x_1: X_1, \dots ,x_n : X_n \vdash R(x_1,\dots,x_n) \type\]
with the additional axiom
\[x_1: X_1, \dots ,x_n : X_n,t_1,t_2: R(x_1,\dots,x_n) \vdash t_1 = t_2.\]
Models of this theory are exactly $\Sigma$-structures, and elements of $\Lb^{T_{\Sigma,=}}_{\omega,\omega}$ are essentially the same as usual first-order formulas in this signature. Elements of  $\Lb^{T_{\Sigma,=}}_{\lambda,\kappa}$ correspond to infinitary first-order formulas using $\lambda$-small conjunction and disjunction and where $\exists$ and $\forall$ quantifiers can quantify over $\kappa$-small set of variables.  
\end{example}

\subsection{Categories of models and their weak factorization systems }
\label{sec:models_of_clans}

In this section and the next we will abstract the notion of the first-order language of a generalized algebraic theory in terms of its category of models, this will allow us to generalize this notion of language to an arbitrary category. To be more precise, we will abstract it terms in terms of the category of models together with a certain weak factorization system we will introduce in this section, and in the next section we will generalize this to an arbitrary category equipped with a weak factorization system. 

Recall (see \cref{appendix-b4-models}), that given a generalized $\kappa$-algebraic theory $T$, the category of models of $T$ is defined as the category of morphisms of contextual categories from the syntactic category of $T$ to a certain contextual category $\cfam_\kappa$ of families of sets. Though, see \cref{rk:models}, this is exactly equivalent to the naive definition of a model where for each type axiom we have a family of sets, for each term axiom we have an operation and for each equation axiom, the corresponding equality is satisfied. Note however, that in the presence of type equality axioms this means that certain equalities between sets need to be satisfied.

We also recall from \cref{cstr:representable_models} that for any context $\Gamma$, there is a ``representable'' model $\Gamma^*$ such that for any other model $M$,
\[ \Hom(\Gamma^*,M ) = M(\Gamma).\]
This forms a functor $\C_T^\op \to \Mod(T)$, which sends pullbacks along display maps to pushouts and limit of $\kappa$-small towers of display maps to colimits.

Now, it can be shown that any locally $\kappa$-presentable category is the category of models of a generalized $\kappa$-algebraic theory, but crucially, the category of models of $T$ comes with additional structure: 

\begin{definition} \label{wfs-models} Given a generalized $\kappa$-algebraic theory $T$, we consider the weak factorization system on the category $\Mod(T)$ which is cofibrantly generated by the maps
  \[ \Gamma^* \cofibration \Gamma'^*  \]
  where $\Gamma' \fibration  \Gamma$ is a (generalized) display map  in $\C_T$. The elements of the left class will be called \emph{cofibrations} and the elements of right class \emph{anodyne fibrations}.
\end{definition}

\begin{remark}
  In most of the paper, we will assume that the category $\Mod(T)$ of models of $T$ is equiped with a model structure (or at least weak model category.) whose trivial fibrations are these anodyne fibrations. However, we want to reserve the use of ``trivial fibration'' to the case where there is indeed a (weak) model category involved. 
\end{remark}

We also recall the closely related notion of models of clans: A $\kappa$-clan is

\begin{definition} A \emph{clan}, or \emph{$\omega$-clan}, is a category $\catC$ endowed with a class of maps called \emph{fibrations} such that:
  \begin{enumerate}
  \item $\catC$ has a terminal object $1$, and for every $X \in \catC$ the unique map $X \fibration 1$ is a fibration,
  \item Isomorphisms are fibrations, the composite of two fibrations is a fibration,
  \item Pullback of fibrations exist and are fibrations.
  \end{enumerate}

  For $\kappa$ a regular cardinal, a $\kappa$-clan is a clan which further satisfies:

  \begin{enumerate}
  \item[4] For any ordinal $\lambda < \kappa$, if $A_\bullet: \lambda^\op \to \catC$ is a diagram in which all the transition maps $A_\beta \fibration A_\alpha$ for $\alpha < \beta$ are fibrations, then the limits
    \[\Lim_{\alpha < \lambda} A_\alpha\]
    exist, and all the projection maps $\pi_\beta : \Lim_{\alpha < \lambda} A_\alpha \fibration A_\beta$ are fibrations. We refer to these as \emph{limits of $\kappa$-small chains of fibrations}. 
  \end{enumerate}

  A \emph{morphism of clans} is a functor that sends fibrations to fibrations, preserves the terminal object and pullbacks of fibrations. A \emph{morphism of $\kappa$-clans} is in addition required to preserve the limits of $\kappa$-small chains of fibrations.

  A \emph{model} of a $\kappa$-clan $\catC$ is a morphism of $\kappa$-clans $\catC \to \set$, where $\set$ has the $\kappa$-clan structure where every map is a fibration.
  
\end{definition}

\begin{remark} For a generalized $\kappa$-algebraic theory $T$, the syntactic category $\C_T$ is an example of a $\kappa$-clan, and we show in \cref{appendix-b} that every $\kappa$-clan is equivalent to such a syntactic category $\C_T$. As discussed in \cref{rk:ModelsOfT_vs_ModelsOfClans} and \cref{rk:Models_of_clam_mismatch}, models of a generalized algebraic theory $T$ are closely related to models of the $\kappa$-clan $\C_T$, but they are not the same thing in general. It can be shown they agree, in the case of theories without type equality axioms, but not in general. Replacing the notion of model of a theory by that of models of a clan everywhere in the paper has no consequences anywhere and the reader should feel free to do so. The cofibration/anodyne fibrations weak factorization on $\Mod(T)$ can be defined in the exact same way (using the Yoneda embedding) on the category $\Mod(\catC)$ of models of a clan.
\end{remark}

\begin{remark} \label{cofibrant-objects:wfs-models}
  In the special case $\kappa = \omega$, this weak factorization was defined in \cite[Definition~2.4.2]{henry2016algebraic} and extensively studied in \cite{frey2025duality} in the context of models of clans. In particular, Jonas Frey gave in \cite{frey2025duality} a complete characterization of which pairs of a category and a weak factorization system can be obtained in this way from an $\omega$-clan -- or equivalently from a generalized algebraic theory with no type equality axioms (see the discussion in \cref{rk:ModelsOfT_vs_ModelsOfClans} and \cref{rk:Models_of_clam_mismatch}). The methods used by Frey can be extended to the $\kappa$-case to obtain a similar characterization. Frey also shows that (in the $\kappa = \omega$ case) the $\omega$-presentable cofibrant object in $\Mod(\catC)$ are exactly the retracts of representable models. The same proof generalizes to the $\kappa$-case to show that if $\catC$ is a $\kappa$-clan, then $\kappa$-presentable cofibrant objects are exactly the retracts of representables. We only mention these results for context, we will not directly use them.\end{remark}

\begin{lemma}\label{lem:trivial_fibration_as_weak_pullback}
  Given a generalized $\kappa$-algebraic theory $T$, a morphism $f:M \to N$ of $T$-models is an anodyne fibration if and only if for every generalized display map $p:X \fibration Y$ in $\C_T$, the naturality square:
\[
  \begin{tikzcd}
    M(X) \ar[r] \ar[d] & M(Y) \ar[d] \\
    N(X) \ar[r] & N(Y) 
  \end{tikzcd}
\]
is a \emph{weak pullback square}, that is, if the induced map $M(X) \to N(X) \times_{N(Y)} M(Y)$ is a surjection.
\end{lemma}

\begin{proof}
  By the Yoneda lemma, there is a one-to-one correspondence between elements of $M(X)$ and morphisms of models $X^* \to M$. The map $M(X) \to M(Y)$ is obtained as the composite $Y^* \to X^* \to M$, and the map $M(X) \to N(X)$ as the composite $X^* \to M \to N$. An element of $N(X) \times_{N(Y)} M(Y)$ is hence the data of maps $X^* \to N$ and $Y^* \to M$ such that the composite $Y^* \to M \to N$ and $Y^* \to X^* \to N$ coincide. This is exactly a commutative square:
  \[
    \begin{tikzcd}
      Y^* \ar[r] \ar[d,"p^*"swap] & M \ar[d,"f"] \\
      X^* \ar[r] & N.
    \end{tikzcd}
  \]
  An element of $M(X)$ whose image in  $N(X) \times_{N(Y)} M(Y)$ is then exactly a dotted diagonal filling in the square above:
  \[
    \begin{tikzcd}
      Y^* \ar[r] \ar[d,"p^*"swap] & M \ar[d,"f"] \\
      X^* \ar[r] \ar[ur,dotted] & N.
    \end{tikzcd}
  \]
 Hence the surjectivity of this map is equivalent to the fact that $f$ has the right lifting property against $Y^* \to X^*$ for all fibrations $X \fibration Y$, which concludes the proof.
\end{proof}

\subsection{The Category theoretic approach: The first-order language of a $\kappa$-clans}
\label{sec:category_approach}

In this section we present another equivalent approach to the definition of the language, which is more categorical in spirit, and strongly inspired from Lawvere's theory of hyperdoctrines (\cite{lawvere1969adjointness}, \cite{lawvere1970equality}). This approach, while much more abstract, has several advantages over the syntactic one. Mainly, it allows working directly with the category $\Mod(T)$ of models equipped with the weak factorization system on the category of models constructed in the previous subsection, without referring to the theory $T$ at all, and to generalize it to an arbitrary category with a weak factorization system.  This will be useful later on to define the language of a model category without having to build explicitly a syntax for it.

As before, we fix $\lambda$ a regular cardinal. A $\lambda$-boolean algebra is a boolean algebra which admits joins (and hence intersections) of $\lambda$-small families. We denote by $\bool_\lambda$ the category whose objects are $\lambda$-boolean algebras and whose morphisms are boolean algebra morphisms preserving $\lambda$-small joins (and hence intersections).

We introduce the notion of $\lambda$-boolean algebra over a clan $\catC,$ which we can think of as an axiomatization of the structure that the $\Lb^T_\lambda$ from \cref{sec:syntactic-approach} have over the contextual category of $T$.

\begin{definition}\label{def:Boolan_alg_over_clan}
	Given $\catC$ a clan and $\lambda$ a regular cardinal, a $\lambda$-\emph{boolean algebra} over $\catC$ is a functor
	\[ \catB: \catC^{op} \to \bool_\lambda \]
	such that:
	\begin{enumerate}
		\item For each fibration $\pi:Z \twoheadrightarrow X$ in $\catC$, $\pi^*:\catB(X) \to \catB(Z)$ has a left adjoint:
		\[\exists_\pi : \catB(Z) \leftrightarrows \catB(X) : \pi^*. \]
		\item The Beck-Chevalley condition holds for each pullback square along a fibration. That is, given any pullback square:
		\[\begin{tikzcd}
			Z' \ar[r,"f'"] \ar[d,->>,"\pi'" swap] \ar[dr,phantom,"\lrcorner" very near start] &  Z \ar[d,->>,"\pi"] \\
			X' \ar[r,"f" swap] & X 
		\end{tikzcd}\]
with $\pi$ a fibration, we have $f^* \exists_{\pi} = \exists_{\pi'} f'^*$.
	\end{enumerate}
	Morphisms of $\lambda$-boolean algebras over $\catC$ are natural transformations that commute with the $\exists_\pi$. We call weak morphisms the natural transformations with no additional conditions.
\end{definition}

\begin{remark}
  If $\catB$ is a $\lambda$-boolean algebra over $\catC$, then for each $X \in \catC$, the negation $\neg : \catB(X) \to \catB(X)^{op}$ is a contravariant equivalence. Therefore, if $\pi:Z \to X$ is a fibration, then the map $\pi^*:\catB(X) \to \catB(Z)$ also has a right adjoint defined by:
\[\forall_\pi(\phi) \coloneqq \neg (\exists_\pi \neg \phi). \]
From this definition, we immediately have the other Beck-Chevalley condition $f^* (\forall_\pi) = \forall_\pi f^*$ and the fact that morphisms of boolean algebras over $\catC$ are also compatible with $\forall_\pi$, simply because $f^*$ is compatible with both $\exists_\pi$ and the negation.
\end{remark}

\begin{remark} \Cref{def:Boolan_alg_over_clan} will in practice be applied to $\catC$ a $\kappa$-clan (and not just a clan). The only reason it is stated like that is because the definition actually does not explicitly involve $\kappa$. This is related to the fact that the dependencies in $\kappa$ of the language defined in the previous subsection are only through the choice of which context can our variables (including bound variables) be taken from: taking a larger $\kappa$ means we can quantify over more variables at the same time.  Similarly, the dependency on $\kappa$ is hidden in the dependency on $\catC$, as $\catC$ is playing the role of the category of $\kappa$-contexts.
\end{remark}

Let us start with our main example of such a boolean algebra over a clan, which is the motivating example for the notion:

\begin{theorem}\label{th:Lb_is_initial}
  Let $T$ be a generalized $\kappa$-algebraic theory and $\catC_T$ the corresponding $\kappa$-contextual category, seen as a clan. Then the construction $X \mapsto \Lb^T_\lambda(X)$ from \cref{def:Lb} (see also \cref{folgat} and \ref{def:vdash_relation}) is a $\lambda$-boolean algebra over $\catC_T$. In fact, it is an initial object in the category of $\lambda$-boolean algebras over $\catC_T$.
\end{theorem}
\begin{proof}
We first check that $\Lcal^T_\lambda$ is a $\lambda$-boolean algebra over $\catC_T$. We have mentioned in \cref{rk:logical_operation_compatible} that all the logical operations $\vee,\wedge,\neg,\exists$ and so on are compatible with the equivalence relation $\dashv \vdash$. Therefore, they all induce operations on the quotient $\Lb^T_\lambda$. The first four points of \cref{def:vdash_relation} immediately show that each $\Lb^T_\lambda(X)$ is a boolean algebra whose order relation is given by $\vdash$, and with $\lambda$-small unions. By \cref{cstr:f*}, the map $f^* : \Lcal^T_\lambda(X) \to \Lcal^T_\lambda(Y)$ is compatible with all the logical operations, so it gives rise to a morphism of boolean algebras $\Lb^T_\lambda(X) \to \Lb^T_\lambda(Y)$. We get a functor $\catC_T \to \bool_\lambda$, the conditions $(g\circ f)^*(\phi) = f^* g^*(\phi)$ and $id^*(\phi) = \phi$ follow immediately by induction. Next, the last two conditions of \cref{def:vdash_relation} show that $\exists $ and $\forall$ define left and right adjoints to $\pi^*$. Finally, the Beck-Chevalley condition follows from how $f^*$ is defined on formulas starting with a $\exists$ quantifier:
\[f^* \left( \exists \{x_\beta :\Gamma_\beta\}_{\gamma \leqslant \beta < \alpha } \Phi \right) =  \exists \{x_\beta :f^*\Gamma_\beta\}_{\gamma \leqslant \beta < \alpha } f^* \Phi, \]
which (after passing to the quotient $\Lcal \to \Lb$) exactly says that $f^* \exists_\pi = \exists \pi f^*$ where $\pi$ is the generalized display map corresponding to forgetting the variables $\{x_\beta \}_{\gamma \leqslant \beta < \alpha} \in X_\alpha$.

We now check that it is an initial object in the category of $\lambda$-boolean algebras over $\catC_T$. Let $\catB$ be any $\lambda$-boolean algebra over $\catC$. Any morphism $v:\Lb^T_\lambda \to \catB$ has to satisfy:
\begin{enumerate}
\item $v(\bot) = \bot_\catB$ and $v(\top) = \top_\catB$.
\item $v(\neg \Phi) = \neg v(\Phi)$.
\item $\displaystyle v(\bigvee_{i \in I} \Phi_i) = \bigvee_{i \in I} v(\Phi_i)$ and $\displaystyle v(\bigwedge_{i \in I} \Phi_i) = \bigwedge_{i \in I} v(\Phi_i)$.
\item \[ v( \exists \{x_\beta :\Gamma_\beta\}_{\gamma \leqslant \beta < \alpha } \Phi)  = \exists \{x_\beta :\Gamma_\beta\}_{\gamma \leqslant \beta < \alpha } v(\Phi)\] and  \[v( \forall \{x_\beta :\Gamma_\beta\}_{\gamma \leqslant \beta < \alpha } \Phi)  = \forall \{x_\beta :\Gamma_\beta\}_{\gamma \leqslant \beta < \alpha } v(\Phi).\]
\end{enumerate}
These form an inductive definition for a function $\Lcal^T_\lambda \to \catB$. So there is a unique such function $v:\Lcal^T_\lambda \to \catB$. To conclude, we only need to check that this function $v$ descends to a function $\Lb^T_\lambda \to \catB$ and is a morphism of $\lambda$-boolean algebras over $\catC$. But this is rather immediate: We first observe, by induction over \cref{def:vdash_relation}, that if $\Phi \vdash \Psi$ then $v(\Phi) \leqslant v(\Psi)$. This implies that if $\Phi \dashv \vdash \Psi$ then $v(\Phi) = v(\Psi)$, so $v$ does define a function $\Lb^T_\lambda \to \catB$. The naturality condition
\[ v(f^*( \Phi)) = f^*(v(\Phi)) \]
can be proved by induction on the formula $\Phi$, and the compatibility of $v$ with all the boolean algebra operations and the quantifiers follows immediately from the definition of $v$.
\end{proof}

\begin{proposition} Given any (small) clan $\catC$ and $\lambda$ a regular cardinal, there is an initial $\lambda$-boolean algebra over $\catC$, which we denote by $\Lb^\catC_{\lambda}$. \end{proposition}

Note that by \cref{th:Lb_is_initial}, if $T$ is a generalized $\kappa$-algebraic theory, with $\catC_T$ its $\kappa$-contextual category, then
\[ \Lb^{\catC_T}_\lambda = \Lb^T_\lambda. \]
This provides a way to define (or at least to characterize) the first-order language of any clan without having to explicitly give a syntactic description of the clan.

\begin{proof}
  We can either remark that the $\lambda$-boolean algebras over $\catC$ are (by their definition) the models of a multi-sorted $\lambda$-algebraic theory (with one sort for each object $c \in \catC$) and hence there is an initial object by usual results on algebraic theories. Alternatively, we can use (see \cref{appendix-c}) that every clan is equivalent to the contextual category of a generalized algebraic theory and use \cref{th:Lb_is_initial} to conclude.
\end{proof}

Next, we mention a few more examples:

\begin{example}\label{ex:Power_Set_boolean_algebra}
	
	\begin{enumerate}
		\item[]	
		
		\item Let $\set$ be the category of sets, considered as a clan where every arrow is a fibration. The contravariant power-set functor $\Pcal: \set^{op} \rightarrow \bool_{\lambda}$ is a $\lambda$-Boolean algebra over $\set$. The Beck-Chevalley condition follows from \cref{lem:BC_Set} below.
		
		\item Given $F: \catC \rightarrow \catD$ a morphism of clans, if $\catB$ is a $\lambda$-boolean algebra over $\catD$, then $F^* \catB$ defined by $F^* \catB (\Gamma)= \catB (F(\Gamma))$ is a $\lambda$-boolean algebra over $\catC$. 
		
		\item Combining the two observations above, given any model $M$ of a clan $\catC$, that is, a morphism of clans $M: \catC \to \set$, one has a boolean algebra $\Pcal(M)$ over $\catC$ given by pulling back example $1$ along the morphism $M : \catC \rightarrow \set$. More explicitly:
		\[\begin{array}{c c c c}
			\Pcal(M) : & \catC^{op} &\rightarrow &\set \\
			& \Gamma & \mapsto & \Pcal(M(\Gamma)).
		\end{array} \]
		
	\end{enumerate}
	
\end{example}

\begin{lemma} \label{lem:BC_Set} Given a square of sets,	
	\[
	\begin{tikzcd}
		W \ar[r,"f"] \ar[d,"g"] & X \ar[d,"h"] \\
		Y \ar[r,"k"] & Z,
	\end{tikzcd}
	\]
	then the power set functor satisfies the Beck-Chevalley condition on this square, \ie $k^* \exists_h =\exists_g f^* $ as maps $\Pcal(X) \to \Pcal(Y)$ if and only if the square is a weak pullback square \ie if and only if the cartesian gap map $W \rightarrow Y \times_Z X$ is surjective. 
\end{lemma}

\begin{proof}
	Given a subset $P \subset X$ one has:
	\[ k^* h_! P = \{ y \in Y | k(y) = h(p) \text{ for some $p \in P$}  \}, \]
	
	\[ g_! f^* P = \{ g(w) | f(w) \in P  \}. \]
	Surjectivity of the map $W \rightarrow Y \times_Z X$ gives a canonical way to make any element of $k^* h_! P$ into an element of $g_! f^* P$, and conversely, applying the equality to $P=\{p\}$ produces the surjectivity of $W \rightarrow Y \times_Z X$.
\end{proof}

In this new setting with just a clan $\catC$, one can still define the set of formulas $\Lb^\catC_ \lambda$ as the initial $\lambda$-boolean algebra over $\catC$. We now explain what it means for formulas defined in this way to be ``true'' or ``false'' given a model and an interpretation of its variables in the model.

\begin{construction}\label{cstr:validity_formulas_categorical}
   Given a clan $\catC$ and a model of $M:\catC \to \set $ we have, as explained in \cref{ex:Power_Set_boolean_algebra}, a $\lambda$-boolean algebra over $\catC$ defined by $c \mapsto \Pcal(M(c))$. By initiality of the $\lambda$-boolean algebra $\Lb^\catC_\lambda$, there exists a unique morphism of $\lambda$-boolean algebras over $\catC$:
   \[ |-|_M: \Lb^\catC_\lambda \to \Pcal(M).  \]
   This morphism associates each formula $\phi$ in context $\Gamma$ to a subset $|\phi|_M \subseteq M(\Gamma)$. An element $x\in M(\Gamma)$ is said to \emph{satisfy} $\phi$ if $x\in |\phi|_M$. With some abuse of notation, we say that ``$\phi(x)$ is true'' in this case. We also write
\[ M \vdash \phi(x) \]
when we want to insist on which model we are talking about.
 When $\Gamma$ is the terminal object of $\catC$ \ie $\phi$ is a closed formula, then $M(\Gamma)=\{*\}$. Therefore, $\Pcal(M(\Gamma))=\{\bot,\top\}$ so that $|\phi|_M$ is simply a proposition. One then says that $M$ satisfies $\phi$, and we write $M \vdash \phi$.
\end{construction}

\begin{lemma}
  When $\catC = \catC_T$ is the $\kappa$-contextual category of a $\kappa$-generalized algebraic theory, then through the identification $\Lb^T_\lambda = \Lb^\catC_\lambda$, the two definitions of validity of a formula on elements of a model given by \cref{cstr:validity_of_formula_syntactic} and \cref{cstr:validity_formulas_categorical} are equivalent.
\end{lemma}

\begin{proof}
  Defining the validity of formulas as in \cref{cstr:validity_formulas_categorical} it is immediate to verify all the explicit conditions of the inductive definition given in \cref{cstr:validity_of_formula_syntactic} simply because the map $\Lb^\catC_\lambda \to \Pcal(M)$ is a morphism of $\lambda$-boolean algebras. Hence, it immediately follows by induction on formulas that the two definitions are equivalent.
\end{proof}

\begin{construction} \label{action-functor:formulas}
  Let $F: \catC \to \catD$ be a morphism of clans. And let $\Lb^\catC_\lambda$ and $\Lb^\catD_\lambda$ be their respective initial $\lambda$-boolean algebras. From the fact that $\Lb^\catC_\lambda$ is initial, there is a morphism of $\lambda$-boolean algebras
\[ \alpha^F: \Lb^\catC_\lambda \to F^*\left( \Lb^\catD_\lambda \right). \]
For any $\Gamma \in \catC$ and any formula $\Phi \in \Lb^\catC_\lambda(\Gamma)$ we denote $F(\Phi) \coloneqq \alpha^F_\Gamma(\Phi) $ which is a formula in context $F(\Gamma)$ \ie an element of $\Lb^\catD_\lambda(F(\Gamma))$. The following is immediate from the definition above:
\end{construction}

\begin{proposition} \label{pullback-model:prop}
  Let $M:\catD \to \set$ a model of the clan $\catD$, $\Phi \in \Lb^\catC_\lambda(\Gamma)$ a formula in context $\Gamma$ and $x \in M(F(\Gamma))$. Then, $M \vdash \alpha_F(\Phi)(x)$ if and only if $ F^*M \vdash \Phi(x) $.
\end{proposition}
Of course this also applies to models of a generalized $\kappa$-algebraic theory.

Finally, we finish this section by showing the key property of invariance of formulas along anodyne fibrations. An invariance property will be established in the next section assuming we are working with a model category, but this first invariance property is purely algebraic. This is also the key observation in Makkai FOLDS \cite{makkai95} and it is directly inspired from it.

We start with the following observation: let $\catC$ be a clan and $f:M\to N$ a morphism of two $\catC$-models, then we have an obvious map $f^*:\Pcal (N)\to \Pcal(M)$ which sends a subset $A \subset N(c)$ for $c \in \catC$ to
\[f_c^{-1}(A) \subset M(c)\]
this map is easily seen to be a \emph{weak} morphism of boolean algebras over $\catC$. It is compatible with the boolean algebra operations and the ordinary contravariant functoriality, but it does not have to be compatible with the covariant functoriality $\exists_\pi$ along fibrations.  However, one has:

\begin{lemma}\label{morphism-bol-alg-iff-trivial-fib}
Let $\catC$ be a clan and let $f:M\to N$ be a morphism between two $\catC$-models. Then $f$ is an anodyne fibration if and only if $f^*:\Pcal(N)\to \Pcal(M)$ is a morphism of $\lambda$-boolean algebras. 
\end{lemma}
\begin{proof}
 We only need to show that for every fibration $p:X \to Y$ the following square
 \[
 \begin{tikzcd}
  \Pcal(N(X)) \ar[r,"f_X^*"] \ar[d,"\exists"] & \Pcal(M(X)) \ar[d,"\exists"] \\
  \Pcal(N(Y)) \ar[r,"f_Y^*"] & \Pcal(M(Y)).
 \end{tikzcd}
 \]
commutes. From \cref{lem:BC_Set} this is equivalent to saying that the dotted map in
 \[
 \begin{tikzcd}
  M(X) \ar[rrd,"f_X",bend left] \ar[ddr,"\pi_*",bend right] \ar[rd, dashrightarrow] & & \\
  & P \ar[r] \ar[d] & N(X) \ar[d,"\pi_*"] \\
  & M(Y) \ar[r,"f_Y"] & N(Y)
 \end{tikzcd}
 \]
 is surjective. But this is exactly the characterization of anodyne fibrations given in \cref{lem:trivial_fibration_as_weak_pullback}.
\end{proof}

This allows us to deduce the key result of invariance of formulas along anodyne fibrations of models. Basically, the validity of formulas is preserved by anodyne fibrations of models:
\begin{corollary} \label{cor:Invariance_triv_fib_clan}
  Let $\catC$ be a clan and let $f:M \fibration N$ be an anodyne fibration between two $\catC$-models. For $c \in \catC$, let $x \in M(c)$ and $\phi \in \Lb_\lambda^\catC$ be any formula. Then
\[ M \vdash \phi(x) \Leftrightarrow N \vdash \phi(f(x)) \]
\end{corollary}

\begin{proof}
  As $f: M \to N$ is an anodyne fibration, it follows from \cref{morphism-bol-alg-iff-trivial-fib} that the map $f^*:\mathcal{P}(N) \to \mathcal{P}(M)$ is a morphism of boolean algebra over $\catC$. Hence, by initiality of $\Lb_\lambda^\catC$, the unique morphism $| \_ |_M : \Lb_\lambda^\catC \to \mathcal{P}(M)$ is obtained as a composite $$\Lb_\lambda^\catC \overset{|\_|_N}{\to} \mathcal{P}(N) \overset{f^*}{\to} \mathcal{P}(M).$$ By definition, $M \vdash \phi(x)$ means that $x \in |\phi|_M$ while $N \vdash \phi(f(x))$ means that $x \in f^* |\phi|_N$, hence the result immediately follows.
\end{proof}

\subsection{The language of a weak model category and two invariance theorems}
\label{sec:wms_language}

\begin{construction} \label{language-model:construction}
  Given $\catM$ a weak model category, the category $\catM^\cof$ of cofibrant objects with cofibrations between them forms a coclan. We define the language of $\catM$ to be the language of the coclan $\catM^\cof$. For any regular cardinal $\lambda$, we denote by $\Lb^\catM_\lambda$ the $\lambda$-boolean algebra $\Lb^{\catM^\cof}_\lambda$ over $\catM^\cof$. 

  Note that for each \emph{cofibrant} object $X \in \catM$, we have a set (or possibly a class if $\catM$ is large) of formulas $\Lb^\catM_\lambda(X)$.

\end{construction}

\begin{remark} There is a size issue to be mentioned here. In most practical examples, $\catM^\cof$ is a large category while the construction of $\Lb^{\catM^\cof}_\lambda$ developed in \cref{sec:category_approach} assumes it is a small category. We can deal with this by invoking a larger Grothendieck universe, but this has a practical consequence: The set of formulas $\Lb^\catM_\lambda(X)$ might not be a small set. Indeed, it lives in the same Grothendieck universe as the one in which $\catM^\cof$ is small.
\end{remark}

\begin{construction}\label{cstr:object_to_model_yoneda}
  If $X \in \catM$ then we can define a model of the coclan $\catM^\cof$ using the restricted Yoneda embedding:
  \[\yoneda_X:
    \begin{array}{ccc}
      (\catM^\cof)^\op & \to & \set \\
      c & \mapsto & \Hom(c,X),
    \end{array}
  \]
 which defines a functor $\yoneda: \catM \to \Mod(\catM^\cof)$.
\end{construction}

\begin{definition}\label{def:wms_language}
  Let $\catM$ be a weak model category. For $c \in \catM$ a cofibrant object, and $X \in \catM$ any object, $v:c \to X$ and  $\phi \in \Lb^\catM_\lambda(c)$ we write
  \[ X \vdash \phi(v) \]
  to mean
  \[\yoneda_X \vdash \phi(v) \]
  where $v$ is seen as an element of $\yoneda_X(c) = \Hom(c,X)$. 
\end{definition}

\begin{remark}
  In the special case where $\catM = \Mod(T)$ is the category of models of a generalized $\kappa$-algebraic theory (or more generally of a $\kappa$-coclan), then $\Lb^\catM_\lambda$ is the initial $\lambda$-boolean algebra over the coclan of all cofibrant objects of $\catM$, while the syntactic category of $T$ is equivalent to a full sub-$\kappa$-coclan of that. In particular, there is a morphism of $\lambda$-boolean algebras over the syntactic category $\catC_T$
  \[\Lb^T_\lambda(X) \to \Lb^\catM_\lambda(X) \qquad \text{(For $X \in \catC_T$).}\]
  If we denote this map by $i$ then for $X$ any model of $T$ we can easily check that
  \[ X \vdash \phi(v) \Leftrightarrow X \vdash i(\phi)(v) \]
for any $c \in \catC_T$ and $\phi \in \Lb^T_\lambda(c)$, where the left-hand side is interpreted in the sense of \cref{folgat} while the right-hand side is in terms of \cref{def:wms_language}.

  Note that we do expect these to be the same. Informally, $\Lb^T_\lambda$ corresponds to an $\Lcal_{\kappa,\lambda}$ logic, in the sense that quantifiers can only be applied to formulas in $\kappa$-small contexts --- applied to less than $\kappa$-many variables at the same time---while $\Lb^\catM_\lambda$ corresponds to an $\Lcal_{\infty,\lambda}$ logic, where quantifiers can be applied to arbitrarily many formulas at the same time.
  
\end{remark}

\begin{theorem} \label{invariance-theorems}

  Let $\catM$ be a weak model category, $c \in \catM$ a cofibrant object and $\phi \in \Lb^\catM_\lambda(c)$.

  \begin{itemize}
  \item \textbf{\First :} Let $v_1,v_2: c \to X$ be two homotopically equivalent maps with $X$ fibrant. Then 
      \[ X \vdash \phi(v_1) \quad \Leftrightarrow \quad X \vdash \phi(v_2). \]
    \item \textbf{\Second :} Let $f:X \to Y$ be a weak equivalence between two fibrant objects and $v:c \to X$ any map. Then
        \[ X \vdash \phi(v) \quad \Leftrightarrow \quad Y \vdash \phi(fv). \]
  \end{itemize}
  
\end{theorem}

\begin{proof}
  We start by first observing that the second invariance theorem in the special case where $f$ is a trivial fibration immediately follows from \cref{cor:Invariance_triv_fib_clan} as a trivial fibration $f$ has the right lifting property against all core cofibrations and hence is sent to an anodyne fibration in $\Mod(\catM^\cof)$ by the functor from \cref{cstr:object_to_model_yoneda}.

  We use this to prove the \first: If $v_1,v_2:c \to X$ are homotopic then there exists a map $h$:
\[\begin{tikzcd}
   & X \\
  c \ar[dr,"v_1"swap] \ar[ur,"v_2"] \ar[r,"h",dotted] & PX \ar[d,"p_1"] \ar[u,"p_2"swap]  \\
  & X.
\end{tikzcd}\]
The two maps $p_1,p_2: PX \to X$ are trivial fibrations (they are both fibrations and weak equivalences), $v_1 = p_1 \circ h$ and $v_2 = p_2 \circ h$. By the observation above, we have:
\[
  \begin{array}{cccc}
    & X &\vdash& \phi(v_1)  \\
    \Leftrightarrow & X &\vdash & \phi(p_1 h) \\
    \Leftrightarrow & PX &\vdash &\phi( h) \\
    \Leftrightarrow & X &\vdash &\phi(p_2 h) \\
    \Leftrightarrow & X &\vdash& \phi(v_2) \\
  \end{array}
  \]
  This concludes the proof of the \first.

  Next, we observe it is enough to prove the second invariance theorem when $X$ and $Y$ are both bifibrant. Indeed, starting from $f:X \to Y$ a weak equivalence between fibrant objects, $v:c \to X$ and $\phi \in \Lb^\catM_\lambda(c)$ as in the theorem. We can replace both $X$ and $Y$ by bifibrant objects
\[\begin{tikzcd}
    X^\cof \ar[r,"\sim","f'"swap] \ar[d,->>,"\sim"] & Y^\cof \ar[d,->>,"\sim"] \\
    X \ar[r,"\sim","f"swap] & Y.
  \end{tikzcd}\]
First replacing $X$ by a cofibrant object $X^\cof$ and then factoring the map $X^\cof \to Y$, which is a weak equivalence, as a trivial cofibration followed by a trivial fibration. The map $v:c \to X$, can be lifted to a map $v':c \to X^\cof$. As we can already apply the \Second\ to trivial fibrations, we have that:
\[ X \vdash \phi(v) \Leftrightarrow X^\cof \vdash \phi(v') \]
\[ Y \vdash \phi(fv) \Leftrightarrow Y^\cof \vdash \phi(f'v'). \]
Therefore, it is enough to show the \Second\ for bifibrant objects.

This last step is achieved essentially using a ``Brown factorization'': any weak equivalence between bifibrant objects can be factored as a section of a trivial fibration followed by a trivial fibration. Indeed, if $f:X \to Y$ is a map between bifibrant objects we can form the pullbacks:
\[\begin{tikzcd}
  X \ar[r,"f"] \ar[dr,phantom,"\lrcorner"{description,very near start}] \ar[d,"e'"]  & Y \ar[d,"e"] \\
 X \times_Y PY \ar[d,->>] \ar[r] \ar[dr,phantom,"\lrcorner"{description,very near start}] & PY \ar[d,->>] \\
 X \times Y \ar[r] \ar[d,"\pi_1",->>] \ar[dr,phantom,"\lrcorner"{description,very near start}] & Y \times Y \ar[d,"\pi_1",->>] \\
 X \ar[r,"f"] & Y.
\end{tikzcd}\]
Note that because the fibrations $PY \to Y$ are trivial fibrations, the map $X \times_Y PY \to X$ in the diagram above is also a trivial fibration. The total vertical maps are both the identity.
Which gives us a diagram:
\[
  \begin{tikzcd}
    &X \ar[dd,"e'"] \ar[ddr,"f"] \ar[ddl,"id_X"swap] & \\
    & & \\
    X & X \times_Y PY \ar[r,->>,"p"swap] \ar[l,->>,"q","\sim"swap]   & Y
  \end{tikzcd}
\]
Where $p$ is the map $X \times_Y PY \fibration X \times Y \overset{\pi_2}{\fibration} Y$. Note that all maps in this diagram are weak equivalences due to the $2$-out-of-$3$ condition.
We can now prove the theorem, we have
\[X \vdash \phi(v) \Leftrightarrow X\times_Y PY \vdash \phi(e'v)\]
because $v = q e'v$ and $q$ is a trivial fibration, and
\[ X\times_Y PY \vdash \phi(e'v) \Leftrightarrow Y \vdash \phi(fv)\]
because $p$ is a trivial fibration and $fv = pe'v$. Hence, combining the two
 \[ X \vdash \phi(v) \Leftrightarrow Y \vdash \phi(fv) \]
  
\end{proof}

Finally, we explain how Quillen adjunctions act on formulas. A \emph{Quillen adjunction} between two weak model categories is an adjunction
\[ L : \catC \leftrightarrows \catD : R\]
where the left adjoint $L$ sends cofibrations to cofibrations and the right adjoint $R$ sends fibrations to fibrations.

\begin{remark}
There is also a more general notion called ``weak Quillen functors'' introduced in \cite{henry20weak} which is sometimes more convenient. The functor $L$ is only defined on cofibrant objects and $R$ on fibrant objects, and they are only required to preserve core (co)fibrations -- all results in this section below, as well as the \Fourth\ from \cref{sec:invariance} apply to weak Quillen adjunctions too. We restrict ourselves to Quillen adjunctions in the paper, unless otherwise stated, for simplicity, and because this already cover most of the applications.
\end{remark}

\begin{construction}\label{cstr:Quillen_functor_acts_on_formulas}

Given a Quillen adjunction\footnote{Or more generally a weak Quillen adjunction in the sense of \cite{henry20weak}.} $L : \catC \leftrightarrows \catD : R$. Then, $L$ restricts to a coclan morphism $L:\catC^\cof \to \catD^\cof$, which following \cref{action-functor:formulas} we have a (unique) comparison map
\[\alpha_L: \Lb^{\catC}_\lambda \to L^* \Lb^{\catD}_\lambda\]
obtained from the fact that $\Lb^{\catC}_\lambda$ is an initial boolean algebra over $\catC$. As before, if $\phi \in \Lb^\catC_\lambda(C)$, we often write $L(\phi)$ instead of $\alpha_L(\Phi)$. Note that $L(\phi) \in \Lb^{\catD}_\lambda(L(C))$.

Finally, exactly as in \cref{action-functor:formulas}, we have:
  
\end{construction}

\begin{proposition} For a Quillen adjunction $ L : \catC \leftrightarrows \catD : R$, any\footnote{If $L$ and $R$ are only a weak Quillen adjunction, then $X$ needs to be fibrant.} object $X \in \catD$, and cofibrant object $C \in \catC$, any map $v:C \to R(X)$ corresponding to $\tilde{v} : LC \to X$, and $\phi \in \Lb^\catC_\lambda$ we have

\[ R(X) \vdash \phi(v) \Leftrightarrow X \vdash L(\phi)(\tilde{v}). \]
  
\end{proposition}

\begin{proof}
  See \cref{action-functor:formulas}.
\end{proof}

The \Fourth\ theorem that we will establish in \cref{sec:invariance} as \cref{invariance-theorems-34} shows that for a Quillen equivalence, this construction gives an equivalence between the language of $\catC$ and of $\catD$ in an appropriate sense.

%%%%%%%%%%%%%%%%%% 

%%% Local Variables:
%%% mode: latex
%%% TeX-master: "main"
%%% End:

        % Examples of languages of model categories

        \section{Examples of languages of model categories}\label{sec:examples}

In this section, we examine some examples of the language associated to a model category by applying the construction as described in \cref{fol-gat}. We include examples we believe to be of interest. Furthermore, we start with some general considerations that allow us to construct the language of a model category.

When applying the theory introduced in \cref{fol-gat} to a model category $\catM$, we have two possible approaches: we can manipulate formulas as elements of the free Boolean algebra over $\catM^\cof$, following the approach from \cref{sec:category_approach}, or we can try to build a generalized algebraic theory whose first-order language is the same as the language of $\catM$. For example, we could try to realize $\catM$ as the category of models of some generalized $\kappa$-algebraic theory, or if that is not possible we could try to realize the category of $\kappa$-presentable cofibrant objects of $\catM$ as the opposite of the syntactic category of some generalized $\kappa$-algebraic theory.

We believe that, once we are familiar with how this language works, the first approach is simpler. But in order to build familiarity with the languages, in all the examples we will cover below we will try to use the second approach and build a more or less explicit generalized algebraic theory associated to each example, in order to show the reader what can be done in the logic of each case.

\bigskip

It is shown in \cref{appendix-b} that any $\kappa$-clan is equivalent to the syntactic category of a generalized $\kappa$-algebraic theory. So in general, given $\catM$ a combinatorial (weak) model category, we can always find a regular cardinal $\kappa$ and a generalized $\kappa$-algebraic theory such that the language associated to $\catM$ is the language of this generalized algebraic theory. Unfortunately, the construction of this theory following \cref{appendix-b} is extremely unexplicit.

What we would like to do here is to give some tools to help ``guess'' a simpler generalized algebraic theory that works on concrete examples. Given that our goal is only to guess the correct theory for a few examples, we will not try to make this completely formal and rigorous -- though it might be possible.

\bigskip

To that end, let us recall some facts about a generalized $\kappa$-algebraic theory $T$, and of the $\kappa$-contextual category $\C_T$ associated to it. \Cref{contextdefinition} states inductively what it means for a judgment $\Gamma \vdash \Delta \, \type$ in a $\kappa$-pretheory to be well-formed in $T$; this is the case whenever $\Gamma$ is a context, and this itself entails that any constituent of $\Gamma$ is obtained from a derived rule of the $\kappa$-pretheory $T$. In turn, each derived rule is deduced from the list of \cref{derivedrules}, or using a rule previously derived. In a generalized $\kappa$-algebraic theory, each type introduction axiom (derived judgment) is well-formed by \cref{well-formed-type-intro:prop}. Concretely, this means that in order to build new types in context $\Gamma'$ we must know that all the variables used in $\Gamma'$ must previously be constructed in some context $\Gamma$. In a sense, each type must be constructed from more primitive types.

We can use the above in the following:
    \begin{remark} \label{syntatic:well-founded}
      Let $T$ be a generalized $\kappa$-algebraic theory and $\C_T$ the syntactic $\kappa$-contextual category of $T$ with the natural $\kappa$-clan structure \ie in which the fibrations are the generalized display maps. Each type axiom $\Gamma \vdash A\, \type $ of $T$ corresponds to a display map $(\Gamma.A \display \Gamma)$. Now, the set of axioms of $T$ admits a well-founded transitive relation $<$ such that for each type axiom $\Gamma \vdash A\, \type$ we can show that $\Gamma$ is a context using only type axioms ``smaller'' than $\Gamma \vdash A\, \type$. In particular, it means that only types ``smaller than A'' can appear in the context $\Gamma$. Formulated categorically, this means that the map $\Gamma \to 1$ can be constructed as $\kappa$-small composite of pullbacks of display maps $\Gamma'.B \to \Gamma'$, for $\Gamma' \vdash B\, \type $ type axioms that are smaller than $\Gamma \vdash A \, \type$.
      Recall from \cref{wfs-models} that $\Mod(T)$ has a weak factorization system which is cofibrantly generated by the set
      \[I=\{ \yoneda_A \cofibration \yoneda_B \in \Mod(T) | B \display A \in \C_T\}.\]
 Given that every display map is a $\kappa$-small composite of pullback of the display map corresponding to type axioms. We can restrict the set of generators to the display maps corresponding to type axioms, which then comes with this additional well-founded relation.
    \end{remark}

    The previous example motivates:
\begin{definition} \label{cofibrations:order}
  Let $\catC$ be a model category and $\cof(\catC)$ the class of cofibrations. Assume that the cofibrations are generated by a set $I$. We say that the set of generating cofibrations is \emph{well-founded} if there exists a well-founded relation $<$ on $I$ such that for all $i \in I$, the map $\emptyset \to Dom(i)$ can be written as a $\kappa$-composite of pushouts of maps $j\in I$ with $j<i$.
\end{definition}

    \begin{example} \label{cofibrations-models:well-founded}
   As explained in \cref{syntatic:well-founded}, if $T$ is a generalized $\kappa$-algebraic theory, then the weak factorization from \cref{wfs-models} on $\Mod(T)$ has a well-founded set of generators corresponding to the type of axioms of $T$.
    \end{example}

The general idea is; if we start from a combinatorial weak factorization system, and we want to see it as coming from an explicitly given generalized algebraic theory, we start by finding a well-founded set of generators, and then we build a theory whose type axioms correspond to these generators. 
    
Note that in particular, we need the factorization system to be generated by map with cofibrant domain, or at least we need the generating cofibrations to have cofibrant domain. Most model structures we work with in practice, in fact all the examples we will encounter here have this property (this is closely related to the notion of tractable model category from \cite{barwick2010left}). But in general this is not an obstruction: lemma 4.7 of \cite{henry2023combinatorial} allows to modify any combinatorial or accessible model category into one that has this property---maybe at the cost of moving to semi-model category.

\begin{proposition}[{\cite[4.7 Lemma]{henry2023combinatorial}}] \label{htract}
  Fix $\kappa$ an uncountable regular cardinal. Let $(L_1,R_1)$ and $(L_2,R_2)$ two $\kappa$-accessible weak factorization systems on a locally $\kappa$-presentable category $\catC$ such that $L_1 \subset L_2$ or $R_2 \subset R_1$. There is a $\kappa$-accessible weak factorization system $(L_3,R_3)$ on $\catC$ such that $R_3$ is the class of maps that have the right lifting property against all $L_1$-maps whose domain is $L_2$-cofibrant. If $(L_1,R_1)$ is $\kappa$-combinatorial, then $(L_3,R_3)$ is also $\kappa$-combinatorial.
\end{proposition}

\begin{observation} \label{tractable}
  If $\catM$ is a combinatorial weak model category, then there exists another combinatorial weak model category structure on $\catM$ with the same core cofibrations and core acyclic cofibrations, but where the cofibrations and acyclic cofibrations are generated by core cofibrations. In order to see this, we apply \cref{htract} taking $(L_1,R_1)$=( acyclic cofibrations, fibrations) and $(L_2,R_2)$=(cofibrations, acyclic fibrations). This produces a weak factorization system $(L_3,R_3)$ where the class $R_3$ of fibrations is generated by acyclic cofibrations with cofibrant domain. We apply the result again, but on (cofibrations, acyclic fibrations)=$(L_2,R_2)$=$(L_1,R_1)$ to get another weak factorization system $(L_3', R_3')$ where the class $R_3'$ is generated by cofibrations with cofibrant domain. Note this process does not change the core (acyclic) cofibrations or core (acyclic) fibrations, and hence preserves the fact that we have a weak model category.
\end{observation}
    
    Once we have generating cofibrations with cofibrant domain, there is always an easy way to get a well-founded set of generators:
   
   \begin{example} \label{combinatorial:well-founded}
 If $L$ is a set of generating cofibrations with cofibrant domain of a combinatorial weak model category, then we can get a well-founded class of cofibrations by setting $L' \coloneqq \{\emptyset \to Dom(l)| l \in L  \} \coprod L$. In this case, we can set $(\emptyset \to Dom(l)) < f$ for $f\in L$ and $l \in L$.
\end{example}
    
\Cref{cofibrations-models:well-founded} shows that starting with a $\kappa$-clan, one can get a cofibrantly generated weak factorization system on the category of models $\Mod(\catC)$ such that the generating set of cofibrations is well-founded. We can reverse this process in the sense that if we are given a weak factorization system with a well-founded set of generating cofibrations, then we can produce a generalized $\kappa$-algebraic theory from it, and therefore the $\kappa$-clan associated to it.

The next example is similar to \cref{syntatic:well-founded}.
    
    \begin{construction} \label{theory-wfs:construction}
      Let $\catC$ be a $\kappa$-clan. Assume that $\catC$ has a weak factorization system that is cofibrantly generated by a set $I$ with a  well-founded relation. Recall that this means that for a cofibration $i: A \cofibration B$ the map $\emptyset \to A$ is a $\kappa$-composite of pushouts of maps $j\in I$ with $j<i$. Therefore, we can introduce a type axiom:
      \[
        \overline{A} \vdash \overline{B} \, \type
      \]
      for $i:A \cofibration B \in I$. The notation $\overline{A}$ denotes the context in which the new type $\overline B$ is built, and the context $\overline A$ is obtained using types strictly smaller than $\overline{B}$, which reflects the decomposition of the map $\emptyset \cofibration A$ as a $\kappa$-composite of pushouts of maps $j \in I$ smaller than $i$.
    \end{construction}

    We can think of this construction as similar to the functor $U:\kcon \to \kgat$ from \cref{functor-contextual-to-gat} which produces a generalized $\kappa$-algebraic theory $U(\catC)$ from a $\kappa$-contextual category $\catC$. In particular, for a display map $B_{\lambda+1} \display B_\lambda \in \catC$ it gives a type axiom $\overline{B_{\lambda}} \vdash \overline{B_{\lambda+1}} \, \type$.
    
    \begin{remark}
      For each of the examples below, we start with a Quillen model category $\catM$ and apply \cref{theory-wfs:construction} to obtain a theory $T_\catM$. In general, this is the guiding principle that will allow us to identify the statements, and the language, to which the invariance theorems apply.

      Furthermore, using the theory $T_\catM$ we can consider the category $\Mod(T_\catM)$ and use \cref{wfs-models} to obtain a weak factorization system. Through this process, the cofibrations and trivial fibrations we obtain coincide with those from the Quillen model category we start with. However, in general we do not have an equivalence of categories $\Mod(T_\catM)\cong\catM$.
    \end{remark}

\subsection{Categories} \label{examples:categories}

Let us illustrate our construction on this prime example we have been referring to throughout the paper. Recall that $\mathbf{0}$ is the empty category, $\mathbf{1} \coloneqq \{0\}$ is the category with a single object, $ \mathbf{2} \coloneqq \{0 \to 1\}$ the arrow category and $P \coloneqq \{0\rightrightarrows 1\} $ the category with two parallel arrows. Finally, $\Jcal \coloneqq \{0 \leftrightarrows 1 \} $ denotes the walking isomorphism category. The following result appears in \cite{rezk1996model}.
\begin{theorem} \label{folkmodel:categories}
  There is Quillen model structure on the category $\cat$ such that:
  \begin{enumerate}
  \item Weak equivalences are the equivalences of categories,
  \item Cofibrations are the functors injective on objects,
  \item Fibrations are the isofibrations.
  \end{enumerate}
  Furthermore, this models structure is cofibrantly generated. The sets
  \[I \coloneqq \{ \mathbf{0} \overset{u}{\to} \mathbf{1},\, \{0\}\sqcup \{1\} \overset{v}{\to} \mathbf{2},\, P \overset{w}{\to} \mathbf{2} \} \text{ and } J \coloneqq \{ \mathbf{1} \to \Jcal\} \]
  are the generating cofibrations and trivial cofibrations respectively.
\end{theorem}

In this model structure all objects are cofibrant. We can immediately associate for each generator in $I$ a sort in the following way:
% https://q.uiver.app/#q=WzAsNixbMCwwLCJcXGVtcHR5c2V0IFxcdG8gXFxtYXRoYmZ7MH0iXSxbMSwwLCJcXHZkYXNoIFxcb2J0eXAgXFwsIFxcdHlwZSJdLFswLDEsIiBcXHswXFx9XFxzcWN1cCBcXHsxXFx9IFxcdG8gXFxtYXRoYmZ7MX0iXSxbMSwxLCJ4LHk6XFxvYnR5cCBcXHZkYXNoIFxcaG9tdHlwKHgseSkgXFwsIFxcdHlwZSJdLFsxLDIsIngseTpcXG9idHlwLCBmLGc6XFxob210eXAoeCx5KSBcXHZkYXNoIFxcZXF0eXAoZixnKSBcXCwgXFx0eXBlIl0sWzAsMiwiUCJdLFswLDEsIiIsMCx7InNob3J0ZW4iOnsic291cmNlIjoyMCwidGFyZ2V0Ijo1MH0sInN0eWxlIjp7InRhaWwiOnsibmFtZSI6Im1hcHMgdG8ifX19XSxbMiwzLCIiLDAseyJzaG9ydGVuIjp7InNvdXJjZSI6MTAsInRhcmdldCI6MjB9LCJzdHlsZSI6eyJ0YWlsIjp7Im5hbWUiOiJtYXBzIHRvIn19fV0sWzUsNCwiIiwwLHsic2hvcnRlbiI6eyJzb3VyY2UiOjUwfSwic3R5bGUiOnsidGFpbCI6eyJuYW1lIjoibWFwcyB0byJ9fX1dXQ==
\[\begin{tikzcd}[row sep=0.1em]
	{\mathbf{0} \to \mathbf{1}} & {\vdash \obtyp \, \type} \\
	{ \{0\}\sqcup \{1\} \to \mathbf{2}} & {x,y:\obtyp \vdash \homtyp(x,y) \, \type} \\
	P & {x,y:\obtyp, f,g:\homtyp(x,y) \vdash \eqtyp(f,g) \, \type}
	\arrow[shorten <=19pt, shorten >=46pt, maps to, from=1-1, to=1-2]
	\arrow[shorten <=4pt, shorten >=9pt, maps to, from=2-1, to=2-2]
	\arrow[shorten <=15pt, maps to, from=3-1, to=3-2]
\end{tikzcd}\]

Note that while the type $\obtyp$ has no dependencies, the type $\homtyp(x,y)$ depends on two elements of type $\obtyp$, which is encoded in the cofibration $ \{0\}\sqcup \{1\} \to \mathbf{2}$. The same situation applies with the type $\eqtyp$ which furthermore has dependencies on the types $\obtyp$ and $\homtyp$, now the cofibration $P \cofibration \mathbf{2}$ expresses this.

\begin{remark}
  The reason the previous association is well-defined is that the set of generating cofibrations $I$ of the model structure on $\cat$ from \cref{folkmodel:categories} has a natural well-founded order---in the sense of \cref{cofibrations:order}. Indeed, we can set $\mathbf{0} \to \mathbf{1}$ as the least element. Since the domain of the cofibration $\{0\}\sqcup \{1\} \to \mathbf{2}$ is a pushout of $\mathbf{0} \to \mathbf{1}$, we can declare $(\mathbf{0} \to \mathbf{1}) <(\{0\}\sqcup \{1\} \to \mathbf{2})$. Following the same reasoning, we see that the domain of the cofibration $P \to \mathbf{2}$ is the pushout of two copies of $\{0\}\sqcup \{1\} \to \mathbf{2}$. Therefore, we can also set $(\{0\}\sqcup \{1\} \to \mathbf{2}) < (P \to \mathbf{2}) $. This completely determines the order $<$ on $I$, which is well-founded by construction. For all the subsequent examples, one can induce the corresponding well-founded orders analogously.
\end{remark}

The resulting theory is what we introduced earlier, $Cat_=$, which for convenience we recall here. This is defined as:	
	\begin{enumerate}
		\item Type of objects: $ \vdash \, \obtyp \, \type $.
		\item Type of morphisms: $ x : \obtyp, \,y : \obtyp \vdash \, \homtyp(x,y) \, \type $.
                \item Equality type: $ x,y:\obtyp, f,g : \homtyp(x,y) \vdash \eqtyp(f,g) \type $
		\item Composition operation: $ x, y, z : \obtyp, \, f:\homtyp(x,y), \, g:\homtyp(y,z) \vdash g\circ f:\homtyp(x,z) $.
		\item Identity operator: $ x: \obtyp \vdash \, \idx_x:\homtyp(x,x) $.
	\end{enumerate}
	Subject to the following axioms:
        \begin{itemize}
        \item $x : \obtyp, \, y:\obtyp, \, f:\homtyp(x,y) \vdash \idx_y\circ f\equiv f $.
        \item $x : \obtyp ,\, y:\obtyp,f:\homtyp(x,y) \vdash f\circ\idx_x\equiv f $.
        \item $x :\obtyp, \, y :\obtyp, \, z :\obtyp, \, w :\obtyp, \, f:\homtyp(x,y),g:\homtyp(y,z),h:\homtyp(z,w) \vdash (h\circ g)\circ f\equiv h\circ (g\circ f)$.
        \item $ x,y:\obtyp, f: \homtyp(x,y) \vdash r_f : \eqtyp(f,f) $.
        \item $ x,y:\obtyp, f,g : \homtyp(x,y), a: \eqtyp(f,g) \vdash f \equiv g $.
        \item $x,y:\obtyp, f,g : \homtyp(x,y), a: \eqtyp(f,g) \vdash a \equiv r_f$.
        \end{itemize}

        \begin{remark}
          In the example above, we have imposed additional axioms for terms of type $\homtyp$ and $\eqtyp$. The reason behind this is solely so that the models of the theory $Cat_=$ are exactly the categories.
        \end{remark}
        
        As pointed out in \cref{ex:cat=} the language we obtain is the same as the one given by \cite{blanc78} and \cite{freyd76}. In the introduction we presented the formula for an object $x$ to be terminal:
        \[ \forall y \in \text{Ob}, \left( \exists v \in \Hom(y,x) \wedge \forall u,w \in \Hom(y,x), \eqtyp(u,w) \right).\]
        Such formula is written in the language of categories.

        \begin{observation}
          We verify the above differently to showcase the fact that we do not need to explicitly know the language (type theory) associated to a model category, we only need to know that it can be constructed out of cofibrations. The formula above is constructed by first quantifying universally over the cofibration $\mathbf{0} \to \mathbf{1}$ to give $\forall y\in \obtyp$. Note that applying the existential quantifier to $\{0\}\sqcup \{1\} \to \mathbf{2} $ gives us $\exists v\in\Hom(y,x)$ and the universal quantifier on $\mathbf{1} \to \Jcal$. In the end, the formula can be seen as a composition pushouts ``in context $x$.'' Building the context of a formula is not an easy task, however, it might be easier to describe a pushout.
        \end{observation}

        \begin{remark}
          We mentioned at the beginning of the section that the association we do from cofibrations to types is not extremely formal. Again, the reason is that the equivalence between $\kappa$-clans and generalized $\kappa$-algebraic theories, \cref{appendix-b}, is not explicit. The association we make, for categories and the other examples below, is the obvious one and ad-hoc to the expected theory. From the start, we know what our intended models are, so once we have the types we define the operations and impose the equations that our intended models satisfy. We stress that this is informal and not very precise.
          \end{remark}

          \begin{remark}
          In general, a cofibration in a model category could be decomposed as a pushouts of cofibrations in more than one way. Depending on our choices, it might happen that we end up with different, but equivalent, theories. 

          One of the worst case scenarios is when we do not have a straightforward well-ordering. See the case for unbounded chain complexes below \cref{unbounded-chains}.
        \end{remark}
        
        \subsection{$2$-categories and Bicategories}  \label{bicats}

In this section we examine the language associated to the canonical model structures on the categories $\twocat$ and $\bicat$, respectively. The model structure for these two categories was defined in \cite{lack2002model} and \cite{lack2004model}.

Given a category $C$ its suspension $\sum C$, is defined as the $2$-category with two objects $X,Y$, the hom categories are $\sum C(X,X)=\sum C(Y,Y)=\sum C(Y,X)=\emptyset$ and $\sum C(X,Y)=C$. Furthermore, each bicategory $\mathscr{B} \in \bicat$ has an underlying $\cat$\textbf{-graph}, in the sense of \cite{wolff1974v}. This  induces a functor $U:\bicat \to \cat\textbf{-graph} $ which has left adjoint $F$; this gives us the free bicategory generated by a $\cat$-graph. The suspension of a category $C$ can be seen as a $\cat$-graph associated to $C$. The free bicategory generated by the suspension of a category is denoted by $ \sum C$. Moreover, this construction is functorial.

\cite[Theorem 3]{lack2004model} constructs a model structure for the category of bicategories. This model structure is cofibrantly generated with generating cofibrations given by the suspension of the generating cofibrations of the canonical model structure on $\cat$ and an additional functor we specify below. Finally, $\mathscr{E}$ is the ``free-living adjoint equivalence '' is the bicategory with objects $x,y$,  freely generated by 1-cells $f:x \to y$ and $g:y \to x$, and two invertible 2-cells $\eta:1_x \Rightarrow gf$, $\varepsilon:fg \Rightarrow 1_y$ satisfying the familiar triangle identities.

\begin{theorem}[{\cite[Theorem 3]{lack2004model}}] \label{model:bicategories}
  There is a model structure on the category $\bicat$ of bicategories and strict bifunctors such that:
  \begin{enumerate}
  \item Weak equivalences are the biequivalences,
  \item Fibrations are the strict bifunctors with the equivalence lifting property. 
  \end{enumerate}
  Furthermore, the model structure is cofibrantly generated by the sets
  \[ I \coloneqq \{ \mathbb{0} \to \mathbbm{1}, \Sigma u, \Sigma v, \Sigma w \}  \text{ and } J \coloneqq \{ \mathbbm{1} \to \mathscr{E} \} \]
  where $\mathbb{0}$ is the empty bicategory, $\mathbbm{1}$ is the bicategory with a single object and no non-identity 2-cells, the functors $u,v,w$ come from \cref{folkmodel:categories}, and the bifunctor in $J$ picks the object $x$.
\end{theorem}
When we analyze the set of generating cofibrations $I$ we rediscover the generalized algebraic theory of bicategories $Bicat_=$:
\begin{itemize}
\item $\begin{tikzcd}
	{\mathbb{0}\to \mathbb{1}} & {\vdash \obtyp \, \type}
	\arrow[maps to, from=1-1, to=1-2]
      \end{tikzcd}$
    
\item % https://q.uiver.app/#q=WzAsMyxbMCwwLCJcXHt4XFx9XFxzcWN1cFxce3lcXH0iXSxbMSwwLCJ4XFx0byB5Il0sWzIsMCwieCx5Olxcb2J0eXAgXFx2ZGFzaCBcXGhvbXR5cCh4LHkpIl0sWzAsMSwiXFxzdW0gdiJdLFsxLDIsIiIsMCx7InN0eWxlIjp7InRhaWwiOnsibmFtZSI6Im1hcHMgdG8ifSwiYm9keSI6eyJuYW1lIjoic3F1aWdnbHkifX19XV0=
$\begin{tikzcd}[sep=small]
	{\{x\}\sqcup\{y\}} & \{x\to y\} & {x,y:\obtyp \vdash \homtyp(x,y)}
	\arrow["{\sum u}", from=1-1, to=1-2]
	\arrow[maps to, from=1-2, to=1-3]
      \end{tikzcd}$
    
  \item % https://q.uiver.app/#q=WzAsNSxbMCwwLCJ4Il0sWzEsMCwieSJdLFsyLDAsIngiXSxbMywwLCJ5Il0sWzQsMCwieCx5Olxcb2J0eXAsIFxcLCBmLGc6XFxob210eXAoeCx5KSBcXFxcIFxcdmRhc2ggXFxob210eXAoZixnKVxcLCBcXHR5cGUiXSxbMCwxLCIwIiwwLHsiY3VydmUiOi0xfV0sWzAsMSwiMSIsMix7ImN1cnZlIjoxfV0sWzIsMywiMCIsMCx7ImN1cnZlIjotMX1dLFsyLDMsIjEiLDIseyJjdXJ2ZSI6MX1dLFsxLDIsIlxcc3VtIHciXSxbMyw0LCIiLDIseyJzdHlsZSI6eyJ0YWlsIjp7Im5hbWUiOiJtYXBzIHRvIn0sImJvZHkiOnsibmFtZSI6InNxdWlnZ2x5In19fV0sWzcsOCwiIiwyLHsic2hvcnRlbiI6eyJzb3VyY2UiOjIwLCJ0YXJnZXQiOjIwfX1dXQ==
$\begin{tikzcd} [sep=small]
	x & y & x & y & {x,y:\obtyp, \, f,g:\homtyp(x,y) \vdash \homtyp(f,g)\, \type}
	\arrow["0", bend left, from=1-1, to=1-2]
	\arrow["1"', bend right, from=1-1, to=1-2]
	\arrow["{\sum v}", from=1-2, to=1-3]
	\arrow[""{name=0, anchor=center, inner sep=0}, "0", bend left, from=1-3, to=1-4]
	\arrow[""{name=1, anchor=center, inner sep=0}, "1"', bend right, from=1-3, to=1-4]
	\arrow[maps to, from=1-4, to=1-5]
	\arrow[shorten <=2pt, shorten >=2pt, Rightarrow, from=0, to=1]
\end{tikzcd}$

\item $
        % https://q.uiver.app/#q=WzAsNCxbMCwwLCJ4Il0sWzEsMCwieSJdLFsyLDAsIngiXSxbMywwLCJ5Il0sWzAsMSwiMCIsMCx7ImN1cnZlIjotMX1dLFswLDEsIjEiLDIseyJjdXJ2ZSI6MX1dLFsyLDMsIjAiLDAseyJjdXJ2ZSI6LTF9XSxbMiwzLCIxIiwyLHsiY3VydmUiOjF9XSxbMSwyLCJcXHN1bSB3Il0sWzYsNywiIiwyLHsic2hvcnRlbiI6eyJzb3VyY2UiOjIwLCJ0YXJnZXQiOjIwfX1dLFs0LDUsIiIsMCx7Im9mZnNldCI6MSwic2hvcnRlbiI6eyJzb3VyY2UiOjIwLCJ0YXJnZXQiOjIwfX1dLFs0LDUsIiIsMix7Im9mZnNldCI6LTEsInNob3J0ZW4iOnsic291cmNlIjoyMCwidGFyZ2V0IjoyMH19XV0=
\begin{tikzcd}
	x & y & x & y
	\arrow[""{name=0, anchor=center, inner sep=0}, "0", bend left, from=1-1, to=1-2]
	\arrow[""{name=1, anchor=center, inner sep=0}, "1"', bend right, from=1-1, to=1-2]
	\arrow["{\sum w}", from=1-2, to=1-3]
	\arrow[""{name=2, anchor=center, inner sep=0}, "0", bend left, from=1-3, to=1-4]
	\arrow[""{name=3, anchor=center, inner sep=0}, "1"', bend right, from=1-3, to=1-4]
	\arrow[shift right, shorten <=2pt, shorten >=2pt, Rightarrow, from=0, to=1]
	\arrow[shift left, shorten <=2pt, shorten >=2pt, Rightarrow, from=0, to=1]
	\arrow[shorten <=2pt, shorten >=2pt, Rightarrow, from=2, to=3]
\end{tikzcd} \mapsto 
\begin{cases}
  x,y:\obtyp, \, f,g:\homtyp(x,y), \\
  \alpha,\beta:\homtyp(f,g) \vdash \eqtyp(\alpha,\beta)\,\type
\end{cases}
      $
\end{itemize}
Moreover, we can also introduce the composition and identity operations for arrows and cells:

\begin{itemize}
\item Composition operation for arrows: $ x : \obtyp, \, y : \obtyp, \, z : \obtyp, \, f:\homtyp(x,y), \, g:\homtyp(y,z) \vdash g\circ f:\homtyp(x,z) $.
\item Identity operator for arrows: $ x: \obtyp \vdash \, \idx_x:\homtyp(x,x) $.
\item Vertical composition of cells: $x,y:\obtyp, f,g,h:\homtyp(x,y), \alpha:\homtyp(f,g), \beta:\homtyp(g,h) \vdash \beta \circ \alpha : \homtyp(f,h)$.
\item Horizontal composition of cells: $x,y,z:\obtyp, f,g:\homtyp(x,y), h,k:\homtyp(y,z), \alpha:\homtyp(f,g), \beta:\homtyp(h,k) \vdash  \alpha *\beta : \homtyp(h \circ f,k\circ g)$.
\item Identity operator for cells: $ x,y: \obtyp, f: \homtyp(x,y) \vdash \, \idx_f:\homtyp(f,f) $.
\end{itemize}
One can also attempt to list all the axioms that the above theory ought to satisfy, with the risk of running out of space. We simply exemplify this with the associator:
\begin{multline*}
    w,x,y,z:\obtyp,f:\homtyp(w,x),g:\homtyp(x,y),h:\homtyp(y,z), \\ \alpha: \homtyp((h\circ g)\circ f, h\circ (g\circ f)),\beta: \homtyp((h\circ (g\circ f),h\circ g)\circ f) \\ \vdash r: \eqtyp(\alpha \circ \beta, \idx_{(h\circ (g\circ f)} ) \wedge s: \eqtyp(\beta \circ \alpha, \idx_{(h\circ g)\circ f} ).
  \end{multline*}
  We also include the axioms for $\eqtyp$ --- the same ones as for categories --- that gives us the expected behaviour. 

  \begin{remark}
    If we now try to obtain the associated theory $2Cat_=$ using the generating cofibration of \cite{lack2004model}, we see that the resulting theory has similar types and operations as the theory $Bicat_=$ of bicategories. The notable differences are that we do not need associators or unitors, but we need to include equations for the associativity and unitality of the composition of arrows and cells, and also the interchange law relating horizontal and vertical composition of cells. All these axioms are the appropriate ones to obtain 2-categories as the models of the theory $2Cat_=$.
  \end{remark}

\begin{definition} \label{twoterminal:def}
  Let $\catC$ be a 2-category. An object $x\in \catC$ is \emph{bi-terminal} if for all $y\in \catC$ there is an equivalence of categories $\catC(y,x)\cong \mathbf{1}$.
\end{definition}

Note that $f:a \to b$ being an equivalence can be written as
\[\exists h:\Hom(b,a),\exists \eta:\Hom(\idx_a, h\circ f), \exists\varepsilon:\Hom(f\circ h,\idx_b), \textsf{\textup{isIso}}(\eta) \wedge \textsf{\textup{isIso}}(\varepsilon),\top.\]
Observe that the statement $\textsf{\textup{isIso}}(\eta)$, which says that $\eta:f\Rightarrow g$ is a natural isomorphism, only involves equality of natural transformations:
\[
  \textsf{\textup{isIso}}(\eta)) \coloneqq \exists \epsilon: \Hom(g,f),s:\eqtyp(\epsilon \circ \eta, \idx_f) \wedge r:\eqtyp(\eta \circ \epsilon, \idx_g),\top.
\]
We can then conclude that the notion of bi-terminal object is invariant.

\begin{remark}
  Other natural, but somewhat different, higher categories to consider in this progression are the double categories. Fortunately, this question has been described in Paula Verdugo's PhD thesis \cite{verdugo2024}, or \cite{verdugo2025invariance}. In particular, she builds a model structure on double categories where the fibrant objects are the \emph{equipments}. The language for this model structure produces formulas that express properties of equipments. Therefore, we can use our invariance theorems for this ``language of equipments''. The details of this are exposed in Verdugo's PhD thesis cited above.
\end{remark}

\subsection{Bounded below chain complexes}

In this section, we examine the language of the projective model structure on bounded below chain complexes $Ch(R)$ over a commutative ring $R$. We start by recalling some facts about this model structure. The detailed proofs can be found elsewhere, e.g. \cite{hovey1999models}.

Given an $R$-module $M$, for each $n \in \Z $ define $S^n(M)\in Ch(R)$ by
\[
  S^n(M)_k \coloneqq
  \begin{cases}
    M, \, k=n \\
    0, \, k \neq n.
  \end{cases}
\]

Similarly, $D^n(M)\in Ch(R)$ is defined as
\[
  D^n(M)_k \coloneqq
  \begin{cases}
    M, \, k=n-1,\, n \\
    0, \, \text{ otherwise.}
  \end{cases}
\]
where the only non-trivial differential $d_n:M \to M$ is the identity. Obviously, we get an inclusion $S^{n-1}(M) \to D^n(M)$.

These constructions induce functors $S^n:R\text{-}Mod \to Ch(R)$ and $D^n:R\text{-}Mod \to Ch(R)$ for each $n \in \Z$. Both functors have right adjoints $Z_n:Ch(R) \to R\text{-}Mod$ and $Ev_n:Ch(R) \to R\text{-}Mod$, respectively, where $Z_nX \coloneqq Ker(d_n)$ and $Ev_n X\coloneqq X_n $.

In particular, when $M=R$ the chains above are denoted by $ S^n $ and $D^n$, respectively. We can define the sets $$I \coloneqq \{ S^{n-1} \to D^n \vert n \in \Z \} \text{ and } J \coloneqq \{ 0 \to D^n \vert n \in \Z \}.$$

All constructions above work on unbounded chain complexes too. In the next result we restrict to bounded below chains, \ie $n\geq 0$. By definition $(D^{0})_{-1}=0$, so that $S^0=D^0$. With this information, what we need to know about the projective model structure is summarized in the following:

\begin{theorem}[\cite{quillen2006axiomatic}] \label{projective-model:bounded-chains}
  The category of chain complexes $Ch(R)$ admits a model structure where:
  \begin{enumerate}
  \item Weak equivalences are the quasi-isomorphisms
  \item Fibrations are the degree-wise epimorphisms.
  \item Cofibrations are the degree-wise monomorphisms with projective cokernel.
  \end{enumerate}
  Furthermore, this model structure is proper, cofibrantly generated and combinatorial. Cofibrations and trivial cofibrations are generated by $I$ and $J$, respectively.
\end{theorem}

The cofibrant objects in the mode structure from \cref{projective-model:bounded-chains} are complexes such that each $R$-module is projective. However, this is not the case for unbounded chain complexes, where not every chain complex with projective modules is cofibrant. Nevertheless, in both cases, all objects are fibrant.

\begin{remark} \label{adjunction:remark}
  Using the adjunction $ S^n \dashv Z_n$, for any chain complex $X$, a map $S^n \to X$ is simply a map $R \to Z_nX$ of $R$-modules. And from $D^n \dashv Ev_n$, a map $D^n \to X$ corresponds to $ y \in X_n $. Therefore, a commutative square
% https://q.uiver.app/#q=WzAsNCxbMCwwLCJTXntuLTF9Il0sWzEsMCwiWCJdLFswLDEsIkRebiJdLFsxLDEsIlkiXSxbMCwyLCJpX24iLDJdLFswLDEsIngiXSxbMSwzLCJmIl0sWzIsMywieSIsMl1d
\[\begin{tikzcd}
	{S^{n-1}} & X \\
	{D^n} & Y
	\arrow["x", from=1-1, to=1-2]
	\arrow["{i_n}"', from=1-1, to=2-1]
	\arrow["f", from=1-2, to=2-2]
	\arrow["Y"', from=2-1, to=2-2]
      \end{tikzcd}\]
    means that $x \in Z_{n-1}X \subseteq X_{n-1} $ \ie $d_{n-1}x=0$ and that $fx=y \in Y_n. $ Therefore, taking a pushout simply means we freely add $(n-1)$-cycles to $X_{n-1}$ with a specified boundary.
  
\end{remark}

    The first element \ie $n=0$, of the set $I$ is the cofibration
    % https://q.uiver.app/#q=WzAsMTAsWzIsMCwiMCJdLFsyLDEsIlIiXSxbMywwLCIwIl0sWzMsMSwiMCJdLFsxLDEsIjAiXSxbMSwwLCIwIl0sWzQsMCwiXFxjZG90cyJdLFs0LDEsIlxcY2RvdHMiXSxbMCwwLCIwIl0sWzAsMSwiRF4wIl0sWzAsNV0sWzIsMF0sWzYsMl0sWzcsM10sWzEsNF0sWzMsMV0sWzAsMV0sWzIsM10sWzUsNF0sWzgsOSwiaV8wIiwyXV0=
\[\begin{tikzcd}
	0 & 0 & 0 & 0 & \cdots \\
	{D^0} & 0 & R & 0 & \cdots
	\arrow["{i_0}"', from=1-1, to=2-1]
	\arrow[from=1-2, to=2-2]
	\arrow[from=1-3, to=1-2]
	\arrow[from=1-3, to=2-3]
	\arrow[from=1-4, to=1-3]
	\arrow[from=1-4, to=2-4]
	\arrow[from=1-5, to=1-4]
	\arrow[from=2-3, to=2-2]
	\arrow[from=2-4, to=2-3]
	\arrow[from=2-5, to=2-4]
\end{tikzcd}\]
For any $n\geq 1$ we have cofibrations $i_n$
% https://q.uiver.app/#q=WzAsMTQsWzMsMCwiUiJdLFszLDEsIlIiXSxbNCwwLCIwIl0sWzQsMSwiUiJdLFsxLDEsIjAiXSxbMSwwLCIwIl0sWzUsMCwiMCJdLFs1LDEsIjAiXSxbNiwxLCJcXGNkb3RzIl0sWzYsMCwiXFxjZG90cyJdLFswLDAsIlNee24tMX0iXSxbMCwxLCJEXm4iXSxbMiwwLCJcXGNkb3RzIl0sWzIsMSwiXFxjZG90cyJdLFsyLDBdLFs2LDJdLFs3LDNdLFszLDEsIjFfUiIsMV0sWzAsMSwiMV9SIiwxXSxbMiwzXSxbNiw3XSxbOSw2XSxbOCw3XSxbNSw0XSxbMTAsMTEsImlfbiIsMl0sWzEzLDRdLFsxMiw1XSxbMCwxMl0sWzEsMTNdXQ==
\[\begin{tikzcd}
	{S^{n-1}} & 0 & \cdots & R & 0 & 0 & \cdots \\
	{D^n} & 0 & \cdots & R & R & 0 & \cdots
	\arrow["{i_n}"', from=1-1, to=2-1]
	\arrow[from=1-2, to=2-2]
	\arrow[from=1-3, to=1-2]
	\arrow[from=1-4, to=1-3]
	\arrow["{1_R}"{description}, from=1-4, to=2-4]
	\arrow[from=1-5, to=1-4]
	\arrow[from=1-5, to=2-5]
	\arrow[from=1-6, to=1-5]
	\arrow[from=1-6, to=2-6]
	\arrow[from=1-7, to=1-6]
	\arrow[from=2-3, to=2-2]
	\arrow[from=2-4, to=2-3]
	\arrow["{1_R}"{description}, from=2-5, to=2-4]
	\arrow[from=2-6, to=2-5]
	\arrow[from=2-7, to=2-6]
\end{tikzcd}\]
We then see immediately that $I$ has a natural, well-founded, order, where we can set $i_0$ to be the minimal element of the set.

From \cref{adjunction:remark}, we get cycles $y \in X_n$ and for each $x\in X_{n-1}$ such that $dx=0$ and $\textsf{\textup{C}}_n(x) \coloneqq \{y \in X_n | dy=x\} $, this is for each generating cofibration $i_n:S^{n-1} \to D^n$. This tells us that the $\omega$-generalized algebraic theory has types $\textsf{\textup{C}}_n(x)$ for $n\geq 1$. We sum up the discussion in the following table:
% https://q.uiver.app/#q=WzAsNixbMCwwLCJpXzA6MFxcdG8gRF4wIl0sWzIsMCwiIFxcdmRhc2ggXFx0ZXh0c2Z7XFx0ZXh0dXB7Q319XzAgXFwsXFx0eXBlIl0sWzEsMCwiXFxtYXBzdG8iXSxbMCwxLCJpX246U157bi0xfVxcdG8gRF5uIl0sWzIsMSwieDpcXHRleHRzZntcXHRleHR1cHtDfX1fe24tMX0oMCkgXFx2ZGFzaCBcXHRleHRzZntcXHRleHR1cHtDfX1fbih4KSBcXCwgXFx0eXBlIl0sWzEsMSwiXFxtYXBzdG8iXV0=
\[\begin{tikzcd}[sep=small]
	{i_0:0\to D^0} & \mapsto & { \vdash \textsf{\textup{C}}_0 \,\type} \\
	{i_n:S^{n-1}\to D^n} & \mapsto & {x:\textsf{\textup{C}}_{n-1}(0) \vdash \textsf{\textup{C}}_n(x) \, \type}
\end{tikzcd}\]
for $n \geq 1$. Note that the differential is already included in the information that defines the types $\textsf{\textup{C}}_n(x)$. We should also add, not included in the table, ``+'' operations on each type $\textsf{\textup{C}}_n(x)$, and axioms, that ensure is an abelian group:
\[
a:\textsf{\textup{C}}_n(x), \, b:\textsf{\textup{C}}_n(y) \vdash a+b: \textsf{\textup{C}}_n(x+y).
\]

\begin{observation}
  It is important to note that in the theory we do not have equality between chains. The only possibility is to consider $\textsf{\textup{C}}_n(x)$ for $x:\textsf{\textup{C}}_{n-1}(0)$. However, this is enough to speak about chains satisfying a boundary condition $x-y=d_nz$ which is written in our language as \[\exists z:\textsf{\textup{C}}_n(x-y),\top. \]
\end{observation}

\subsection{Unbounded chain complexes} \label{unbounded-chains}

When we work with unbounded chain complexes \cref{projective-model:bounded-chains} becomes:

\begin{theorem}[{\cite[Theorem~2.3.10]{hovey1999models}}]  \label{projective-model:chains}
  The category of chain complexes $Ch(R)$ admits a model structure where:
  \begin{enumerate}
  \item Weak equivalences are the quasi-isomorphisms
  \item Fibrations are the degree-wise epimorphisms.
  \item Cofibrations are the degree-wise split monomorphisms with cofibrant cokernel.
  \end{enumerate}
  Furthermore, this model structure is proper, cofibrantly generated and combinatorial. Cofibrations and trivial cofibrations are generated by $I$ and $J$, respectively.
\end{theorem}

Unlike the case for bounded chains, the set $I$ of generating cofibrations, is not well-founded in the sense of \cref{cofibrations:order}. However, we can obtain a new generating set of cofibrations following \cref{combinatorial:well-founded}. We consider the new set $I'\coloneqq I \cup \{0 \to S^{n} | n\in \Z\}$. Note that since $0\to S^n$ is already a cofibration, we are not altering the model structure. The resulting theory is similar to the bounded case, we now must have the following association:
% https://q.uiver.app/#q=WzAsNixbMCwwLCIwXFx0byBTXm4iXSxbMiwwLCIgXFx2ZGFzaCBcXHRleHRzZntcXHRleHR1cHtDfX1fbiBcXCxcXHR5cGUgXFxcXCBcXHZkYXNoIDA6IFxcdGV4dHNme1xcdGV4dHVwe0N9fV9uIl0sWzEsMCwiXFxtYXBzdG8iXSxbMCwxLCJpX246U157bi0xfVxcdG8gRF5uIl0sWzIsMSwiIHg6XFx0ZXh0c2Z7XFx0ZXh0dXB7Q319X24gXFx2ZGFzaCBkX254OlxcdGV4dHNme1xcdGV4dHVwe0N9fV97bi0xfSBcXFxcICB4OlxcdGV4dHNme1xcdGV4dHVwe0N9fV97bi0xfSgwKSBcXHZkYXNoIFxcdGV4dHNme1xcdGV4dHVwe0N9fV9uKHgpIFxcLCBcXHR5cGUiXSxbMSwxLCJcXG1hcHN0byJdXQ==
\[\begin{tikzcd}[sep=small]
	{0\to S^n} & \mapsto &  \vdash \textsf{\textup{Z}}_n \,\type \\
	{i_n:S^{n-1}\to D^n} & \mapsto &  x:\textsf{\textup{Z}}_{n-1} \vdash \textsf{\textup{C}}_n(x) \, \type
      \end{tikzcd}\]
    for $n\in \Z$.

    Again, we need to add some non-type axioms. For example, we need each $Z_n$ to contain an element $0$, and $C_n(0) = Z_n$, then each $C_n$ has an abelian group structure as in the case of bounded complexes.

\subsection{Topological spaces}

Here we recall the Quillen model structure on the category of topological spaces $\topo$ \cite{quillen2006axiomatic}. Recall that a map $f:X \to Y \in \topo$ is a \emph{weak homotopy equivalence} if for all $x \in X$ and $n \geq 1$ the induced map $f_*: \pi_n(X,x) \to \pi_n(Y,f(x))$ is an isomorphism of groups and for $n=0$ is a bijection. Additionally, the map $f$ is a \emph{Serre fibration} if for any $CW$-complex $W$ the following square has a diagonal filler:
% https://q.uiver.app/#q=WzAsNCxbMCwwLCJBXFx0aW1lc1xcezBcXH0iXSxbMCwxLCJBXFx0aW1lcyBbMCwxXSJdLFsxLDAsIlgiXSxbMSwxLCJZIl0sWzIsMywiZiJdLFswLDFdLFswLDJdLFsxLDNdLFsxLDIsIiIsMSx7InN0eWxlIjp7ImJvZHkiOnsibmFtZSI6ImRhc2hlZCJ9fX1dXQ==
\[\begin{tikzcd}
	{A\times\{0\}} & X \\
	{A\times [0,1]} & Y.
	\arrow[from=1-1, to=1-2]
	\arrow[from=1-1, to=2-1]
	\arrow["f", from=1-2, to=2-2]
	\arrow[dashed, from=2-1, to=1-2]
	\arrow[from=2-1, to=2-2]
      \end{tikzcd}\]

    \begin{theorem}
      The category $\topo$ has a model category structure such that:
      \begin{enumerate}
      \item Weak equivalences are the weak homotopy equivalences.
      \item Fibrations are the Serre fibrations.
      \item Cofibrations are the maps with the left lifting property against trivial fibrations.
      \end{enumerate}
      Moreover, this model structure is cofibrantly generated. The generating cofibrations is the set of boundary inclusions $\{ S^{n-1} \to D^{n} | n\in \N \}$. The set $\{ D^n \to D^n\times [0,1]| n\in \N  \}$ generates trivial cofibrations.
    \end{theorem}

    We can immediately write some of the relevant type axiom of the resulting theory:

    \begin{itemize}
    \item $\vdash 0\textsf{-CW}  \, \type$.
    \item $x,y: 0\textsf{-CW} \vdash 1\textsf{-CW}(x,y)  \, \type$.
    \item $x : 0\textsf{-CW}, \gamma: 1\textsf{-CW}(x,x) \vdash 2\textsf{-CW}(x,\gamma)  \, \type$.
    \item $\vdots$
    \end{itemize}

    Note that the language associated to the model structure allows us to express properties of topological spaces without relying on a specific set of axioms. However, this presents a limitation coming from the fact that we do not have an equality type. It is a classic result that there is no finitary presentation of a topological space. But in our setting, when $X$ is a CW-complex \ie it is obtained as an iterated pushout of cells, then a continuous map $D^n \to X$ can be written in the language above.
    
    \begin{example}
      We cannot write the formula $$\exists x: 0\textsf{-CW}\, \forall y: 0\textsf{-CW},\, x=y. $$
   The only possibility is to write $$\forall x,y:0\textsf{-CW}\, \exists \alpha:1\textsf{-CW}(x,y),\top$$ which simply says that a space is path-connected. Moreover, we can not say that two paths $\alpha,\beta:1\textsf{-CW}(x,x)$ are homotopic in the usual sense, only that there exists $\sigma:3\textsf{-CW}$ connecting the two loops.
    \end{example}

    \subsection{Kan complexes and quasi-categories}

    In this section, we analyze two very well-known model structures on the category of simplicial sets $\sset$; the Kan--Quillen and the Joyal model structures. One interesting feature is that we obtain the same theory for both models, but under the light of \cref{invariance-theorems} meaningful statements are delimited by the fibrant objects. In the first model we are interested in Kan complexes, while in the second model in the quasi-categories. The first model appears in \cite{quillen2006axiomatic} and the second in \cite{joyal2008volii}. These are the first references one can find, but the literature is ample for both models.
    
    Recall that a map $f:X \to Y$ between simplicial sets is a \emph{Kan fibration} if it has the right lifting property for all horn inclusions, \ie the solid diagram below a diagonal filler
    % https://q.uiver.app/#q=WzAsNCxbMCwwLCJcXExhbWJkYV5rW25dIl0sWzAsMSwiXFxEZWx0YVtuXSJdLFsxLDAsIlgiXSxbMSwxLCJZIl0sWzIsMywiZiJdLFswLDJdLFswLDEsIiIsMix7InN0eWxlIjp7InRhaWwiOnsibmFtZSI6Imhvb2siLCJzaWRlIjoidG9wIn19fV0sWzEsM10sWzEsMiwiIiwxLHsic3R5bGUiOnsiYm9keSI6eyJuYW1lIjoiZGFzaGVkIn19fV1d
\[\begin{tikzcd}
	{\Lambda^k[n]} & X \\
	{\Delta[n]} & Y
	\arrow[from=1-1, to=1-2]
	\arrow[hook, from=1-1, to=2-1]
	\arrow["f", from=1-2, to=2-2]
	\arrow[dashed, from=2-1, to=1-2]
	\arrow[from=2-1, to=2-2]
      \end{tikzcd}\]
    for all $0 \leq k \leq n \in \N$. The simplicial set $X$ is a \emph{Kan complex} if the unique map to the terminal presheaf is a Kan fibration. This is the result from \cite{quillen2006axiomatic}:
    \begin{theorem}
      The category of simplicial sets $\sset$ carries a model structure in which:
      \begin{enumerate}
      \item Weak equivalences are maps $f:X \to Y$ whose geometric realization $|f|: |X| \to |Y|$ is a weak homotopy equivalence in the category of topological spaces $\topo$. These are called Kan equivalences.
      \item Fibrations are the Kan fibrations.
      \item Cofibrations are the monomorphisms.
      \end{enumerate}
      The class of cofibrations is generated by $I \coloneqq \{ \partial\Delta[n] \hookrightarrow \Delta[n]| n\in \N \}$ and trivial cofibrations are generated by $J \coloneqq \{ \Lambda^k[n] \to \Delta[n]| n \in \N \text{ and } 0\leq k \leq n \}.$
    \end{theorem}

    Similarly, a map $f:X \to Y$ between simplicial sets is an \emph{inner Kan fibration} if it has the right lifting property for all inner horn inclusions, \ie the solid diagram below a diagonal filler
    % https://q.uiver.app/#q=WzAsNCxbMCwwLCJcXExhbWJkYV5rW25dIl0sWzAsMSwiXFxEZWx0YVtuXSJdLFsxLDAsIlgiXSxbMSwxLCJZIl0sWzIsMywiZiJdLFswLDJdLFswLDEsIiIsMix7InN0eWxlIjp7InRhaWwiOnsibmFtZSI6Imhvb2siLCJzaWRlIjoidG9wIn19fV0sWzEsM10sWzEsMiwiIiwxLHsic3R5bGUiOnsiYm9keSI6eyJuYW1lIjoiZGFzaGVkIn19fV1d
\[\begin{tikzcd}
	{\Lambda^k[n]} & X \\
	{\Delta[n]} & Y
	\arrow[from=1-1, to=1-2]
	\arrow[hook, from=1-1, to=2-1]
	\arrow["f", from=1-2, to=2-2]
	\arrow[dashed, from=2-1, to=1-2]
	\arrow[from=2-1, to=2-2]
      \end{tikzcd}\]
    for all $0 < k < n \in \N$. The simplicial set $X$ is a \emph{quasi-category} if the unique map to the terminal presheaf is an inner Kan fibration. This is the result from \cite{joyal2008volii}:
    \begin{theorem}
      The category of simplicial sets $\sset$ carries a model structure in which:
      \begin{enumerate}
      \item Weak equivalences are the weak categorical equivalences.
      \item Fibrations are the inner Kan fibrations.
      \item Cofibrations are the monomorphisms.
      \end{enumerate}
      The class of cofibrations is generated by $I \coloneqq \{ \partial\Delta[n] \hookrightarrow \Delta[n]| n\in \N \}$, the set of boundary inclusions.
    \end{theorem}

    Notice that both model structures have the same class of generating cofibrations. Hence, we expect that they have the same theories. We get a type for each cofibration in $I$. The first elements in this list of types are:
    \begin{itemize}
    \item $\vdash 0\simp \,\type$.
    \item $\sigma_0,\sigma_1:0\simp \vdash 1\simp(\sigma_0,\sigma_1) \, \type$.
    \item $\sigma_0,\sigma_1,\sigma_2:0\simp, \quad \sigma_{01}:1\simp(\sigma_0,\sigma_1), \quad\sigma_{12}:1\simp(\sigma_1,\sigma_2), \quad \sigma_{02}:1\simp(\sigma_0,\sigma_2) \vdash 2\simp(\sigma_0,\sigma_1,\sigma_2,\sigma_{01},\sigma_{12},\sigma_{02}) \, \type$.
    \item $\vdots$
    \end{itemize}

    The picture we should have in mind on the dependency of types is the usual one about simplices. A 1-simplex depend on two 0-simplicies, a 2-simplex consist of three 0-simplicies and three 1-simplicies connecting them, and so forth.

    One can see that the faces of an $n$-simplex are obtained via the dependencies, or context in which is defined. However, we can still adopt the usual notation for faces. Specifically, for each $n\in \N$ one has the faces $d_i(\sigma_{0123\dots (i-1)i(i+1)\dots n})\coloneqq \sigma_{0123\dots (i-1)(i+1)\dots n}$ is the $(n-1)$-simplex ``opposite'' to the $i$-th vertex of $\sigma_{012\dots n} $. This simplex is already defined, and it is used in the construction of $\sigma_{012\dots n}$. We emphasize that this is not part of the theory, but just a convenient and familiar shortcut.

    The \emph{degeneracy} operator is part of the theory and needs to be introduced: \[ \sigma_{0123\dots (i-1)i(i+1)\dots n}:n\simp \vdash s_i(\sigma_{0123\dots (i-1)i(i+1)\dots n}):(n+1)\simp\]
    where $s_i(\sigma_{0123\dots (i-1)i(i+1)\dots n})\coloneqq \sigma_{0123\dots (i-1)\hat{i}(i+1)\dots n}$ is the $(n+1)$-simplex that contains $\sigma_{0123\dots (i-1)i(i+1)\dots n}$ as its $i$-th and $(i+1)$-faces. We have one of such operations for $0\leq i \leq n$. The way we have introduced this operation is not completely correct as we are missing the dependencies for $n\simp$ and $(n+1)\simp$ and the context, nevertheless we can infer them. For example: \[x,y:0\simp, \, f:1\simp(x,y)\vdash s_1(f):2\simp(x,y,y,f,s_0(y),f)\]
    where $s_0(y)$ is the degeneracy of $y$ or the ``identity of $y$'' and is constructed previously.

    We also expect the simplicial identities to be satisfied. However, we do not need to postulate all of them as axioms of the theory since some of them are given via dependencies or by the typing of the operations. The only equation we postulate is $s_is_j=s_{j+1}s_i$ for $i\leq j$. On the one hand, the usual equation $d_id_j=d_{j-1}d_i$ for $i<j$ only involves faces, therefore everything is encoded in the dependency. On the other hand, the equation
    \[
      d_is_j=
      \begin{cases}
        s_{j-1}d_i, & i<j \\
        Id, & i=j,j+1 \\
        s_jd_{i-1}, & i>j+1
      \end{cases}
    \]
    is valid from the definition of degeneracies and dependency of the faces.

    We should note again that there is no visible difference in the language of the Joyal model structure and the language of the Kan-Quillen model structure as these have the same cofibrations. The only difference is that the language of the Kan-Quillen model structure is only meant to be applied to Kan complexes, while the language of the Joyal model structure can be applied to quasi-categories.

    \begin{example} \label{contractible:example}
    A Kan complex $X$ is contractible if it is weakly homotopy equivalent to $\pointtyp$. This is just to say that for any $n\geq 0$ we can find a lift
    % https://q.uiver.app/#q=WzAsNCxbMCwwLCJcXHBhcnRpYWxcXERlbHRhXm4iXSxbMCwxLCJcXERlbHRhXm4iXSxbMSwwLCJYIl0sWzEsMSwiXFxwb2ludHR5cCJdLFswLDJdLFswLDFdLFsxLDNdLFsyLDNdLFsxLDIsIiIsMSx7InN0eWxlIjp7ImJvZHkiOnsibmFtZSI6ImRhc2hlZCJ9fX1dXQ==
\[\begin{tikzcd}
	{\partial\Delta^n} & X \\
	{\Delta^n} & \pointtyp
	\arrow[from=1-1, to=1-2]
	\arrow[from=1-1, to=2-1]
	\arrow[from=1-2, to=2-2]
	\arrow[dashed, from=2-1, to=1-2]
	\arrow[from=2-1, to=2-2]
      \end{tikzcd}\]
    which expresses the fact that the unique map $X \to \pointtyp$ is a weak homotopy equivalence. Note that $X$ must satisfy an infinite number of conditions:
    \begin{itemize}
    \item For $n=0$ this says: $\exists \sigma_0: 0\simp, $
    \item For $n=1$ this says: $\forall \sigma_0,\sigma_1:0\simp, \, \exists \sigma_{01}: 1\simp(\sigma_0,\sigma_1),$
    \item For $n=2$ this says:
      \begin{multline*}
        \forall \sigma_0,\sigma_1:0\simp \, \sigma_{01}: 1\simp(\sigma_0,\sigma_1), \sigma_{12}: 1\simp(\sigma_1,\sigma_2), \\ \sigma_{02}: 1\simp(\sigma_0,\sigma_2), \exists \sigma_{012}:2\simp(\sigma_0,\sigma_1,\sigma_2,\sigma_{01},\sigma_{12},\sigma_{02}).
      \end{multline*}
    \end{itemize}
    One continues unpacking the conditions and takes the infinite conjunction of the formulas.

    Alternatively, we can note that the domain of a trivial cofibration $i_n:\partial \Delta^n \cofibration \Delta^n$ give us the context, or hypotheses, of the statement. In this case, the codomain gives us the type where the conclusion holds. If we accept this, let us write,  $t\in \Lb^\sset(\partial \Delta^n)$ for a term (formula) which expresses a property in the context $ \partial\Delta^n$, similarly $t'\in \Lb^\sset(\Delta^n)$ for a formula in the context $\Delta^n$. With this convention, we do not have to use the theory explicitly. When we apply the quantifiers, universal or existential, we move these formulas to $\Lb^\sset(\emptyset)$ and ask whether a fibrant object satisfies the resulting formula. For $ \top\in \Lb^\sset(\Delta^n)$ then for $i_n: \partial \Delta^n \cofibration \Delta^n$ and $j_n: \emptyset \to \partial \Delta^n$ we get maps \[\exists_{i_n}: \Lb^\sset(\Delta^n)  \to \Lb^\sset(\partial \Delta^n) \text{ and }\forall_{j_n}: \Lb^\sset(\partial \Delta^n) \to \Lb^\sset(\emptyset),\]
    and thus the formula $\forall_{j_n}\exists_{i_n}\top: \Lb^\sset(\emptyset)$ would say that a Kan complex satisfies the corresponding lifting problem. For a Kan complex to be contractible, it needs to satisfy formulas for all $n \in \N$. Therefore,
    \[\iscontr(X) \coloneqq (X \vdash \bigwedge_{n\in \N}\forall_{j_n}\exists_{i_n}\top).\]
  \end{example}

  We are now convinced that contractibility can be written in the language we just described. \Cref{contractible:example} indicates that we might not need to get an explicit syntax from the generating set of cofibrations. Instead, we might just quantify over the required cofibrations. The main reason this is preferable over explicitly defining the syntax is that in general such syntax is complicated to write, see for example \cref{segal-spaces:example}. The previous example shows that we might prefer to choose simplifications that make our sentences easier to read. This is specially true for contexts like the ones covered in the following section.
    
\subsection{Reedy languages} \label{reedy-languages:sec}

The purpose of this subsection is to describe the language for the category $\catM^{K^\op}$, where $K$ is a Reedy category and $\catM$ is a model category whose language we know. This encompasses some of the previous examples and opens the door to further applications.

Recall that if $\catM$ is a cofibrantly generated model category whose cofibrations are generated by a well-founded set of cofibrations $I$, then for each cofibration $A \cofibration B \in I$ we can associate a type introduction axiom $ \bar{A} \vdash \bar{B} \, \type$, where $\bar{A}$ is a well-formed context previously constructed.

Let $K$ be a Reedy category with degree function $\deg:K \to \omega $. This restriction is artificial since we could consider more general Reedy categories, however, for the examples this construction is aimed at, this is enough. The objects of $K$ have a well-founded order relation induced by the degree function.

\begin{construction} \label{reedy-language-generators}
  Let $\partial \yoneda_k$ be the latching object of the representable functor $\yoneda_k$ and $d_k: \partial \yoneda_k \to \yoneda_k$ the induced map. There is a bifunctor
  \[
    \otimes: \set^{K^\op} \times \catM \to \catM^{K^\op}
  \]
  defined by $( A \otimes X )_k \coloneqq \coprod_{A_k} X$. Let $I$ be as above, given $i: X \to Y \in I$ and $k\in K$ we apply the usual Leibniz construction and obtain the dashed arrow below
  % https://q.uiver.app/#q=WzAsNSxbMCwwLCJcXHBhcnRpYWxcXHlvbmVkYV9rIFxcb3RpbWVzIFgiXSxbMSwwLCJcXHlvbmVkYV9rIFxcb3RpbWVzIFgiXSxbMCwxLCJcXHBhcnRpYWxcXHlvbmVkYV9rIFxcb3RpbWVzIFkiXSxbMiwyLCJcXHlvbmVkYV9rIFxcb3RpbWVzIFkiXSxbMSwxLCJcXHBhcnRpYWxcXHlvbmVkYV9rIFxcb3RpbWVzIFlcXGNvcHJvZF97XFxwYXJ0aWFsXFx5b25lZGFfayBcXG90aW1lcyBYfSBcXHlvbmVkYV9rIFxcb3RpbWVzIFgiXSxbMCwyXSxbMCwxXSxbMSwzXSxbMiwzXSxbMSw0XSxbMiw0XSxbNCwzLCJkX2tcXGhhdHtcXG90aW1lc30gaSIsMCx7InN0eWxlIjp7ImJvZHkiOnsibmFtZSI6ImRhc2hlZCJ9fX1dXQ==
\[\begin{tikzcd}[sep=small]
	{\partial\yoneda_k \otimes X} & {\yoneda_k \otimes X} \\
	{\partial\yoneda_k \otimes Y} & {\partial\yoneda_k \otimes Y\coprod_{\partial\yoneda_k \otimes X} \yoneda_k \otimes X} \\
	&& {\yoneda_k \otimes Y.}
	\arrow[from=1-1, to=1-2]
	\arrow[from=1-1, to=2-1]
	\arrow[from=1-2, to=2-2]
	\arrow[from=1-2, to=3-3,bend left]
	\arrow[from=2-1, to=2-2]
	\arrow[from=2-1, to=3-3, bend right=25]
	\arrow["{d_k\hat{\otimes} i}", dashed, from=2-2, to=3-3]
\end{tikzcd}\]

We now consider the set of maps $K \hat{\otimes} I \coloneqq \{ d_k \hat{\otimes} i | k \in K,\,  i \in I \}$. By identifying each map $d_k \hat{\otimes} i \in K\hat{\otimes} I $ with a pair $(k,i)$, we see that $K \hat{\otimes} I$ is also a well-founded set with a relation, which we denote by $\leq_\otimes$. Here the relation is defined entry by entry \ie $(k',i')\leq_\otimes (k,i)$ if and only if $\deg(k') \leq \deg(k)$ and $i'\leq_I i$, where $\leq_I$ is the well-founded relation on $I$.
\end{construction}

The previous construction is further justified by \cite[Proposition 2.3.22]{barton2019model} for premodel categories, but a similar description is abundant in the literature for Quillen model categories.

\begin{proposition}
  The Reedy weak factorization system on $\catM^{K^\op}$ is generated by $K \hat{\otimes} I$, and therefore the Reedy model category structure on $\catM^{K^\op}$ is combinatorial whenever $\catM$ is combinatorial.
\end{proposition}

A useful result we can have in mind is the following:

\begin{lemma} \label{transpose-lebniz:lemma}
  Given any $i:A \to B \in \catM$, a morphism $f:X \to Y \in \catM^{K^\op}$ has the lifting property with respect to $d_k \hat\otimes i$, if and only if $\hat f^k:X_k \to Y_k\times_{M_kY}M_kX$ has the right lifting property with respect to $i$.
\end{lemma}
\begin{proof}
  As written, this is \cite[Lemma 2.3.21]{barton2019model}, but it is also a classical result found in \cite{hovey1999models}.
\end{proof}

\begin{remark} \label{matching:remark}
  The matching objects in \cref{transpose-lebniz:lemma} are computed with respect to the Reedy structure of $K^\op$. This means that the relevant diagram in $M_kX$ is given by maps in $(K^\op)_{-}=K_+$.
\end{remark}

\begin{observation}
  Many models for higher categories are built starting with presheaves over a Reedy category. Then to obtain the desired model one takes a left Bousfield localization for an appropriate class of maps. Importantly, this localization does not change the generating cofibrations. This is just to say that the language of $\catM^{K^\op}$ remains unchanged after localization.
\end{observation}

The cofibrations for the Reedy model structure are usually rather complicated, we can sometimes proceed as in \cref{contractible:example}. This is, if $\Gamma' \cofibration \Gamma$ is a generating cofibration, then we might simply consider a formula $\phi' \in \Lb^{\catM^{K^\op}}(\Gamma')$ or $\phi \in \Lb^{\catM^{K^\op}}(\Gamma)$ with no explicit description of the type associated to the cofibration.

As an interesting case, in the following section we examine the Reedy language for Segal spaces. However, the construction applies to any other model category constructed similarly.

\subsection{Segal spaces} \label{segal-spaces:example}

We denote $\bisim \coloneqq [\Delta^\op, \sset]=[\Delta^\op \times \Delta^\op,\set]$ as the category of simplicial spaces, or bisimplicial sets. This category has two model structures that are obtained as left Bousfield localizations of the Reedy model structure. For both of these localizations, we use the Kan--Quillen model structure from the previous section. Recall that this model structure is cofibrantly generated. The set of generating cofibrations is the set of boundary inclusions. We will use the following facts and notation.

\begin{itemize}
\item There is an adjunction of two variables $\- \Box \-: \sset \times \sset \to \bisim $ defined as $(X\Box Y)_{mn}\coloneqq X_m \times Y_n$ for each $m,n \in \N$. This is called the box product.
\item $\sset$ can be seen as vertically embedded into $\bisim$. If $X\in \sset$, then it can be seen as a simplicial space $X \Box \Delta[0]$. There is also a horizontal embedding by setting $\Delta[0] \Box X$.

\item For $[m]\in \Delta $ we write $F(n)\coloneqq \Delta[n] \Box \Delta[0]$ and $\partial F(n) \coloneqq \partial \Delta[n] \Box \Delta[0]$.
  
\item The simplicial spaces $F(n)$ represent the $n$-th mapping space functors, respectively $Map(F(n),X) = X_n$.
\end{itemize}

There is map $\iota:F(1)\coprod_{F(0)} \cdots \coprod_{F(0)} F(1)  \to F(n)$, where the colimit on left has $n$ factors. The following two model category structures were constructed by Rezk \cite{rezk}.

\begin{theorem}
  The category admits a unique simplicial model category structure such that:
  \begin{enumerate}
  \item The cofibrations are the monomorphisms.
  \item Fibrant objects are simplicial spaces $X$ such that the map
    \[X_n \to X_1 \times_{X_0} \cdots \times_{X_0} X_1 \]
    induced by $\iota$ is a Kan equivalence. The fibrant objects are called Segal spaces.
  \item The weak equivalences are the maps $f:X \to Y \in \bisim$ such that $$Map(f,W): Map(Y,W) \to Map(X,W)$$ is a Kan equivalence for every Segal space $W$.
  \item A map $f:X \to Y$ between Segal spaces is a fibration (weak equivalence) if and only if is a Reedy fibration (Reedy weak equivalence).
  \end{enumerate}
\end{theorem}

Recall that $\catJ$ denotes the category with two objects and two arrows that are mutually inverses. It is usual to denote by $E(1)$ to the Segal space which is obtained by considering the nerve $N\catJ$ as a discrete simplicial space. This produces a map $F(1) \to E(1)$.

\begin{theorem}
  The category admits a unique simplicial model category structure such that:
  \begin{enumerate}
  \item The cofibrations are the monomorphisms.
  \item Fibrant objects are Segal spaces $X$ such that the map $$Map(E(1),X) \to Map(F(0),X) $$ is a Kan equivalence. The fibrant objects are called complete Segal spaces.
  \item The weak equivalences are the maps $f:X \to Y \in \bisim$ such that $$Map(f,W): Map(Y,W) \to Map(X,W)$$ is a Kan equivalence for every complete Segal space $W$.
  \item A map $f:X \to Y$ between complete Segal spaces is a fibration (weak equivalence) if and only if is a Reedy fibration (Reedy weak equivalence).
  \end{enumerate}
\end{theorem}

These models are cofibrantly generated. The set of generating cofibrations can be described using the box product \cite[Proposition 2.2]{joyal2007quasi}. This set is given by $\hat I \coloneqq \{ d_m \hat\Box d_n | m,n\in \N \}$. Explicitly, a map in $\hat I$ is of the form
\[
  d_m \hat\Box d_n: \partial\Delta[m] \Box \Delta[n] \coprod_{\partial\Delta[m] \Box \partial\Delta[n]} \Delta[m] \Box \partial \Delta[n]  \to \Delta[m] \Box \Delta[n]
\]
We can obtain the generalized algebraic theory for (complete) Segal space. The domains of these maps provide the context in which a new type is formed. To get a sense of the theory, consider the following picture of a bisimplicial set $X$:
% https://q.uiver.app/#q=WzAsNyxbMCwwLCJYX3swMH0iXSxbMCwxLCJYX3sxMH0iXSxbMSwwLCJYX3swMX0iXSxbMiwwLCJcXGRvdHMiXSxbMSwxLCJYX3sxMX0iXSxbMiwyLCJcXHRleHR7fSJdLFswLDIsIlxcdmRvdHMiXSxbMCwxXSxbMSwwLCIiLDAseyJvZmZzZXQiOjJ9XSxbMSwwLCIiLDEseyJvZmZzZXQiOi0yfV0sWzAsMl0sWzIsMCwiIiwxLHsib2Zmc2V0IjotMn1dLFsyLDAsIiIsMSx7Im9mZnNldCI6Mn1dLFsyLDMsIiIsMSx7Im9mZnNldCI6LTJ9XSxbMiwzLCIiLDEseyJvZmZzZXQiOjJ9XSxbMywyXSxbMywyLCIiLDEseyJvZmZzZXQiOjN9XSxbMywyLCIiLDEseyJvZmZzZXQiOi0zfV0sWzIsNF0sWzQsMiwiIiwxLHsib2Zmc2V0IjoyfV0sWzQsMiwiIiwxLHsib2Zmc2V0IjotMn1dLFsxLDRdLFs0LDEsIiIsMSx7Im9mZnNldCI6Mn1dLFs0LDEsIiIsMSx7Im9mZnNldCI6LTJ9XSxbNiwxLCIiLDEseyJvZmZzZXQiOi0zfV0sWzYsMV0sWzYsMSwiIiwxLHsib2Zmc2V0IjozfV0sWzEsNiwiIiwxLHsib2Zmc2V0IjoyfV0sWzEsNiwiIiwxLHsib2Zmc2V0IjotMn1dLFs0LDUsIlxcZGRvdHMiLDEseyJzdHlsZSI6eyJib2R5Ijp7Im5hbWUiOiJub25lIn0sImhlYWQiOnsibmFtZSI6Im5vbmUifX19XV0=
\[\begin{tikzcd}
	{X_{00}} & {X_{01}} & \dots \\
	{X_{10}} & {X_{11}} \\
	\vdots && {\text{}}
	\arrow[from=1-1, to=1-2]
	\arrow[from=1-1, to=2-1]
	\arrow[shift left=2, from=1-2, to=1-1]
	\arrow[shift right=2, from=1-2, to=1-1]
	\arrow[shift left=2, from=1-2, to=1-3]
	\arrow[shift right=2, from=1-2, to=1-3]
	\arrow[from=1-2, to=2-2]
	\arrow[from=1-3, to=1-2]
	\arrow[shift right=3, from=1-3, to=1-2]
	\arrow[shift left=3, from=1-3, to=1-2]
	\arrow[shift right=2, from=2-1, to=1-1]
	\arrow[shift left=2, from=2-1, to=1-1]
	\arrow[from=2-1, to=2-2]
	\arrow[shift right=2, from=2-1, to=3-1]
	\arrow[shift left=2, from=2-1, to=3-1]
	\arrow[shift right=2, from=2-2, to=1-2]
	\arrow[shift left=2, from=2-2, to=1-2]
	\arrow[shift right=2, from=2-2, to=2-1]
	\arrow[shift left=2, from=2-2, to=2-1]
	\arrow["\ddots"{description}, draw=none, from=2-2, to=3-3]
	\arrow[shift left=3, from=3-1, to=2-1]
	\arrow[from=3-1, to=2-1]
	\arrow[shift right=3, from=3-1, to=2-1]
      \end{tikzcd}\]
    The arrows indicate the degeneracy and face maps. Now we go back to consider the maps $d_m\Box d_n$. When $m=n=0$ then we simply get a map $\emptyset \to \Delta[0]\Box \Delta[0]$, and allow us to introduce the type
    \[\vdash \spc_{00} \, \type. \]
    When $n=0$ the resulting subset of maps is of the form \[ d_m \hat \Box \Delta[0]: \partial \Delta[m] \Box \Delta[0] \to \Delta[m] \Box \Delta[0]. \]
    In this setting, since for $m=0$ we obtain the previous cofibration $\emptyset \to \pointtyp$, for each $m\geq 1$ we can write the following types:
\begin{itemize}
\item $x,y: \spc_{00} \vdash \spc_{10}(x,y) \, \type$.
\item $x,y,z:\spc_{00}, f:\spc_{10}(x,y),g:\spc_{10}(y,z),h:\spc_{10}(x,z) \vdash \spc_{20}(x,y,z,f,g,h)$.
\item $\vdots$
\end{itemize}
When $m=0$ we obtain the theory of the categorical direction. Now suppose that $m=1=n$, then resulting generating cofibration is the map
    \[
  d_1 \hat\Box d_1: \partial\Delta[1] \Box \Delta[1] \coprod_{\partial\Delta[1] \Box \partial\Delta[1]} \Delta[1] \Box \partial \Delta[1]  \to \Delta[1] \Box \Delta[1]
\]
From here we see that the type associated to this map has the following form:
\begin{multline*}
x_0,x_1,x_2,x_3:\spc_{00}, f_{01}:\spc_{01}(x_0,x_1), f_{23}:\spc_{01}(x_2,x_3),f_{02}:\spc_{10}(x_0,x_2), \\ f_{13}:\spc_{10}(x_1,x_3) \vdash \spc_{11}(x_0,x_1,x_2,x_3,f_{01},f_{23},f_{02},f_{13}) \, \type. 
\end{multline*}
We think of this new type as the type of squares where the solid boundary is the given context
\begin{center}
  \begin{tikzpicture}[scale=0.7] % Adjust the scale here
    % Fill the interior of the square while avoiding the vertices
    \fill[cyan!30] (0.1,0.1) rectangle (1.9,1.9);

    % Draw the square with shorter arrows
    \draw[->] (0.4,2) -- (1.6,2) node[midway, above] {$f_{01}$};
    \draw[->] (0,1.6) -- (0,0.4) node[midway, left] {$f_{02}$};
    \draw[->] (2,1.6) -- (2,0.4) node[midway, right] {$f_{13}$};
    \draw[->] (0.4,0) -- (1.6,0) node[midway, below] {$f_{23}$};

    % Place the vertices
    \node at (0,2) (A) {$x_0$};
    \node at (2,2) (B) {$x_1$};
    \node at (0,0) (C) {$x_2$};
    \node at (2,0) (D) {$x_3$};
\end{tikzpicture}
\end{center}
For different $m,n$ the context are simply more involved, but the dependencies can be inferred. Note we still need to add the degeneracy operators satisfying the usual axioms. We can see that as we build more complex contexts, it will be computationally difficult to obtain an explicit description of the types. We might instead proceed as in \cref{contractible:example}.

\begin{example}
  Two elements $x,y:\spc_{00}$ are said to be \emph{homotopic} if there exists $\alpha:\spc_{10}(x,y)$. This sentence only involves types in the language of Segal spaces. In contrast to topological spaces, we can express the fact that two maps are homotopic.
\end{example}

\begin{remark}
  Note in particular that the language of spaces or Kan complexes is available for us to use. This in combination with our construction in \cref{reedy-languages:sec} allow us to realize many properties of (complete) Segal spaces, for example the ones found in \cite{rasekh2023yoneda}, are written in this language.
\end{remark}

\subsection{Functors and Isofibrations}

We denote $[1] \coloneqq \{0 \to 1\}$ the category with two objects and single non-identity arrow. This category can be viewed as a Reedy category in two ways. The first one respects the direction of the arrow, so we take $[1]_+$ to be the non-identity map, while for the second we take the same map to be in $[1]_-$. Recall that if $K$ is a Reedy category, then $K^\op$ is also a Reedy category where $(K^\op)_+=K_{-}$ and $(K^\op)_{-}=K_+$. In order to match the computations of \cref{reedy-language-generators}, we use the same notation as there. By which we mean that for a model category $\catC$ we use $\catC^{([1]_+)^\op}$ and $\catC^{([1]_{-})^\op}$ with the corresponding Reedy model structures, ignoring the fact that $\catC^{([1]_+)^\op}=\catC^{[1]_{-}}$ and $\catC^{([1]_{-})^\op}=\catC^{[1]_+}$.

\begin{proposition}
  The Reedy model structure on $\catC^{([1]_{-})^\op}_{Reedy}$ coincides with the projective model structure. In particular, weak equivalences and fibrations are the level-wise weak equivalences and fibrations in $\catC$.
\end{proposition}
\begin{proof}
  This is a classical and well-known result.
\end{proof}
We are interested in the particular case of $\catC = \cat$. It is immediate to see that all objects are fibrant. The language we obtain should be the language for functors. Since $\cat$ is cofibrantly generated by $I=\{ \mathbf{0} \overset{u}{\to} \mathbf{1},\, \{0\}\sqcup \{1\} \overset{v}{\to} \mathbf{2},\, P \overset{w}{\to} \mathbf{2} \} $ we have that $[1]\hat\otimes I $ generates $\catC^{([1]_{-})^\op}_{Reedy}$, by \cref{reedy-language-generators}. This gives us the set of maps \[\{d_0\hat\otimes u, d_0\hat\otimes v, d_0\hat\otimes w, d_1\hat\otimes u, d_1\hat\otimes v, d_1\hat\otimes w \}.\]
To explain what it means for a map $f:X\to Y$ to have the lifting property against these cofibration we can use \cref{transpose-lebniz:lemma}, for which we need the matching objects. We observe from \cref{matching:remark} that $M_0X=1=M_1X$ since $ ([1]_{-})_+$ has no non-identity maps, and the same applies to $Y$. Therefore, for $i\in I$ and $k=0,1$ we have $(d_k\hat\otimes i) \pitchfork f$ in $\cat^{[1]_{-}^\op}$ if and only if $i\pitchfork \hat f^k$, but $\hat f^k$ is either $X_0 \to Y_0$ or $X_1\to Y_1$. Diagrammatically we have:
% https://q.uiver.app/#q=WzAsOCxbMCwwLCJcXHBhcnRpYWxcXHlvbmVkYV9rIFxcb3RpbWVzIGEgXFxjb3Byb2Rfe1xccGFydGlhbFxceW9uZWRhX2sgXFxvdGltZXMgYX0gXFx5b25lZGFfayBcXG90aW1lcyBhIl0sWzAsMSwiXFx5b25lZGFfa1xcb3RpbWVzIGIiXSxbMSwwLCJYIl0sWzEsMSwiWSJdLFsyLDAsImEiXSxbMiwxLCJiIl0sWzMsMCwiWF9rIl0sWzMsMSwiWV9rIl0sWzAsMSwiZF9rXFxoYXRcXG90aW1lcyBpIiwyXSxbMiwzLCJmIl0sWzAsMl0sWzEsM10sWzQsNSwiaSIsMl0sWzYsNywiXFxoYXQgZl5rIl0sWzQsNl0sWzUsN10sWzEsMiwiIiwxLHsic3R5bGUiOnsiYm9keSI6eyJuYW1lIjoiZGFzaGVkIn19fV0sWzUsNiwiIiwxLHsic3R5bGUiOnsiYm9keSI6eyJuYW1lIjoiZG90dGVkIn19fV0sWzksMTIsIiIsMCx7InNob3J0ZW4iOnsic291cmNlIjo0MCwidGFyZ2V0Ijo0MH19XV0=
\[\begin{tikzcd}
	{\partial\yoneda_k \otimes b \coprod_{\partial\yoneda_k \otimes a} \yoneda_k \otimes a} & X & a & {X_k} \\
	{\yoneda_k\otimes b} & Y & b & {Y_k}
	\arrow[from=1-1, to=1-2]
	\arrow["{d_k\hat\otimes i}"', from=1-1, to=2-1]
	\arrow[""{name=0, anchor=center, inner sep=0}, "f", from=1-2, to=2-2]
	\arrow[from=1-3, to=1-4]
	\arrow[""{name=1, anchor=center, inner sep=0}, "i"', from=1-3, to=2-3]
	\arrow["{\hat f^k}", from=1-4, to=2-4]
	\arrow[dashed, from=2-1, to=1-2]
	\arrow[from=2-1, to=2-2]
	\arrow[dotted, from=2-3, to=1-4]
	\arrow[from=2-3, to=2-4]
	\arrow[shorten <=13pt, shorten >=13pt, Leftrightarrow, from=0, to=1]
\end{tikzcd}\]
Specializing to $Y=\mathbf{1}$, it gives us an idea of how types are introduced:
% https://q.uiver.app/#q=WzAsOSxbMCwwLCJcXG1hdGhiZiAwIl0sWzAsMSwiXFxtYXRoYmYgMSJdLFsxLDAsIlhfayJdLFsyLDAsIlxcezBcXH1cXHNxY3VwXFx7MVxcfSJdLFsyLDEsIlxcbWF0aGJmIDIiXSxbMywwLCJYX2siXSxbNCwwLCJQIl0sWzQsMSwiXFxtYXRoYmYgMiJdLFs1LDAsIlhfayJdLFswLDEsInUiLDJdLFswLDJdLFszLDQsInYiLDJdLFszLDVdLFs2LDcsInciLDJdLFs2LDhdXQ==
\[\begin{tikzcd}
	{\mathbf 0} & {X_k} & {\{0\}\sqcup\{1\}} & {X_k} & P & {X_k} \\
	{\mathbf 1} && {\mathbf 2} && {\mathbf 2}
	\arrow[from=1-1, to=1-2]
	\arrow["u"', from=1-1, to=2-1]
	\arrow[from=1-3, to=1-4]
	\arrow["v"', from=1-3, to=2-3]
	\arrow[from=1-5, to=1-6]
	\arrow["w"', from=1-5, to=2-5]
      \end{tikzcd}\]
    for $k=0,1$. This means that we introduce objects, arrows between two objects and equality between arrows to $X_0$ or $X_1$. This indicates that corresponding generating cofibration produce the following type axioms:
    % https://q.uiver.app/#q=WzAsMyxbMCwwLCJcXHZkYXNoIFhfayBcXCxcXHR5cGUiXSxbMSwwLCJhLGI6WF9rIFxcdmRhc2ggWF9rKGEsYikgXFwsIFxcdHlwZSJdLFsyLDAsImYsZzpYX2soYSxiKSBcXHZkYXNoIGY9X3tYX2t9ZyBcXCwgXFx0eXBlIl1d
\[\begin{tikzcd}[sep=tiny]
    {\vdash X_0 \,\type} & {a,b:X_0 \vdash X_0(a,b) \, \type} & {a,b:X_0,f,g:X_0(a,b) \vdash f=_{X_0}g \, \type} \\
    {\vdash X_1 \,\type} & {a,b:X_k \vdash X_k(a,b) \, \type} & {a,b:X_1,f,g:X_k(a,b) \vdash f=_{X_k}g \, \type}
      \end{tikzcd}\]
    and we introduce the operation symbol for the functor as an operation
    % https://q.uiver.app/#q=WzAsMixbMCwwLCJhOlhfMCBcXHZkYXNoIEZhOlhfMSJdLFsxLDAsImEsYjpYXzAsXFwsIGY6WF9rKGEsYikgXFx2ZGFzaCBGZjpYXzEoRmEsRmIpIl1d
\[\begin{tikzcd}[column sep=small]
	{a:X_0 \vdash Fa:X_1} & { f:X_0(a,b) \vdash Ff:X_1(Fa,Fb)}
      \end{tikzcd}\]
    On top of it, we add the usual axioms that ensure we have the expected behaviour with respect to the identity and composition operations. Let us call denote this language by $\Lb^{Fun}.$

    Now we examine the language for the other model structure.

\begin{proposition}
  The Reedy model structure on $\catC^{([1]_+)^\op}_{Reedy}$ coincides with the injective model structure. In particular, weak equivalences and cofibrations are the level-wise weak equivalences and cofibrations in $\catC$.
\end{proposition}
\begin{proof}
  The result is folklore.
\end{proof}
We find that fibrant objects are those such that $X_0 \to X_1$ is an isofibration. Therefore, the language in this case refers to isofibrations. Again, this model structure has generating cofibrations
\[\{d_0\hat\otimes u, d_0\hat\otimes v, d_0\hat\otimes w, d_1\hat\otimes u, d_1\hat\otimes v, d_1\hat\otimes w \}.\]
Next, observe that $\partial \yoneda_0=0$ and $\partial\yoneda_1=\yoneda_0$. We have the maps $d_0: 0 \to \yoneda_0$ and $d_1: \yoneda_0 \to \yoneda_1$. Therefore, if $i:a \to b \in I$, then this give us the following cofibrations
\begin{itemize}
\item $ \yoneda_0 \otimes a \to \yoneda_0 \otimes b$,
\item $\yoneda_1 \otimes a \coprod_{\yoneda_0\otimes a} \yoneda_0 \otimes b \to \yoneda_1 \otimes b$.
\end{itemize}
The map $ \yoneda_0 \otimes a \to \yoneda_0 \otimes b$ for $i\in I$ corresponds to the following type introduction:
\[\begin{tikzcd}[column sep=tiny]
    {\vdash X_0 \,\type} & {x,y:X_0 \vdash X_0(x,y) \, \type} & {x,y:X_0,f,g:X_0(x,y) \vdash f=_{X_0}g \, \type}
  \end{tikzcd}\]
which we can think of as a category. The analysis of the second map is more intricate. Let us denote the evaluation of the representables by $\yoneda_{k0}$ and $\yoneda_{k1}$ for $k=0,1$, and for simplicity we keep the $`\otimes' $ symbol. Evaluating the cofibration
$ \yoneda_1 \otimes a \coprod_{\yoneda_0\otimes a} \yoneda_0 \otimes b \to \yoneda_1 \otimes b
$ at $[1]_+^\op$ give us the square,
% https://q.uiver.app/#q=WzAsNCxbMCwwLCJcXHlvbmVkYV97MTF9IFxcb3RpbWVzIGEgXFxjb3Byb2Rfe1xceW9uZWRhX3sxMH1cXG90aW1lcyBhfSBcXHlvbmVkYV97MTB9IFxcb3RpbWVzIGIgIl0sWzAsMSwiXFx5b25lZGFfezExfSBcXG90aW1lcyBiIl0sWzEsMCwiXFx5b25lZGFfezAxfSBcXG90aW1lcyBhIFxcY29wcm9kX3tcXHlvbmVkYV97MDB9XFxvdGltZXMgYX0gXFx5b25lZGFfezAwfSBcXG90aW1lcyBiICJdLFsxLDEsIlxceW9uZWRhX3swMX0gXFxvdGltZXMgYiJdLFswLDJdLFswLDFdLFsyLDNdLFsxLDNdXQ==
\[\begin{tikzcd}
	{\yoneda_{11} \otimes a \coprod_{\yoneda_{10}\otimes a} \yoneda_{10} \otimes b } & {\yoneda_{01} \otimes a \coprod_{\yoneda_{00}\otimes a} \yoneda_{00} \otimes b } \\
	{\yoneda_{11} \otimes b} & {\yoneda_{01} \otimes b},
	\arrow[from=1-1, to=1-2]
	\arrow[from=1-1, to=2-1]
	\arrow[from=1-2, to=2-2]
	\arrow[from=2-1, to=2-2]
      \end{tikzcd}\]
    where the horizontal arrows are induced by the diagram $[1]_+^\op$. This simplifies to
    % https://q.uiver.app/#q=WzAsNCxbMCwwLCJhIl0sWzAsMSwiYiJdLFsxLDAsImFcXGNvcHJvZF9hYiJdLFsxLDEsImIiXSxbMCwxXSxbMiwzXSxbMCwyXSxbMSwzXV0=
\[\begin{tikzcd}
	a & {a\coprod_ab} \\
	b & b,
	\arrow[from=1-1, to=1-2]
	\arrow[from=1-1, to=2-1]
	\arrow[from=1-2, to=2-2]
	\arrow[from=2-1, to=2-2]
\end{tikzcd}\]
which we now compute for $i\in I$, so the pictures take the following form:
% https://q.uiver.app/#q=WzAsMTIsWzAsMCwiXFxtYXRoYmYgMCJdLFsxLDAsIlxcbWF0aGJmMSJdLFswLDEsIlxcbWF0aGJmIDEiXSxbMSwxLCJcXG1hdGhiZjEiXSxbMiwwLCJcXHswXFx9XFxzcWN1cFxcezFcXH0iXSxbMywwLCJcXG1hdGhiZjIiXSxbMiwxLCJcXG1hdGhiZjIiXSxbMywxLCJcXG1hdGhiZjIiXSxbNCwwLCJQIl0sWzQsMSwiXFxtYXRoYmYyIl0sWzUsMCwiXFxtYXRoYmYyIl0sWzUsMSwiXFxtYXRoYmYyIl0sWzAsMV0sWzIsM10sWzAsMl0sWzEsM10sWzQsNV0sWzYsN10sWzUsN10sWzQsNl0sWzgsMTBdLFsxMCwxMV0sWzgsOV0sWzksMTFdXQ==
\[\begin{tikzcd}
	{\mathbf 0} & \mathbf1 & {\{0\}\sqcup\{1\}} & \mathbf2 & P & \mathbf2 \\
	{\mathbf 1} & \mathbf1 & \mathbf2 & \mathbf2 & \mathbf2 & \mathbf2.
	\arrow[from=1-1, to=1-2]
	\arrow[from=1-1, to=2-1]
	\arrow[from=1-2, to=2-2]
	\arrow[from=1-3, to=1-4]
	\arrow[from=1-3, to=2-3]
	\arrow[from=1-4, to=2-4]
	\arrow[from=1-5, to=1-6]
	\arrow[from=1-5, to=2-5]
	\arrow[from=1-6, to=2-6]
	\arrow[from=2-1, to=2-2]
	\arrow[from=2-3, to=2-4]
	\arrow[from=2-5, to=2-6]
\end{tikzcd}\]
    From the above we deduce that the type axioms introduced by these cofibrations take, respectively, the following form:
% https://q.uiver.app/#q=WzAsMyxbMCwwLCJ4OlhfMFxcdmRhc2ggWF8xKHgpIFxcLCBcXHR5cGUiXSxbMSwwLCJ4LHk6WF8wLGY6WF8wKHgseSksYTpYXzEoeCksYjpYXzEoeSkgXFx2ZGFzaCBYXzEoYSxiKSBcXCxcXHR5cGUiXSxbMSwxLCJ4LHk6WF8wLGY6WF8wKHgseSksYTpYXzEoeCksYjpYXzEoeSksaixrOlhfMShhLGIpIFxcdmRhc2ggaj1fe1hfMShhLGIpfSBrIFxcLFxcdHlwZSJdXQ==
\[\begin{tikzcd}[row sep=small]
    {x:X_0\vdash X_1(x) \, \type,} & \\
    {x,y:X_0,f:X_0(x,y),a:X_1(x),b:X_1(y) \vdash X_1(a,b,f) \,\type,} & \\
    {x,y:X_0,f:X_0(x,y),a:X_1(x),b:X_1(y),j,k:X_1(a,b,f) \vdash j=_{X_1(a,b,f)} k \,\type.} &
\end{tikzcd}\]
Unlike the language for functors $\Lb^{Fun}$, here we do not need a symbol for $F:X_0\to X_1$. We denote this language for isofibrations as $\Lb^{Iso}$.

For the observation below, it will be useful to remember that given a functor $F:X \to Y$, an arrow $f:x\to y\in X$ is \emph{cartesian} if for any $h:x' \to y$ and $w:F(x') \to F(x)$ with $F(f) \circ w=F(h)$, there exists a unique $u:x'\to x$ such that $f \circ u=h$. The following diagram illustrates this definition:
% https://q.uiver.app/#q=WzAsOCxbMiwxLCJ4Il0sWzIsMywiRngiXSxbMywzLCJGeSJdLFszLDEsInkiXSxbMSwyLCJGeCciXSxbMSwwLCJ4JyJdLFswLDAsIlgiXSxbMCwzLCJZIl0sWzEsMiwiRmYiLDJdLFswLDMsImYiLDJdLFs0LDIsIkZoIl0sWzQsMSwidyIsMl0sWzUsMywiaCJdLFs1LDAsIlxcZXhpc3RzISB1IiwyLHsic3R5bGUiOnsiYm9keSI6eyJuYW1lIjoiZGFzaGVkIn19fV0sWzYsNywiRiIsMl0sWzAsMSwiIiwxLHsic2hvcnRlbiI6eyJzb3VyY2UiOjIwLCJ0YXJnZXQiOjYwfSwic3R5bGUiOnsidGFpbCI6eyJuYW1lIjoibWFwcyB0byJ9fX1dXQ==
\[\begin{tikzcd}[row sep=small]
	X & {x'} \\
	&& x & y \\
	& {Fx'} \\
	Y && Fx & Fy
	\arrow["F"', from=1-1, to=4-1]
	\arrow["{\exists! u}"', dashed, from=1-2, to=2-3]
	\arrow["\forall h",bend left=25, from=1-2, to=2-4]
	\arrow["f"', from=2-3, to=2-4]
	\arrow[shorten <=6pt, shorten >=19pt, maps to, from=2-3, to=4-3]
	\arrow["\forall w"', from=3-2, to=4-3]
	\arrow["Fh",bend left=15, from=3-2, to=4-4]
	\arrow["Ff"', from=4-3, to=4-4]
\end{tikzcd}\]
A \emph{Grothendieck fibration} is a functor $F:X \to Y$ such that for any $y\in Y$ and $f:a \to F(y)$, there exists a cartesian arrow $\phi_f:f^*y \to y$ such that $F(\phi_f)=f$. The functor $F:X\to Y$ is a \emph{Street fibration} if for any $y\in Y$ and $f:a \to F(y)$, there exists a cartesian arrow $\hat f:e\to y$ and an isomorphism $F(e)\cong a$ that makes the resulting triangle commutative.

\begin{remark}
  It is a classical result that a Grothendieck fibration is the same as a Street fibration which is also an isofibration. On the one hand, note that a Grothendieck fibration can be written in the language $\Lb^{Iso}$ of isofibrations, but not in $\Lb^{Fun}$ of functors since it contains an equality between objects, such equality is salvaged in $\Lb^{Iso}$ thanks to the dependencies. On the other hand, a Street fibration is a formula in $\Lb^{Fun}$. We also know that the two Reedy model structures on the category $\cat^{[1]}$ are Quillen equivalent. The above result can also be automatically obtained as an elementary application of \Fourth, whose proof is the heart of the next section.
\end{remark}

%%% Local Variables:
%%% mode: latex
%%% TeX-master: "main"
%%% End:

        % Invariance of language under Quillen equivalences

        \section{Language invariance under Quillen equivalences}\label{sec:invariance}

\subsection{The third and fourth invariance theorem}

The main goal of this section is to show two more invariance properties of the first order language from \cref{sec:wms_language}, that we can phrase informally\footnote{The precise statement is just below as \cref{invariance-theorems-34}.} as:

\begin{enumerate}
\item \Third: If two cofibrant objects $X$ and $Y$ are equivalent, then any formula in context $X$ can be translated into a formula in context $Y$.
\item \Fourth: If two (weak) model categories $\catM$ and $\catN$ are Quillen equivalent, then any formula in the language of $\catM$ can be translated into a formula in the language of $\catN$. 
\end{enumerate}

These ``translations'' are equivalent to the original formula in the sense that they are interpreted in the same way in any fibrant model, but they might not be equivalent in the more syntactic sense introduced in \cref{def:Lb}. More precisely, we introduce the following equivalence relation on formulas:

\begin{definition}\label{def:semantical_equiv}
  Let $A$ be a cofibrant object of $\catM$. Two formulas $\phi, \psi \in \Lb_\lambda^\catM(A) $ are said to be \emph{semantically equivalent} if for all fibrant objects $X \in \catM$ we have $|\phi|_X=|\psi|_X$. In this situation we write $\phi \approx \psi$.

  We define $h\Lb_\lambda^\catM(A)$ to be the quotient of $\Lb_\lambda^\catM(A)$ by the relation $\approx$. We easily check that this is still a Boolean algebra.
\end{definition}

By definition of $\approx$ we have that for $\phi, \psi \in \Lb_\lambda^\catM(\Gamma) $,  $\phi \approx \psi$ if and only if all maps $v:\Gamma \to X$ with $X$ fibrant
\[ \Gamma \vdash \phi(v) \Leftrightarrow \Gamma \vdash \psi(v). \]

We can now state our theorems.

\begin{theorem} \label{invariance-theorems-34} \label{third-invariance:thm}

  \begin{itemize}
    \item[]
  \item \textbf{\Third :} Let $A,B \in \catM$ two cofibrant objects of a weak Quillen model category $\catM$ and $f:A \to B$ a weak equivalence between them. Then the map $f^*: \Lb_\lambda(B) \to \Lb_\lambda(A)$ induces a bijection 
    \[ h\Lb_\lambda(B) \simeq h\Lb_\lambda(A). \]
  \item \textbf{\Fourth :} If $F:\catM \to \catN$ is a left Quillen equivalence between two weak model categories, then for any cofibrant object $A \in \catM $ the induced map
  \[ h\Lb F_A: h\Lb_\lambda^\catM(A) \to h\Lb_\lambda^\catN(FA)  \]
  from \cref{left-quillen:map-languages} is an isomorphism.
  \end{itemize}
  
\end{theorem}

\begin{remark} Note that if $F: \catM \rightleftarrows \catN:G$ a Quillen equivalence between weak model categories and $B$ is a cofibrant object of $\catN$ which is not of the form $F(A)$ for $A \in \catM$, then one can still use the \Fourth\ to transfer a formula in $h\Lb(B)$ to a formula in $\catM$. We do this by first finding an object of the form $F(A)$ which is homotopically equivalent to $B$, which is always possible as $F$ is a Quillen equivalence, and then transferring our formula $\phi \in h\Lb(B)$ to a formula in $h\Lb(F(A))$ using the \Third. \end{remark}

\begin{observation} \label{bifibrant-semantic-relation}
  For any cofibrant object $\Gamma \in \catM$, $\phi,\psi\in \Lb_\lambda^\catM(\Gamma) $ we defined $\phi \approx \psi$ if and only if $|\phi|_X = |\psi|_X$ for all fibrant objects. However, note that if we take a cofibrant replacement $X^\cof$ of $X$, then by \cref{invariance-theorems} (\Second) we have, $X \vdash \phi(fv)$ if and only if $X^\cof \vdash \phi(v)$, where $f:X^\cof \trivialfib X$ and $v: \Gamma \to X^\cof$.

  Therefore, when testing the relation $\approx$, it is enough to use bifibrant objects. More precisely, define $\phi \approx_b \psi$ if $|\phi|_X = |\psi|_X$ for any bifibrant object $X$. Then \begin{center}$\phi \approx \psi$ if and only if $\phi \approx_b \psi$.\end{center}
\end{observation}

We now explain the construction of the map $h\Lb F_A: h\Lb_\lambda^\catM(A) \to h\Lb_\lambda^\catN(FA) $ mentioned in the \Fourth.

\begin{construction} \label{left-quillen:map-languages}
  The map $h\Lb F_A$ in the \Fourth\ is the map coming from $\Lb F_A : \Lb_\lambda^\catM(A) \to \Lb_\lambda^\catN(FA) $ constructed in \cref{cstr:Quillen_functor_acts_on_formulas}. It just comes from the fact that $\Lb_\lambda^\catM$ is the initial boolean algebra. Recall that it satisfies the formula:
  \[ G(X) \vdash \phi(v) \Leftrightarrow X \vdash F(\phi)(\tilde{v}). \]
  for any object $X \in \catN$, and cofibrant object $C \in \catM$, any map $v:C \to G(X)$ corresponding to $\tilde{v} : F(C) \to X$, and $\phi \in \Lb^\catM_\lambda(C)$.

  This immediately imply the following proposition that shows that the map $h\Lb_A$ mentioned in the \Fourth\ is well-defined.
\end{construction}

\begin{proposition}\label{induced-map:homotopy-bool-alg} For any Quillen adjunction $F: \catM \leftrightarrows \catN: G$ and $A \in \catM$ a cofibrant object, the map $F: \Lb_\lambda(A) \to \Lb_\lambda(FA)$ is compatible with the relation $\approx$ and induces a morphism of $\lambda$-boolean algebras
  \[ F: h\Lb_\lambda(A) \to h\Lb_\lambda(FA). \]
\end{proposition}

\begin{proof}
  If $\phi$ and $\psi$ are semantically equivalent formulas in $\Lb_\lambda(A)$, then for any fibrant object $X \in \catN$, and a map $\tilde{v}:FA \to X$ corresponding to $v:A \to GX$ we  have
  \[  X \vdash F(\phi)(\tilde{v}) \Leftrightarrow  G(X) \vdash \phi(v) \Leftrightarrow G(X) \vdash \psi(v) \Leftrightarrow X \vdash F(\psi)(\tilde{v}) \]
  which shows that $F(\phi) \approx F(\psi)$ and concludes the proof.
\end{proof}

We are now ready prove the \Third. We start with a special case:

\begin{lemma}\label{lemma:thrid_for_triv_cof}
  Let $\Gamma,\Gamma' \in \catM^\cof$ and $\pi:\Gamma \trivialcof \Gamma'$ be a core trivial cofibration, then the induced map $h\Lb_\lambda^\catM(\Gamma) \to h\Lb_\lambda^\catM(\Gamma')$ is an isomorphism of $\lambda$-boolean algebras.
\end{lemma}
\begin{proof}
  Assume that $\pi:\Gamma \trivialcof \Gamma'$ is a core trivial cofibration. Since to define the language of $\catM$ we take the $\kappa$-clan $(\catM^\cof)^\op$, when constructing the language we get a covariant functor $\catM^\cof\to \bool_\lambda$. Therefore, we obtain a map $\pi^*: \Lb_\lambda^\catM(\Gamma) \to \Lb_\lambda^\catM(\Gamma')$ and its left adjoint $\exists_\pi:\Lb_\lambda^\catM(\Gamma') \to \Lb_\lambda^\catM(\Gamma) $, which furthermore descends to the adjoint pair $h \exists_\pi:h\Lb_\lambda^\catM(\Gamma') \rightleftarrows h\Lb_\lambda^\catM(\Gamma): h \pi^*$ between the $\lambda$-boolean algebras.

  We claim that $h\exists_\pi$ is the inverse for $h\pi^*$. It is enough to show that for any $\phi:\Lb_\lambda^\catM(\Gamma)$ and $\psi \in \Lb_\lambda^\catM(\Gamma')$ we have $\exists_\pi\pi^*(\phi) \approx \phi$ and $\pi^*\exists_\pi(\psi)\approx \psi$.

  Firstly, let $X \in \catM^\fib$ be a fibrant object and $x: \Gamma \to X$. Note that $x\in |\exists_\pi \psi|_X\subseteq \hom_\catM(\Gamma,X)$ if and only if there exists $x':\Gamma' \to X$ such that $ x' \in |\psi|_X \subseteq \hom_\catM(\Gamma',X)$ and that makes the following triangle commutative:
  % https://q.uiver.app/#q=WzAsMyxbMCwwLCJcXEdhbW1hIl0sWzAsMSwiXFxHYW1tYSciXSxbMSwwLCJYIl0sWzAsMSwiXFxwaSIsMl0sWzAsMiwieCJdLFsxLDIsIngnIiwyXV0=
\[\begin{tikzcd}
	\Gamma & X \\
	{\Gamma'.}
	\arrow["x", from=1-1, to=1-2]
	\arrow["\pi"',"\sim",hook, from=1-1, to=2-1]
	\arrow["{x'}"', from=2-1, to=1-2]
\end{tikzcd}\]
Since $X$ is fibrant, the map $x'$ always exists. Such $x'$ is not necessarily unique, however, in a situation in which we have two arrows
% https://q.uiver.app/#q=WzAsMyxbMCwwLCJcXEdhbW1hIl0sWzAsMSwiXFxHYW1tYSciXSxbMSwwLCJYIl0sWzAsMSwiXFxwaSIsMix7InN0eWxlIjp7InRhaWwiOnsibmFtZSI6Imhvb2siLCJzaWRlIjoidG9wIn19fV0sWzAsMiwieCJdLFsxLDJdLFsxLDIsIiIsMSx7Im9mZnNldCI6Mn1dXQ==
\[\begin{tikzcd}
	\Gamma & X \\
	{\Gamma'}
	\arrow["x", from=1-1, to=1-2]
	\arrow["\pi"',"\sim", hook, from=1-1, to=2-1]
	\arrow["y",from=2-1, to=1-2]
	\arrow["z"',shift right=2, from=2-1, to=1-2]
\end{tikzcd}\]
that make the triangle commutative, then using that $\pi$ is a trivial cofibration we see that $y$ and $z$ are homotopic. By the first invariant theorem (\cref{invariance-theorems}) we have $y \in |\psi|_X$ if and only if $z\in |\psi|_X$. Therefore, the existence of $x' \in |\psi|_X$ is independent of choices.

From here, the result is immediate: $x\in |\exists_\pi \pi^* \phi|_X$ if and only if there exists $x':\Gamma' \to X$ such that $x'\pi=x$ such that $X\vdash \phi(\pi^*x')$ \ie if and only $x \in |\phi|_X$. This shows that $|\exists_\pi \pi^* \phi|_X= |\phi|_X$ for any fibrant object. Conversely, for $y:\Gamma'\to X$ we have $y\in |\pi^*\exists_\pi \psi|$ if and only if there exists $z:\Gamma' \to X$ such that $z\pi=y\pi$ and $X \vdash \psi(z)$, which is equivalent to $y\in |\psi|_X$, showing that $|\exists_\pi\pi^*\psi|_X=|\psi|_X$. This concludes the proof that $h\exists_\pi$ is the inverse for $h\pi^*$.
\end{proof}

We are now ready to prove the \Third:

\begin{proof}[Proof of the \Third:] The idea is to use \cref{lemma:thrid_for_triv_cof} together with Brown's factorization lemma from \cite{brown1973abstract}, or rather an adaptation of it to the setting of weak model structures that we present now. If $f:X \to Y$ is a weak equivalence between cofibrant objects in a weak model category. In general we cannot form a cylinder object for $X$, but instead a ``weak cylinder'' for $X$, that is a diagram:
\[ \begin{tikzcd}
    X \coprod X \ar[d,hook] \ar[r,"\nabla"] &  X \ar[d,hook,"\sim"] \\
    IX \ar[r,"\sim"] & DX,
  \end{tikzcd}\]
we then take the pushout of this whole diagram by the map $X \to Y$, using either of the two canonical maps $X \to X \coprod X$:
  \begin{equation}
    \label{Weak-brown-diag}\begin{tikzcd}
    X \coprod Y \ar[d,hook] \ar[r,"(id{,}f)"] &  Y \ar[d,hook,"\sim"] \\
    IX\coprod_X Y \ar[r,"\sim"] & DX \coprod_X Y
  \end{tikzcd}
  \end{equation}
and by precomposing with the coproduct inclusion $X \to X \coprod Y$, we obtain a diagram:
\[\begin{tikzcd}
     X \ar[d,hook] \ar[r,"f"] &  Y \ar[d,hook,"\sim"] \\
    IX\coprod_X Y \ar[r,"\sim"] & DX \coprod_X Y
  \end{tikzcd}\]
three of the four maps here are weak equivalence, so it follows by the $2$-out-of-$3$ property that the left vertical map is also a weak equivalence, hence a trivial cofibration. Applying $h\Lb$ we obtain a diagram:
\[\begin{tikzcd}
     h\Lb(X)  &  h\Lb(Y) \ar[l,"f^*"]  \\
    h\Lb\left(IX\coprod_X Y \right) \ar[u]  & \ar[l] \ar[u] h\Lb\left( DX \coprod_X Y \right) 
  \end{tikzcd}\]
The two vertical arrows are bijections because of \cref{lemma:thrid_for_triv_cof}, so in order to show that $f^*$ is a bijection, it is enough to show that the bottom map is a bijection.
This bottom horizontal map fits into a commutative diagram:
\[\begin{tikzcd}
   &  Y  \ar[d,hook,"\sim"] \ar[dl,hook,"\sim"swap] \\
    IX\coprod_X Y \ar[r,"\sim"] & DX \coprod_X Y,
\end{tikzcd}\]
where the arrow $Y \to IX \coprod_X Y$ is obtained as the pushout:
\[\begin{tikzcd}
    X \ar[r] \ar[d,hook,"\sim"] \ar[dr,phantom,"\ulcorner"very near end] & Y \ar[d,hook,"\sim"] \\
    IX \ar[r] & IX \coprod_X Y 
  \end{tikzcd}\]
Applying the $h\Lb$ functor, we get a triangle:

\[\begin{tikzcd}
   &  h\Lb \left(Y\right)  \\
    h\Lb\left(IX\coprod_X Y\right)  \ar[ur] & h\Lb\left( DX \coprod_X Y \right) \ar[l] \ar[u]
  \end{tikzcd}\]
the two vertical and diagonal arrows are bijections because of \cref{lemma:thrid_for_triv_cof}, and so the third, horizontal, arrows also is, which concludes the proof.
  
\end{proof}

We can also show the injectivity part of the \Fourth.

\begin{lemma}\label{induced-functor:injective}
Let $F: \catM \rightleftarrows \catN:G$ a Quillen equivalence. Then, for any cofibrant object $\Gamma \in \catM$, the induced map $h\Lb F_\Gamma: h\Lb_\lambda^\catM(\Gamma) \to h\Lb_\lambda^\catN(F\Gamma)  $ is injective.  
\end{lemma}
\begin{proof}
  Let $\phi$ and $\psi$ be formulas in $\Lb_\lambda^\catM(\Gamma)$ such that $F(\phi) \approx F(\psi)$ \ie $F(\phi)$ and $F(\psi)$ are equal in $h\Lb_\lambda^\catN(F\Gamma)$. We must show that $\psi \approx \phi$. Alternatively, by \cref{bifibrant-semantic-relation} we can show that $\psi \approx_b \phi$. The Quillen equivalence induces an equivalence between homotopy categories $Ho(G):Ho(\catN^{\bif}) \to Ho(\catM^{\bif})$. Hence, there is a bifibrant object $Y\in \catN$ such that $ GY$ is isomorphic to $X$ in $Ho(\catM^{\bif})$.
  Given any $x:\Gamma \to X$, denote by $y: \Gamma \to GY$ any map such that the following triangle
  % https://q.uiver.app/#q=WzAsMyxbMSwwLCJYIl0sWzAsMCwiQSJdLFsxLDEsIkdZIl0sWzEsMiwieSIsMl0sWzEsMCwieCJdLFswLDIsIlxcY29uZyJdXQ==
\[\begin{tikzcd}
	A & X \\
	& GY
	\arrow["x", from=1-1, to=1-2]
	\arrow["y"', from=1-1, to=2-2]
	\arrow["\cong", from=1-2, to=2-2]
      \end{tikzcd}\]
    commutes in $\Ho(\catM^{\bif})$. Lastly, let $y': F\Gamma \to Y$ the transpose of $y$ via the Quillen adjunction. It follows from the first invariance theorem \cref{invariance-theorems} that $X \vdash \phi (x)$ if and only if $GY \vdash \phi(y)$. From \cref{induced-map:homotopy-bool-alg}, this is equivalent to $Y \vdash F(\psi)(y') $. By assumption $F(\phi) \approx F(\psi)$, so $Y \vdash F(\psi)(y')$. Again, this is $GY  \vdash \psi (y)$ and $X \vdash \psi (x)$. This establishes the equality $|\phi|_X=|\psi|_X$ for all $X \in \catM$ bifibrant, which proves $\psi \approx_b \phi$, and hence $\psi \approx \phi$. This concludes the proof of the statement.
  \end{proof}

  We now explain our strategy to prove the rest of \cref{invariance-theorems-34}, that is the surjectivity part of the \Fourth.

  In \cite{barton2019model}, Reid Barton constructs a model 2-category structure on the 2-category of simplicial model categories. The trivial fibrations satisfy a property, that Barton called ``extensible'' (see \cref{extensible-morphism:definition}). In this section, we introduce a version of these in the non-enriched case, and we call those functors \emph{Barton trivial fibrations}. In \cref{invariance-along-trivial-fibrations} we show that the result holds for Barton trivial fibrations. After that, the idea is to use the same strategy as for the proof of the \Third\ based on this modified Brown factorization lemma to conclude the result holds for general Quillen equivalences.  We could do this immediately for combinatorial simplicial model categories using Brown's lemma in Barton's model structure, but for the general case we give a direct proof of the existence of the appropriate diagram which is inspired by how it would be done in Barton's model structure, but without relying on it directly. This is done in \cref{quillen-equivalence:factorization} using \cref{path-objects}.

\subsection{Invariance along Barton trivial fibrations} \label{invariance-along-trivial-fibrations}

In this section we introduce a class of left Quillen functor that we call \emph{Barton trivial fibrations} as they are essentially a non-simplicial version of the trivial fibrations of the model structure constructed by Barton in \cite{barton2019model}, and we establish that \cref{third-invariance:thm} holds for these particular functors.

\begin{definition} \label{extensible-morphism:definition}
  Let $F: \catC \to \catD$ a morphism between $\kappa$-coclans. We say that $F$ is \emph{extensible} if for every object in $ X \in \catC$ and for any cofibration $ g: FX \cofibration Y \in \catD$ there exists $f: X \cofibration Z$ and an isomorphism $ F(Z) \cong Y$ making the obvious triangle commutative.

  Dually, $F: \catC \to \catD$ a morphism between $\kappa$-clans is \emph{extensible} if the induced map of $\kappa$-coclans $F^\op : \catC^\op \to \catD^\op$ is extensible.

In our setting, a functor $F:\catM \to \catN$ between weak model categories will be called extensible if the coclan morphism $F: \catM^\cof \to \catN^\cof$ is extensible.
\end{definition}

The terminology \emph{extensible} in the definition above for both clans and coclans, instead of ``extensible'' and ``co-extensible'', is simply because it is always clear whether it refers to cofibrations or fibrations. This is because, for example, when considering a morphism between clans the relevant structure that ought to be preserved is that related to fibrations. The name extensible from \cref{extensible-morphism:definition} is adapted from Reid Barton's PhD thesis \cite[Definition 8.3.1]{barton2019model}.

\begin{definition}
  A left Quillen functor $F:\catM \to \catN$ between weak model categories is called \emph{weakly conservative} if for any core cofibration $x\cofibration y \in \catM^\cof$ such that $h:Fx\trivialcof Fy$ is a trivial cofibration, the map $x\cofibration y$ is a trivial cofibration.
\end{definition}

The `weakly' part in the previous definition does not come from weak model categories, but rather from the fact that core trivial cofibrations are weak equivalences.

\begin{definition} \label{barton-trivialfib:def}
   Let $F: \catM \to \catN$ a left Quillen functor between weak model categories. We say that $F$ is a \emph{Barton trivial fibration} if it is extensible as a morphism between of the coclans $\catM^\cof$ and $\catN^\cof$ and weakly conservative.
 \end{definition}

 \begin{remark}
   Barton trivial fibrations which are also simplicial Quillen functors between combinatorial simplicial model categories are exactly the trivial fibrations in \cite{barton2019model} in the model 2-category of pre-model categories. As the reader might anticipate, the notion of fibration between (simplicial) model categories exists as well, but we will make no use of it.
 \end{remark}

 The reason why we isolate the two properties in the definition above is because the first one is well-behaved with respect to the language we constructed, see \cref{trivial-fibration:surjective}. In \cref{extensible-conservative:quillen-equiv}, we justify the ``trivial'' part of \cref{barton-trivialfib:def} by showing that an extensible and weakly conservative left Quillen functor is a left Quillen equivalence, to do this we need an intermediate result.
 
\begin{lemma} \label{lift-correspondence:lemma}
  Let be $F:\catM \to \catN$ a left Quillen functor which is extensible and weakly conservative. Suppose there are diagrams
  % https://q.uiver.app/#q=WzAsNyxbMCwwLCJBIl0sWzAsMSwiQiJdLFsxLDAsIkMiXSxbMiwwLCJGQSJdLFsyLDEsIkZCIl0sWzMsMCwiRkMiXSxbMywxLCJaIl0sWzAsMSwiaSIsMix7InN0eWxlIjp7InRhaWwiOnsibmFtZSI6Imhvb2siLCJzaWRlIjoidG9wIn19fV0sWzAsMiwiZiJdLFszLDQsIkZpIiwyLHsic3R5bGUiOnsidGFpbCI6eyJuYW1lIjoiaG9vayIsInNpZGUiOiJ0b3AifX19XSxbMyw1LCJGZiJdLFs0LDYsInUiLDJdLFs1LDYsIlxcc2ltIiwwLHsic3R5bGUiOnsidGFpbCI6eyJuYW1lIjoiaG9vayIsInNpZGUiOiJ0b3AifX19XV0=
\[\begin{tikzcd}
	A & C & FA & FC \\
	B && FB & Z
	\arrow["f", from=1-1, to=1-2]
	\arrow["i"', hook, from=1-1, to=2-1]
	\arrow["Ff", from=1-3, to=1-4]
	\arrow["Fi"', hook, from=1-3, to=2-3]
	\arrow["\sim","v"', hook, from=1-4, to=2-4]
	\arrow["u"', from=2-3, to=2-4]
      \end{tikzcd}\]
    in $\catM$ and $\catN$, respectively, where $C \in \catM^\bif$ and $Z\in \catN^\bif$ are bifibrant and the right square is commutative. Then, there exists $g:B \to C$ that makes the triangle commutative and such that in the diagram
    % https://q.uiver.app/#q=WzAsNCxbMCwwLCJGQSJdLFswLDEsIkZCIl0sWzEsMCwiRkMiXSxbMSwxLCJaIl0sWzAsMSwiRmkiLDIseyJzdHlsZSI6eyJ0YWlsIjp7Im5hbWUiOiJob29rIiwic2lkZSI6InRvcCJ9fX1dLFswLDIsIkZmIl0sWzEsMywidSIsMl0sWzIsMywiXFxzaW0iLDAseyJzdHlsZSI6eyJ0YWlsIjp7Im5hbWUiOiJob29rIiwic2lkZSI6InRvcCJ9fX1dLFsxLDIsIkZnIiwxXV0=
\[\begin{tikzcd}
	FA & FC \\
	FB & Z
	\arrow["Ff", from=1-1, to=1-2]
	\arrow["Fi"', hook, from=1-1, to=2-1]
	\arrow["\sim","v"', hook, from=1-2, to=2-2]
	\arrow["Fg"{description}, from=2-1, to=1-2]
	\arrow["u"', from=2-1, to=2-2]
      \end{tikzcd}\]
    the lower triangle commutes up to homotopy relative to $FA$.
  \end{lemma}
  \begin{proof}
    Since $F$ is left Quillen then we have $F(B\coprod_A C) \cong FB\coprod_{FA} FC$ and is cofibrant. Up to this isomorphism, we factor the map $F(B\coprod_A C) \to Z$ as $F(B\coprod_A C) \cofibration Y \trivialfib Z$. Since $F$ is extensible we can lift this cofibration to a cofibration $B\coprod_A C \cofibration D$ together with the isomorphism $FD \cong Y$ making the resulting triangle commutative, which also implies that $FD$ is bifibrant since $Y$ is. Furthermore, this produces a commutative diagram as on the left,
    % https://q.uiver.app/#q=WzAsOSxbMywwLCJGQyJdLFszLDEsIkZEIl0sWzQsMSwiWSJdLFs0LDAsIloiXSxbMCwwLCJBIl0sWzAsMSwiQiJdLFsxLDAsIkMiXSxbMSwxLCJCXFxjb3Byb2RfQUMiXSxbMiwyLCJEIl0sWzEsMiwiXFxjb25nIiwxXSxbMCwyLCIiLDEseyJzdHlsZSI6eyJ0YWlsIjp7Im5hbWUiOiJob29rIiwic2lkZSI6InRvcCJ9fX1dLFsyLDMsIlxcc2ltIiwyLHsic3R5bGUiOnsiaGVhZCI6eyJuYW1lIjoiZXBpIn19fV0sWzAsMywiXFxzaW0iLDAseyJzdHlsZSI6eyJ0YWlsIjp7Im5hbWUiOiJob29rIiwic2lkZSI6InRvcCJ9fX1dLFswLDEsIkZoIiwyLHsic3R5bGUiOnsidGFpbCI6eyJuYW1lIjoiaG9vayIsInNpZGUiOiJ0b3AifX19XSxbNiw3LCIiLDEseyJzdHlsZSI6eyJ0YWlsIjp7Im5hbWUiOiJob29rIiwic2lkZSI6InRvcCJ9fX1dLFs0LDYsImYiXSxbNCw1LCJpIiwyLHsic3R5bGUiOnsidGFpbCI6eyJuYW1lIjoiaG9vayIsInNpZGUiOiJ0b3AifX19XSxbNSw3XSxbNyw4LCJsIiwxLHsic3R5bGUiOnsidGFpbCI6eyJuYW1lIjoiaG9vayIsInNpZGUiOiJ0b3AifX19XSxbNiw4LCJoIiwxXSxbNSw4LCJrIiwxXSxbNywxLCJGIiwwLHsibGFiZWxfcG9zaXRpb24iOjYwLCJzaG9ydGVuIjp7InNvdXJjZSI6NDAsInRhcmdldCI6MjB9LCJzdHlsZSI6eyJ0YWlsIjp7Im5hbWUiOiJtYXBzIHRvIn19fV1d
\[\begin{tikzcd}
	A & C && FC & Z \\
	B & {B\coprod_AC} && FD & Y \\
	&& D
	\arrow["f", from=1-1, to=1-2]
	\arrow["i"', hook, from=1-1, to=2-1]
	\arrow[hook, from=1-2, to=2-2]
	\arrow["h"{description}, from=1-2, to=3-3]
	\arrow["\sim", hook, from=1-4, to=1-5]
	\arrow["Fh"', hook, from=1-4, to=2-4]
	\arrow[hook, from=1-4, to=2-5]
	\arrow[from=2-1, to=2-2]
	\arrow["k"{description}, from=2-1, to=3-3]
	\arrow["F"{pos=0.6}, shorten <=18pt, shorten >=9pt, maps to, from=2-2, to=2-4]
	\arrow["l"{description}, hook, from=2-2, to=3-3]
	\arrow["\cong"{description}, from=2-4, to=2-5]
	\arrow["\sim"', two heads, from=2-5, to=1-5]
\end{tikzcd}\]
    while the diagram on the right is the result of applying $F$, we introduce the name $\rho:FD \trivialfib Z$ for the evident resulting trivial fibration. We can use the 2-out-of-3 property of weak equivalences between cofibrant-fibrant objects to conclude that $ FC \cofibration Y$ is a weak equivalence, and hence a trivial cofibration. Since $F$ is weakly conservative, the map $C \cofibration D$ must be a weak equivalence too. Using that $C$ is bifibrant we can obtain a dashed arrow which is a homotopy inverse of $h$
% https://q.uiver.app/#q=WzAsNSxbMCwwLCJBIl0sWzAsMSwiQiJdLFsxLDEsIkQiXSxbMSwwLCJDIl0sWzIsMCwiQyJdLFswLDMsImYiXSxbMCwxLCJpIiwyLHsic3R5bGUiOnsidGFpbCI6eyJuYW1lIjoiaG9vayIsInNpZGUiOiJ0b3AifX19XSxbMSwyLCJrIiwyXSxbMywyLCJcXHNpbSIsMCx7InN0eWxlIjp7InRhaWwiOnsibmFtZSI6Imhvb2siLCJzaWRlIjoidG9wIn19fV0sWzMsNCwiSWQiXSxbMiw0LCJyIiwyLHsic3R5bGUiOnsiYm9keSI6eyJuYW1lIjoiZGFzaGVkIn19fV1d
\[\begin{tikzcd}
	A & C & C \\
	B & D,
	\arrow["f", from=1-1, to=1-2]
	\arrow["i"', hook, from=1-1, to=2-1]
	\arrow["Id", from=1-2, to=1-3]
	\arrow["\sim","h"', hook, from=1-2, to=2-2]
	\arrow["k"', from=2-1, to=2-2]
	\arrow["r"', dashed, from=2-2, to=1-3]
      \end{tikzcd}\]
    we can take $g\coloneqq rk$ to be a diagonal filler of the square. Observe that when we apply $F$ to the resulting diagram, it gives us the square and the diagonal in the diagram
% https://q.uiver.app/#q=WzAsNSxbMCwwLCJGQSJdLFswLDEsIkZCIl0sWzIsMCwiRkMiXSxbMiwxLCJGRCJdLFszLDEsIloiXSxbMSwyLCJGZyIsMV0sWzAsMSwiRmkiLDJdLFswLDIsIkZmIl0sWzEsMywiRmsiLDFdLFszLDQsIlxcc2ltIiwwLHsic3R5bGUiOnsiaGVhZCI6eyJuYW1lIjoiZXBpIn19fV0sWzIsNCwiXFxzaW0iLDAseyJzdHlsZSI6eyJ0YWlsIjp7Im5hbWUiOiJob29rIiwic2lkZSI6InRvcCJ9fX1dLFsxLDQsInUiLDIseyJjdXJ2ZSI6Mn1dLFsyLDMsIkZoIiwyLHsic3R5bGUiOnsidGFpbCI6eyJuYW1lIjoiaG9vayIsInNpZGUiOiJ0b3AifX19XV0=
\[\begin{tikzcd}
	FA && FC \\
	FB && FD & Z
	\arrow["Ff", from=1-1, to=1-3]
	\arrow["Fi"', from=1-1, to=2-1]
	\arrow["Fh"',"\sim", hook, from=1-3, to=2-3]
	\arrow["\sim","v"', hook, from=1-3, to=2-4]
	\arrow["Fg"{description}, from=2-1, to=1-3]
	\arrow["Fk"{description}, from=2-1, to=2-3]
	\arrow["u"', bend right, from=2-1, to=2-4]
	\arrow["\sim","\rho"', two heads, from=2-3, to=2-4]
\end{tikzcd}\]
where a priori the outer triangle involving $u$ is not commutative. However, we can realize this diagram in the homotopy category $\Ho(FA/\catN)$. So working in the homotopy category we have $hr=Id$ and $FhFr=Id$. By construction, we also get $Fg=FrFk$, therefore $FhFg=FhFrFk=Fk$ in the homotopy category, and $\rho:FD \trivialfib Z$ becoming an isomorphism implies $vFg=u$ up to homotopy relative to $FA$.
  \end{proof}

\begin{corollary} \label{extensible-conservative:quillen-equiv}
  Let $F:\catM \to \catN$ a left Quillen functor between weak model categories. Assume that $F:\catM^\cof \to \catN^\cof$ is extensible and weakly conservative, then $F$ is a left Quillen equivalence.
\end{corollary}
\begin{proof}
  We show directly that $F$ induces an equivalence of categories between the homotopy categories.
  
  Assume that $X \in \catN^\cof$ is cofibrant. Then we can use that $F$ is extensible for the cofibration $0 \cofibration X$ to obtain a cofibrant object $A \in \catM^\cof$ and an isomorphism $FA \cong X \in \catN$. This shows that the induced functor is essentially surjective.
  
  We now show that for $\Ho(\catM) \to \Ho(\catN)$ is full. Let $B,C \in \catM^\cof$ cofibrant objects. We could take a fibrant replacement $C^\fib$ and use this instead, so we can freely assume that $C$ is bifibrant. A map $FB \to FC \in \Ho(\catN)$ can be represented by a cospan
  \[
    FB \to (FC)^\fib \overset{\sim}{\hookleftarrow} FC \in \catN.
  \]
  Therefore, we can use \cref{lift-correspondence:lemma} to find a map $B \to C$ in $\Ho(\catM)$ which is in the preimage.

  Lastly, we see that the induced functor is faithful. Let $A,C \in \catM^\cof$ cofibrant and two maps $f,g:A\to C \in \catM$ which become equal in $\Ho(\catN)$ under the functor induced by $F$. This is just saying that the maps $F\bar f,F\bar g:FA \to F(C^\fib)$ are homotopic where $\bar f,\bar g: A\to C^\fib$ are maps in $\catM$. It will be enough to show that $\bar f$ and $\bar g$ are homotopic \ie there is a diagonal filler for the diagram
  % https://q.uiver.app/#q=WzAsMyxbMCwwLCJBXFxjb3Byb2QgQSJdLFsyLDAsIkNeXFxmaWIiXSxbMCwxLCJJQSJdLFswLDEsIihcXGJhciBmLFxcYmFyIGcpIl0sWzAsMiwiIiwyLHsic3R5bGUiOnsidGFpbCI6eyJuYW1lIjoiaG9vayIsInNpZGUiOiJ0b3AifX19XV0=
\[\begin{tikzcd}
	{A\coprod A} && {C^\fib} \\
	IA
	\arrow["{(\bar f,\bar g)}", from=1-1, to=1-3]
	\arrow[hook, from=1-1, to=2-1]
\end{tikzcd}\]
where $IA$ is a weak cylinder object for $A$. Since $F$ is a left Quillen functor, we can assume that cylinders are preserved. Furthermore, homotopies are independent of the choice of cylinders. We can express the homotopy between of $F\bar f$ and $F\bar g$ in $\catN$ as the commutative square
% https://q.uiver.app/#q=WzAsNCxbMCwwLCJGKEFcXGNvcHJvZCBBKSJdLFswLDEsIkYoSUEpIl0sWzIsMCwiRihCXlxcZmliKSJdLFsyLDEsIkYoQl5cXGZpYileXFxmaWIiXSxbMCwxLCIiLDAseyJzdHlsZSI6eyJ0YWlsIjp7Im5hbWUiOiJob29rIiwic2lkZSI6InRvcCJ9fX1dLFswLDIsIihGXFxiYXIgZixGXFxiYXIgZykiXSxbMiwzLCIiLDAseyJzdHlsZSI6eyJ0YWlsIjp7Im5hbWUiOiJob29rIiwic2lkZSI6InRvcCJ9fX1dLFsxLDMsImgiLDJdXQ==
\[\begin{tikzcd}
	{F(A\coprod A)} && {F(B^\fib)} \\
	{F(IA)} && {F(B^\fib)^\fib},
	\arrow["{(F\bar f,F\bar g)}", from=1-1, to=1-3]
	\arrow[hook, from=1-1, to=2-1]
	\arrow[hook,"\sim", from=1-3, to=2-3]
	\arrow["h"', from=2-1, to=2-3]
\end{tikzcd}\]
  where $h$ is the homotopy, and the fibrant replacement $F(C^\fib)^\fib$ is necessary since $F(C^\fib)$ is not fibrant as $F$ is only left Quillen. The assumptions of \cref{lift-correspondence:lemma} are now satisfied, so this produces a diagonal as on the left whose image fits on the right square up to homotopy:
% https://q.uiver.app/#q=WzAsNyxbMywwLCJGKEFcXGNvcHJvZCBBKSJdLFszLDEsIkYoSUEpIl0sWzUsMCwiRihDXlxcZmliKSJdLFs1LDEsIkYoQ15cXGZpYileXFxmaWIiXSxbMCwwLCJBXFxjb3Byb2QgQSJdLFsyLDAsIkNeXFxmaWIiXSxbMCwxLCJJQSJdLFswLDEsIiIsMCx7InN0eWxlIjp7InRhaWwiOnsibmFtZSI6Imhvb2siLCJzaWRlIjoidG9wIn19fV0sWzAsMiwiKEZcXGJhciBmLEZcXGJhciBnKSJdLFsyLDMsIlxcc2ltIiwwLHsic3R5bGUiOnsidGFpbCI6eyJuYW1lIjoiaG9vayIsInNpZGUiOiJ0b3AifX19XSxbMSwzLCJoIiwyXSxbNCw2LCIiLDAseyJzdHlsZSI6eyJ0YWlsIjp7Im5hbWUiOiJob29rIiwic2lkZSI6InRvcCJ9fX1dLFs0LDUsIihcXGJhciBmLFxcYmFyIGcpIl0sWzYsNSwiSCIsMl0sWzEsMiwiRkgiLDFdXQ==
\[\begin{tikzcd}
	{A\coprod A} && {C^\fib} & {F(A\coprod A)} && {F(C^\fib)} \\
	IA &&& {F(IA)} && {F(C^\fib)^\fib}
	\arrow["{(\bar f,\bar g)}", from=1-1, to=1-3]
	\arrow[hook, from=1-1, to=2-1]
	\arrow["{(F\bar f,F\bar g)}", from=1-4, to=1-6]
	\arrow[hook, from=1-4, to=2-4]
	\arrow["\sim", hook, from=1-6, to=2-6]
	\arrow["H"', from=2-1, to=1-3]
	\arrow["FH"{description}, from=2-4, to=1-6]
	\arrow["h"', from=2-4, to=2-6]
      \end{tikzcd}\]
    The above shows that $\Ho(\catM) \to \Ho(\catN)$ is faithful, concluding the proof that $F$ is a left Quillen equivalence.
\end{proof}

We now return to show the \Fourth\ for the case in which the functor is a Barton trivial fibration. First, we observe that extensible functors always induce a surjection between the languages of clans.

\begin{lemma} \label{trivial-fibration:surjective}
  Let $F:\catM \to \catN$ be an extensible morphism between $\kappa$-clans and $\Gamma \in \catM $. Then, any formula $\Phi \in \Lb_\lambda^\catN (F\Gamma)$ is the image under $F$ of a formula $\Phi_0 \in \Lb_\lambda^\catM(\Gamma)$.
\end{lemma}
\begin{proof}
  Since every $\kappa$-clan is of the form $\C_T$ for some $T$ generalized $\kappa$-algebraic theory, it is enough to show the result is valid for the syntactic definition of language as in \cref{folgat}. We prove by induction on formulas $\Phi \in \Lb_\lambda^\catN(\Delta) $ that, given any context $\Gamma$ and $f:\Delta \cong F(\Gamma)$, there is a formula $\Phi_0 \in \Lb_\lambda^\catM(\Gamma)$ such that $f^*(F\Phi_0)=\Phi$.
  \begin{enumerate}
  \item When $\Phi=\top$ or $\Phi=\bot$, then this can clearly be lifted to $\top$ and $\bot$.
  \item If $\Phi= \neg \Psi$ or $\Phi= \bigvee_{i \in I} \Psi_i$ or $\Phi= \bigwedge_{i \in I} \Psi_i$ then it is also clear that $\Phi$ can be lifted. Indeed, we can simply use the inductive hypothesis to lift each $\Psi_i$ and then use the boolean algebra structure to conclude.
  \item Suppose that $\Phi$ is of the form $\exists_\pi \Psi$ or $\forall_\pi \Psi$ for some fibration $\pi: \Gamma' \fibration F(\Gamma)$. The formula $\Psi\in \Lb_\lambda^\catN(\Gamma')$, so $\Phi \in \Lb_\lambda^\catN(F\Gamma)$. Furthermore, we assume that $\Psi$ can be lifted. Since $F$ is extensible, there is a lift $\bar{\pi}: \bar{\Gamma'} \to \Gamma \in \catM$ of $\pi:\Gamma' \fibration F(\Gamma)$, which comes with an isomorphism $g: \Gamma' \cong F(\bar{\Gamma'})$ such that the following triangle commutes
    % https://q.uiver.app/#q=WzAsMyxbMCwxLCJGKFxcYmFye1xcR2FtbWEnfSkiXSxbMCwwLCJcXEdhbW1hJyJdLFsxLDAsIkYoXFxHYW1tYSkiXSxbMSwwLCJcXGNvbmciLDJdLFsxLDIsIlxccGkiLDAseyJzdHlsZSI6eyJoZWFkIjp7Im5hbWUiOiJlcGkifX19XSxbMCwyLCJGKFxcYmFyXFxwaSkiLDIseyJzdHlsZSI6eyJoZWFkIjp7Im5hbWUiOiJlcGkifX19XV0=
\[\begin{tikzcd}
	{\Gamma'} & {F(\Gamma)} \\
	{F(\bar{\Gamma'}).}
	\arrow["\pi", two heads, from=1-1, to=1-2]
	\arrow["\cong"',"g", from=1-1, to=2-1]
	\arrow["{F(\bar\pi)}"', two heads, from=2-1, to=1-2]
\end{tikzcd}\]
    Therefore, we get a commutative square as in the left below, and at the level of languages as on the right below
  % https://q.uiver.app/#q=WzAsOCxbMCwxLCJGKFxcYmFye1xcR2FtbWEnfSkiXSxbMCwwLCJcXEdhbW1hJyJdLFsxLDAsIlxcRGVsdGEiXSxbMSwxLCJGKFxcR2FtbWEpIl0sWzMsMCwiXFxMYl9cXGxhbWJkYV5cXGNhdE4oRihcXGJhclxcR2FtbWEnKSkiXSxbMywxLCJcXExiX1xcbGFtYmRhXlxcY2F0TihcXEdhbW1hJykiXSxbNCwxLCJcXExiX1xcbGFtYmRhXlxcY2F0TihcXERlbHRhKSJdLFs0LDAsIlxcTGJfXFxsYW1iZGFeXFxjYXROKEYoXFxHYW1tYSkpIl0sWzEsMCwiXFxjb25nIiwyXSxbMiwzLCJcXGNvbmciXSxbMSwyLCJcXHBpJyIsMCx7InN0eWxlIjp7ImhlYWQiOnsibmFtZSI6ImVwaSJ9fX1dLFswLDMsIkYoXFxiYXJcXHBpKSIsMix7InN0eWxlIjp7ImhlYWQiOnsibmFtZSI6ImVwaSJ9fX1dLFs0LDUsImdeKiIsMl0sWzcsNiwiZl4qIl0sWzQsNywiXFxleGlzdHNfe1xccGknfSJdLFs1LDYsIlxcZXhpc3RzX3tGKFxcYmFyXFxwaSl9IiwyXV0=
\[\begin{tikzcd}
	{\Gamma'} & \Delta && {\Lb_\lambda^\catN(F(\bar\Gamma'))} & {\Lb_\lambda^\catN(F(\Gamma))} \\
	{F(\bar{\Gamma'})} & {F(\Gamma)} && {\Lb_\lambda^\catN(\Gamma')} & {\Lb_\lambda^\catN(\Delta).}
	\arrow["{\pi'}", two heads, from=1-1, to=1-2]
	\arrow["\cong"',"g", from=1-1, to=2-1]
	\arrow["\cong","f"', from=1-2, to=2-2]
	\arrow["{\exists_{\pi'}}", from=1-4, to=1-5]
	\arrow["{g^*}"', from=1-4, to=2-4]
	\arrow["{f^*}", from=1-5, to=2-5]
	\arrow["{F(\bar\pi)}"', two heads, from=2-1, to=2-2]
	\arrow["{\exists_{F(\bar\pi)}}"', from=2-4, to=2-5]
\end{tikzcd}\]
    By assumption $\psi \in \Lb_\lambda^\catN(\Gamma')$ can be lifted. Hence, there is a formula $\Psi_0 \in \Lb_\lambda^\catM(\bar\Gamma')$ such that $g^*(F\Psi_0)=\Psi$. Using the right hand square above, one can see that $\exists_{\bar\pi}\Psi_0 $ is a lift for $\Phi$.
  \end{enumerate}
  This shows that the map is surjective.
\end{proof}

As an immediate consequence of \cref{trivial-fibration:surjective}, we can establish the \Fourth\ in the special case where $F:\catM \to \catN$ is a Barton trivial fibration. We will use this result to be able to establish \Fourth\ for the general case later on.

\begin{theorem} \label{trivial-fibration:invariance}
  Let $F: \catM \to \catN$ be a Barton trivial fibration between weak model categories. Then for any cofibrant $\Gamma \in \catM$ the induced map $h\Lb F_A: h\Lb_\lambda^\catM(\Gamma) \to h\Lb_\lambda^\catN(F\Gamma)$  is an isomorphism.
\end{theorem}
\begin{proof}
  By the previous \cref{induced-functor:injective} we know that $h\Lb F_\Gamma: h\Lb_\lambda^\catM(\Gamma) \to h\Lb_\lambda^\catN(F\Gamma)$ is injective. Next we can use \cref{trivial-fibration:surjective} by observing that this surjectivity also descends to the level of $h\Lb F_\Gamma: h\Lb_\lambda^\catM(\Gamma) \to h\Lb_\lambda^\catN(F\Gamma)$.
\end{proof}

Since our next goal is to prove \fourth, with \cref{trivial-fibration:invariance} at hand, we simply need to reduce our problem to the case in which we have Barton trivial fibrations. The constructions to come are essentially the necessary steps for this reduction process.

\subsection{Path objects for weak model categories} \label{path-objects}

The next step is to build some sort of ``path object'' for (weak) model category so that we can emulate Brown Factorization lemma to factor a general Quillen equivalence into a retract of a Barton trivial fibration followed by a Barton fibration. Ideally, we would want for a model category $\catM$, we would like to build a diagram of left Quillen functors
\[\catM \to P\catM \to \catM \times \catM \]
where the maps $P\catM \to \catM$ are Barton trivial fibrations, and then try to use it to follow the proof of Brown's Factorization Lemma. Unfortunately, that is not going to be quite possible: we will not be able to construct a map $\catM \to P\catM$. Instead, similarly to the proof of the \third, we will construct, a diagram of the form
\[
  \begin{tikzcd}
    R\catM \ar[d,"p"] \ar[r] & P\catM \ar[d] \\
    \catM \ar[r] & \catM \times \catM
  \end{tikzcd}
\]
where the arrow $p$ is a Barton trivial fibration. This will turn out to be sufficient to build our desired Brown style factorization. The weak model categories $R\catM$ and $P\catM$ will be constructed as certain category of functors $\catM^J$ and $\catM^I$, equipped with certain localization of the Reedy model structure. So we get a diagram
  % https://q.uiver.app/#q=WzAsNCxbMCwwLCJcXGNhdE1eSiJdLFswLDEsIlxcY2F0TSJdLFsxLDAsIlxcY2F0TV5JIl0sWzEsMSwiXFxjYXRNXFx0aW1lc1xcY2F0TSJdLFswLDEsIlxcc2ltIiwyLHsic3R5bGUiOnsiaGVhZCI6eyJuYW1lIjoiZXBpIn19fV0sWzIsM10sWzAsMl0sWzEsM11d
\[\begin{tikzcd}
	{\catM^J} & {\catM^I} \\
	\catM & {\catM\times\catM}
	\arrow[from=1-1, to=1-2]
	\arrow["\sim"', two heads, from=1-1, to=2-1]
	\arrow[from=1-2, to=2-2]
	\arrow[from=2-1, to=2-2]
      \end{tikzcd}\]
where the arrow on the left and the two maps $\catM^I \to \catM$ induced by the projections are Barton trivial fibrations. More precisely, the construction we do takes as input a left Quillen equivalence $F: \catM \to \catN$ between weak model categories and produces a diagram
    % https://q.uiver.app/#q=WzAsNCxbMCwwLCJcXGNhdE1eSiJdLFswLDEsIlxcY2F0TSJdLFsxLDAsIlxcY2F0Tl9GXkkiXSxbMSwxLCJcXGNhdE5cXHRpbWVzXFxjYXROIl0sWzAsMSwiXFxzaW0iLDIseyJzdHlsZSI6eyJoZWFkIjp7Im5hbWUiOiJlcGkifX19XSxbMiwzXSxbMCwyXSxbMSwzXV0=
\[\begin{tikzcd}
	{\catM^J} & {\catN_F^I} \\
	\catM & {\catN\times\catM}
	\arrow[from=1-1, to=1-2]
	\arrow["\sim"', two heads, from=1-1, to=2-1]
	\arrow[from=1-2, to=2-2]
	\arrow[from=2-1, to=2-2]
      \end{tikzcd}\]
    where again the arrow on the left and the two maps $\catN_F^I \to \catN$ induced by the projections are Barton trivial fibrations. Hence, the first diagram is a particular case when $F=Id_\catM$.
    \begin{remark}
      The core idea of the outline above is already present in the proof of the \third. This can be seen as the analogue (or rather a dual) of the diagram (\ref{Weak-brown-diag}) that appears in the proof of the \Third, and it will play the exact same role. In both cases, the idea is to obtain some sort of Brown factorization.
    \end{remark}

    The bulk of the work lies in endowing the categories $\catM^I$ and $\catM^J$ with the correct weak model structures. This can be summarized as follows: We start with the Reedy weak model structure on the category $\catM^J$, or $\catN^I$, and perform a ``right Bousfield localization'' to obtain our desired models.

    \begin{remark}
 The weak model structure on $\catN^I$ encodes a pair of objects $A,B$ in $\catN$ with a ``correspondence'' between them; that is, a homotopy equivalence encoded by a cofibration $A \coprod B \to C$ where both maps $A \to C$ and $B \to C$ are trivial cofibrations. The weak model structure we obtain on $\catM^J$ encodes objects $X$ in $\catM$ equipped with a (weak) cylinder object, so that we can send such an object $X$ with a cylinder $IX$ to the correspondence $X \coprod X \to IX$.
    \end{remark}
  
  \subsubsection{Weak model for objects with weak cylinders}
  
  We start by fixing a weak model category $\catM$ and let $ J $ be the category
  \[\begin{tikzcd}
     a \ar[r,"i",shift left] \ar[r,"j"',shift right] & b  \ar[r,"k"] & c
   \end{tikzcd}\]
 such that $ ki=kj $. Consider the degree function making $J$ into a direct category, $ \deg(a)=0,\; \deg(b)=1,\; \deg(c)=2 $. Our first goal is to prove:

 \begin{theorem} \label{model-paths:wms}
  The category of diagrams $\catM^J$ has a weak model structure where:
  \begin{enumerate}
  \item A map between diagrams $X \to Y$ is a cofibration if
    \begin{enumerate}
    \item It is a Reedy cofibration,
    \item $Y_a \sqcup_{X_a} X_c \trivialcof Y_c$ and $Y_b \sqcup_{X_b} X_c \trivialcof Y_c$ are trivial cofibrations in $\catM$.
    \end{enumerate}
  \item Fibrations are level-wise fibrations.
  \end{enumerate}
\end{theorem}

\begin{remark}
  The theorem above make reference to Reedy cofibrations, therefore we must justify first that $\catM^J$ carries the Reedy weak model structure. Fortunately, this has been addressed in \cref{reedy-model:theorem}.
\end{remark}

\begin{notation}
  For the sake of clarity, we denote by $\catM^J_{Reedy}$ when referring to the Reedy weak model structure and $\catM^J_{Loc}$ for the weak model structure of \cref{model-paths:wms}. Of course, \textit{a priori,} we have yet to prove that the latter is indeed a weak model structure. Therefore, whenever we say, for example, that a map $f:X \to Y$ is a cofibration we just mean that $f$ satisfies the corresponding condition of \cref{model-paths:wms}.
\end{notation}

We will justify that the following construction, which is simply the conditions of the theorem, is the correct one.
 
 \begin{observation}\label{trivialcofibrations-between-cofibrants}
   One can verify that in this new model structure, the core fibrations and core trivial cofibrations coincide with the ones in the Reedy weak model structure (see \cref{localized-triv-cof:lem}).

   The reader might suspect that this is not a fortuitous coincidence, these suspicions are well justified. As we mentioned, what we have done is a right Bousfield localization of a Reedy weak model structure on $\catM^J$. Such localizations are studied in \cite{henry2023combinatorial} in the case when $\catM$ is a combinatorial (accessible) weak model category. Due to the lack of a general theorem that justifies the existence of these localizations producing a weak model category, we verify all required conditions by hand. 
 \end{observation}
 
 We examine the  class of cofibrations. For a diagram $ X\in \catM^J $, the latching objects are $ L_aX=\emptyset $, $ L_bX= X_a\sqcup X_a $ and $ L_cX= X_b\sqcup_{X_a} X_b$. These are cofibrant in $\catM$. Then a map $ f:X\to Y $ being a cofibration means that $ X_a\cofibration Y_a $,
		\[
		X_b\sqcup_{X_a\sqcup X_a} (Y_a\sqcup Y_a) \cofibration Y_b \text{ and }
		X_c\sqcup_{(X_b\sqcup_{X_a} X_b)} (Y_b\sqcup_{Y_a} Y_b) \cofibration Y_c \]
		are cofibrations in $ \catM $, and additionally $ Y_a\sqcup_{X_a} X_c \trivialcof Y_c $ and $ Y_b\sqcup_{X_b} X_c \trivialcof Y_c $ are trivial cofibrations in $ \catM $.
		
		Therefore, a diagram $ Y \in \catM^J$ is \emph{cofibrant} if $ Y_a $ is a cofibrant object in $ \catM $,
		\[
		Y_a\sqcup Y_a \cofibration Y_b \text{ and }
		Y_b\sqcup_{Y_a} Y_b \cofibration Y_c \]
              are cofibrations, and additionally $ Y_a \trivialcof Y_c $ and $ Y_b \trivialcof Y_c $ are trivial cofibrations. Spelling out the second Reedy condition gives us the following commutative diagram:
              \[
		\begin{tikzcd}
			\emptyset \ar[r,hook] \ar[d,hook] \ar[rd,phantom,"\ulcorner", very near end] & Y_a \ar[d,hook] \ar[rdd,bend left] &  \\
		Y_a \ar[r,hookrightarrow] \ar[drr,bend right] & Y_a \sqcup Y_a \ar[rd,hook]& \\
		 & & Y_b
		\end{tikzcd}
		\]
		This says that both maps
                $ \begin{tikzcd}
		Y_a \ar[r,"Yi",shift left] \ar[r,"Yj"',shift right] & Y_b
              \end{tikzcd} $ are cofibrations. We can use this on the following diagram 
              \[
		\begin{tikzcd}
		Y_a \ar[r,hook] \ar[d,hook] \ar[rd,phantom,"\ulcorner", very near end] & Y_b \ar[d,hook] \ar[rdd,bend left] &  \\
		Y_b \ar[r,hookrightarrow] \ar[drr,bend right] & Y_b \sqcup_{Y_a} Y_b \ar[rd,hook] & \\
		& & Y_c
		\end{tikzcd}
		\]
		to conclude that $ Y_b\cofibration Y_c $ is a cofibration. Of course this is in principle not necessary since we also have $ Y_b \trivialcof Y_c $ is a trivial cofibration, but the novel aspect is that this follows only from Reedy cofibrancy. We also have a trivial cofibration $ Y_a\trivialcof Y_c $, by the two-out-of-three property the maps $ \begin{tikzcd}
		Y_a \ar[r,"Yi",shift left] \ar[r,"Yj"',shift right] & Y_b
              \end{tikzcd} $ are trivial cofibrations. We collect the above in the following:
              \begin{remark}
                If $ Y $ is cofibrant then we obtain the following diagram:
		\[
		\begin{tikzcd}
		 Y_a\sqcup Y_a \ar[d,hook] \ar[r,"\triangledown"] & Y_a \ar[d,"\sim",hook] \\
		 Y_b \ar[r,"\sim",hook] & Y_c.
		\end{tikzcd}
              \]
              This is just to say that cofibrant diagrams of $\catM^J_{Loc}$ encode objects of $\catM$ for which a weak cylinder exists in the sense of \cref{cstr:weak_path_and_cylinder}.
              \end{remark}

              We reiterate that our goal is to show that the category of diagrams $\catM^J_{Loc}$ has a weak model structure on it, where the cofibrations are the ones as specified in \cref{model-paths:wms}. We begin by showing the following lemmas which are expected results in the theory of right Bousfield localizations.

               \begin{lemma} \label{cofibration-loc-eq-cofibration}
        Let $X,Y \in \catM^J_{Loc}$ cofibrant. Then, a map $X \to Y$ is a cofibration in $\catM^J_{Loc}$ if and only if it is a cofibration in $\catM^J_{Reedy}$.
      \end{lemma}
      \begin{proof}
        We only prove the interesting direction; assume that $X,Y$ are cofibrant in $\catM_{Loc}^J$ and that $X \to Y \in \catM^J_{Reedy}$ is a Reedy cofibration. Remains to show that \[
          X_c\sqcup_{X_a} Y_a \to Y_c \text{ and }  X_c\sqcup_{X_b} Y_b \to Y_c
        \]
        are trivial cofibrations. The fact that the maps are weak equivalences follows by applying the 2-out-of-3 property to the diagrams:
\[\begin{tikzcd}
	{X_a} & {Y_a} && {X_b} & {Y_b} \\
	{X_c} & {X_c\sqcup_{X_a} Y_a} && {X_c} & {X_c\sqcup_{X_b} Y_b} \\
	&& {Y_c} &&& {Y_c}
	\arrow[hook, from=1-1, to=1-2]
	\arrow["\sim", hook, from=1-1, to=2-1]
	\arrow["\ulcorner"{description, pos=0.9}, draw=none, from=1-1, to=2-2]
	\arrow["\sim", hook, from=1-2, to=2-2]
	\arrow["\sim", bend left, hook, from=1-2, to=3-3]
	\arrow[hook, from=1-4, to=1-5]
	\arrow["\sim", hook, from=1-4, to=2-4]
	\arrow["\ulcorner"{description, pos=0.9}, draw=none, from=1-4, to=2-5]
	\arrow["\sim", hook, from=1-5, to=2-5]
	\arrow["\sim", bend left, hook, from=1-5, to=3-6]
	\arrow[hook, from=2-1, to=2-2]
	\arrow[hook, from=2-1, to=3-3]
	\arrow[dashed, from=2-2, to=3-3]
	\arrow[hook, from=2-4, to=2-5]
	\arrow[hook, from=2-4, to=3-6]
	\arrow[dashed, from=2-5, to=3-6]
      \end{tikzcd}\]
    The vertical maps $X_a \trivialcof X_c$, $X_b \trivialcof X_c$, $Y_a \trivialcof Y_c$ and $Y_b \trivialcof Y_c$, are trivial cofibrations since $X$ and $Y$ are cofibrant in $\catM^J_{Loc}$. Remains to see that they are cofibrations. From the Reedy condition we have that the map $X_c\sqcup_{L_c X} L_cY \cofibration Y_c$ is a cofibration, and observe that the domains of the maps $X_c\sqcup_{X_a} Y_a \to Y_c$ and $X_c\sqcup_{X_b} Y_b \to Y_c$ are contained in the colimit $X_c\sqcup_{L_c X} L_cY$. Therefore, the maps factor as composition of cofibrations
    \[
      X_c\sqcup_{X_a} Y_a \cofibration X_c\sqcup_{L_c X} L_cY \cofibration Y_c \text{ and }
      X_c\sqcup_{X_b} Y_b \cofibration X_c\sqcup_{L_c X} L_cY \cofibration Y_c,
    \]
    which concludes the proof.
      \end{proof}
              
              \begin{lemma} \label{localized-triv-cof:lem}
                Let $X\in \catM^J_{Loc}$ cofibrant and $X\to Z \in \catM^J_{Reedy}$ a Reedy trivial cofibration. Then $Z$ is cofibrant in $\catM^J_{Loc}$. Furthermore, $X \to Z$ is a trivial cofibration in $\catM^J_{Loc}$.
      \end{lemma}
      \begin{proof}
     Since $ X \trivialcof Z $ is a Reedy trivial cofibration, then $ X_a\trivialcof Z_a $, $X_b\sqcup_{X_a\sqcup X_a} (Z_a\sqcup Z_a) \trivialcof Z_b$ and $ X_c\sqcup_{(X_b\sqcup_{X_a} X_b)} (Z_b\sqcup_{Z_a} Z_b) \trivialcof Z_c $ are trivial cofibrations. We then obtain the following diagram:
    	\[
    	\begin{tikzcd}
    	X_a\sqcup X_a \ar[d,"\sim", hook] \ar[r] \ar[rd,phantom, "\ulcorner", very near end] & X_b \ar[d,"\sim",hook] \ar[rdd,bend left]&  \\
    	Z_a\sqcup Z_a \ar[r]   & \bullet \ar[rd, "\sim",hook] & \\
    	\text{} &  & Z_b
    	\end{tikzcd}
    	\]
    	 This shows that $ X_b\trivialcof Z_b $ is a trivial cofibration. Since $ X $ is cofibrant then all the maps in the diagram
    	\[\begin{tikzcd}
    	X_a \ar[r,shift left] \ar[r,shift right] & X_b  \ar[r] & X_c
      \end{tikzcd}\]
    are trivial cofibrations. Consider the commutative diagram where the back and front faces are pushouts
    	\[
    	\begin{tikzcd}
    	 X_a \ar[rrdd,phantom,"\ulcorner",very near end] \ar[rr,"\sim",hook] \ar[dd,"\sim",hook] \ar[rd,"\sim",hook] &  & X_b \ar[rd,"\sim",hook] \ar[dd,"\sim",near start,hook] &  \\
    	  & Z_a \ar[rr,"\sim",near start,hook] \ar[dd,"\sim",near start,hook] \ar[rrdd,phantom,"\ulcorner",very near end] &  & Z_b \ar[dd,"\sim",hook] \\
    	 X_b \ar[rr,"\sim",near start,hook] \ar[rd,"\sim",hook] &  & X_b\sqcup_{X_a}X_b \ar[rd] &  \\
    	  & Z_b \ar[rr,"\sim",hook] &  & Z_b\sqcup_{Z_a}Z_b,
    	\end{tikzcd}
    	\]
    	which, by the two-out-of-three, shows that $ X_b\sqcup_{X_a}X_b \trivialcof Z_b\sqcup_{Z_a}Z_b$ is a trivial cofibration. Remains to prove that $ Z_b\trivialcof Z_c $ is a trivial cofibration. The pushout
        \[
        \begin{tikzcd}
        X_b\sqcup_{X_a} X_b \ar[d,"\sim", hookrightarrow] \ar[r] \ar[rd,phantom, "\ulcorner", very near end] & X_c \ar[d,"\sim",hook] \ar[rdd, bend left] & \\
        Z_b\sqcup_{Z_a} Z_b \ar[r]  & \bullet \ar[rd, "\sim",hookrightarrow] &  \\
        &  & Z_c 
        \end{tikzcd}
        \]
        shows that $ X_c\trivialcof Z_c $ is a trivial cofibration. Note that $ Z $ is Reedy cofibrant, hence $ Z_b\cofibration Z_c $ is a cofibration. By the two-out-of-three property, we can conclude that $ Z_b\trivialcof Z_c $ is indeed a trivial cofibration. The above says that $ Z  $ is cofibrant.

        The second part is also true, since $X\to Z$ is a level-wise weak equivalence.
      \end{proof}

      \begin{corollary} \label{factorization-trivial-cofibration-fibration}
        Any map between diagrams $ f:X \to Y $, where $X$ is a cofibrant diagram $ X $ and $Y$ is a fibrant diagram in $ \catM^J_{Loc} $, can be factored as a trivial cofibration followed by a fibration.
      \end{corollary}
      \begin{proof}
        We factor $ f:X \to Y $ in $ \catM^J_{Reedy}$ to obtain $X \trivialcof Z \fibration Y$. $Z \fibration Y$ is also a fibration in $ \catM^J_{Loc} $ as is it is level-wise. Finally, $X \trivialcof Z \in \catM^J_{Loc}$ by the previous \cref{localized-triv-cof:lem}.
      \end{proof}

      %%%%%%%%%%%%%%%%%%%%

      For the factorization of a diagram map $f:X \to Y$ in $\catM^J$, with $X$ cofibrant and $Y$ fibrant, into a cofibration followed by a trivial fibration we will need an auxiliary class of diagrams.

      \begin{construction} \label{homotopy-limit:cons}
        Denote by $ K $ the category $ J $ with the opposite Reedy structure given above (the degree function reversed). We endow $ \catM^K $ with the Reedy model structure. Then a diagram $ Y\in\catM^K_{Reedy} $ is fibrant if $ Y_c\display 1 $, $ Y_b \display Y_c $ and $ Y_a \display Y_b\times_{Y_c} Y_b $ are fibrations in $ \catM $. In this situation $ Y_b $ is also fibrant.
      \end{construction}

      The limit of a diagram $Y \in \catM^K$ is simply the equalizer $ Eq(Y_i,Y_j) $. Note that the following pullback also computes the limit of $ Y $:
	\[
	\begin{tikzcd}
		P \ar[r] \ar[d,twoheadrightarrow] \ar[rd,phantom,"\lrcorner",very near start] & Y_a \ar[d,twoheadrightarrow] \\
		Y_b \ar[r] & Y_b\times_{Y_c} Y_b.
	\end{tikzcd}
	\]
	From this we conclude that $ \Lim Y $ is a fibrant object of $ \catM $ if $Y \in \catM^K_{Reedy}$ is fibrant, and letting $ Z $ to denote the constant diagram at $ \Lim Y $ then this comes with a diagram map $ Z \to Y $ of the following form
	\[
	\begin{tikzcd}
	 \Lim Y \ar[d,twoheadrightarrow] \ar[r,shift left] \ar[r,shift right] & \Lim Y \ar[d,twoheadrightarrow] \ar[r] & \Lim Y \ar[d,twoheadrightarrow]\\
	 Y_a \ar[r,shift left] \ar[r,shift right] & Y_b  \ar[r] & Y_c
	\end{tikzcd}
	\]
	where all top arrows are identities. Finally, note that $ Y $ being fibrant in $ \catM^K_{Reedy} $ implies that both maps $ \begin{tikzcd}
	Y_a \ar[r,shift left] \ar[r,shift right] & Y_b
	\end{tikzcd} $ are fibrations. This can be deduced from the following diagram:
	\[
	 \begin{tikzcd}
	  Y_a \ar[rrd,bend left] \ar[ddr,bend right] \ar[rd,twoheadrightarrow] & & \\
	   & Y_b\times_{Y_c} Y_b \ar[r,twoheadrightarrow] \ar[d,twoheadrightarrow] & Y_b \ar[d,twoheadrightarrow] \\
	   & Y_b \ar[r,twoheadrightarrow] & Y_c
	 \end{tikzcd}
	\]
	
        \begin{observation}
          Recall that the fibrations in $\catM^J_{Loc}$ are the level-wise fibrations. Since $ Z \in \catM^K$ is point-wise fibrant then it is Reedy fibrant in $ \catM^J_{Loc} $. Similarly, $ Y $ is Reedy fibrant in $ \catM^K_{Reedy} $, in particular, implies that is object-wise fibrant, so it is fibrant in $ \catM^J_{Loc}$. We will use this diagram $Z$ throughout this section.
        \end{observation}

\begin{lemma} \label{cofacycfib1}
	The map $ Z \to Y $ from above is a trivial fibration in $\catM^J_{Loc}$.
      \end{lemma}
      \begin{proof}
        We show that the map has the right lifting property against any cofibration $A \cofibration B \in \catM_{Loc}^J$. First, assume that $ A=\emptyset $, and $B$ is a cofibrant object in $ \catM^J_{Loc} $ and $ Y $ a fibrant diagram in $ \catM^K_{Reedy} $. We consider the lifting problem in $ \catM^J_{Loc}$:
	\[
	\begin{tikzcd}
	\emptyset \ar[r] \ar[d,hook] & Z \ar[d] \\
	B \ar[r] & Y
	\end{tikzcd}
	\]	
	From the discussion above we obtain the following commutative diagram:	
	\[
	\begin{tikzcd}
	B_a \ar[r,shift left,hook,"\sim"] \ar[r,shift right,"\sim"',hook] \ar[d] & B_b  \ar[r,"\sim",hook] \ar[d] & B_c \ar[d] \\
	Y_a \ar[r,shift left,twoheadrightarrow] \ar[r,shift right,twoheadrightarrow] & Y_b  \ar[r,twoheadrightarrow] & Y_c
	\end{tikzcd}
	\]
	Thus, we obtain the following lifts:
	\[
	\begin{tikzcd}[sep=small]
	 B_a \ar[d,hook,"\sim","Bi"'] \ar[r] & Y_a \ar[d,twoheadrightarrow,"Yi"] \\
	 B_b \ar[r] \ar[ru,dashrightarrow,"l_i"'] & Y_b
	\end{tikzcd} \quad
	\begin{tikzcd}[sep=small]
	B_a \ar[d,hook,"\sim","Bj"'] \ar[r] & Y_a \ar[d,twoheadrightarrow,"Yj"] \\
	B_b \ar[r] \ar[ru,dashrightarrow,"l_j"'] & Y_b
	\end{tikzcd} \quad
	\begin{tikzcd}[sep=small]
	B_b \ar[d,hook,"\sim","Bk"'] \ar[r] & Y_b \ar[d,twoheadrightarrow,"Yk"] \\
	B_c \ar[r] \ar[ru,dashrightarrow,"l_k"'] & Y_c
	\end{tikzcd}
	\]
	Using this we can construct the following commutative diagram:	
	\[
	\begin{tikzcd}
	B_a \ar[r,hook,"\sim"] \ar[d,hook,"\sim"] \ar[rd,phantom,"\ulcorner",very near end] & B_b \ar[d,hook,"\sim"] \ar[rd,"l_j"] &  &  \\
	B_b \ar[r,hook,"\sim"] \ar[rd,"B_k"'] & B_b\sqcup_{B_a} B_b \ar[d,hook,"\sim"] \ar[r] & Y_a \ar[rd,"Y_j"] \ar[d,twoheadrightarrow] &  \\
	 & B_c \ar[rd,"l_k"'] \ar[r] & Y_b\sqcup_{Y_c} Y_b \ar[rd,phantom,"\lrcorner",very near start] \ar[d,twoheadrightarrow] \ar[r,twoheadrightarrow] & Y_b \ar[d,twoheadrightarrow] \\
	 &  & Y_b \ar[r,twoheadrightarrow] & Y_c 
	\end{tikzcd}
	\]
	where the middle trivial cofibration and fibration come from $ B $ being cofibrant in $ \catM^J_{Loc} $ and $ Y $ being fibrant in $ \catM^K_{Reedy} $ respectively. Then there exist a map $ B_c \overset{r}{\to} Y_a $ that fits in the diagram. Furthermore, we readily see from the diagram that $ Y_jr=l_k=Y_ir $. Therefore, there is a unique arrow $ B_c \overset{t}{\to} Eq(Y_i,Y_j)=\Lim Y $ making the obvious triangle commutative. By taking the appropriate compositions with the map $ t $ we can construct a diagram map $ B \to Z $ such that is a solution to the lifting problem.
        
	For the general case
	\[
	\begin{tikzcd}
	A \ar[r] \ar[d,hook] & Z \ar[d] \\
	B \ar[r] & Y
	\end{tikzcd}
	\]
	one can play the same game, the only change is that the diagram is a bit more involved.
\end{proof}

The diagram $Z$ from \cref{homotopy-limit:cons} is not necessarily Reedy cofibrant, but it is almost cofibrant in $\catM^J_{Loc}$ as the maps in it are trivial cofibrations. The only missing part is that $\lim Y$ is not cofibrant in $\catM$. In order to obtain cofibrant diagram in $\catM^J_{Loc}$, we include the following result.

\begin{lemma} \label{acycfib2}
  If $Y \in \catM^K_{Reedy}$ is fibrant then there exists a trivial fibration $W \fibration Y \in \catM^J_{Loc}$ with $W \in \catM^J_{Loc}$ cofibrant.
\end{lemma}
\begin{proof}
  Since $Y$ is fibrant in $\catM^K_{Reedy}$, then it is fibrant in $\catM^J_{Loc}$ as these are level-wise fibrant. Similarly, $Z$ from \cref{homotopy-limit:cons} is fibrant in $\catM^J_{Loc}$, which also comes with a trivial fibration $Z \trivialfib Y$ by \cref{cofacycfib1}. We can take a Reedy cofibrant replacement $W \trivialfib Z$. Since this last map is in particular a level-wise weak equivalence, it implies that the maps in $W$ are weak equivalences. By 2-out-of-3 property, the maps in $W$ are trivial cofibrations. This makes $W$ a cofibrant replacement in $\catM^J$ of $Y$ by composing the trivial fibrations $W\trivialfib Z \trivialfib Y$.
  \end{proof}

  Before giving the factorization, we need a technical result that follows from the next lemma.

  \begin{remark} \label{slice:wms}
    From \cite[2.1.11 Proposition]{henry20weak}, if $A \in \catM$ is cofibrant then the coslice category $A/\catM$ inherits a weak model structure from $\catM$ where a map in $A/\catM$ is cofibration, fibration and weak equivalences if it is one in $\catM$. Dually, one induces a weak model structure on the slice $\catM/Y$ if $Y$ is fibrant.
  \end{remark}

  \begin{construction} \label{double-slice:construction}
    Consider a map $f:A \to Y$ in $\catM$ where $A$ is cofibrant and $Y$ is fibrant. Consider $A/\catM$ with the weak model structure described in the previous \cref{slice:wms}. 

    The map $f: A \to Y$ allows us to see $Y$ as an object in $A/\catM$, which is fibrant as $Y$ is fibrant in $\catM$.  So, we can take the slice $(A/\catM)/Y$. Objects of $(A/\catM)/Y$ are factorizations of the form
  % https://q.uiver.app/#q=WzAsMyxbMCwwLCJBIl0sWzEsMSwiWSJdLFswLDEsIlciXSxbMCwxXSxbMCwyXSxbMiwxXV0=
\[\begin{tikzcd}
	A \\
	W & Y.
	\arrow[from=1-1, to=2-1]
	\arrow[from=1-1, to=2-2,"f"]
	\arrow[from=2-1, to=2-2]
      \end{tikzcd}\]
    
  Let two objects in this category
% https://q.uiver.app/#q=WzAsNyxbMCwwLCJBIl0sWzAsMSwiQiJdLFsxLDEsIlkiXSxbMywwLCJBIl0sWzQsMSwiWSJdLFs0LDAsIlgiXSxbMiwwLCJcXHRleHR7IGFuZCB9Il0sWzAsMV0sWzEsMl0sWzAsMiwiZiJdLFszLDVdLFs1LDRdLFszLDQsImYiLDJdXQ==
\[\begin{tikzcd}
	A && {\text{ and }} & A & X \\
	B & Y &&& Y
	\arrow[from=1-1, to=2-1]
	\arrow["f", from=1-1, to=2-2]
	\arrow[from=1-4, to=1-5]
	\arrow["f"', from=1-4, to=2-5]
	\arrow[from=1-5, to=2-5]
	\arrow[from=2-1, to=2-2]
\end{tikzcd}\]
    which we refer to as $B$ and $X$. A map from $B$ to $X$ is a diagonal filler of the resulting commutative square:
    % https://q.uiver.app/#q=WzAsNCxbMCwwLCJBIl0sWzAsMSwiQiJdLFsxLDEsIlkiXSxbMSwwLCJYIl0sWzAsMV0sWzEsMl0sWzAsM10sWzMsMl0sWzEsM11d
\[\begin{tikzcd}
	A & X \\
	B & Y
	\arrow[from=1-1, to=1-2]
	\arrow[from=1-1, to=2-1]
	\arrow[from=1-2, to=2-2]
	\arrow[from=2-1, to=1-2,dashed]
	\arrow[from=2-1, to=2-2]
\end{tikzcd}\]

A cofibrant object in $(A/\catM)/Y$ is one in which the first map is a cofibration in $\catM$, and a fibrant object when the last map is a fibration \ie
% https://q.uiver.app/#q=WzAsNyxbMCwwLCJBIl0sWzAsMSwiQiJdLFsxLDEsIlkiXSxbMywwLCJBIl0sWzQsMSwiWSJdLFs0LDAsIlgiXSxbMiwwLCJcXHRleHR7IGFuZCB9Il0sWzAsMSwiIiwwLHsic3R5bGUiOnsidGFpbCI6eyJuYW1lIjoiaG9vayIsInNpZGUiOiJ0b3AifX19XSxbMSwyXSxbMCwyLCJmIl0sWzMsNV0sWzUsNCwiIiwwLHsic3R5bGUiOnsiaGVhZCI6eyJuYW1lIjoiZXBpIn19fV0sWzMsNCwiZiIsMl1d
\[\begin{tikzcd}
	A && {\text{ and }} & A & X \\
	B & Y &&& Y
	\arrow[hook, from=1-1, to=2-1]
	\arrow["f", from=1-1, to=2-2]
	\arrow[from=1-4, to=1-5]
	\arrow["f"', from=1-4, to=2-5]
	\arrow[two heads, from=1-5, to=2-5]
	\arrow[from=2-1, to=2-2]
      \end{tikzcd}\]
    respectively. Also note that the category $(A/\catM)/Y$ coincides with $A/(\catM/Y)$, both as categories and as model categories.
    \end{construction}

    \begin{observation} \label{ho-adjunction:weak}
      \cite[2.4.3 Proposition]{henry20weak} observed that the Quillen adjunction descends to the homotopy categories: If $F: \catC \rightleftarrows \catD:G$ is a Quillen pair,  then we obtain a natural isomorphism $$\Ho(\catC^{\bif})(W,G(Z))\cong \Ho(\catD^{\bif})(F(W),Z)$$ of the homotopy categories.

      The category $\Ho(\catC^{\bif})$ is the localization of the subcategory of bifibrant objects at trivial (co)fibrations. This is the content of \cite[2.2.6 Theorem]{henry20weak}, which also proves that there are equivalences $$\Ho(\catC^\cof) \cong \Ho(\catC^{\bif}) \cong \Ho(\catC^\fib)$$ where the first category is the localization of $\catC^\cof$ at trivial cofibrations, and  the second is the localization of $\catC^\fib$ at trivial fibrations. Therefore, up to these equivalences of categories, we say that $\Ho(F): \Ho(\catC^\cof) \to \Ho(\catD^\cof)$ and $\Ho(G):\Ho(\catD^\fib) \to \Ho(\catC^\fib)$ are ``adjoint''.
    \end{observation}
  
  \begin{lemma}\label{liftequiv}
	For all $ i: A \cofibration B$ and $ i': A' \cofibration B'$ cofibrations between cofibrant objects, for all $ p:X \fibration Y $ fibration between fibrant objects, if there is a commutative diagram:
	\[
	\begin{tikzcd}
	A \ar[r,"\sim","k"'] \ar[d,hookrightarrow,"i"'] & A' \ar[d,hookrightarrow,"i'"] \\
	B \ar[r,"\sim","l"'] & B'
	\end{tikzcd}
	\]
	then  $ i\pitchfork p $ if and only if $ i'\pitchfork p $. The dual statement also holds: For all $ i: A \cofibration B$ core cofibrations, for all $ p:X \fibration Y $ and $ p':X' \fibration Y' $ fibrations between fibrant objects, if there is a commutative diagram:
	\[
	\begin{tikzcd}
	X \ar[r,"\sim","m"'] \ar[d,twoheadrightarrow,"p"'] & X' \ar[d,twoheadrightarrow,"p'"] \\
	Y \ar[r,"\sim","n"'] & Y'
	\end{tikzcd}
	\]
	then  $ i\pitchfork p $ if and only if $ i\pitchfork p' $. 
      \end{lemma}
      \begin{proof}
        We prove the first part of the lemma, the second part is dual. We have the following commutative squares
        % https://q.uiver.app/#q=WzAsMTIsWzAsMCwiQSJdLFswLDEsIkIiXSxbMSwwLCJBJyJdLFsxLDEsIkInIl0sWzIsMCwiQSJdLFsyLDEsIkIiXSxbMywwLCJYIl0sWzMsMSwiWSJdLFs0LDAsIkEnIl0sWzQsMSwiQiciXSxbNSwwLCJYIl0sWzUsMSwiWSJdLFswLDEsImkiLDIseyJzdHlsZSI6eyJ0YWlsIjp7Im5hbWUiOiJob29rIiwic2lkZSI6InRvcCJ9fX1dLFsyLDMsImknIiwwLHsic3R5bGUiOnsidGFpbCI6eyJuYW1lIjoiaG9vayIsInNpZGUiOiJ0b3AifX19XSxbMCwyLCJrIl0sWzEsMywibCIsMl0sWzQsNSwiaSIsMix7InN0eWxlIjp7InRhaWwiOnsibmFtZSI6Imhvb2siLCJzaWRlIjoidG9wIn19fV0sWzYsNywicCIsMCx7InN0eWxlIjp7ImhlYWQiOnsibmFtZSI6ImVwaSJ9fX1dLFs1LDcsImciLDJdLFs0LDYsImYiXSxbOCw5LCJpJyIsMix7InN0eWxlIjp7InRhaWwiOnsibmFtZSI6Imhvb2siLCJzaWRlIjoidG9wIn19fV0sWzEwLDExLCJwIiwwLHsic3R5bGUiOnsiaGVhZCI6eyJuYW1lIjoiZXBpIn19fV0sWzksMTEsImcnIiwyXSxbOCwxMCwiZiciXV0=
\[\begin{tikzcd}
	A & {A'} & A & X & {A'} & X \\
	B & {B'} & B & Y & {B'} & Y
	\arrow["k","\sim"', from=1-1, to=1-2]
	\arrow["i"', hook, from=1-1, to=2-1]
	\arrow["{i'}", hook, from=1-2, to=2-2]
	\arrow["f", from=1-3, to=1-4]
	\arrow["i"', hook, from=1-3, to=2-3]
	\arrow["p", two heads, from=1-4, to=2-4]
	\arrow["{f'}", from=1-5, to=1-6]
	\arrow["{i'}"', hook, from=1-5, to=2-5]
	\arrow["p", two heads, from=1-6, to=2-6]
	\arrow["l"',"\sim", from=2-1, to=2-2]
	\arrow["g"', from=2-3, to=2-4]
	\arrow["{g'}"', from=2-5, to=2-6]
\end{tikzcd}\]
The proof relies heavily on \cref{double-slice:construction}: The middle square above corresponds to a pair of objects $B,X$ in a double slice category $A/\catM/Y$, and a diagonal filler witnessing that $ i \pitchfork p $ is a map in this double slice category.

We start with the induced weak model structure on the slice $\catM/Y$. Note that from \cite[2.4.2 Example]{henry20weak} the weak equivalence $k:A \to A'$ induces a weak Quillen equivalence $P_k: A/(\catM/Y) \leftrightarrows A'/(\catM/Y):U_k$. Observe that $B$, $B'$ are cofibrant and $Y$ is fibrant. In what follows we leave $Y$ implicit as we work in the slice $(A/\catM)/Y$, here we use that $(A/\catM)/Y=A/(\catM/Y)$ from \cref{double-slice:construction}.

The functor $P_k$ takes a cofibration $A \cofibration C$ along $k:A \to A'$, while $U_k$ precomposes with $k$. Using the following diagram, since $P_kB$ is cofibrant, by the 2-out-of-3 property
    % https://q.uiver.app/#q=WzAsNSxbMCwwLCJBIl0sWzAsMSwiQiJdLFsxLDAsIkEnIl0sWzEsMSwiUCJdLFsyLDIsIkInIl0sWzAsMSwiaSIsMix7InN0eWxlIjp7InRhaWwiOnsibmFtZSI6Im1vbm8ifX19XSxbMCwyLCJrIl0sWzEsMywiXFxzaW0iLDJdLFsyLDMsIiIsMCx7InN0eWxlIjp7InRhaWwiOnsibmFtZSI6Im1vbm8ifX19XSxbMCwzLCJcXGxyY29ybmVyIiwxLHsibGFiZWxfcG9zaXRpb24iOjkwLCJzdHlsZSI6eyJib2R5Ijp7Im5hbWUiOiJub25lIn0sImhlYWQiOnsibmFtZSI6Im5vbmUifX19XSxbMiw0LCIiLDIseyJzdHlsZSI6eyJ0YWlsIjp7Im5hbWUiOiJtb25vIn19fV0sWzEsNCwiXFxzaW0iLDIseyJjdXJ2ZSI6MX1dLFszLDQsIiIsMix7InN0eWxlIjp7ImJvZHkiOnsibmFtZSI6ImRhc2hlZCJ9fX1dXQ==
\[\begin{tikzcd}
	A & {A'} \\
	B & P_kB \\
	&& {B'}
	\arrow["k", from=1-1, to=1-2]
	\arrow["i"', hook, from=1-1, to=2-1]
	\arrow["\lrcorner"{description, pos=0.9}, draw=none, from=1-1, to=2-2]
	\arrow[hook, from=1-2, to=2-2]
	\arrow[hook, from=1-2, to=3-3]
	\arrow["\sim"', from=2-1, to=2-2]
	\arrow["\sim"', bend right, from=2-1, to=3-3]
	\arrow[dashed, from=2-2, to=3-3]
      \end{tikzcd}\]
    we see that there is a weak equivalence $P_kB \overset{\sim}{\to} B' $, this implies they are isomorphic in $\Ho(A'/(\catM/Y))$. We have:
    \begin{align*}
      \Hom_{\Ho(A'/(\catM/Y))}(B',X) &\cong \Hom_{\Ho(A'/(\catM/Y))}(P_k(B),X) \\
                                     &\cong \Hom_{\Ho(A/(\catM/Y))}(B, U_k(X)) \\
                                     &\cong \Hom_{\Ho(A/(\catM/Y))}(B,X).
    \end{align*}
    The first isomorphism follows from $B' \cong P_k(B)$ in $\Ho(A'/(\catM/Y))$, the second is the weak Quillen adjunction $P_k \dashv U_k$ applied to the cofibrant object $B \in (A/\catM)/Y$ and the fibrant object $X \in (A'/\catM)/Y$. We crucially use \cref{ho-adjunction:weak}, so the second isomorphism is really up to some equivalence of categories.

    Now we use $ \Hom_{\Ho(A'/(\catM/Y))}(B',X) \cong \Hom_{\Ho(A/(\catM/Y))}(B,X)$ to conclude. First, recall that a diagonal filler of
    % https://q.uiver.app/#q=WzAsNCxbMCwwLCJBIl0sWzAsMSwiQiJdLFsxLDEsIlkiXSxbMSwwLCJYIl0sWzAsMSwiIiwwLHsic3R5bGUiOnsidGFpbCI6eyJuYW1lIjoiaG9vayIsInNpZGUiOiJ0b3AifX19XSxbMywyLCIiLDAseyJzdHlsZSI6eyJoZWFkIjp7Im5hbWUiOiJlcGkifX19XSxbMSwyXSxbMCwzXV0=
\[\begin{tikzcd}
	A & X \\
	B & Y
	\arrow[from=1-1, to=1-2]
	\arrow[hook, from=1-1, to=2-1]
	\arrow[two heads, from=1-2, to=2-2]
	\arrow[from=2-1, to=2-2]
      \end{tikzcd}\]
    is the same as a map $B \to X $ in $A/\catM/Y$, and similarly for $B'$ and $X$. Assume that $i \pitchfork p$, this give us a map $B \to X$ in $\Ho(A/\catM/Y)$. Using the isomorphism, we have a map $B' \to X$ in $\Ho(A'/\catM/Y)$, from which we can select a representative of the homotopy class, which implies that $i'\pitchfork p$. Similarly, we get that $i'\pitchfork p$ implies $i \pitchfork p$.
      \end{proof}

\begin{lemma}\label{factorization-cofibration-trivial-fibration}
	Let $ X \to Y $ be a map in $ \catM^J $ with $ X $ cofibrant and $ Y $ fibrant. Then such a map can be factored as a cofibration followed by a trivial fibration.
\end{lemma}
\begin{proof}
  Observe first that $Y$ can be assumed to be Reedy cofibrant in $\catM^J$. Indeed, we can simply take a Reedy cofibrant replacement $Y' \trivialfib Y$, and instead use the dashed arrow
  % https://q.uiver.app/#q=WzAsNCxbMCwxLCJYIl0sWzAsMCwiMCJdLFsxLDAsIlknIl0sWzEsMSwiWSJdLFsyLDMsIlxcc2ltIiwwLHsic3R5bGUiOnsiaGVhZCI6eyJuYW1lIjoiZXBpIn19fV0sWzAsM10sWzEsMCwiIiwyLHsic3R5bGUiOnsidGFpbCI6eyJuYW1lIjoiaG9vayIsInNpZGUiOiJ0b3AifX19XSxbMSwyLCIiLDAseyJzdHlsZSI6eyJ0YWlsIjp7Im5hbWUiOiJob29rIiwic2lkZSI6InRvcCJ9fX1dLFswLDIsIiIsMSx7InN0eWxlIjp7ImJvZHkiOnsibmFtZSI6ImRhc2hlZCJ9fX1dXQ==
\[\begin{tikzcd}
	0 & {Y'} \\
	X & Y.
	\arrow[hook, from=1-1, to=1-2]
	\arrow[hook, from=1-1, to=2-1]
	\arrow["\sim", two heads, from=1-2, to=2-2]
	\arrow[dashed, from=2-1, to=1-2]
	\arrow[from=2-1, to=2-2]
      \end{tikzcd}\]
    Under this assumption, $Y$ is point-wise cofibrant, whence Reedy cofibrant in $\catM^K$. Therefore, we can take a fibrant replacement in $\catM^K$, $Y \trivialcof Y'$. Using \cite[Corollary 2.4.4]{henry20weak} equivalences are preserved under pullbacks along fibrations, so we get the pullback square
    % https://q.uiver.app/#q=WzAsNCxbMSwwLCJXIl0sWzEsMSwiWSciXSxbMCwxLCJZIl0sWzAsMCwiTFkiXSxbMiwxLCJcXHNpbSIsMCx7InN0eWxlIjp7InRhaWwiOnsibmFtZSI6Imhvb2siLCJzaWRlIjoidG9wIn19fV0sWzAsMSwiIiwyLHsic3R5bGUiOnsiaGVhZCI6eyJuYW1lIjoiZXBpIn19fV0sWzMsMiwiIiwwLHsic3R5bGUiOnsiaGVhZCI6eyJuYW1lIjoiZXBpIn19fV0sWzMsMCwiXFxzaW0iXV0=
\[\begin{tikzcd}
	LY & W \\
	Y & {Y'.}
	\arrow["\sim", from=1-1, to=1-2]
	\arrow[two heads, from=1-1, to=2-1]
	\arrow[two heads, from=1-2, to=2-2]
	\arrow["\sim", hook, from=2-1, to=2-2]
      \end{tikzcd}\]
    Furthermore, we know from \cref{acycfib2} that $W \fibration Y'$ is a trivial fibration in $\catM^J$. Therefore, it has the right lifting property against any cofibration between cofibrant objects in $\catM^J$. We can use \cref{liftequiv} to conclude that $LY \fibration Y$ satisfies the same property, \ie it is a trivial fibration in $\catM^J$. Since $X$ is cofibrant, we obtain a lift
    % https://q.uiver.app/#q=WzAsNCxbMCwwLCIwIl0sWzAsMSwiWCJdLFsxLDAsIkxZIl0sWzEsMSwiWSJdLFswLDEsIiIsMCx7InN0eWxlIjp7InRhaWwiOnsibmFtZSI6Imhvb2siLCJzaWRlIjoidG9wIn19fV0sWzIsMywiXFxzaW0iLDAseyJzdHlsZSI6eyJoZWFkIjp7Im5hbWUiOiJlcGkifX19XSxbMCwyLCIiLDEseyJzdHlsZSI6eyJ0YWlsIjp7Im5hbWUiOiJob29rIiwic2lkZSI6InRvcCJ9fX1dLFsxLDNdLFsxLDIsIiIsMSx7InN0eWxlIjp7ImJvZHkiOnsibmFtZSI6ImRhc2hlZCJ9fX1dXQ==
\[\begin{tikzcd}
	0 & LY \\
	X & Y.
	\arrow[hook, from=1-1, to=1-2]
	\arrow[hook, from=1-1, to=2-1]
	\arrow["\sim", two heads, from=1-2, to=2-2]
	\arrow[dashed, from=2-1, to=1-2]
	\arrow[from=2-1, to=2-2]
      \end{tikzcd}\]
The map $X \to LY$ can be factored in the Reedy model structure $\catM^J$ as $X \cofibration X' \trivialfib LY$. The diagram $X'$ is cofibrant in $\catM^J$ since it is equivalent to the cofibrant diagram $LY$, and $X$ is cofibrant by assumption. Therefore, it follows from \cref{localized-triv-cof:lem} that the Reedy cofibration $X \cofibration X'$ is a cofibration in the model $\catM^J$. This gives us the desired factorization in $\catM^J$, $X \cofibration X' \trivialfib Y$.
\end{proof}

All the previous work can be summarized in the following proof of \cref{model-paths:wms}. This proves that the category of diagrams $\catM^J$ has a weak model structure with the specified cofibrations and fibrations, which, as explained above, encodes objects with a weak cylinder object. We remark that our proof will show that the conditions of \cite[2.1.10 Definition]{henry20weak} are satisfied instead of \cref{def:wms}. The reason is for this is that in \cref{model-paths:wms} we do not have an explicit class of weak equivalences. More precisely, we will use \cite[2.3.3 Proposition]{henry20weak} which gives some alternative criteria to obtain a weak model structure in this sense.

\begin{proof}(\cref{model-paths:wms})
  Note first that we have the Reedy weak model structure on $\catM^J$ by virtue of \cref{reedy-model:theorem}. Also, the existence of initial and terminal diagrams is clear. We must justify that the class of (co)fibrations form a class of (co)fibrations in $\catM^J$. For fibrations, since these are level-wise, it is immediate that: the terminal diagram is fibrant, any isomorphism with fibrant codomain is a fibration, the class is closed under compositions, and stable under pullbacks along maps between fibrant objects.

  The dual conditions must be verified for the class of cofibrations. That the initial diagram is cofibrant it is immediate to verify. To see other stability conditions, we observe these are true for $\catM^J_{Reedy}$. In addition, for stability under isomorphisms we use repeatedly that maps in $\catM$ isomorphic to trivial cofibration are also trivial cofibrations. This simply because the new condition we added involves the requirement that certain maps trivial cofibrations. Stability under pushouts follows from the stability in $\catM^J_{Reedy}$ and the fact that trivial cofibrations in the weak model $\catM$ are pushout stable.

  The factorization of a map $f:X \to Y$, where $X$ is cofibrant and $Y$ is fibrant, into a cofibration followed by a trivial fibration is the content of \cref{factorization-cofibration-trivial-fibration}.

  The factorization of a map $f:X \to Y$, where $X$ is cofibrant and $Y$ is fibrant, into a trivial cofibration followed by a fibration is guaranteed by \cref{factorization-trivial-cofibration-fibration}.

  In order to conclude, we use \cite[2.3.3 Proposition]{henry20weak}. For which we need to verify that a cofibration $X \to Y \in \catM^J_{Loc}$ with $X$ cofibrant and $Y$ fibrant admit a relative strong cylinder object. Firstly, we know that the map admits a relative cylinder object in $\catM^J_{Reedy}$:
  % https://q.uiver.app/#q=WzAsMyxbMCwwLCJZXFxjb3Byb2RfWFkiXSxbMSwwLCJZIl0sWzAsMSwiSV9YWSJdLFswLDIsIiIsMCx7InN0eWxlIjp7InRhaWwiOnsibmFtZSI6Imhvb2siLCJzaWRlIjoidG9wIn19fV0sWzAsMV0sWzIsMV1d
\[\begin{tikzcd}
	{Y\coprod_XY} & Y \\
	{I_XY}
	\arrow[from=1-1, to=1-2]
	\arrow[hook, from=1-1, to=2-1]
	\arrow[from=2-1, to=1-2]
      \end{tikzcd}\]
    with $Y \cofibration Y \coprod_X Y \cofibration I_XY$ a Reedy trivial cofibration. Since $Y$ is cofibrant in $\catM^J_{Loc}$ we can use \cref{localized-triv-cof:lem} to conclude that $I_XY$ is also cofibrant in $\catM^J_{Loc}$, and that the map $Y\to I_XY$ is a trivial cofibration in $\catM^J_{Loc}$. Now we have cofibrant objects $Y \coprod_X Y$, $I_XY$ in $\catM^J_{Loc}$ and a Reedy cofibration between them, so we use \cref{cofibration-loc-eq-cofibration} to conclude it is actually a cofibration in $\catM^J_{Loc}$. This gives us the relative cylinder objects.

  Finally, the 2-out-of-3 property for trivial cofibrations between bifibrant objects follows using that $\catM^J_{Reedy}$ is a weak model category, so the property is true in this Reedy weak model structure. By which we mean that the property is true for the underlying Reedy trivial cofibrations between bifibrant objects of $\catM^J_{Loc}$. \Cref{localized-triv-cof:lem} allows us to conclude that such Reedy trivial cofibrations are indeed trivial cofibrations in $\catM^J_{Loc}$. Now \cite[2.3.3 Proposition]{henry20weak} allows us to conclude that $\catM^J_{Loc}$, with the specified classes of maps, is a weak model category.
\end{proof}

\subsubsection{Weak model on correspondences}

Next, we consider another diagram category $I$:
\[ 0 \to 2 \leftarrow 1 \]
Where $\deg(0) = \deg(1) = 0$ and $\deg(2) = 1$.  Similarly to the previous section, we construct a ``right Bousfield localization'' of the Reedy weak model structure on $\catN^I$.

\begin{theorem} \label{correspondences:wms}
  There is a weak model structure $\catN^I_{Loc}$ on the category of diagrams $\catN^I$ obtained from the Reedy weak model structure $\catN^I_{Reedy}$, where:
  \begin{enumerate}
  \item A map between diagrams $X \to Y$ is a cofibration if
    \begin{enumerate}
    \item It is a Reedy cofibration,
    \item $X_2 \sqcup_{X_1} Y_1 \trivialcof Y_2 $ and $X_2 \sqcup_{X_0} Y_0 \trivialcof Y_2$ are trivial cofibrations in $\catN$.
    \end{enumerate}
  \item Fibrations are level-wise fibrations.
  \end{enumerate}
\end{theorem}

It will be useful to have in mind that for an object $X \in \catN^I$ we have $L_0X=0$ and $L_1X=X_0\sqcup X_1$. So a map $X \to Y$ is a Reedy cofibration if the maps $X_0 \cofibration Y_0$, $X_1 \cofibration Y_1$ and $(Y_0\sqcup Y_1) \sqcup_{(X_0\sqcup X_1)} X_2 \cofibration Y_2$ are cofibrations.

\begin{observation}
Unwinding the definitions, a diagram $X \in \catN^I_{Loc}$ is cofibrant if both maps $X_0 \trivialcof X_2$ and $X_1 \trivialcof X_2$ are trivial cofibrations.  
\end{observation}

The proof of the theorem is completely analogous to \cref{model-paths:wms}. We state the lemmas necessary for this and only comment on the proofs when adequate. 

\begin{lemma} \label{cofibration-loc-eq-cofibration-2}
        Let $X,Y \in \catN^I_{Loc}$ cofibrant. Then, a map $X \to Y$ is a cofibration in $\catN^I_{Loc}$ if and only if it is a cofibration in $\catN^I_{Reedy}$.
      \end{lemma}
      \begin{proof}
        Just as in \cref{cofibration-loc-eq-cofibration} we only prove the interesting direction; assume that $X,Y$ are cofibrant in $\catN_{Loc}^I$ and that $X \to Y \in \catN^I_{Reedy}$ is a Reedy cofibration. Remains to show that \[
          X_2\sqcup_{X_0} Y_0 \to Y_2 \text{ and }  X_2\sqcup_{X_1} Y_1 \to Y_2
        \]
        are trivial cofibrations. Again, the fact that the maps are weak equivalences follow from $X,Y$ being cofibrant and the 2-out-of-3 property. To see that they are cofibrations we can use the Reedy condition just as in \cref{cofibration-loc-eq-cofibration}.
      \end{proof}

      \begin{lemma} \label{localized-triv-cof:lem-2}
                Let $X\in \catN^I_{Loc}$ cofibrant and $X\to Z \in \catN^I_{Reedy}$ a Reedy trivial cofibration. Then $Z$ is cofibrant in $\catN^I_{Loc}$. Furthermore, $X \to Z$ is a trivial cofibration in $\catN^I_{Loc}$.
              \end{lemma}
              \begin{proof}
                The difficult part is to show that $Z$ is cofibrant. Since $X \to Z$ is a Reedy trivial cofibration, then by \cref{cor:core_cof_are_levelwise} we have it is a levelwise trivial cofibration. Then $Z$ is cofibrant by the 2-out-of-3 property.
      \end{proof}

              \begin{corollary} \label{factorization-trivial-cofibration-fibration-2}
        Any map between diagrams $ f:X \to Y $, where $X$ is a cofibrant diagram and $Y$ is a fibrant diagram in $ \catN^I_{Loc} $, can be factored as a trivial cofibration followed by a fibration.
      \end{corollary}
      \begin{proof}
        Now that we have \cref{localized-triv-cof:lem-2}, we can proceed as in \cref{factorization-trivial-cofibration-fibration} by first taking the factorization in $\catN^I_{Reedy}$.
      \end{proof}

      \begin{construction} \label{homotopy-limit:cons-2}
        Denote by $ K' $ the category $ I $ with the opposite Reedy structure given above (the degree function reversed). We endow $ \catN^{K'} $ with the Reedy model structure. Then a diagram $ Y\in\catN^{K'}_{Reedy} $ is fibrant if $ Y_2\display 1 $, $ Y_0 \fibration Y_2 $ and $ Y_1 \fibration  Y_2 $ are fibrations in $ \catN $.

        In this situation we can see that $\lim Y= Y_0 \times_{Y_2} Y_1$ and is fibrant in $\catN$. We can again take a $Z \in \catN^I$ to be the correspondence with constant value $ \lim Y$. So it comes with a map $Z \to Y$.
      \end{construction}

      \begin{lemma} \label{cofacycfib1-2}
	The map $ Z \to Y $ from above is a trivial fibration in $\catN^I_{Loc}$.
      \end{lemma}
      \begin{proof}
        The same idea as in \cref{cofacycfib1} carries over here. The diagrams are even simpler.
      \end{proof}

      \begin{lemma} \label{acycfib2-2}
  If $Y \in \catN^{K'}_{Reedy}$ is fibrant then there exists a trivial fibration $W \fibration Y \in \catN^I_{Loc}$ with $W \in \catN^I_{Loc}$ cofibrant.
\end{lemma}
\begin{proof}
  The argument of \cref{acycfib2} applies here too.
\end{proof}

\begin{lemma}\label{factorization-cofibration-trivial-fibration-2}
	Let $ X \to Y $ be a map in $ \catN^I $ with $ X $ cofibrant and $ Y $ fibrant. Then such a map can be factored as a cofibration followed by a trivial fibration.
\end{lemma}
\begin{proof}
  We have all ingredients to proceed as in \cref{factorization-cofibration-trivial-fibration}. Firstly, we can assume that $Y$ is Reedy cofibrant in $\catN^I$ and we can take a fibrant replacement in $\catN^{K}$. So we can construct the following pullback square:
  % https://q.uiver.app/#q=WzAsNCxbMSwwLCJXIl0sWzEsMSwiWSciXSxbMCwxLCJZIl0sWzAsMCwiTFkiXSxbMiwxLCJcXHNpbSIsMCx7InN0eWxlIjp7InRhaWwiOnsibmFtZSI6Imhvb2siLCJzaWRlIjoidG9wIn19fV0sWzAsMSwiIiwyLHsic3R5bGUiOnsiaGVhZCI6eyJuYW1lIjoiZXBpIn19fV0sWzMsMiwiIiwwLHsic3R5bGUiOnsiaGVhZCI6eyJuYW1lIjoiZXBpIn19fV0sWzMsMCwiXFxzaW0iXSxbMywxLCIiLDAseyJzdHlsZSI6eyJuYW1lIjoiY29ybmVyIn19XV0=
\[\begin{tikzcd}
	LY & W \\
	Y & {Y'.}
	\arrow["\sim", from=1-1, to=1-2]
	\arrow[two heads,"\sim"', from=1-1, to=2-1]
	\arrow["\lrcorner"{anchor=center, pos=0.125}, draw=none, from=1-1, to=2-2]
	\arrow[two heads,"\sim"', from=1-2, to=2-2]
	\arrow["\sim", hook, from=2-1, to=2-2]
      \end{tikzcd}\]
Then we can obtain a map $X \to LY$. Factoring this map as $X \cofibration X' \trivialfib LY$, the first map is moreover a cofibration in $\catN^I_{Loc}$ in view of \cref{localized-triv-cof:lem-2}. This produces the factorization $X \cofibration X' \trivialfib Y$.
\end{proof}

The proof of \cref{correspondences:wms} is a carbon copy from the one of \cref{model-paths:wms}, the lemmas of this section provide us with all the required steps.

\subsubsection{Projections are Barton trivial fibrations}

 \begin{lemma} \label{first-projection:extensible}
   The functor $ \catN^I \to \catN$ such that $A \to B \leftarrow C \in \catN^I \mapsto A \in \catN$, is extensible. Also, the functor $ \catN^I \to \catN$ such that $A \to B \leftarrow C \in \catN^I \mapsto C \in \catN$ is extensible.
 \end{lemma}
 \begin{proof}
   Let $A \coloneqq a \trivialcof b \lefttrivialcofib c  \in \catN^I_{Loc}$ be a cofibrant diagram and $x \in \catN^\cof$ a cofibrant object and a cofibration $a \cofibration x$. We take the fibrant replacement of $x$ and consider the pushout as indicated below, and we obtain a solution to the lifting problem on the right:
   % https://q.uiver.app/#q=WzAsNixbMSwwLCJhIl0sWzEsMSwiYiJdLFsyLDAsIngiXSxbMCwxLCJjIl0sWzIsMSwiYlxcc3FjdXBfYSB4Il0sWzMsMCwieF57ZmlifSJdLFszLDEsIlxcc2ltIiwwLHsic3R5bGUiOnsidGFpbCI6eyJuYW1lIjoiaG9vayIsInNpZGUiOiJ0b3AifX19XSxbMCwxLCJcXHNpbSIsMix7InN0eWxlIjp7InRhaWwiOnsibmFtZSI6Imhvb2siLCJzaWRlIjoidG9wIn19fV0sWzAsMiwiIiwwLHsic3R5bGUiOnsidGFpbCI6eyJuYW1lIjoiaG9vayIsInNpZGUiOiJ0b3AifX19XSxbMSw0XSxbMiw0LCJcXHNpbSIsMCx7InN0eWxlIjp7InRhaWwiOnsibmFtZSI6Imhvb2siLCJzaWRlIjoidG9wIn19fV0sWzAsNCwiXFxscmNvcm5lciIsMSx7ImxhYmVsX3Bvc2l0aW9uIjo5MCwic3R5bGUiOnsiYm9keSI6eyJuYW1lIjoibm9uZSJ9LCJoZWFkIjp7Im5hbWUiOiJub25lIn19fV0sWzIsNSwiXFxzaW0iLDAseyJzdHlsZSI6eyJ0YWlsIjp7Im5hbWUiOiJob29rIiwic2lkZSI6InRvcCJ9fX1dLFs0LDUsIiIsMCx7InN0eWxlIjp7ImJvZHkiOnsibmFtZSI6ImRhc2hlZCJ9fX1dXQ==
\[\begin{tikzcd}
	& a & x & {x^{fib}} \\
	c & b & {b\sqcup_a x}
	\arrow[hook, from=1-2, to=1-3]
	\arrow["\sim"', hook, from=1-2, to=2-2]
	\arrow["\lrcorner"{description, pos=0.9}, draw=none, from=1-2, to=2-3]
	\arrow["\sim", hook, from=1-3, to=1-4]
	\arrow["\sim", hook, from=1-3, to=2-3]
	\arrow["\sim", hook, from=2-1, to=2-2]
	\arrow[from=2-2, to=2-3]
	\arrow[dashed, from=2-3, to=1-4]
      \end{tikzcd}\]
    The resulting map $c \to x^{fib}$ can be factored as $ c \cofibration z \trivialfib x^{fib}$. We can take further pushouts
    % https://q.uiver.app/#q=WzAsOSxbMSwwLCJhIl0sWzEsMSwiYiJdLFsyLDAsIngiXSxbMCwxLCJjIl0sWzIsMSwiYlxcc3FjdXBfYSB4Il0sWzMsMCwieF57ZmlifSJdLFswLDIsInoiXSxbMSwyLCJ6XFxzcWN1cF9jIGIiXSxbMiwyLCJQIl0sWzMsMSwiXFxzaW0iLDAseyJzdHlsZSI6eyJ0YWlsIjp7Im5hbWUiOiJob29rIiwic2lkZSI6InRvcCJ9fX1dLFswLDEsIlxcc2ltIiwyLHsic3R5bGUiOnsidGFpbCI6eyJuYW1lIjoiaG9vayIsInNpZGUiOiJ0b3AifX19XSxbMCwyLCIiLDAseyJzdHlsZSI6eyJ0YWlsIjp7Im5hbWUiOiJob29rIiwic2lkZSI6InRvcCJ9fX1dLFsxLDQsIiIsMCx7InN0eWxlIjp7InRhaWwiOnsibmFtZSI6Imhvb2siLCJzaWRlIjoidG9wIn19fV0sWzIsNCwiXFxzaW0iLDAseyJzdHlsZSI6eyJ0YWlsIjp7Im5hbWUiOiJob29rIiwic2lkZSI6InRvcCJ9fX1dLFswLDQsIlxcbHJjb3JuZXIiLDEseyJsYWJlbF9wb3NpdGlvbiI6OTAsInN0eWxlIjp7ImJvZHkiOnsibmFtZSI6Im5vbmUifSwiaGVhZCI6eyJuYW1lIjoibm9uZSJ9fX1dLFsyLDUsIlxcc2ltIiwwLHsic3R5bGUiOnsidGFpbCI6eyJuYW1lIjoiaG9vayIsInNpZGUiOiJ0b3AifX19XSxbNCw1XSxbMyw2LCIiLDIseyJzdHlsZSI6eyJ0YWlsIjp7Im5hbWUiOiJob29rIiwic2lkZSI6InRvcCJ9fX1dLFs2LDcsIlxcc2ltIiwyLHsic3R5bGUiOnsidGFpbCI6eyJuYW1lIjoiaG9vayIsInNpZGUiOiJ0b3AifX19XSxbMSw3LCIiLDIseyJzdHlsZSI6eyJ0YWlsIjp7Im5hbWUiOiJob29rIiwic2lkZSI6InRvcCJ9fX1dLFszLDcsIlxcbHJjb3JuZXIiLDEseyJsYWJlbF9wb3NpdGlvbiI6OTAsInN0eWxlIjp7ImJvZHkiOnsibmFtZSI6Im5vbmUifSwiaGVhZCI6eyJuYW1lIjoibm9uZSJ9fX1dLFs3LDgsIiIsMSx7InN0eWxlIjp7InRhaWwiOnsibmFtZSI6Imhvb2siLCJzaWRlIjoidG9wIn19fV0sWzQsOCwiIiwxLHsic3R5bGUiOnsidGFpbCI6eyJuYW1lIjoiaG9vayIsInNpZGUiOiJ0b3AifX19XSxbMSw4LCJcXGxyY29ybmVyIiwxLHsibGFiZWxfcG9zaXRpb24iOjkwLCJzdHlsZSI6eyJib2R5Ijp7Im5hbWUiOiJub25lIn0sImhlYWQiOnsibmFtZSI6Im5vbmUifX19XSxbNiw1LCIiLDEseyJjdXJ2ZSI6NX1dXQ==
\[\begin{tikzcd}
	& a & x & {x^{fib}} \\
	c & b & {b\sqcup_a x} \\
	z & {z\sqcup_c b} & P.
	\arrow[hook, from=1-2, to=1-3]
	\arrow["\sim"', hook, from=1-2, to=2-2]
	\arrow["\lrcorner"{description, pos=0.9}, draw=none, from=1-2, to=2-3]
	\arrow["\sim", hook, from=1-3, to=1-4]
	\arrow["\sim", hook, from=1-3, to=2-3]
	\arrow["\sim", hook, from=2-1, to=2-2]
	\arrow[hook, from=2-1, to=3-1]
	\arrow["\lrcorner"{description, pos=0.9}, draw=none, from=2-1, to=3-2]
	\arrow[hook, from=2-2, to=2-3]
	\arrow[hook, from=2-2, to=3-2]
	\arrow["\lrcorner"{description, pos=0.9}, draw=none, from=2-2, to=3-3]
	\arrow[from=2-3, to=1-4]
	\arrow[hook, from=2-3, to=3-3]
	\arrow[bend right=70,"\sim"',twoheadrightarrow, from=3-1, to=1-4]
	\arrow["\sim"', hook, from=3-1, to=3-2]
	\arrow[hook, from=3-2, to=3-3]
      \end{tikzcd}\]
    There is a map $P \to x^{fib}$ which we can factor as $P \cofibration y \trivialfib x^{fib}$, and the resulting diagram we get
\[\begin{tikzcd}
	& a & x & {x^{fib}} \\
	c & b & {b\sqcup_a x} & y \\
	z & {z\sqcup_c b} & P
	\arrow[hook, from=1-2, to=1-3]
	\arrow["\sim"', hook, from=1-2, to=2-2]
	\arrow["\lrcorner"{description, pos=0.9}, draw=none, from=1-2, to=2-3]
	\arrow["\sim", hook, from=1-3, to=1-4]
	\arrow["\sim", hook, from=1-3, to=2-3]
	\arrow["\sim", hook, from=2-1, to=2-2]
	\arrow[hook, from=2-1, to=3-1]
	\arrow["\lrcorner"{description, pos=0.9}, draw=none, from=2-1, to=3-2]
	\arrow[hook, from=2-2, to=2-3]
	\arrow[hook, from=2-2, to=3-2]
	\arrow["\lrcorner"{description, pos=0.9}, draw=none, from=2-2, to=3-3]
	\arrow[from=2-3, to=1-4]
	\arrow[hook, from=2-3, to=3-3]
	\arrow["\sim"', two heads, from=2-4, to=1-4]
	% \arrow[bend right=70, from=3-1, to=1-4]
	\arrow["\sim"', hook, from=3-1, to=3-2]
	\arrow[hook, from=3-2, to=3-3]
	\arrow[hook, from=3-3, to=2-4]
\end{tikzcd}\]
Furthermore, there is a map $b \sqcup_a x \to y $ which is a cofibration as it is the composite of the two cofibrations. Using the 2-out-of-3 property repeatedly, one concludes that the map $z \sqcup_c b \to y$ is a trivial cofibration. Thus, we have constructed the cofibrant object $ X \coloneqq z \trivialcof y \lefttrivialcofib x \in \catN^I_{Loc}$. The induced map $A \to X$ is a level-wise cofibration. The maps $b \sqcup_a x \to y$ and $b \sqcup_a z \to y$ are trivial cofibrations.

Remains to show that $A \to X$ is a Reedy cofibration. We already have that $a \to x$ and $c \to z$ are cofibrations. We now need to show that the induced map
% https://q.uiver.app/#q=WzAsNSxbMCwwLCJhXFxzcWN1cCBjIl0sWzAsMSwieFxcc3FjdXAgeiJdLFsxLDAsImIiXSxbMSwxLCIoeCBcXHNxY3VwIHopXFxzcWN1cF97YSBcXHNxY3VwIGN9IGIiXSxbMiwyLCJ5Il0sWzIsMywiIiwwLHsic3R5bGUiOnsidGFpbCI6eyJuYW1lIjoiaG9vayIsInNpZGUiOiJ0b3AifX19XSxbMCwyXSxbMCwxLCIiLDIseyJzdHlsZSI6eyJ0YWlsIjp7Im5hbWUiOiJob29rIiwic2lkZSI6InRvcCJ9fX1dLFsxLDNdLFsxLDQsIiIsMCx7ImN1cnZlIjoxfV0sWzIsNCwiIiwwLHsiY3VydmUiOi0xfV0sWzMsNCwiIiwwLHsic3R5bGUiOnsiYm9keSI6eyJuYW1lIjoiZGFzaGVkIn19fV1d
\[\begin{tikzcd}
	{a\sqcup c} & b \\
	{x\sqcup z} & {(x \sqcup z)\sqcup_{a \sqcup c} b} \\
	&& y
	\arrow[from=1-1, to=1-2]
	\arrow[hook, from=1-1, to=2-1]
	\arrow[hook, from=1-2, to=2-2]
	\arrow[bend left=30, from=1-2, to=3-3]
	\arrow[from=2-1, to=2-2]
	\arrow[bend right=20, from=2-1, to=3-3]
	\arrow[dashed, from=2-2, to=3-3]
\end{tikzcd}\]
is a cofibration. By diagram chasing, one can show that the diagram
% https://q.uiver.app/#q=WzAsNCxbMCwwLCJhIFxcc3FjdXAgYyJdLFsxLDAsImIiXSxbMCwxLCJ4IFxcc3FjdXAgeiJdLFsxLDEsIih6XFxzcWN1cF9jIGIpXFxzcWN1cF9iIChiXFxzcWN1cF9hIHgpIl0sWzAsMl0sWzAsMV0sWzEsM10sWzIsM11d
\[\begin{tikzcd}
	{a \sqcup c} & b \\
	{x \sqcup z} & {(z\sqcup_c b)\sqcup_b (b\sqcup_a x)}
	\arrow[from=1-1, to=1-2]
	\arrow[from=1-1, to=2-1]
	\arrow[from=1-2, to=2-2]
	\arrow[from=2-1, to=2-2]
      \end{tikzcd}\]
    commutes. One shows that the bottom right corner computes the pushout of the span. Using that the map $P \cofibration y$ is a cofibration one concludes that $(x \sqcup) \sqcup_{a\sqcup c}b \to y$ is also a cofibration. This concludes the proof that $A \to X$ is a Reedy core cofibration in $\catN^I$. Therefore, it must a cofibration. We summarize our construction with the following diagram:
    % https://q.uiver.app/#q=WzAsOCxbMCwyLCJhIl0sWzEsMiwieCJdLFswLDEsImIiXSxbMCwwLCJjIl0sWzEsMSwieSJdLFsxLDAsInoiXSxbMCwzLCJhIl0sWzEsMywieCJdLFswLDEsIiIsMCx7InN0eWxlIjp7InRhaWwiOnsibmFtZSI6Imhvb2siLCJzaWRlIjoidG9wIn19fV0sWzMsMiwiXFxzaW0iLDAseyJzdHlsZSI6eyJ0YWlsIjp7Im5hbWUiOiJob29rIiwic2lkZSI6InRvcCJ9fX1dLFswLDIsIlxcc2ltIiwyLHsic3R5bGUiOnsidGFpbCI6eyJuYW1lIjoiaG9vayIsInNpZGUiOiJib3R0b20ifX19XSxbMSw0LCJcXHNpbSIsMix7InN0eWxlIjp7InRhaWwiOnsibmFtZSI6Imhvb2siLCJzaWRlIjoiYm90dG9tIn19fV0sWzUsNCwiXFxzaW0iLDAseyJzdHlsZSI6eyJ0YWlsIjp7Im5hbWUiOiJob29rIiwic2lkZSI6InRvcCJ9fX1dLFszLDUsIiIsMSx7InN0eWxlIjp7InRhaWwiOnsibmFtZSI6Imhvb2siLCJzaWRlIjoidG9wIn19fV0sWzIsNCwiIiwxLHsic3R5bGUiOnsidGFpbCI6eyJuYW1lIjoiaG9vayIsInNpZGUiOiJ0b3AifX19XSxbNiw3LCIiLDEseyJzdHlsZSI6eyJ0YWlsIjp7Im5hbWUiOiJob29rIiwic2lkZSI6InRvcCJ9fX1dLFswLDYsIiIsMCx7Im9mZnNldCI6LTUsInN0eWxlIjp7InRhaWwiOnsibmFtZSI6Im1hcHMgdG8ifX19XV0=
\[\begin{tikzcd}
	c & z \\
	b & y \\
	a & x \\
	a & x
	\arrow[hook, from=1-1, to=1-2]
	\arrow["\sim", hook, from=1-1, to=2-1]
	\arrow["\sim", hook, from=1-2, to=2-2]
	\arrow[hook, from=2-1, to=2-2]
	\arrow["\sim"', hook', from=3-1, to=2-1]
	\arrow[hook, from=3-1, to=3-2]
	\arrow[shift left=8, maps to, from=3-1, to=4-1]
	\arrow["\sim"', hook', from=3-2, to=2-2]
	\arrow[hook, from=4-1, to=4-2]
\end{tikzcd}\]
This cofibration is a (strict) lift of $a \cofibration x$, showing that the functor $\catN^I \to N$ is an extensible functor. The second part of the lemma is analogous.
\end{proof}

\begin{observation} \label{lift-equivalences:obs}
  Note that in the previous \cref{first-projection:extensible}, using 2-out-of-3 property, if we start with a trivial cofibration $a \trivialcof x$ then we obtain a level-wise equivalence between cofibrant objects in $\catN^I_{Loc}$. We conclude that the projections are weakly conservative.
\end{observation}

\begin{corollary} \label{first-projection:trivial-fibration}
  The functor $ \catN^I \to \catN$ such that $A \to B \leftarrow C \in \catN^I \mapsto A \in \catN$, is a Barton trivial fibration. Also, the functor $ \catN^I \to \catN$ such that $A \to B \leftarrow C \in \catN^I \mapsto C \in \catN$, is a Barton trivial fibration.
\end{corollary}
\begin{proof}
  We saw in \cref{first-projection:extensible} that the projections are extensible and from \cref{lift-equivalences:obs} that is weakly conservative. It is also straightforward to see that it preserve cofibrations and trivial cofibrations.
\end{proof}

 We now want to see that any left Quillen functor $F:\catM \to \catN$ part of a Quillen equivalence between weak model categories admits a Brown-like factorization. To this end, consider the following:

 \begin{construction} \label{fcylinders:wms}
   We define the category of diagrams
 $$ \catN_F^I \coloneqq \{Fa \to b \leftarrow c | a \in \catM^\cof, \, b,c \in \catN  \}.  $$
 The weak model structure on this category is similar to that of $\catN^I$, the only difference is that $X \to Y$ is a cofibration if $X_b \sqcup_{FX_a} FY_a \to Y_b$ is a trivial cofibration.
 \end{construction}
 
 When $F$ is the identity functor we recover $\catN^I$ from \cref{correspondences:wms}. A cofibrant object in  $\catN_F^I$ is a diagram of the form \[
   \begin{tikzcd}
     Fa \ar[r,"\sim",hook] & b & c. \ar[l,"\sim"',hook'] 
   \end{tikzcd}
 \]

 \begin{observation}
   With the set up above, it follows from \cref{first-projection:trivial-fibration} that the projection $\pi_1: \catN_F^I \to \catM$, sending each diagram $Fa \to b \leftarrow c$ to $a$, is a Barton trivial fibration.
 \end{observation}

 To show that the projection from $\pi_2: \catN_F^I\to\catN$ sending each diagram $Fa \to b \leftarrow c$ to $c \in\catN$ is a trivial fibration we make use of the following:

 \begin{lemma} \label{transpose:factorization}
   Let $F:\catM \to \catN $ be a left Quillen equivalence between weak model categories. For any objects $x \in \catM^\cof$, $y \in \catN^\fib$ and a map $f: Fx \to y$ there exists an object $z\in \catM^\cof$ such that $f$ factors as
   \[
     \begin{tikzcd}
       Fx \ar[rr, "f"] \ar[rd,hook]& & y \\
       & Fz. \ar[ru,"\sim"' ]&
     \end{tikzcd}
   \]
 \end{lemma}
 \begin{proof}
   We know that there is an isomorphism $$\varphi:\Hom_\catN(Fx,y) \simeq \Hom_\catM(x,Gy):\varphi^{-1}$$ given by the Quillen adjunction, natural in $x\in\catM^\cof$ and $y \in \catN^\fib$. Recall from \cite[2.4.3 Proposition]{henry20weak} that $F: \catM^\cof \to \catN^\cof$ and $G: \catN^\fib \to \catM^\fib$  preserve equivalences. Take $\varphi f$ the adjoint transpose of $f$. We can take a factorization \[
     \begin{tikzcd}
       & z \ar[rd,"\sim"',"s",two heads]  &  \\
       x \ar[ru,hook,"r"] \ar[rr,"\varphi f"'] & & Gy
     \end{tikzcd}
   \]
   By naturality, one checks that $f=\varphi^{-1}s F r$ where $Fr$ is a cofibration. Since the Quillen pair is an equivalence, we deduce from \cite[2.4.5 Proposition (i)]{henry20weak} that $\varphi^{-1}s$ is an equivalence.
 \end{proof}

 \begin{corollary} \label{second-projection:trivial-fibration}
   Let $F: \catM \rightleftarrows \catN:G $ be a Quillen equivalence. Then the projection $\pi_2:\catN_F^I \to \catN $sending each diagram $ Fa \to b \leftarrow c$ to $c \in\catN$ is a Barton trivial fibration.
 \end{corollary}
 \begin{proof}
   We show that in a situation as in the diagram
   % https://q.uiver.app/#q=WzAsNixbMCwyLCJjIl0sWzEsMiwiXFx0ZXh0e30iXSxbMCwxLCJiIl0sWzAsMCwiRmEiXSxbMCwzLCJjIl0sWzEsMywieSJdLFswLDEsIiIsMCx7InN0eWxlIjp7ImJvZHkiOnsibmFtZSI6Im5vbmUifSwiaGVhZCI6eyJuYW1lIjoibm9uZSJ9fX1dLFszLDIsIlxcc2ltIiwwLHsic3R5bGUiOnsidGFpbCI6eyJuYW1lIjoiaG9vayIsInNpZGUiOiJ0b3AifX19XSxbMCwyLCJcXHNpbSIsMix7InN0eWxlIjp7InRhaWwiOnsibmFtZSI6Imhvb2siLCJzaWRlIjoiYm90dG9tIn19fV0sWzQsNSwiIiwxLHsic3R5bGUiOnsidGFpbCI6eyJuYW1lIjoiaG9vayIsInNpZGUiOiJ0b3AifX19XSxbMCw0LCIiLDAseyJvZmZzZXQiOi01LCJzdHlsZSI6eyJ0YWlsIjp7Im5hbWUiOiJtYXBzIHRvIn19fV1d
\[\begin{tikzcd}
	Fa \\
	b \\
	c & {\text{}} \\
	c & z,
	\arrow["\sim", hook, from=1-1, to=2-1]
	\arrow["\sim"', hook', from=3-1, to=2-1]
	\arrow[draw=none, from=3-1, to=3-2]
	\arrow[maps to, from=3-1, to=4-1]
	\arrow[hook, from=4-1, to=4-2]
      \end{tikzcd}\]
    there is a cofibrant object over $z$ that projects onto $c \cofibration z$. By taking a fibrant replacement, we can assume that the diagram is point-wise fibrant. From \cite[2.2.3 Proposition]{henry20weak} there exists a homotopy inverse of $c \trivialcof b$, this give us a map $ Fa \to c$. Using \cref{transpose:factorization} this last map can be factored as $Fa \cofibration Fx \overset{\sim}{\to} c$. The rest of the proof continues as in \cref{first-projection:trivial-fibration}.
 \end{proof}

 \begin{theorem}\label{quillen-equivalence:factorization}
   Given $F: \catM \rightleftarrows \catN$ be a left Quillen equivalence between weak model categories. Then, we have a diagram of week model categories
% https://q.uiver.app/#q=WzAsNCxbMCwwLCJcXGNhdE1eSiJdLFswLDEsIlxcY2F0TSJdLFsxLDEsIlxcY2F0TVxcdGltZXNcXGNhdE4iXSxbMSwwLCJcXGNhdE5eSV9GIl0sWzAsMywiRyJdLFswLDEsIkIiLDJdLFsxLDIsIihJZF9cXGNhdE0sRikiLDJdLFszLDIsIihcXHBpXzEsXFxwaV8yKSJdXQ==
\[\begin{tikzcd}
	{\catM^J} & {\catN^I_F} \\
	\catM & {\catM\times\catN,}
	\arrow["H", from=1-1, to=1-2]
	\arrow["B"', from=1-1, to=2-1]
	\arrow["{(\pi_1,\pi_2)}", from=1-2, to=2-2]
	\arrow["{(Id_\catM,F)}"', from=2-1, to=2-2]
\end{tikzcd}\]
   where $\pi_1$ and $\pi_2$ are Barton trivial fibrations.
 \end{theorem}
 \begin{proof}
   The work we have done produces a diagram as on the left below, and the action of the functors on objects is spelled out on the right:
   % https://q.uiver.app/#q=WzAsOCxbMiwwLCJYX2FcXHJpZ2h0cmlnaHRhcnJvd3MgWF9iIFxcdG8gWF9jIl0sWzIsMSwiWF9hIl0sWzMsMCwiRlhfYSBcXHJpZ2h0cmlnaHRhcnJvd3MgRlhfYiJdLFszLDEsIihYX2EsRlhfYSkiXSxbMCwwLCJcXGNhdE1eSiJdLFswLDEsIlxcY2F0TSJdLFsxLDEsIlxcY2F0TVxcdGltZXNcXGNhdE4iXSxbMSwwLCJcXGNhdE5eSV9GIl0sWzAsMSwiQiIsMix7InN0eWxlIjp7InRhaWwiOnsibmFtZSI6Im1hcHMgdG8ifX19XSxbMCwyLCJHIiwwLHsic3R5bGUiOnsidGFpbCI6eyJuYW1lIjoibWFwcyB0byJ9fX1dLFsyLDMsIiIsMix7InN0eWxlIjp7InRhaWwiOnsibmFtZSI6Im1hcHMgdG8ifX19XSxbMSwzLCIiLDAseyJzdHlsZSI6eyJ0YWlsIjp7Im5hbWUiOiJtYXBzIHRvIn19fV0sWzQsNywiRyJdLFs0LDUsIkIiLDJdLFs1LDYsIihJZF9cXGNhdE0sRikiLDJdLFs3LDYsIihcXHBpXzEsXFxwaV8yKSJdXQ==
\[\begin{tikzcd}
	{\catM^J} & {\catN^I_F} & {X_a\rightrightarrows X_b \to X_c} & {FX_a \rightrightarrows FX_b} \\
	\catM & {\catM\times\catN} & {X_a} & {(X_a,FX_a)}
	\arrow["H", from=1-1, to=1-2]
	\arrow["B"', from=1-1, to=2-1]
	\arrow["{(\pi_1,\pi_2)}", from=1-2, to=2-2]
	\arrow["H", maps to, from=1-3, to=1-4]
	\arrow["B"', maps to, from=1-3, to=2-3]
	\arrow[maps to, from=1-4, to=2-4]
	\arrow["{(Id_\catM,F)}"', from=2-1, to=2-2]
	\arrow[maps to, from=2-3, to=2-4]
\end{tikzcd}\]
We have shown in \cref{first-projection:trivial-fibration} and \cref{second-projection:trivial-fibration} that both projections are Barton trivial fibrations.
\end{proof}

It will be essential to highlight that there is a diagonal functor which is a Barton trivial fibration, making the lower triangle commutative.

\begin{corollary}\label{factorization:corollary}
  Let $F:\catM \to \catM$ be a left Quillen equivalence. There exists a Barton trivial fibration $P:\catN_F^I \to \catM$.
\end{corollary}
\begin{proof}
  \Cref{quillen-equivalence:factorization} can be further specialized to a diagram
  % https://q.uiver.app/#q=WzAsNCxbMCwwLCJcXGNhdE1eSiJdLFswLDEsIlxcY2F0TSJdLFsxLDAsIlxcY2F0Tl5JX0YiXSxbMSwxLCJcXGNhdE4iXSxbMSwzLCJGIiwyXSxbMiwzLCJcXHBpXzIiXSxbMCwxXSxbMCwyXV0=
\[\begin{tikzcd}
	{\catM^J} & {\catN^I_F} \\
	\catM & \catM
	\arrow[from=1-1, to=1-2]
	\arrow[from=1-1, to=2-1]
	\arrow["{\pi_1}", from=1-2, to=2-2]
	\arrow["Id_\catM"', from=2-1, to=2-2]
      \end{tikzcd}\]
    from which we see that there is a functor $P:\catN^I_F\to \catM$. This is an immediate consequence of \cref{quillen-equivalence:factorization}.
\end{proof}

 \subsection{Proof of main theorem} \label{proof-third-invariance}

\begin{theorem}\label{p-third-invariance:thm}
Let $F: \catM \rightleftarrows \catN:G$ a Quillen equivalence. Then, for any cofibrant object $A \in \catM$. The induced map $h\Lb F_A: h\Lb_\lambda^\catM(A) \to h\Lb_\lambda^\catN(FA)  $ is an isomorphism.
\end{theorem}
 
\begin{proof}
  Recall from \cref{induced-functor:injective} that for any cofibrant object $A$ the induced map $h\Lb F_A$ is injective. Remains to show that it is surjective.
  Using \cref{factorization:corollary}, we obtain a diagram
  % https://q.uiver.app/#q=WzAsNCxbMCwwLCJcXGNhdE1eSiJdLFswLDEsIlxcY2F0TSJdLFsxLDAsIlxcY2F0Tl5JX0YiXSxbMSwxLCJcXGNhdE4iXSxbMSwzLCJGIiwyXSxbMiwzLCJcXHBpXzIiXSxbMCwxXSxbMCwyXSxbMiwxLCJQIiwxXV0=
\[\begin{tikzcd}
	{\catM^J} & {\catN^I_F} \\
	\catM & \catN
	\arrow[from=1-1, to=1-2]
	\arrow[from=1-1, to=2-1]
	\arrow["P"{description}, from=1-2, to=2-1]
	\arrow["{\pi_2}", from=1-2, to=2-2]
	\arrow["F"', from=2-1, to=2-2]
      \end{tikzcd}\]
    where $P$ is a Barton trivial fibration. $P: \catN_F^I \to \catM$ induces, for any cofibrant object $X\in \catN_F^I$, an isomorphism $(h\Lb \pi_1)_X:h\Lb_\lambda^{\catN_F^I}(X) \to h\Lb_\lambda^\catM(\pi_1X) $. Indeed, this follows from \cref{trivial-fibration:invariance}. Similarly, the map $(h\Lb \pi_2)_X:h\Lb_\lambda^{\catN_F^I}(X) \to h\Lb_\lambda^\catN(\pi_2X) $ is an isomorphism of $\lambda$-boolean algebras. For $A\in \catM^\cof$ cofibrant we can get a correspondence in $C_{FA}\in \catN^I_F$ with all objects $FA$ and maps the identities. We can conclude that $h\Lb F_A$ is surjective by chasing through the maps $(h\Lb \pi_2)_{C_A}$ and $(h\Lb P)_{C_A}$ which we already know are isomorphisms.
\end{proof}

It is an immediate that:

\begin{corollary}
  For any Quillen equivalence $F: \catM \rightleftarrows \catN:G$. The functors $ Ho(F) \circ h\Lb_\lambda^\catM $ and $h\Lb_\lambda^\catN:Ho(\catN) \to \bool_\lambda$ are naturally isomorphic via $h\Lb F$.
\end{corollary}

%%%%%%%%%%%%%%%%%%%%%%%%%%%%%%%% 

%%% Local Variables:
%%% mode: latex
%%% TeX-master: "main"
%%% End:

        \appendix

        % Appendix A: Infinitary cartmell theories

        \section{Infinitary Cartmell theories} \label{appendix-a}

We introduce a generalization of \textit{Cartmell theories}, also known as \textit{generalized algebraic theories}, Cartmell \cite{cartmell1978}. This is straightforward and most of the proofs will be omitted since they are similar to those in \cite{cartmell1978}. In very few cases we will need to provide new proofs. We claim no originality other than the generalization itself.
We begin by recalling some definitions given in \textit{Ibid}. We assume to have a set of variables $V$ whose size is $\aleph_0$ and an alphabet $A$. Informally, a \emph{Cartmell generalized algebraic theory} consists of:
\begin{enumerate}[i)]
	\item A set $S$, called the set of \textit{sort symbols},
	\item A set $O$, called the set of \textit{operation symbols},
	\item An introductory rule for each sort symbol,
	\item An introductory rule for each operation symbol,
	\item A set of axioms.
\end{enumerate}

To understand our generalization let us examine the previous definition in more detail, for this we need some preliminary notions. An \emph{expression} is a finite sequence of $A\cup V\cup \{(\}\cup \{)\}\cup \{,\}$. Inductively:
\begin{enumerate}[i)]
	\item Elements of $V$ and $A$ are expressions,
	\item If $f\in A$ and $e_1,\, e_2,...,e_n$ are expressions, then $f(e_1,\, e_2,...,\,e_n)$ is an expression.
\end{enumerate}

The set of expressions is denoted by $E$. This is simply to say that an expression is a finite string taken from the set $A\cup V\cup \{(\}\cup \{)\}\cup \{,\}$. A \emph{premise} is a finite (possibly empty) sequence of $V\times E$. A \emph{conclusion} is an n-tuple of expressions, i.e. any element of $E^n$ for some $n\in \N$. Finally, a \emph{rule} is given by a premise $P$ and a conclusion $C$. Rules are written as: $ P \vdash C $. This intends to convey the idea that under the premise $P$, the conclusion $C$ is a valid expression.
Whenever $P$ is a premise we will write $x_1:\Delta_1,\, x_2:\Delta_2,...,x_n:\Delta_n$. For a conclusion, this is slightly more involved since we differentiate depending on the size of the tuple. For example, if we have a 1-tuple $\Delta$, then we write $\Delta \type$. We favour the notation ``:'' from type theory instead of the set theoretic one ``$\epsilon$'' used by Cartmell. Furthermore, we will take advantage of conventions and notation from type theory.

The most important definition we will need to change is that of a \textit{context}. In a Cartmell theory, a \emph{context} is the premise such that a rule

\[
x_1:\Delta_1,\, x_2:\Delta_2(x_1),...,x_n:\Delta_n(x_1,x_2,\cdots,x_{n-1}) \vdash \Delta(x_1,x_2,\cdots, x_n) \, \type 
\]

is a \textit{derived rule}.

The only difference between Cartmell theories and infinitary Cartmell theories is that in we allow infinitely many variables in the contexts. Just as any Cartmell theory gives rise to a contextual category, the same is true for the infinitary case with the appropriate generalized version of a contextual category.

\subsection{Generalized algebraic theories}

In this section, we give the formal definition of an infinitary Cartmell theory. We follow Cartmell \cite{cartmell1978} to develop the theory; however, there will be some instances where a change has to be made. We could say that by changing in the definition every instance of ``finite'' by ``size strictly less than $\kappa$'' we get the correct notion, this is indeed the case. We carve out the definition with a fair amount of detail, since the applications we have in mind benefit from having an explicit syntax. The technicalities and motivations for introducing a generalized algebraic in the following way are presented in Cartmell \cite{cartmell1978}.

From now on, we fix a regular cardinal $\kappa$, unless otherwise stated, all other ordinals mentioned will be strictly smaller than $\kappa$. \\
Let $V$ be a set such that $|V|=\kappa$, this set will be called the set of \emph{variables}. We make an additional assumption on this set: Its elements have \emph{canonical names}, this is $V=\{x_\alpha \}_{\alpha<\kappa}$. This is also known as an \emph{enumeration}. This is a minor assumption that allows to change variables. Otherwise, we would need to prove a result similar to \cite[Corollary, pp 1.32]{cartmell1978}\footnote{This result states that under the substitution property the derived rules are stable under substitution of variables by another variables}.
Let $A$ be any set, which as before is called \emph{alphabet}. Following \cite{cartmell1978} we define inductively the collection of \emph{expressions} $ A^* $ over the alphabet $ A $. An expression is any $ \lambda $-sequence of $ A\cup V \cup \{(\}\cup \{)\}\cup \{,\} $ subject to:
\begin{enumerate}[i)]
	\item If $x_\alpha\in V$ then $x_\alpha \in A^*$,
	\item If $ F\in A $ then $ F \in A^* $,
	\item If $F\in A$ and $\{e_{\alpha}\}_{\alpha<\lambda}\subseteq A^*$ then $F(e_{\alpha})_{\alpha<\lambda} \in A^*$.
\end{enumerate}

A \emph{premise} is any $ \lambda $-sequence of $ V\times A^* $. We will usually write premises as $ \{x_\alpha:\Delta_\alpha\}_{\alpha<\lambda}$ where $ x_\alpha $ are variables and $ \Delta_\alpha $ are expressions for $ \alpha<\lambda $. Suppose we have a premise $ \Gamma $, or later a \textit{context}, and we need an extra premise (or \textit{context}), according to our variable numbering, formally, we must write $\Gamma, \, \{x_\alpha:\Delta_\alpha\}_{\lambda\leq\alpha<\mu} $, where $ \lambda $ represents the number of variables in $ \Gamma $. This is clearly a problem when the expression complexity increases. In order to avoid overloading the notation, we choose to reset the variable counting to only the essential variables in use. Under this convention, we will write $\Gamma, \, \{x_\alpha:\Delta_\alpha\}_{\alpha<\lambda} $ instead. We will freely assume that $ \Gamma $ is a premise unless otherwise specified.

\begin{definition}\label{judgment1}
	A \emph{judgment} is an expression over the alphabet $ A $ that has one of the following forms:
	\begin{enumerate}
		\item Type judgment: $ \Gamma \vdash \Delta\, \type. $
		\item Element judgment: $\Gamma \vdash t:\Delta.$
		\item Type equality judgment: $\Gamma \vdash \Delta\equiv \Delta'.$
		\item Term equality judgment: $\Gamma \vdash t\equiv_\Delta t'. $
	\end{enumerate}
	where $ \Gamma $ is a premise.
\end{definition}

Given a premise $\Gamma$, $\{e_{\alpha}\}_{\alpha<\lambda}$ expression and $\{x_{\alpha}\}_{\alpha<\lambda}$ variables then the new expression
\[\Gamma[e_{\alpha}|x_{\alpha}]_{\alpha<\lambda}\]
it is obtained by simultaneously changing the variables in $\Gamma$ by the expressions. This process, unsurprisingly, is called \emph{substitution} of variables. Along with the infinitary substitutions, we will also allow operations to have possibly infinite arity. This is made explicit:
\begin{definition}
	A $\kappa$-\emph{pretheory} $T$ consists of the following data:
	\begin{enumerate}[i)]
		\item A set $S$, called the set of \textit{sort symbols},
		\item A set $O$, called the set of \textit{operation symbols},
		\item For each sort symbol $B$, a judgment of the form:
		\[
		\{x_\alpha:\Delta_\alpha\}_{\alpha<\lambda} \vdash B(x_{\alpha})_{\alpha<\lambda}\, \type
		\]
		where $\lambda$ is some ordinal strictly smaller than $\kappa$,
		\item For each operator symbol $F$, a judgment:
		\[
		\{x_\alpha:\Delta_\alpha\}_{\alpha<\lambda} \vdash F(x_\alpha)_{\alpha<\lambda}:\Delta
		\]
		where $\lambda$ is an ordinal strictly smaller than $\kappa$,
		\item A set of judgments, each of which is either a type equality judgment or a term equality judgment, listed in \cref{judgment1}. This is the set of \emph{axioms} of the $\kappa$-pretheory.
	\end{enumerate}
\end{definition}

The following definitions are of inductive nature:

\begin{definition}\label{contextdefinition}
	\begin{enumerate}
		\item A premise $\{x_\alpha:\Delta_\alpha\}_{\alpha<\lambda}$ is a \emph{context} if the judgment   
		\[ \{x_\beta:\Delta_\beta\}_{\beta<\alpha} \vdash \Delta_{\alpha} \, \type \]
		
		is a \textit{derived judgment} of $T$ for every $\alpha<\lambda$. Whenever we want to specify that a premise $ \Gamma $ is a context we will write $ \vdash \Gamma\, \context $.
		
		\item The judgment
		\[ \{x_\alpha : \Delta_\alpha \}_{\alpha<\lambda} \vdash \Delta \,\type \]
		is a \emph{well-formed judgment} of $ T $ if and only if $\{x_\alpha : \Delta_\alpha \}_{\alpha<\lambda}$ is a context.
		
		\item The judgment
		\[\{x_\alpha:\Delta_\alpha\}_{\alpha<\lambda} \vdash t:\Delta \]
		is \emph{well-formed} if and only if 
		\[ \{x_\alpha : \Delta_\alpha \}_{\alpha<\lambda} \vdash \Delta \, \type \] is a \textit{derived judgment} of $ T $.
	\end{enumerate}
\end{definition}

\begin{definition}\label{derivedrules}
	Let $T$ be a $\kappa$-pretheory. The set of \emph{derived judgments} of $T$ are the ones that can be derived from the following list of rules:
	\begin{enumerate}
		\item \[\inferrule{\Gamma \vdash A\, \type}{\Gamma\vdash A \equiv A}\]
		\label{drule1}
		
		\item  \[\inferrule{\Gamma \vdash t:A}{\Gamma\vdash t\equiv_A t}\]
		\label{drule2}
		
		\item \[ \inferrule{\Gamma\vdash A_1\equiv A_2}{\Gamma \vdash A_2\equiv A_1} \] \label{drule3}
		
		\item \[\inferrule{\Gamma\vdash t_1\equiv_A t_2}{\Gamma \vdash t_2\equiv_A t_1} \]
		\label{drule4}
		
		\item \[\inferrule{\Gamma \vdash A_1\equiv A_2 \\ \Gamma\vdash A_2\equiv A_3 }{\Gamma \vdash A_1\equiv A_3 }\]
		\label{drule5}
		
		\item \[ \inferrule{\Gamma \vdash t_1\equiv_A t_2 \\ \Gamma \vdash t_2\equiv_A t_3 }{\Gamma \vdash t_1\equiv_A t_3 }\]
		\label{drule6}
		
		\item \[\inferrule{\Gamma\vdash A_1\equiv A_2 \\ \Gamma \vdash t_1\equiv_{A_1}t_2  }{\Gamma\vdash  t_2\equiv_{A_2} t_1 }\]
		\label{drule8}
		
		\item \[\inferrule{\Gamma \vdash A_1\equiv A_2 \\ \Gamma \vdash t:A_1}{\Gamma \vdash t:A_2}\]
                \label{drule9}
       
		\item 
		\[\inferrule{\Gamma,\, \{ x_\delta:A_\delta \}_{\delta<\beta<\lambda} \vdash A_\beta\, \type }{\Gamma,\, \{ x_\alpha:A_\alpha \}_{\alpha<\lambda} \vdash x_\alpha:A_\alpha }\]
		\label{drule10}
		
		\item For any $B$ sort symbol with a well-formed introduction type judgment:
			  \[
			  \inferrule{ \{x_\alpha:A_\alpha\}_{\alpha<\lambda} \vdash B(x_\lambda) \,\type, \quad \vdash \Gamma \, \context, \quad \Gamma \vdash t_\alpha:B[t_{\alpha}|x_{\alpha}]}{\Gamma \vdash B(t_\lambda)\,\type}
			  \]    	
			  \label{drule11} 		   
		\item For any $F$ operator symbol with a well-formed introduction type element judgment:
		     \[
		    \inferrule{\Gamma,\, \{ x_\alpha:\Delta_\alpha \}_{\alpha<\lambda} \vdash F(x_\lambda):\Delta, \quad \Gamma \vdash t_\alpha:\Delta_\alpha[t_\alpha|x_\alpha] }{\Gamma,\, \{t_\alpha:\Delta_\alpha[t_\alpha|x_\alpha]\}_{\alpha<\lambda}\vdash F(t_\lambda):\Delta[t_\lambda\mid x_\lambda]}
		   \]
                   \label{drule12}
		
		\item \[ \inferrule{\vdash \Gamma\, \context \\ \{x_\alpha:\Delta_\alpha\}_{\alpha<\lambda} \vdash \Delta\equiv \Delta' \\\\ \quad \Gamma, \, t_\alpha:\Delta_\alpha[t_\beta\mid x_\beta]_{\beta<\alpha}, \, t_\alpha':\Delta_\alpha'[t_\beta'\mid x_\beta]_{\beta<\alpha} \vdash t_\alpha\equiv_{\Delta_\alpha[t_\beta|x_\beta]_{\beta<\alpha}}t_\alpha'}{\Gamma,\, \{t_\alpha:\Delta_\alpha[t_\beta\mid x_\beta]_{\beta<\alpha}\}_{\alpha<\lambda},\, \{t_\alpha':\Delta_\alpha'[t_\beta'\mid x_\beta]_{\beta<\alpha}\}_{\alpha<\lambda} \\\\ \vdash \Delta[t_\alpha|x_\alpha]_{\alpha<\lambda} \equiv \Delta'[t_\alpha'|x_\alpha]_{\alpha<\lambda}}
		\] 
		\label{drule13}

		\item \[
		\inferrule{\vdash \Gamma\, \context \\ \{x_\alpha:\Delta_\alpha\}_{\alpha<\lambda} \vdash t\equiv_\Delta t' \\\\ \quad \Gamma,\,  s_\alpha:\Delta_\alpha[s_\beta\mid x_\beta]_{\beta<\alpha},\, s_\alpha':\Delta_\alpha[s_\beta'\mid x_\beta]_{\beta<\alpha} \vdash s_\alpha \equiv_{\Delta_\alpha[s_\beta'\mid x_\beta]_{\beta<\alpha}}s_\alpha'}{\Gamma,\, \{s_\alpha:\Delta_\alpha[s_\beta\mid x_\beta]_{\beta<\alpha}\}_{\alpha<\lambda},\, \{s_\alpha':\Delta_\alpha[s_\beta'\mid x_\beta]_{\beta<\alpha}\}_{\alpha<\lambda} \\\\ \vdash t[s_\alpha\mid x_\alpha]_{\alpha<\lambda}\equiv_{\Delta[s_\alpha\mid x_\alpha]_{\alpha<\lambda}}t'[s_\alpha'\mid x_\alpha]_{\alpha<\lambda} }
		\]
		\label{drule14}
		
		\item If $ \{x_\alpha:\Delta_\alpha\}_{\alpha<\lambda} \vdash \Delta \equiv \Delta'$ is an axiom then \[
		\inferrule{ \{x_\alpha:\Delta_\alpha\}_{\alpha<\lambda} \vdash \Delta \,\type \\ \{x_\alpha:\Delta_\alpha\}_{\alpha<\lambda} \vdash \Delta'\,\type,}{ \{x_\alpha:\Delta_\alpha\}_{\alpha<\lambda} \vdash \Delta\equiv \Delta' }
		\]
		\item If $\{x_\alpha:\Delta_\alpha\}_{\alpha<\lambda} \vdash t\equiv_\Delta t'$ is an axiom then \[
		\inferrule{ \{x_\alpha:\Delta_\alpha\}_{\alpha<\lambda} \vdash t:\Delta  \\ \{x_\alpha:\Delta_\alpha\}_{\alpha<\lambda} \vdash t':\Delta }{ \{x_\alpha:\Delta_\alpha\}_{\alpha<\lambda} \vdash t \equiv_\Delta t' }
		\]    
	\end{enumerate}
\end{definition}   

We are now ready for the following:

\begin{definition}
	A $\kappa$-pretheory $T$ is \emph{well-formed} if all its rules are well-formed. A \emph{generalized $\kappa$-algebraic theory} is a well-formed $\kappa$-pretheory.
\end{definition}

\begin{remark}
	Observe that a generalized algebraic theory as defined by Cartmell \cite{cartmell1978} is the same as an $\omega$-generalized algebraic theory in our sense.
\end{remark}

%\subsection{Examples of \texorpdfstring{$ \kappa $}{Lg}-algebraic theories}

We introduce an important example of $ \kappa $-algebraic theories.

\begin{example} \label{gatcategories}
	Let $ Cat $ denote the $ \omega $-algebraic theory defined in the following way:
	
	\begin{enumerate}
		\item Type of objects: $ \vdash \, \obtyp \, \type $.
		\item Type of morphisms: $ x : \obtyp, \,y : \obtyp \vdash \, \homtyp(x,y) \, \type $.
		\item Composition operation: $ x : \obtyp, \, y : \obtyp, \, z : \obtyp, \, f:\homtyp(x,y), \, g:\homtyp(y,z) \vdash g\circ f:\homtyp(x,z) $.
		\item Identity operator: $ x: \obtyp \vdash \, \idx_x:\homtyp(x,x) $.
	\end{enumerate}
	
	Subject to the following axioms:
	
	\begin{center}
		$
		% left identity
		\inferrule{x : \obtyp, \, y:\obtyp, \, f:\homtyp(x,y)}{
			\idx_y\circ f\equiv f } \quad\quad
		% left identity
		\inferrule{ x : \obtyp ,\, y:\obtyp,f:\homtyp(x,y)}{f\circ\idx_x\equiv f}
		$ \\
		$
		\inferrule{
			x :\obtyp, \, y : :\obtyp, \, z :\obtyp, \, w :\obtyp, \, f:\homtyp(x,y),g:\homtyp(y,z),h:\homtyp(z,w)}{(h\circ g)\circ f\equiv h\circ (g\circ f)}
		$
	\end{center}
\end{example}

\subsection{Substitution property}

Let $T$ be a generalized $\kappa$-algebraic theory. Recall that given $\Delta$, $\{t_{\alpha}\}_{\alpha<\lambda}$ expressions and $\{x_{\alpha}\}_{\alpha<\lambda}$ variables, then the new expression $\Delta[e_\alpha|x_\alpha]_{\alpha<\lambda}$ denotes the substitution of variables by the expressions.

\begin{definition}
	Let $ \{x_\alpha:\Delta_\alpha\}_{\alpha<\lambda} \vdash \Delta $ be a derived judgment of $T$. We say that this judgment has the \emph{substitution property} if for every $\vdash \Gamma \, \context$ and expressions $\{t_\alpha\}_{\alpha<\lambda}$, such that for all $\alpha<\lambda$
	\[ \Gamma,\, \{t_\beta:\Delta_\beta[t_\gamma|x_\gamma]_{\gamma<\beta}\}_{\beta<\alpha} \vdash t_\alpha:\Delta_\alpha[t_\beta|x_\beta]_{\beta<\alpha} \]
	are derived rules, then
	\[  \Gamma \vdash \Delta[t_\alpha|x_\alpha]_{\alpha<\lambda} \]
	is a derived rule of $T$.
\end{definition}

In \cite{cartmell1978} it is proven that all derived judgment of a generalized algebraic theory satisfy the substitution property. This is done through a series of results that can be generalized to our setting. The proofs are omitted since they are the same as in the original reference.

\begin{lemma}
	If $\{x_\alpha:\Delta_\alpha\}_{\alpha<\lambda} \vdash \Delta $ is a derived judgment of $ T $, then the variables that appear in $ \Delta $ is a subset of $ \{x_\alpha\}_{\alpha<\lambda} $
\end{lemma}
\begin{proof}
	See \cite[Lemma 1, Section 1.7]{cartmell1978}.
\end{proof}

\begin{lemma}
	\begin{enumerate}
		\item The premise of a derived judgment is a context.
		\item If $ \vdash \{x_\alpha: \Delta_\alpha \}_{\alpha<\lambda} \context $ then for $ \alpha<\lambda $, we have $$ \{x_\beta:\Delta_\beta\}_{\beta<\alpha} \vdash \Delta_\alpha \, \type$$
	\end{enumerate}
\end{lemma}
\begin{proof}
	See \cite[Lemma 2, Section 1.7]{cartmell1978}.
\end{proof}

\begin{theorem}
	Every derived judgment of a generalized $\kappa$-algebraic theory has the substitution property.
\end{theorem}
\begin{proof}
	The same proof as in \cite[1.7]{cartmell1978} applies. This goes by proving that each judgment has the substitution property. For the last two judgments in \cref{judgment1}, this is a consequence of rules (\ref{drule12}) and (\ref{drule13}) in \cref{derivedrules}. While for the first two it is done by induction on the derivations. It is shown that each derivation rule of \cref{derivedrules} preserve the substitution property. 
\end{proof}

This result has similar consequences of those in \cite{cartmell1978}. The proofs are analogous or the same. For us, it is only relevant to know that our generalized $\kappa$-algebraic theories are well-defined. That is:

\begin{proposition} \label{well-formed-type-intro:prop}
	The derived judgments of a generalized $\kappa$-algebraic theory are well-formed.
\end{proposition}
\begin{proof}
	Again, by induction on the derivations \cite[pp.~1.33]{cartmell1978}.
\end{proof}

Both the statement and proof of the next lemma are the same as The Derivation Lemma \cite[pp.~1.34]{cartmell1978}. The proof does not rely on the context size.

\begin{lemma}\label{derivationlemma}
	\begin{enumerate}
		\item Every derived type judgment of $ T $ is of the form $$ \{x_\beta:\Omega_\beta\}_{\beta<\mu} \vdash A(t_\alpha)_{\alpha<\lambda} $$ for some type symbol $ A $ with introductory rule
		$$ \{x_\alpha:\Delta_\alpha \}_{\alpha<\lambda} \vdash A(x_\alpha)_{\alpha<\lambda} \, \type $$
		and $ \{t_\alpha\}_{\alpha<\lambda} $ are expressions such that for all $ \alpha<\lambda $ the rule $$ \{ x_\beta: \Omega_\beta \}_{\beta<\mu} \vdash t_\alpha:\Delta_\alpha[t_\delta\mid x_\delta ]_{\delta<\alpha}. $$
		\item Every term element judgment of $ T $ is of the form  \[\{x_\beta:\Omega_\beta\}_{\beta<\mu} \vdash x_\beta: \Omega \]
		for some $ x_\beta $ and such that $ \{x_\beta:\Omega_\beta\}_{\beta<\mu} \vdash \Omega_\beta\equiv \Omega $, or is of the form \[ \{x_\beta:\Omega_\beta\}_{\beta<\mu} \vdash f(t_\alpha)_{\alpha<\lambda} : \Omega \] for some operator symbol $ f $ of $ T $ with introductory judgment of the form \[ \{x_\alpha:\Delta_\alpha \}_{\alpha<\lambda} \vdash f(x_\alpha)_{\alpha<\lambda}: \Delta\] such that for each $ \alpha<\lambda $ the rules   \[ \{x_\beta:\Omega_\beta\}_{\beta<\mu} \vdash t_\alpha: \Delta_\alpha[t_\delta\mid x_\delta]_{\delta<\alpha}\] and \[ \{x_\beta:\Omega_\beta\}_{\beta<\mu} \vdash \Delta[t_\alpha\mid x_\alpha]_{\alpha<\lambda} \equiv \Omega \] are derived rules of $ T $.
	\end{enumerate}
      \end{lemma}
      \begin{proof}
        This follows from \cref{derivedrules} (\ref{drule11}) and (\ref{drule12}).
      \end{proof}

\subsection{Equivalence relation on judgments}

Throughout this section we work in a generalized $\kappa$-algebraic theory.  We first introduce a relation that allows us to identify contexts which express the same meaning, but differ on the variables that are used in them \cite[1.13]{cartmell1978}.

There is a relation defined on the judgments of the generalized $\kappa$-algebraic theory $T$.

\begin{definition}
	Let $ \{x_\alpha:\Delta_\alpha\}_{\alpha<\lambda}\vdash \Delta_\lambda \, \type$ and $\{x_\beta:\Omega_\beta\}_{\beta<\mu}\vdash \Omega_\mu \,\type $ be two type judgments of $ T $. We say that
	\[\{x_\alpha:\Delta_\alpha\}_{\alpha<\lambda}\vdash \Delta_\lambda \, \type  \approx \{x_\beta:\Omega_\beta\}_{\beta<\mu}\vdash \Omega_\mu \, \type\]
	if either:
	\begin{enumerate}
		\item Both ordinals are successor such that $\lambda=\mu=\nu+1$ and for all $\alpha \leq \nu$ we have
		\[\{x_\delta:\Delta_\delta\}_{\delta<\alpha} \vdash \Delta_\alpha \equiv \Omega_\alpha\]
		is a derived rule of $ T $.
		\item Both ordinals are limit ordinals with $ \lambda=\mu $ and for any successor ordinal $ \nu+1<\lambda $ we have \[\{x_\alpha:\Delta_\alpha\}_{\alpha<\nu}\vdash \Delta_\nu \, \type  \approx \{x_\beta:\Omega_\beta\}_{\beta<\nu}\vdash \Omega_\nu \, \type.\]
	\end{enumerate}
\end{definition}

\begin{lemma}\label{equivalenceofcontexts}
	The relation $\approx$ is an equivalence relation on type judgments of the theory $ T $.
\end{lemma}
\begin{proof}	
	This is an immediate result since we have assumed canonical names for variables. Otherwise, we could repeat the argument as in \cite[1.13]{cartmell1978}.
	
	%The fact that this is reflexive follows from the principle of derivation from \cref{derivedrules} since in general $\Delta_\alpha\equiv \Delta_\alpha[x_\alpha|x_\alpha]$ for any $\alpha\leq \lambda$.
	
	%For symmetry, suppose that \[\{x_\alpha:\Delta_\alpha\}_{\alpha<\lambda}\vdash \Delta_\lambda \type  \approx \{y_\beta:\Omega_\beta\}_{\beta<\mu}\vdash \Omega_\mu \type.\] We need to show that $\{y_\delta:\Omega_\delta\}_{\delta<\alpha}\vdash \Omega_\alpha\equiv \Delta_\alpha[y_\alpha|x_\alpha]$ for all $\alpha\leq \lambda$.
	
	%By transfinite induction on $\alpha.$ The same proof from Lemma 1 in (\cite{cartmell1978}, 1.13) applies. The inductive step assumes that the property holds for any $ \delta<\alpha $, it is proven then that the property also holds for $ \alpha $. Therefore, the fact that $ \alpha $ is not necessarily a finite ordinal makes no difference in the argument.
	
	%The same argument of Lemma 1 in (\cite{cartmell1978}, 1.13) proves that $ \approx $ is a transitive relation. This concludes the proof that $ \approx $ is an equivalence relation.
	
\end{proof}

\begin{definition}
	Let $\{x_\alpha:\Delta_\alpha\}_{\alpha<\lambda}$ and $\{x_\beta:\Omega_\beta\}_{\beta<\mu}$ be two contexts. We say that
	\[\{x_\alpha:\Delta_\alpha\}_{\alpha<\lambda} \approx \{x_\beta:\Omega_\beta\}_{\beta<\mu}\]
	if and only if $ \lambda=\mu $ and for all $ \alpha<\lambda $
	\[\{x_\delta:\Delta_\delta\}_{\delta<\alpha}\vdash \Delta_\alpha  \, \type  \approx \{x_\gamma:\Omega_\gamma\}_{\gamma<\alpha}\vdash \Omega_\alpha \, \type\]
\end{definition}

It follows that this induces an equivalence relation on contexts. 

\begin{definition}	
	We say that \[ \{x_\alpha:\Delta_\alpha\}_{\alpha<\lambda}\vdash t:\Delta \approx  \{x_\beta:\Omega_\beta\}_{\beta<\mu}\vdash s:\Omega\] if and only if $ \{x_\alpha:\Delta_\alpha\}_{\alpha<\lambda}\vdash \Delta \,\type \approx  \{x_\beta:\Omega_\beta\}_{\beta<\mu}\vdash \Omega \,\type $ and $ \{x_\alpha:\Delta_\alpha\}_{\alpha<\lambda}\vdash t\equiv s  $.
\end{definition}

\begin{remark} \label{remarktransitivityofproofs}
	Let $\{x_\alpha:\Delta_\alpha\}_{\alpha<\lambda}$ and $\{x_\beta:\Omega_\beta\}_{\beta<\mu}$ be two contexts. Assume further that \[\{x_\alpha:\Delta_\alpha\}_{\alpha<\lambda} \approx \{x_\beta:\Omega_\beta\}_{\beta<\mu}.\] Then for all derived rules \[ \{x_\beta:\Omega_\beta\}_{\beta<\mu} \vdash \Omega, \]
	the rule \[ \{x_\alpha:\Delta_\alpha\}_{\alpha<\lambda} \vdash \Omega\]
	is also a derived rule.
\end{remark}

Regardless of its simplicity, this remark is useful in the next:

\begin{corollary}
	The relation $ \approx $ is an equivalence relation on judgments of the form $ \{x_\beta:\Delta_\beta\}_{\beta<\mu} \vdash t:\Delta.$
\end{corollary}
\begin{proof}
	Reflexivity is a consequence of \ref{drule2} from \cref{derivedrules}. Assume that $ \{x_\alpha:\Delta_\alpha\}_{\alpha<\lambda} \vdash t:\Delta \approx  \{x_\alpha:\Omega_\alpha\}_{\alpha<\lambda} \vdash s:\Omega$. Hence, the contexts satisfy $ \{x_\alpha:\Delta_\alpha\}_{\alpha<\lambda}\approx  \{x_\alpha:\Omega_\alpha\}_{\alpha<\lambda}$. Applying the symmetry of the relation $ \approx $ to contexts, and using \cref{remarktransitivityofproofs}, we see that $ \{x_\alpha:\Omega_\alpha\}_{\alpha<\lambda} \vdash t\equiv s $. Then we must have $ \{x_\alpha:\Omega_\alpha\}_{\alpha<\lambda} \vdash s:\Delta $ and $ \{x_\alpha:\Omega_\alpha\}_{\alpha<\lambda} \vdash \Omega \equiv \Delta $.  We can apply \ref{drule4} from \cref{derivedrules} to conclude that $ \{x_\alpha:\Omega_\alpha\}_{\alpha<\lambda} \vdash s\equiv t$, thus proving symmetry. Transitivity is a straightforward application of \cref{remarktransitivityofproofs}.
\end{proof}

\begin{definition}
	A \emph{morphism} between contexts $$ \langle t_\beta\rangle_{\beta<\mu}: \{x_\alpha:\Delta_\alpha\}_{\alpha<\lambda} \to \{x_\beta:\Omega_\beta\}_{\beta<\mu} $$ is $ \mu $-sequence of terms $ \{t_\beta\}_{\beta<\mu} $ such that for all $ \beta<\mu $ we have $$ \{x_\alpha:\Delta_\alpha\}_{\alpha<\lambda} \vdash t_\beta : \Omega_\beta[t_\gamma|x_\gamma]_{\gamma<\beta}.$$
\end{definition}

Just as in the finite case, with the substitution as composition and the obvious identity, it can be shown that contexts form a category with morphisms as defined above. This is called the \emph{category of realizations} of the theory $ T$.
The composition of \[ \langle t_\beta\rangle_{\beta<\mu}: \{x_\alpha:\Delta_\alpha\}_{\alpha<\lambda} \to \{x_\beta:\Omega_\beta\}_{\beta<\mu} \] and \[ \langle s_\delta\rangle_{\delta<\nu}: \{x_\beta:\Omega_\beta\}_{\beta<\mu} \to \{x_\delta:\Omega_\delta'\}_{\delta<\nu} \]
is the map $$ \langle s_\delta\rangle_{\delta<\nu} \circ \langle t_\beta \rangle_{\beta<\mu}: \{x_\alpha:\Delta_\alpha\}_{\alpha<\lambda} \to \{x_\delta:\Omega_\delta'\}_{\delta<\nu} $$
defined as the sequence $ \langle s_\delta[\langle t_\beta| x_\beta \rangle_{\beta<\mu}] \rangle_{\delta<\nu}.$

Using the previous relation $ \approx $ on contexts and rules we induce one on morphisms between contexts.
If we have morphisms $$ \langle t_\beta\rangle_{\beta<\mu}: \{x_\alpha:\Delta_\alpha\}_{\alpha<\lambda} \to \{x_\beta:\Omega_\beta\}_{\beta<\mu} \text{ and } \langle t_\beta'\rangle_{\beta<\mu}:\{x_\alpha:\Delta_\alpha'\}_{\alpha<\lambda} \to \{x_\beta:\Omega_\beta'\}_{\beta<\mu} $$ 
Then $$ \langle t_\beta\rangle_{\beta<\mu}\approx \langle t_\beta'\rangle_{\beta<\mu} $$ if and only if $$ \{x_\beta:\Omega_\beta\}_{\beta<\mu}\approx \{x_\beta':\Omega_\beta'\}_{\beta<\mu} $$ and for all $ \gamma<\mu $ $$ \{x_\beta:\Delta_\beta\}_{\beta<\mu} \vdash t_\gamma: \Omega_\gamma[t_{\gamma'}|x_{\gamma'}]_{\gamma'<\gamma} \approx \{x_\beta:\Delta_\beta'\}_{\beta<\mu} \vdash t_\gamma': \Omega_\gamma'[t_{\gamma'}'|x_{\gamma'}]_{\gamma'<\gamma}.$$

Unfolding the definition this means that $$ \{x_\beta:\Delta_\beta\}_{\beta<\mu} \vdash \Omega_\gamma[t_{\gamma'}|x_{\gamma'}]_{\gamma'<\gamma} \, \type \approx \{x_\beta:\Delta_\beta'\}_{\beta<\mu} \vdash \Omega_\gamma'[t_{\gamma'}'|x_{\gamma'}]_{\gamma'<\gamma} \,\type $$ and that $ \{x_\beta:\Delta_\beta\}_{\beta<\mu} \vdash t_\gamma\equiv t_\gamma' $ for all $ \gamma < \mu. $

The following remarks are results from \cite{cartmell1978} whose proofs are completely similar. However, it is important to make them explicit, since they imply that we can define a composition operation of equivalence classes of morphisms between contexts.

\begin{remark} \label{lemma3}
	Let $ \langle t_\beta\rangle_{\beta<\mu}: \{x_\alpha:\Delta_\alpha\}_{\alpha<\lambda} \to \{x_\beta:\Omega_\beta\}_{\beta<\mu} $ and $ \langle t_\beta'\rangle_{\beta<\mu}: \{x_\alpha:\Delta_\alpha\}_{\alpha<\lambda} \to \{x_\beta:\Omega_\beta'\}_{\beta<\mu} $ two morphisms between contexts with $ \langle t_\beta\rangle_{\beta<\mu} \approx \langle t_\beta'\rangle_{\beta<\mu}  $.
	\begin{enumerate}
		\item If $ \{x_\beta:\Omega_\beta\}_{\beta<\mu} \vdash \Omega \, \type $ and $ \{x_\beta:\Omega_\beta'\}_{\beta<\mu} \vdash \Omega' \, \type $ are derived judgment of the theory such that 
		$$ \{x_\beta:\Omega_\beta,\, x_\mu:\Omega\}_{\beta<\mu} \approx \{x_\beta:\Omega_\beta',\, x_\mu:\Omega'\}_{\beta<\mu} $$ 
		then 
		$$ \{x_\alpha:\Delta_\alpha,\,x_\mu:\Omega[t_\beta|x_\beta]_{\beta<\mu}\}_{\alpha<\lambda} \approx \{x_\alpha:\Delta_\alpha', \, x_\mu:\Omega'[t_\beta'|x_\beta']_{\beta<\mu} \}_{\alpha<\lambda} $$
		
		This follows by unwinding the relation $ \approx $ and applying the principle \ref{drule13} in \cref{derivedrules}. This simply means that we can extend contexts by a fresh variable. Moreover, there is a more general result: 
		
		For all $ \varepsilon > 0 $, if $ \{x_\beta:\Omega_\beta\}_{\beta<\mu+\varepsilon} $ and $ \{x_\beta:\Omega_\beta'\}_{\beta<\mu+\varepsilon}$ are contexts then
		$$ \{x_\alpha:\Delta_\alpha,\,x_\beta:\Omega_\beta[t_\gamma|x_\gamma]_{\gamma<\beta}\}_{\substack{\alpha<\lambda, \\ \mu \leq \beta <\mu+\varepsilon}} \approx \{x_\alpha:\Delta_\alpha', \, x_\beta:\Omega_\beta'[t_\gamma'|x_\gamma]_{\gamma<\beta} \}_{\substack{\alpha<\lambda, \\ \mu \leq \beta <\mu+\varepsilon}} $$

		\item \label{item2} If $ \{x_\beta:\Omega_\beta\}_{\beta<\mu} \vdash s:\Omega $ and $ \{x_\beta:\Omega_\beta'\}_{\beta<\mu} \vdash s':\Omega' $ are derived judgment such that $$ \{x_\beta:\Omega_\beta\}_{\beta<\mu} \vdash s \equiv_\Omega s' .$$ Then $$ \{x_\alpha:\Delta_\alpha\}_{\alpha<\lambda} \vdash s[t_\beta|x_\beta]_{\beta<\mu} \equiv_{\Omega[t_\beta|x_\beta]_{\beta<\mu}} s'[t_\beta'|x_\beta]_{\beta<\mu}. $$
		Observe that the principle \ref{drule14} from \cref{derivedrules} implies this result.
		
	\end{enumerate}
\end{remark}

\begin{remark} \label{lemma4}
	\begin{enumerate}
		\item Let $ \langle t_\beta\rangle_{\beta<\mu}: \{x_\alpha:\Delta_\alpha\}_{\alpha<\lambda} \to \{x_\beta:\Omega_\beta\}_{\beta<\mu} $ be a morphism between two contexts. If $$ \{x_\alpha:\Delta_\alpha\}_{\alpha<\lambda} \approx \{x_\alpha':\Delta_\alpha'\}_{\alpha<\lambda} \text{ and } \{x_\beta:\Omega_\beta\}_{\beta<\mu} \approx \{x_\beta':\Omega_\beta'\}_{\beta<\mu} $$ then $ \langle t_\beta\rangle_{\beta<\mu}: \{x_\alpha':\Delta_\alpha'\}_{\alpha<\lambda} \to \{x_\beta':\Omega_\beta'\}_{\beta<\mu} $ is also a morphism between these contexts.
		
		\item If we have a context $ \{x_\alpha:\Delta_\alpha\}_{\alpha<\lambda+1} $ and $ \{x_\alpha:\Delta_\alpha\}_{\alpha<\lambda}\approx \{x_\alpha':\Delta_\alpha'\}_{\alpha<\lambda} $ then we can extend the context $ \{x_\alpha':\Delta_\alpha'\}_{\alpha<\lambda} $ to $ \{x_\alpha':\Delta_\alpha'\}_{\alpha<\lambda+1} $ such that $ x_\alpha':\Delta_\alpha' $ is $ x_\lambda:\Delta_\lambda.$
	\end{enumerate}
	
\end{remark}

\begin{remark} \label{lemma5}
	Let $ \langle t_\beta\rangle_{\beta<\mu+1}: \{x_\alpha:\Delta_\alpha\}_{\alpha<\lambda} \to \{x_\beta:\Omega_\beta\}_{\beta<\mu+1} $ and $ \langle s_\beta\rangle_{\beta<\mu}: \{x_\alpha:\Delta_\alpha\}_{\alpha<\lambda} \to \{x_\beta:\Omega_\beta\}_{\beta<\mu} $ be morphisms between contexts. Then we have a morphism $$ \langle s_\beta\rangle_{\beta<\mu+1}: \{x_\alpha:\Delta_\alpha\}_{\alpha<\lambda} \to \{x_\beta:\Omega_\beta\}_{\beta<\mu+1} $$ where $ s_\mu\equiv t_\mu$, and such that $ \{s_\beta\}_{\beta<\mu+1}\approx \{t_\beta\}_{\beta<\mu+1} $.
\end{remark}

\subsection{The category of generalized \texorpdfstring{$ \kappa $}{Lg}-algebraic theories}

We construct a category where the objects are generalized $\kappa$-algebraic theories with maps \textit{interpretations}. This is analogous to the category that Cartmell constructs in \cite[1.11]{cartmell1978}, all the results can be copied from there to our setting. Since we work with different theories, the alphabets, expressions and rules are marked accordingly. If $ T $ is a theory then these sets are denoted $ Alp(T), \, Exp(T),\, Rul(T) $ respectively.

Let $ T $ and $ T' $ two generalized $\kappa$-algebraic theories. Let $ I:Alp(T) \to Exp(T') $ be a function. Using this function, we can define a \emph{preinterpretation} $ \widetilde{I}: Exp(T) \to Exp(T') $ by induction on the construction of expressions:

\begin{enumerate}
	\item If $ x\in V $ $$ \widetilde{I}(x):=x, $$
	\item If $ F\in Alp(T) $ $$ \widetilde{I}(F):=I(F), $$
	\item If $ L\in Alp(T) $ is an alphabet symbol and $ \{t_\alpha\}_{\alpha<\lambda} $ are expressions $$ \widetilde{I}(L(t_\alpha)_{\alpha<\lambda}):= I(L)( \widetilde{I}(t_\alpha))_{\alpha<\lambda}. $$
\end{enumerate}

\begin{definition} \label{interpretation}
	Given a preinterpretation $ \widetilde{I} $ we define a new function $ \widehat{I}:Rul(T)\to Rul(T') $.
	
	\begin{enumerate}
		\item $ \widehat{I}(\Gamma \vdash \Delta\, \type):= \widetilde{I}(\Gamma) \vdash \widetilde{I}(\Delta)\, \type $
		\item $\widehat{I}(\Delta \vdash t:\Delta):= \widetilde{I}(\Delta)\vdash \widetilde{I} (t):\widetilde{I}(\Delta)$
		\item $\widehat{I}(\Delta,\,\Delta' \vdash \Delta\equiv \Delta'):= \widetilde{I}(\Delta),\,\widetilde{I}(\Delta') \vdash \widetilde{I}(\Delta)\equiv \widetilde{I}(\Delta').$
		\item $\widehat{I}(\Delta,\, t,t':\Delta \vdash t\equiv_\Delta t'):=  \widetilde{I}(\Delta),\, \widetilde{I}(t),\widetilde{I}(t'):\widetilde{I}(\Delta) \vdash \widetilde{I}(t)\equiv_{\widetilde{I}(\Delta)} \widetilde{I}(t')$.
	\end{enumerate}
	
	This function is an \emph{interpretation} from $ T $ into $ T' $ if all introductory judgments and axioms of $ T $ are sent to derived rules of $ T' $, we will simply denote this as $ I: T\to T' $.
\end{definition}

Just as in \cite{cartmell1978} it is possible to prove that:

\begin{lemma}\label{lemma1112}
	If $ I $ is an interpretation from $ T $ to $ T' $, then it preserves the derived judgments of the theory $ T $.
\end{lemma}
\begin{proof}
	From Lemma 2 \cite[pp 1.52]{cartmell1978}. To illustrate how this is done, we show that the derived judgment \cref{derivedrules} (\ref{drule14}) it is preserved by $ I $. Consider the derived judgment
	\[
	\inferrule{\vdash \Gamma\, \context \\ \{x_\alpha:\Delta_\alpha\}_{\alpha<\lambda} \vdash t\equiv_\Delta t' \\\\ \quad \Gamma,\,  s_\alpha:\Delta_\alpha[s_\beta\mid x_\beta]_{\beta<\alpha},\, s_\alpha':\Delta_\alpha[s_\beta'\mid x_\beta]_{\beta<\alpha} \vdash s_\alpha \equiv_{\Delta_\alpha[s_\beta'\mid x_\beta]_{\beta<\alpha}}s_\alpha'}{\Gamma,\, \{s_\alpha:\Delta_\alpha[s_\beta\mid x_\beta]_{\beta<\alpha}\}_{\alpha<\lambda},\, \{s_\alpha':\Delta_\alpha[s_\beta'\mid x_\beta]_{\beta<\alpha}\}_{\alpha<\lambda} \\\\ \vdash t[s_\alpha\mid x_\alpha]_{\alpha<\lambda}\equiv_{\Delta[s_\alpha\mid x_\alpha]_{\alpha<\lambda}}t'[s_\alpha'\mid x_\alpha]_{\alpha<\lambda} }
	\]
	in the theory $ T $. We may assume that the context $ \Gamma $ is of the form $ \{x_\beta:\Omega_\beta\}_{\beta<\mu} $, so we get 
	\[
	\inferrule{\vdash \{x_\beta:\Omega_\beta\}_{\beta<\mu} \, \context \\ \{x_\alpha:\Delta_\alpha\}_{\alpha<\lambda} \vdash t\equiv_\Delta t' \\\\ \quad \{x_\beta:\Omega_\beta\}_{\beta<\mu},\,  s_\alpha:\Delta_\alpha[s_\beta\mid x_\beta]_{\beta<\alpha},\, s_\alpha':\Delta_\alpha[s_\beta'\mid x_\beta]_{\beta<\alpha} \vdash s_\alpha \equiv_{\Delta_\alpha[s_\beta'\mid x_\beta]_{\beta<\alpha}}s_\alpha'}{\{x_\beta:\Omega_\beta\}_{\beta<\mu},\, \{s_\alpha:\Delta_\alpha[s_\beta\mid x_\beta]_{\beta<\alpha}\}_{\alpha<\lambda},\, \{s_\alpha':\Delta_\alpha[s_\beta'\mid x_\beta]_{\beta<\alpha}\}_{\alpha<\lambda} \\\\ \vdash t[s_\alpha\mid x_\alpha]_{\alpha<\lambda}\equiv_{\Delta[s_\alpha\mid x_\alpha]_{\alpha<\lambda}}t'[s_\alpha'\mid x_\alpha]_{\alpha<\lambda} }
	\]
	Applying the $ I $ to the hypothesis and by \cref{lemma1111} we obtain the following derivations in $ T'. $
	\begin{itemize}
		\item $ \vdash \{x_\beta:\widetilde{I}(\Omega_\beta)\}_{\beta<\mu} \, \context $,
		\item $ \{x_\alpha:\widetilde{I}(\Delta_\alpha)\}_{\alpha<\lambda} \vdash \widetilde{I}(t)\equiv_\Delta \widetilde{I}(t') $,
		\item $ \{x_\beta:\widetilde{I}(\Omega_\beta)\}_{\beta<\mu},\;  s_\alpha:\widetilde{I}(\Delta_\alpha)[\widetilde{I}(s_\beta)\mid x_\beta]_{\beta<\alpha},\;\widetilde{I}( s_\alpha'):\widetilde{I}(\Delta_\alpha)[\widetilde{I}(s_\beta')\mid x_\beta]_{\beta<\alpha} \vdash \widetilde{I}(s_\alpha) \equiv_{\widetilde{I}(\Delta_\alpha)[\widetilde{I}(s_\beta')\mid x_\beta]_{\beta<\alpha}}\widetilde{I}(s_\alpha'). $
	\end{itemize}
	
	We have all the requirements to use \cref{derivedrules} (\ref{drule14}) for the theory $ T' $. Thus,
	\[
	\inferrule{\vdash \{x_\beta:\widetilde{I}(\Omega_\beta)\}_{\beta<\mu} \, \context \\ \{x_\alpha:\widetilde{I}(\Delta_\alpha)\}_{\alpha<\lambda} \vdash \widetilde{I}(t)\equiv_\Delta \widetilde{I}(t') \\\\ \quad \{x_\beta:\widetilde{I}(\Omega_\beta)\}_{\beta<\mu},\;  s_\alpha:\widetilde{I}(\Delta_\alpha)[\widetilde{I}(s_\beta)\mid x_\beta]_{\beta<\alpha},\;\widetilde{I}( s_\alpha'):\widetilde{I}(\Delta_\alpha)[\widetilde{I}(s_\beta')\mid x_\beta]_{\beta<\alpha} \\\\ \vdash \widetilde{I}(s_\alpha) \equiv_{\widetilde{I}(\Delta_\alpha)[\widetilde{I}(s_\beta')\mid x_\beta]_{\beta<\alpha}}\widetilde{I}(s_\alpha')}{\{x_\beta:\widetilde{I}(\Omega_\beta)\}_{\beta<\mu},\, \{\widetilde{I}(s_\alpha):\widetilde{I}(\Delta_\alpha)[\widetilde{I}(s_\beta)\mid x_\beta]_{\beta<\alpha}\}_{\alpha<\lambda},\, \{\widetilde{I}(s_\alpha'):\widetilde{I}(\Delta_\alpha)[\widetilde{I}(s_\beta')\mid x_\beta]_{\beta<\alpha}\}_{\alpha<\lambda} \\\\ \vdash \widetilde{I}(t)[\widetilde{I}(s_\alpha)\mid x_\alpha]_{\alpha<\lambda}\equiv_{\widetilde{I}(\Delta)[\widetilde{I}(s_\alpha)\mid x_\alpha]_{\alpha<\lambda}}\widetilde{I}(t')[\widetilde{I}(s_\alpha')\mid x_\alpha]_{\alpha<\lambda} }
	\]
	is a derived rule of $ T' $. Therefore, the rule is preserved by the interpretation $ I $.
	
\end{proof}

The following lemma fills the gap:

\begin{lemma} \label{lemma1111}
	If $ I $ is an interpretation of $ T $ into $ T' $ and we have expressions $ f $ and $ \{t_\alpha\}_{\alpha<\lambda} $ on the alphabet $ A_T $, then $$ \widetilde{I}(f[t_\alpha\mid x_\alpha]_{\alpha<\lambda})= \widetilde{I}(f)[\widetilde{I}(t_\alpha)\mid x_\alpha]_{\alpha<\lambda} .$$
\end{lemma}
\begin{proof}
	This is done by induction on the length of $ f $ in \cite[Lemma 1, pp. 1.52]{cartmell1978}. The interesting case is when $ f= F(e_\beta)_{\beta<\mu} $ for some $ F $ in the alphabet and expressions $ \{e_\beta\}_{\beta<\mu} $. We assume inductively the result true for the expressions $ \{e_\beta\}_{\beta<\mu} $. Then we have:
	\begin{align*}
		\widetilde{I}(f[t_\alpha\mid x_\alpha]_{\alpha<\lambda}) &= \widetilde{I}(F(e_\beta[t_\alpha\mid x_\alpha]_{\alpha<\lambda})_{\beta<\mu} ) \\
		&= I(F)(\widetilde{I}(e_\beta[t_\alpha\mid x_\alpha]_{\alpha<\lambda}) )_{\beta<\mu} \\
		&= I(F)(\widetilde{I}(e_\beta)[\widetilde{I}(t_\alpha)\mid x_\alpha]_{\alpha<\lambda})_{\beta<\mu}, \text{ by induction hypothesis}\\
		&= I(F)(\widetilde{I}(e_\beta))_{\beta<\mu}[\widetilde{I}(t_\alpha)\mid x_\alpha]_{\alpha<\lambda} \\
		&= \widetilde{I}(F(e_\beta)_{\beta<\mu})[\widetilde{I}(t_\alpha)\mid x_\alpha]_{\alpha<\lambda} \\
		&= \widetilde{I}(f)[\widetilde{I}(t_\alpha)\mid x_\alpha]_{\alpha<\lambda}
	\end{align*}
\end{proof}

There is also a notion of composition of interpretations: If $ I:S\to T $ and $ J: T \to U $ are interpretations, then there is an interpretation $ J\circ I:S\to U $ that is defined in the obvious way. It is also easy to infer what is the identity for this composition. A crucial result to define these compositions is:

\begin{lemma}
	If $ I:S\to T $ and $ J: T \to U $ are interpretations then $ \widetilde{J \circ I}(e)= \widetilde{J}(\widetilde{I}(e)) $
\end{lemma}
\begin{proof}
	This is by induction of the expression $ e $ see \cite[Lemma 3, pp. 1.55]{cartmell1978}.
\end{proof}

We can define the category $ \kgat $ of $ \kappa $-generalized algebraic theories. There is an equivalence relation on interpretations between two theories $ T $ and $ T' $. If $ I,\, J:T\to T' $ are two interpretations, then $ I\approx J $ if an only if for every rule $ r\in R_U $ we have $ I(r)\approx J(r) $ in the theory $ T' $.

\begin{lemma} \label{lemma11114}
	If $ I $ and $ J $ are interpretations from $ T $ to $ T' $ such that $ I\approx J$ then for all type and element judgments $ \mathcal{J} $ of $ U $, $ \widehat{I}(\mathcal{J})\approx \widehat{J}(\mathcal{J})$ in $ T' .$
\end{lemma}
\begin{proof}
	See \cite[Lemma 1, Section 1.14]{cartmell1978}.
\end{proof}

Then \cref{lemma11114} implies that the compositions as given is well-defined. Finally, in order to get the correct morphisms, we need to know that the equivalence relation on interpretations is compatible with the composition. Another advantageous consequence is that this it gives us criteria to establish whether two interpretations are equivalent.

\begin{corollary} \label{corollary2114}
	If $ I $ and $ J $ are interpretations from $ T $ to $ T' $ then $ I\approx J$ if and only if for any type element judgment $ r $, $ \widehat{I}(r)\approx \widehat{J}(r) $.
\end{corollary}
\begin{proof}
	This follows from \cref{lemma11114} and (3) of \cref{contextdefinition}.
\end{proof}

\begin{corollary}
	If $ I $ and $ J $ are interpretations from $ T $ to $ T' $ and $ I' $ and $ J' $ are interpretations from $ T' $ to $ T'' $ then from $ I\approx J $ and $ I'\approx J' $ we conclude that $ I'\circ I\approx J'\circ J $.
\end{corollary}
\begin{proof}
	\cite[pp. 1.72]{cartmell1978}.
\end{proof}

The category $ \kgat $ has morphisms equivalence classes of interpretations \cite[pp. 1.72]{cartmell1978}.

\subsection{Construction and properties of the syntactic category \texorpdfstring{$ \C_T $}{Lg}} \label{syntacticcat}

Let $ T $ be a generalized $\kappa$-algebraic theory. The category $ \C_T $ has the following data:

\begin{itemize}
	\item Objects: Equivalence classes of contexts under the relation $ \approx $. If $ \{x_\alpha:\Delta_\alpha\}_{\alpha<\lambda} $ is a context then the object in $ \C_T $ is denoted $ [\{x_\alpha:\Delta_\alpha\}_{\alpha<\lambda}] $.
	\item Morphisms: A morphism between $ [\{x_\alpha:\Delta_\alpha\}_{\alpha<\lambda}] $ and $ [\{x_\beta:\Omega_\mu\}_{\beta<\mu}] $ is the equivalence class of a map $$ \langle t_\beta\rangle_{\beta<\mu}: \{x_\alpha:\Delta_\alpha\}_{\alpha<\lambda} \to \{x_\beta:\Omega_\beta\}_{\beta<\mu} $$ induced by the relation $ \approx. $ We denote this set by $$ \hom_{\C_T}([\{x_\alpha:\Delta_\alpha\}_{\alpha<\lambda}], [\{x_\beta:\Omega_\mu\}_{\beta<\mu}]). $$
	\item Composition: This is induced by the composition of maps between contexts. This is again well-defined in view of \ref{item2} of \cref{lemma3}.
	\item Identity: For a context $ \{x_\alpha:\Delta_\alpha\}_{\alpha<\lambda} $ its identity is the equivalence class of the obvious map $ \langle x_\alpha \rangle_{\alpha<\lambda}$.
	
\end{itemize}  

\begin{remark}
	The category $ \C_T $ has a unique object $ 1:=[\emptyset] $, the equivalence class of the empty context. Note that this is a terminal object.
\end{remark}

\begin{remark} \label{remarkdisplay}
	Let $ [\{x_\alpha:\Delta_\alpha\}_{\alpha<\lambda}] $ be an object of $ \C_T $. Then for any $ \mu< \lambda $ we get a morphism $ [\langle x_\beta \rangle_{\beta<\mu} ]: [\{x_\alpha:\Delta_\alpha\}_{\alpha<\lambda}] \to [\{x_\beta:\Delta_\beta\}_{\beta<\mu}]$.
	Indeed, since $ \{x_\alpha:\Delta_\alpha\}_{\alpha<\lambda} $ is a context then for any $ \beta<\lambda $ we have $ \{x_\delta:\Delta_\delta\}_{\delta<\delta}\vdash \Delta_\beta \; \type$. Therefore, it follows from (\cref{derivedrules}, \ref{drule10}) that $ \{x_\alpha:\Delta_\alpha\}_{\alpha<\lambda} \vdash x_\alpha:\Delta_\alpha $ for all $ \alpha<\lambda $. In particular this is true for all $ \beta<\mu $, which gives the morphism above.
	
	Following the same argument, if $ \nu<\mu $, then we also have a map $ [\langle x_\gamma \rangle_{\gamma<\nu} ]: [\{x_\beta:\Delta_\beta\}_{\beta<\mu}] \to [\{x_\gamma:\Delta_\gamma\}_{\gamma<\nu}]. $ Furthermore, we get a commutative diagram:	
	% https://q.uiver.app/#q=WzAsMyxbMCwwLCJbXFx7eF9cXGFscGhhOlxcRGVsdGFfXFxhbHBoYVxcfV97XFxhbHBoYTxcXGthcHBhfV0iXSxbMiwwLCJbXFx7eF9cXGJldGE6XFxEZWx0YV9cXGJldGFcXH1fe1xcYmV0YTxcXG11fV0iXSxbMiwxLCJbXFx7eF9cXGdhbW1hOlxcRGVsdGFfXFxnYW1tYVxcfV97XFxnYW1tYTxcXG51fV0iXSxbMCwxLCJbXFxsYW5nbGUgeF9cXGJldGEgXFxyYW5nbGVfe1xcYmV0YTxcXG11fSBdIl0sWzEsMiwiW1xcbGFuZ2xlIHhfXFxnYW1tYSBcXHJhbmdsZV97XFxnYW1tYTxcXG51fSBdIl0sWzAsMiwiW1xcbGFuZ2xlIHhfXFxnYW1tYSBcXHJhbmdsZV97XFxnYW1tYTxcXG51fSBdIiwyXV0=
	\[\begin{tikzcd}
		{[\{x_\alpha:\Delta_\alpha\}_{\alpha<\lambda}]} && {[\{x_\beta:\Delta_\beta\}_{\beta<\mu}]} \\
		&& {[\{x_\gamma:\Delta_\gamma\}_{\gamma<\nu}]}
		\arrow["{[\langle x_\beta \rangle_{\beta<\mu} ]}", from=1-1, to=1-3]
		\arrow["{[\langle x_\gamma \rangle_{\gamma<\nu} ]}", from=1-3, to=2-3]
		\arrow["{[\langle x_\gamma \rangle_{\gamma<\nu} ]}"', from=1-1, to=2-3]
	\end{tikzcd}\]

      \begin{remark} \label{display:generalized}
        Since these morphisms are somewhat canonical we will use the notation $``\display" $, and whenever we use this arrow for a morphism it must be assumed that such map is of this form. These morphisms are called display, which is Cartmell's terminology. In contrast, our `display' maps can be of arbitrary length, which we will often refer to as \emph{generalized display} maps.
      \end{remark}
      
      Suppose there is a context $ [\{x_\alpha:\Delta_\alpha\}_{\alpha<\lambda+\varepsilon}] $ with $ \varepsilon \geq 0 $. Then we can consider an $ \varepsilon $-indexed sequence of display morphisms:
	% https://q.uiver.app/#q=WzAsNCxbMSwwLCJbXFx7eF9cXGFscGhhOlxcRGVsdGFfXFxhbHBoYVxcfV97XFxhbHBoYTxcXGthcHBhKzJ9XSJdLFsyLDAsIltcXHt4X1xcYWxwaGE6XFxEZWx0YV9cXGFscGhhXFx9X3tcXGFscGhhPFxca2FwcGErMX1dIl0sWzMsMCwiW1xce3hfXFxhbHBoYTpcXERlbHRhX1xcYWxwaGFcXH1fe1xcYWxwaGE8XFxrYXBwYX1dIl0sWzAsMCwiXFx0ZXh0e30iXSxbMCwxLCIiLDEseyJzdHlsZSI6eyJoZWFkIjp7Im5hbWUiOiJlcGkifX19XSxbMSwyLCIiLDEseyJzdHlsZSI6eyJoZWFkIjp7Im5hbWUiOiJlcGkifX19XSxbMywwLCJcXGNkb3RzIiwxLHsic3R5bGUiOnsiYm9keSI6eyJuYW1lIjoibm9uZSJ9LCJoZWFkIjp7Im5hbWUiOiJub25lIn19fV1d
	\[\begin{tikzcd}
		{\text{}} & {[\{x_\alpha:\Delta_\alpha\}_{\alpha<\lambda+2}]} & {[\{x_\alpha:\Delta_\alpha\}_{\alpha<\lambda+1}]} & {[\{x_\alpha:\Delta_\alpha\}_{\alpha<\lambda}]}
		\arrow[two heads, from=1-2, to=1-3]
		\arrow[two heads, from=1-3, to=1-4]
		\arrow["\cdots"{description}, draw=none, from=1-1, to=1-2]
	\end{tikzcd}\]
	
	Also, there is a display map $ [\{x_\alpha:\Delta_\alpha\}_{\alpha<\lambda+\varepsilon}] \display [\{x_\alpha:\Delta_\alpha\}_{\alpha<\lambda}] $. This display morphism will be by definition the composition for the sequence. If $ \varepsilon=0 $, then this map is simply the identity. We also get a factorization of the map $ [\{x_\alpha:\Delta_\alpha\}_{\alpha<\lambda}] \display 1 $ via display maps for any $ \lambda \geq 0 $.
\end{remark}

\begin{observation} \label{limitcontext}
	From the previous \cref{remarkdisplay} we can observe that if $ \lambda $ is a limit ordinal then $ [\{x_\alpha:\Delta_\alpha\}_{\alpha<\lambda}] $ is the limit of the sequence
	% https://q.uiver.app/#q=WzAsNCxbMSwwLCJbXFx7eF8xOlxcRGVsdGFfMSx4XzI6XFxEZWx0YV8yXFx9XSJdLFsyLDAsIltcXHt4XzE6XFxEZWx0YV8xXFx9XSJdLFszLDAsIjEiXSxbMCwwLCJcXHRleHR7fSJdLFswLDEsIiIsMSx7InN0eWxlIjp7ImhlYWQiOnsibmFtZSI6ImVwaSJ9fX1dLFsxLDIsIiIsMSx7InN0eWxlIjp7ImhlYWQiOnsibmFtZSI6ImVwaSJ9fX1dLFszLDAsIlxcY2RvdHMiLDEseyJzdHlsZSI6eyJib2R5Ijp7Im5hbWUiOiJub25lIn0sImhlYWQiOnsibmFtZSI6Im5vbmUifX19XV0=
	\[\begin{tikzcd}
		{\text{}} & {[\{x_1:\Delta_1,x_2:\Delta_2\}]} & {[\{x_1:\Delta_1\}]} & 1.
		\arrow[two heads, from=1-2, to=1-3]
		\arrow[two heads, from=1-3, to=1-4]
		\arrow["\cdots"{description}, draw=none, from=1-1, to=1-2]
	\end{tikzcd}
	\]
	If there is another context $ [\{x_\delta:\Gamma_\delta \}_{\delta<\gamma} ] $ and maps \[ [\langle t_\beta \rangle_{\beta<\alpha} ]:[\{x_\delta:\Gamma_\delta \}_{\delta<\gamma} ] \to [\{x_\beta:\Delta_\beta\}_{\beta<\alpha}]  \]
	for all $ \alpha<\lambda $ then we can simply take the map
	\[[\langle t_\alpha \rangle_{\alpha<\lambda} ]: [\{x_\delta:\Gamma_\delta \}_{\delta<\gamma} ]  \to [\{x_\alpha:\Delta_\alpha\}_{\alpha<\lambda}]. \]
	This can be shown to be the cone map (which is unique). This verifies our claim.
\end{observation}

Using \cref{remarkdisplay} we can define a function:
% https://q.uiver.app/#q=WzAsMixbMCwwLCJcXG51Ok9iKFxcQ19UKSJdLFsxLDAsIlxcbGFtYmRhIl0sWzAsMV1d
\[\begin{tikzcd}
	{\nu:Ob(\C_T)} & \kappa
	\arrow[from=1-1, to=1-2]
\end{tikzcd}\]
as $ \nu([\{x_\alpha:\Delta_\alpha\}_{\alpha<\lambda}]):=\lambda$. We call this the \emph{length function}. We can use $ \nu $ to construct a filtration on the objects of $ \C_T $: we define $$ Ob_\lambda(\C_T):= \nu^{-1}(\lambda)$$ then $ Ob(\C_T)=\coprod_{\lambda<\kappa}Ob_\lambda(\C_T) $, and so if $ \alpha\leq\beta $ then $ Ob_\alpha(\C_T) \subseteq Ob_\beta(\C_T) $. Furthermore, if $ p:A\display B $ is a display morphism, then $ \nu(B)\leq \nu(A) $. 
For $ \alpha<\beta $ there are functions $$ \pi_\beta: Ob_\beta(\C_T)\to Ob_\alpha(\C_T)$$ that are defined in the obvious way. Additionally, $ 1\in Ob_0(\C_T) $ is unique.

The proof of the following lemma is the same as in \cite{cartmell1978}.

\begin{lemma} \label{pullbacksyntactic}
	The pullback of a display map along arbitrary morphisms in $ \C_T $ exists, and it is also display.
\end{lemma}
\begin{proof}
	We use induction over the context length. Assume we have the following diagram in $ \C_T $:
	% https://q.uiver.app/#q=WzAsMyxbMiwwLCJbXFx7eF9cXGJldGE6XFxPbWVnYV9cXGJldGFcXH1fe1xcYmV0YTxcXG11K1xcdmFyZXBzaWxvbn1dIl0sWzIsMSwiW1xce3hfXFxiZXRhOlxcT21lZ2FfXFxiZXRhXFx9X3tcXGJldGE8XFxtdX1dIl0sWzAsMSwiW1xce3hfXFxhbHBoYTpcXERlbHRhX1xcYWxwaGFcXH1fe1xcYWxwaGE8XFxrYXBwYX1dIl0sWzAsMSwiW1xcbGFuZ2xlIHhfXFxiZXRhIFxccmFuZ2xlX3tcXGJldGE8XFxtdX1dIiwwLHsic3R5bGUiOnsiaGVhZCI6eyJuYW1lIjoiZXBpIn19fV0sWzIsMSwiW1xcbGFuZ2xlIHRfXFxiZXRhXFxyYW5nbGVfe1xcYmV0YTxcXG11fV0iLDJdXQ==
	\[\begin{tikzcd}
		&& {[\{x_\beta:\Omega_\beta\}_{\beta<\mu+1}]} \\
		{[\{x_\alpha:\Delta_\alpha\}_{\alpha<\lambda}]} && {[\{x_\beta:\Omega_\beta\}_{\beta<\mu}]}
		\arrow["{[\langle x_\beta \rangle_{\beta<\mu}]}", two heads, from=1-3, to=2-3]
		\arrow["{[\langle t_\beta\rangle_{\beta<\mu}]}"', from=2-1, to=2-3]
	\end{tikzcd}\]
	Then the pullback is given using \cref{lemma3}, and the context is 
	\[ [\{x_\alpha:\Delta_\alpha,\, x_\mu:\Omega_\mu[t_\beta\mid x_\beta]_{\beta<\mu} \}_{\substack{\alpha<\lambda}}]. \]
	Therefore we have a commutative square
	\begin{equation} \label{pullback1}
		% https://q.uiver.app/#q=WzAsNCxbMiwwLCJbXFx7eF9cXGJldGE6XFxPbWVnYV9cXGJldGFcXH1fe1xcYmV0YTxcXG11KzF9XSJdLFsyLDEsIltcXHt4X1xcYmV0YTpcXE9tZWdhX1xcYmV0YVxcfV97XFxiZXRhPFxcbXV9XSJdLFswLDEsIltcXHt4X1xcYWxwaGE6XFxEZWx0YV9cXGFscGhhXFx9X3tcXGFscGhhPFxcbGFtYmRhfV0iXSxbMCwwLCJbXFx7eF9cXGFscGhhOlxcRGVsdGFfXFxhbHBoYSxcXCwgeF9cXG11OlxcT21lZ2FfXFxtdVt0X1xcYmV0YVxcbWlkIHhfXFxiZXRhXV97XFxiZXRhPFxcbXV9IFxcfV97XFxzdWJzdGFja3tcXGFscGhhPFxcbGFtYmRhfX1dIl0sWzAsMSwiW1xcbGFuZ2xlIHhfXFxiZXRhIFxccmFuZ2xlX3tcXGJldGE8XFxtdX1dIiwwLHsic3R5bGUiOnsiaGVhZCI6eyJuYW1lIjoiZXBpIn19fV0sWzIsMSwiW1xcbGFuZ2xlIHRfXFxiZXRhXFxyYW5nbGVfe1xcYmV0YTxcXG11fV0iLDJdLFszLDIsIltcXGxhbmdsZSB4X1xcYWxwaGEgXFxyYW5nbGVfe1xcYWxwaGE8XFxsYW1iZGF9XSIsMix7InN0eWxlIjp7ImhlYWQiOnsibmFtZSI6ImVwaSJ9fX1dLFszLDAsIltcXGxhbmdsZSB0X1xcYmV0YSx4X1xcbXVcXHJhbmdsZV97XFxiZXRhPFxcbXV9XSJdXQ==
		\begin{tikzcd}
			{[\{x_\alpha:\Delta_\alpha,\, x_\mu:\Omega_\mu[t_\beta\mid x_\beta]_{\beta<\mu} \}_{\substack{\alpha<\lambda}}]} && {[\{x_\beta:\Omega_\beta\}_{\beta<\mu+1}]} \\
			{[\{x_\alpha:\Delta_\alpha\}_{\alpha<\lambda}]} && {[\{x_\beta:\Omega_\beta\}_{\beta<\mu}]}
			\arrow["{[\langle x_\beta \rangle_{\beta<\mu}]}", two heads, from=1-3, to=2-3]
			\arrow["{[\langle t_\beta\rangle_{\beta<\mu}]}"', from=2-1, to=2-3]
			\arrow["{[\langle x_\alpha \rangle_{\alpha<\lambda}]}"', two heads, from=1-1, to=2-1]
			\arrow["{[\langle t_\beta,x_\mu\rangle_{\beta<\mu}]}", from=1-1, to=1-3]
		\end{tikzcd}
	\end{equation}
	Note that by definition the left vertical morphism is also display. If there is another commutative square
	% https://q.uiver.app/#q=WzAsNCxbMywwLCJbXFx7eF9cXGJldGE6XFxPbWVnYV9cXGJldGFcXH1fe1xcYmV0YTxcXG11KzF9XSJdLFszLDEsIltcXHt4X1xcYmV0YTpcXE9tZWdhX1xcYmV0YVxcfV97XFxiZXRhPFxcbXV9XSJdLFswLDEsIltcXHt4X1xcYWxwaGE6XFxEZWx0YV9cXGFscGhhXFx9X3tcXGFscGhhPFxcbGFtYmRhfV0iXSxbMCwwLCJbXFx7eF9cXHpldGE6XFxHYW1tYV9cXHpldGFcXH1fe1xcemV0YTxcXHhpfV0iXSxbMCwxLCJbXFxsYW5nbGUgeF9cXGJldGEgXFxyYW5nbGVfe1xcYmV0YTxcXG11fV0iLDAseyJzdHlsZSI6eyJoZWFkIjp7Im5hbWUiOiJlcGkifX19XSxbMiwxLCJbXFxsYW5nbGUgdF9cXGJldGFcXHJhbmdsZV97XFxiZXRhPFxcbXV9XSIsMl0sWzMsMiwiW1xcbGFuZ2xlIGZfXFxhbHBoYSBcXHJhbmdsZV97XFxhbHBoYTxcXGxhbWJkYX1dIiwyXSxbMywwLCJbXFxsYW5nbGUgZ19cXGJldGEgXFxyYW5nbGVfe1xcYmV0YTxcXG11KzF9XSJdXQ==
	\[\begin{tikzcd}
		{[\{x_\zeta:\Gamma_\zeta\}_{\zeta<\xi}]} &&& {[\{x_\beta:\Omega_\beta\}_{\beta<\mu+1}]} \\
		{[\{x_\alpha:\Delta_\alpha\}_{\alpha<\lambda}]} &&& {[\{x_\beta:\Omega_\beta\}_{\beta<\mu}],}
		\arrow["{[\langle x_\beta \rangle_{\beta<\mu}]}", two heads, from=1-4, to=2-4]
		\arrow["{[\langle t_\beta\rangle_{\beta<\mu}]}"', from=2-1, to=2-4]
		\arrow["{[\langle f_\alpha \rangle_{\alpha<\lambda}]}"', from=1-1, to=2-1]
		\arrow["{[\langle g_\beta \rangle_{\beta<\mu+1}]}", from=1-1, to=1-4]
	\end{tikzcd}\]
	the map
	\[  [\langle f_\alpha,g_\mu \rangle_{\alpha<\lambda}]:[\{x_\zeta:\Gamma_\zeta\}_{\zeta<\xi}] \to [\{x_\alpha:\Delta_\alpha,\, x_\mu:\Omega_\mu[t_\beta\mid x_\beta]_{\beta<\mu} \}_{\substack{\alpha<\lambda}}] \]
	shows that the square (\ref{pullback1}) is the pullback.\\
	Next, assume that we have a diagram
	% https://q.uiver.app/#q=WzAsMyxbMiwwLCJbXFx7eF9cXGJldGE6XFxPbWVnYV9cXGJldGFcXH1fe1xcYmV0YTxcXG11K1xcdmFyZXBzaWxvbn1dIl0sWzIsMSwiW1xce3hfXFxiZXRhOlxcT21lZ2FfXFxiZXRhXFx9X3tcXGJldGE8XFxtdX1dIl0sWzAsMSwiW1xce3hfXFxhbHBoYTpcXERlbHRhX1xcYWxwaGFcXH1fe1xcYWxwaGE8XFxrYXBwYX1dIl0sWzAsMSwiW1xcbGFuZ2xlIHhfXFxiZXRhIFxccmFuZ2xlX3tcXGJldGE8XFxtdX1dIiwwLHsic3R5bGUiOnsiaGVhZCI6eyJuYW1lIjoiZXBpIn19fV0sWzIsMSwiW1xcbGFuZ2xlIHRfXFxiZXRhXFxyYW5nbGVfe1xcYmV0YTxcXG11fV0iLDJdXQ==
	\[\begin{tikzcd}
		&& {[\{x_\beta:\Omega_\beta\}_{\beta<\mu}]} \\
		{[\{x_\alpha:\Delta_\alpha\}_{\alpha<\lambda}]} && {[\{x_\beta:\Omega_\beta\}_{\beta<\nu}]}
		\arrow["{[\langle x_\beta \rangle_{\beta<\mu}]}", two heads, from=1-3, to=2-3]
		\arrow["{[\langle t_\beta\rangle_{\beta<\nu}]}"', from=2-1, to=2-3]
	\end{tikzcd}\]
	where $ \mu $ is a limit ordinal and $ \mu > \nu $. We simplify the notation as follows:
	% https://q.uiver.app/#q=WzAsMyxbMSwwLCJCX1xcbXUiXSxbMSwxLCJCX1xcbnUiXSxbMCwxLCJBX1xcbGFtYmRhIl0sWzIsMSwiXFxsYW5nbGUgdF9cXGJldGEgXFxyYW5nbGVfe30iLDJdLFswLDEsIiIsMCx7InN0eWxlIjp7ImhlYWQiOnsibmFtZSI6ImVwaSJ9fX1dXQ==
	\[\begin{tikzcd}
		& {B_\mu} \\
		{A_\lambda} & {B_\nu}
		\arrow["{\langle t_\beta \rangle_{\beta<\nu}}"', from=2-1, to=2-2]
		\arrow[two heads, from=1-2, to=2-2]
	\end{tikzcd}
	\]
	Assume that the factorization of the map $ B_\mu \display B_\nu $ is of the form $$ \dots \display B_{\nu+2} \display B_{\nu+1} \display B_\nu $$ and therefore $ B_\mu $ is the limit (obtained similarly as in \cref{limitcontext} and \cref{remarkdisplay}). Then we can take the successive pullback
	% https://q.uiver.app/#q=WzAsOCxbMiwzLCJCX3tcXG51fSJdLFsyLDIsIkJfe1xcbnUrMX0iXSxbMiwxLCJCX3tcXG51KzJ9Il0sWzIsMCwiQl97XFxtdX0iXSxbMCwzLCJBX1xcbGFtYmRhIl0sWzAsMiwiZl4qQl97XFxudSsxfSJdLFswLDEsInEoZixCX3tcXG51KzF9KV4qQl97XFxudSsyfSJdLFswLDAsImZeKkJfXFxtdSJdLFs0LDAsImYiLDJdLFsxLDAsIiIsMCx7InN0eWxlIjp7ImhlYWQiOnsibmFtZSI6ImVwaSJ9fX1dLFsyLDEsIiIsMCx7InN0eWxlIjp7ImhlYWQiOnsibmFtZSI6ImVwaSJ9fX1dLFszLDIsIlxcdmRvdHMiLDEseyJzdHlsZSI6eyJib2R5Ijp7Im5hbWUiOiJub25lIn0sImhlYWQiOnsibmFtZSI6Im5vbmUifX19XSxbNSwxLCJxKGYsQl97XFxudSsxfSkiXSxbNSw0LCIiLDIseyJzdHlsZSI6eyJoZWFkIjp7Im5hbWUiOiJlcGkifX19XSxbNiw1LCIiLDAseyJzdHlsZSI6eyJoZWFkIjp7Im5hbWUiOiJlcGkifX19XSxbNiwyLCJxKHEoZixCX3tcXG51KzF9KSxCX3tcXG51KzJ9KSJdLFs3LDMsInEoZixCX1xcbXUpIl0sWzcsNiwiXFx2ZG90cyIsMSx7InN0eWxlIjp7ImJvZHkiOnsibmFtZSI6Im5vbmUifSwiaGVhZCI6eyJuYW1lIjoibm9uZSJ9fX1dLFs1LDAsIlxcbHJjb3JuZXIiLDEseyJsYWJlbF9wb3NpdGlvbiI6MCwic3R5bGUiOnsiYm9keSI6eyJuYW1lIjoibm9uZSJ9LCJoZWFkIjp7Im5hbWUiOiJub25lIn19fV0sWzYsMSwiXFxscmNvcm5lciIsMSx7ImxhYmVsX3Bvc2l0aW9uIjowLCJzdHlsZSI6eyJib2R5Ijp7Im5hbWUiOiJub25lIn0sImhlYWQiOnsibmFtZSI6Im5vbmUifX19XSxbNywwLCJcXGxyY29ybmVyIiwxLHsibGFiZWxfcG9zaXRpb24iOjAsInN0eWxlIjp7ImJvZHkiOnsibmFtZSI6Im5vbmUifSwiaGVhZCI6eyJuYW1lIjoibm9uZSJ9fX1dXQ==
	\begin{equation} \label{limitpullback}
		\begin{tikzcd}
			{f^*B_\mu} && {B_{\mu}} \\
			{q(f,B_{\nu+1})^*B_{\nu+2}} && {B_{\nu+2}} \\
			{f^*B_{\nu+1}} && {B_{\nu+1}} \\
			{A_\lambda} && {B_{\nu}}
			\arrow["f"', from=4-1, to=4-3]
			\arrow[two heads, from=3-3, to=4-3]
			\arrow[two heads, from=2-3, to=3-3]
			\arrow["\vdots"{description}, draw=none, from=1-3, to=2-3]
			\arrow["{q(f,B_{\nu+1})}", from=3-1, to=3-3]
			\arrow[two heads, from=3-1, to=4-1]
			\arrow[two heads, from=2-1, to=3-1]
			\arrow["{q(q(f,B_{\nu+1}),B_{\nu+2})}", from=2-1, to=2-3]
			\arrow["{q(f,B_\mu)}", from=1-1, to=1-3]
			\arrow["\vdots"{description}, draw=none, from=1-1, to=2-1]
			\arrow["\lrcorner"{description, pos=0}, draw=none, from=3-1, to=4-3]
			\arrow["\lrcorner"{description, pos=0}, draw=none, from=2-1, to=3-3]
			\arrow["\lrcorner"{description, pos=0}, draw=none, from=1-1, to=4-3]
		\end{tikzcd}
	\end{equation}
	where at each successor stage it is given as before, $ f:= \langle t_\beta \rangle_{\beta<\nu}$, the context $$ f^*B_\mu:= [\{x_\alpha:\Delta_\alpha,\, x_\beta:\Omega_\beta[t_\delta\mid x_\delta ]_{\delta<\beta}\}_{\substack{\alpha<\lambda \\ \nu < \beta<\mu}}] $$
	is the limit of the sequence on the left-hand side, with the obvious display maps to each object in the sequence, and $$ q(f,B_\mu):= [\langle t_\beta,\, x_\gamma \rangle_{\substack{\beta< \nu<\gamma<\mu}}].$$
	This makes the outer rectangle in (\ref{limitpullback}) commutative. Moreover, the map $ q(f,B_\mu) $ is the unique cone map induced by the family of maps $$ \{[\langle t_\beta,\, x_\gamma \rangle_{\substack{\beta<\nu<\gamma<\delta}}]: f^*B_\mu \to B_\delta\}_{\nu<\delta<\mu}. $$
\end{proof}

Using the same notation as in the lemma above, we have:
\begin{remark}
	\begin{enumerate}
		\item If $ f=Id_{B_\nu} $ then $ (Id_{B_\nu})^*B_\mu=B_\mu $ and $ q(Id_{B_\nu},B_\mu)=Id_{B_\mu} $.
		\item For a diagram
		% https://q.uiver.app/#q=WzAsNCxbMiwwLCJBIl0sWzIsMSwiQiJdLFsxLDEsIkMiXSxbMCwxLCJEIl0sWzIsMSwiZiJdLFszLDIsImciXSxbMCwxLCJwIiwwLHsic3R5bGUiOnsiaGVhZCI6eyJuYW1lIjoiZXBpIn19fV1d
		\[\begin{tikzcd}
			&& A \\
			D & C & B,
			\arrow["f", from=2-2, to=2-3]
			\arrow["g", from=2-1, to=2-2]
			\arrow["p", two heads, from=1-3, to=2-3]
		\end{tikzcd}\]
		we have that $ g^*(f^*(A))= (fg)^*(A) $ and $ q(fg,A)=q(f,A)(g,f^*A). $
	\end{enumerate}
\end{remark}

We will refer to the category $ \C_T $ as the \emph{syntactic category} associated to the generalized $\kappa$-algebraic theory $ T $.

\begin{observation} \label{explicitpullbacks}
	We note that \cref{pullbacksyntactic} give us an explicit construction of pullbacks in $ \C_T $, as well as the pullback of the maps and an explicit description of $ q(f,B_\mu)$.
      \end{observation}

      We finish this section by characterizing the display maps in the category $\C_T$. This result says that display maps are somehow generic. We start with a preparatory result.

      \begin{lemma} \label{display:pullback-axiom}
     Let $T$ be a generalized $\kappa$-algebraic theory and $\C_T$ its syntactic $\kappa$-contextual category. Assume that there is a $ f: \Delta \to \Gamma $, then any display map $ B \display \Delta$ of length 1 can be obtained as a pullback of the form
   % https://q.uiver.app/#q=WzAsNCxbMCwwLCJCIl0sWzAsMSwiXFxEZWx0YSJdLFsxLDEsIlxcR2FtbWEiXSxbMSwwLCJcXEdhbW1hJyJdLFszLDIsIiIsMCx7InN0eWxlIjp7ImhlYWQiOnsibmFtZSI6ImVwaSJ9fX1dLFswLDNdLFswLDEsIiIsMix7InN0eWxlIjp7ImhlYWQiOnsibmFtZSI6ImVwaSJ9fX1dLFsxLDIsImYiLDJdLFswLDIsIiIsMSx7InN0eWxlIjp7Im5hbWUiOiJjb3JuZXIifX1dXQ==
\[\begin{tikzcd}
	B & {\Gamma'} \\
	\Delta & \Gamma
	\arrow[two heads, from=1-2, to=2-2]
	\arrow[from=1-1, to=1-2]
	\arrow[two heads, from=1-1, to=2-1]
	\arrow["f"', from=2-1, to=2-2]
	\arrow["\lrcorner"{anchor=center, pos=0.125}, draw=none, from=1-1, to=2-2]
      \end{tikzcd}\]
    where $\Gamma' \display \Gamma$ is of length 1.
\end{lemma}
\begin{proof}
  This is simply a reformulation of \cref{derivationlemma}. Assume that $$f=[\langle t_\beta \rangle_{\beta<\mu}]:[\{x_\alpha:\Delta_\alpha\}_{\alpha<\lambda}] \to [\{x_\beta:\Gamma_\beta\}_{\beta<\mu}]. $$ Therefore, when the display map is of the form $$ [\{x_\alpha:\Delta_\alpha\}_{\alpha<\lambda+1}] \display [\{x_\alpha:\Delta_\alpha\}_{\alpha<\lambda}].$$ We can construct the square
  % https://q.uiver.app/#q=WzAsNCxbMCwwLCJbXFx7eF9cXGFscGhhOlxcRGVsdGFfXFxhbHBoYVxcfV97XFxhbHBoYTxcXGxhbWJkYSsxfV0iXSxbMCwxLCJbXFx7eF9cXGFscGhhOlxcRGVsdGFfXFxhbHBoYVxcfV97XFxhbHBoYTxcXGxhbWJkYX1dIl0sWzIsMSwiW1xce3g6XFxHYW1tYV9cXGJldGFcXH1fe1xcYmV0YTxcXG11fV0iXSxbMiwwLCJbXFx7eDpcXEdhbW1hX1xcYmV0YSx4X1xcbGFtYmRhOlxcRGVsdGFfXFxsYW1iZGFcXH1fe1xcYmV0YTxcXG11fV0iXSxbMywyLCIiLDAseyJzdHlsZSI6eyJoZWFkIjp7Im5hbWUiOiJlcGkifX19XSxbMCwzLCJcXGxhbmdsZSB0X1xcYmV0YSx4X1xcbGFtYmRhIFxccmFuZ2xlX3tcXGJldGE8XFxtdX0iXSxbMCwxLCIiLDIseyJzdHlsZSI6eyJoZWFkIjp7Im5hbWUiOiJlcGkifX19XSxbMSwyLCJcXGxhbmdsZSB0X1xcYmV0YSBcXHJhbmdsZV97XFxiZXRhPFxcbXV9IiwyXV0=
\[\begin{tikzcd}
	{[\{x_\alpha:\Delta_\alpha\}_{\alpha<\lambda+1}]} && {[\{x:\Gamma_\beta,x_\lambda:\Delta_\lambda\}_{\beta<\mu}]} \\
	{[\{x_\alpha:\Delta_\alpha\}_{\alpha<\lambda}]} && {[\{x:\Gamma_\beta\}_{\beta<\mu}]}.
	\arrow[two heads, from=1-3, to=2-3]
	\arrow["{\langle t_\beta,x_\lambda \rangle_{\beta<\mu}}", from=1-1, to=1-3]
	\arrow[two heads, from=1-1, to=2-1]
	\arrow["{\langle t_\beta \rangle_{\beta<\mu}}"', from=2-1, to=2-3]
      \end{tikzcd}\]
    Since for all $\beta<\mu$, $x_\beta$ does not occur in $\Delta_\lambda$ we have that $\Delta_\lambda[t_\beta|x_\beta]_{\beta<\mu}\equiv \Delta_\lambda $. Hence, it follows from the construction of pullbacks in $\C_T$ (\cref{pullbacksyntactic}) that the square above is indeed a pullback diagram.
  \end{proof}
      
  We are ready to give the full description of display maps.
  
  \begin{proposition} \label{generalized-display:limit-display}
    Every display map $B \display \Delta$ in $\C_T$ is a limit of a $\kappa$-small tower $V:\lambda \to \C_T$ where for each limit ordinal $\beta < \lambda$
        \[V(\beta) = \Lim_{\alpha < \beta } V(\alpha) \]
and the map $V(\alpha+1) \to V(\alpha)$ is a pullback of a length one display map of the form $(\Gamma,A) \display \Gamma $ where $\Gamma \vdash A \, \type$ is a type axiom of the theory $T$.
\end{proposition}
\begin{proof}
Each display map in $\C_T$ has a length $\lambda$. Just as in \cref{remarkdisplay} it admits a decomposition into display maps. It will be enough to prove the second claim, and this follows by an inductive argument in conjunction with the previous \cref{display:pullback-axiom}. The inductive step provides us with the required map $f:V(\alpha) \to \Gamma $ in \cref{display:pullback-axiom}.
\end{proof}

      %%%%%%%%%%%%%%%%%%%%%%%%%%%%5

%%% Local Variables:
%%% mode: latex
%%% TeX-master: "main"
%%% End:

        % Appendix B: Contextual categories and cartmell theories

        \section{Contextual categories and Cartmell theories}\label{appendix-b}

This section is the most relevant part. We will show that from the syntax of a generalized $\kappa$-algebraic theory we can construct a category, called $\kappa$-\textit{contextual category}, which we now introduce.

\subsection{\texorpdfstring{$\kappa$}{Lg}-contextual categories}

The discussion in \cref{syntacticcat} on the properties of the syntactic category $ \C_T $ can be summarized with the next definition, which is the natural generalization of Cartmell's \cite{cartmell1978} or \cite{kapulkinlumsdaine2018}. We present our definition in the same way as in the latter reference. Recall that $ \kappa $ is a regular cardinal.

\begin{definition}\label{contextualcat}
	A category $ \catC $ is said to be a $ \kappa $-\emph{contextual category} if:
	\begin{enumerate}
		\item The objects of $ \catC $ have grading $ Ob(\catC)=\coprod_{\lambda<\kappa}Ob_\lambda (\catC) $. This grading determines the \emph{height} of any object $ B \in \catC $, which we write as $ ht(B) $.
		\item There is a terminal object $ 1\in \catC $, it is unique up to equality and has height 0.
		\item There is a wide subcategory $ Dis(\catC) $ with distinguished maps $ ``\display'' $ called \emph{display morphisms},
		\item The subcategory $ Dis(\catC) $ is closed under transfinite compositions: if we have	
		% https://q.uiver.app/#q=WzAsNSxbNCwwLCJCX1xcbXUiXSxbMywwLCJCX3tcXG11KzF9Il0sWzIsMCwiQl97XFxtdSsyfSJdLFsxLDAsIkJfe1xcbXUrM30iXSxbMCwwLCJcXGNkb3RzIl0sWzIsMSwiIiwwLHsic3R5bGUiOnsiaGVhZCI6eyJuYW1lIjoiZXBpIn19fV0sWzEsMCwiIiwwLHsic3R5bGUiOnsiaGVhZCI6eyJuYW1lIjoiZXBpIn19fV0sWzMsMiwiIiwwLHsic3R5bGUiOnsiaGVhZCI6eyJuYW1lIjoiZXBpIn19fV0sWzQsMywiIiwwLHsic3R5bGUiOnsiaGVhZCI6eyJuYW1lIjoiZXBpIn19fV1d
		\[\begin{tikzcd}
			\cdots & {B_{3}} & {B_{2}} & {B_{1}} & {B_0}
			\arrow[two heads, from=1-3, to=1-4]
			\arrow[two heads, from=1-4, to=1-5]
			\arrow[two heads, from=1-2, to=1-3]
			\arrow[two heads, from=1-1, to=1-2]
		\end{tikzcd}\]
		a $ \lambda $-sequence of display maps, then there is a unique object $ B $ in $ Dis(\catC) $ with height $ \lambda $ and for each $ \mu\leq \lambda $ a display map $ B \display B_\mu $ such that for any $ \alpha<\lambda $ we have a factorization
		% https://q.uiver.app/#q=WzAsMyxbMCwwLCJCIl0sWzIsMCwiQl8wIl0sWzEsMSwiQl9cXGFscGhhIl0sWzAsMiwiIiwwLHsic3R5bGUiOnsiaGVhZCI6eyJuYW1lIjoiZXBpIn19fV0sWzAsMSwiIiwyLHsic3R5bGUiOnsiaGVhZCI6eyJuYW1lIjoiZXBpIn19fV0sWzIsMSwiIiwwLHsic3R5bGUiOnsiaGVhZCI6eyJuYW1lIjoiZXBpIn19fV1d
		\[\begin{tikzcd}[row sep=tiny]
			B && {B_0} \\
			& {B_\alpha}
			\arrow[two heads, from=1-1, to=2-2]
			\arrow[two heads, from=1-1, to=1-3]
			\arrow[two heads, from=2-2, to=1-3]
		\end{tikzcd}\]

		\item The inclusion functor preserves $i:Dis(\catC) \hookrightarrow \catC$ transfinite compositions.
		
		\item  If $ A \display B $ is an arrow in $ Dis(\catC) $ then $ B \in Ob_\mu(\catC) $ and $ A \in Ob_\lambda(\catC) $ for some ordinals $ \lambda,\, \mu $ with $ \mu\leq\lambda $.
		
		\item For any object $ A \in Ob_\lambda(\catC) $ and any $ \mu\leq\lambda $ there exists a unique object $ B \in Ob_\mu(\catC) $ and a unique display map $ A \display B $. The \emph{length} of this display map is the unique ordinal $ \alpha $ such that $ \lambda=\mu + \alpha $, is such situation, we write $ lt(p) $.
		
		\item For any $ A\in Ob_\lambda(\catC) $, a map $ A \display B $ and any map $ f:C \to B $ there is a  pullback square
		% https://q.uiver.app/#q=WzAsNCxbMSwwLCJBIl0sWzEsMSwiXFxwaV9cXG11KEEpIl0sWzAsMSwiQyJdLFswLDAsImZeKkEiXSxbMiwxLCJmIiwyXSxbMCwxLCJwIiwwLHsic3R5bGUiOnsiaGVhZCI6eyJuYW1lIjoiZXBpIn19fV0sWzMsMiwiZl4qcCIsMix7InN0eWxlIjp7ImhlYWQiOnsibmFtZSI6ImVwaSJ9fX1dLFszLDAsInEoZixBKSJdLFszLDEsIlxcbHJjb3JuZXIiLDEseyJsYWJlbF9wb3NpdGlvbiI6MCwic3R5bGUiOnsiYm9keSI6eyJuYW1lIjoibm9uZSJ9LCJoZWFkIjp7Im5hbWUiOiJub25lIn19fV1d
		\[\begin{tikzcd}
			{f^*A} & A \\
			C & B
			\arrow["f"', from=2-1, to=2-2]
			\arrow["p", two heads, from=1-2, to=2-2]
			\arrow["{f^*p}"', two heads, from=1-1, to=2-1]
			\arrow["{q(f,A)}", from=1-1, to=1-2]
			\arrow["\lrcorner"{description, pos=0}, draw=none, from=1-1, to=2-2]
		\end{tikzcd}\]
		called \emph{canonical pullback} of $ A $ along $ f $, and we require $ lt(f^*p)=lt(p) $.
		
		\item Canonical pullbacks are strictly functorial: for ordinals with $ \mu\leq\lambda $, $ A\in Ob_\lambda(\catC) $
		\begin{enumerate}
			\item If $ f=id_B $ then $ id_B^*A=A $ and $ q(id_B,A)=id_A $.
			\item For a diagram
			% https://q.uiver.app/#q=WzAsNSxbMiwwLCJBIl0sWzIsMSwiXFxwaV9cXG11KEEpIl0sWzEsMSwiQyJdLFswLDEsIkQiXSxbMywxXSxbMiwxLCJmIl0sWzMsMiwiZyJdLFswLDEsInAiLDAseyJzdHlsZSI6eyJoZWFkIjp7Im5hbWUiOiJlcGkifX19XV0=
			\[\begin{tikzcd}
				&& A \\
				D & C & B, & {}
				\arrow["f", from=2-2, to=2-3]
				\arrow["g", from=2-1, to=2-2]
				\arrow["p", two heads, from=1-3, to=2-3]
			\end{tikzcd}\]
			we have that $ g^*(f^*(A))= (fg)^*(A) $ and $ q(fg,A)=q(f,A)q(g,f^*A). $
		\end{enumerate}
              \item Given display maps $p:A \display B $ and $q:B \to C$ and any $f: X \to C$, in the diagram
                % https://q.uiver.app/#q=WzAsNixbMiwwLCJBIl0sWzIsMSwiQiJdLFsyLDIsIkMiXSxbMCwyLCJYIl0sWzAsMSwiZl4qQiJdLFswLDAsInEoZixCKV4qQSJdLFswLDEsInAiLDAseyJzdHlsZSI6eyJoZWFkIjp7Im5hbWUiOiJlcGkifX19XSxbMSwyLCJyIiwwLHsic3R5bGUiOnsiaGVhZCI6eyJuYW1lIjoiZXBpIn19fV0sWzMsMiwiZiIsMl0sWzQsMywiZl4qciIsMix7InN0eWxlIjp7ImhlYWQiOnsibmFtZSI6ImVwaSJ9fX1dLFs0LDEsInEoZixCKSIsMl0sWzUsMCwicShxKGYsQiksQSkiXSxbNSw0LCJxKGYsQileKnAiLDJdLFs0LDIsIiIsMix7InN0eWxlIjp7Im5hbWUiOiJjb3JuZXIifX1dLFs1LDEsIiIsMix7InN0eWxlIjp7Im5hbWUiOiJjb3JuZXIifX1dXQ==
\[\begin{tikzcd}
	{q(f,B)^*A} && A \\
	{f^*B} && B \\
	X && C,
	\arrow["{q(q(f,B),A)}", from=1-1, to=1-3]
	\arrow["{q(f,B)^*p}"',two heads, from=1-1, to=2-1]
	\arrow["\lrcorner"{anchor=center, pos=0.125}, draw=none, from=1-1, to=2-3]
	\arrow["p", two heads, from=1-3, to=2-3]
	\arrow["{q(f,B)}"', from=2-1, to=2-3]
	\arrow["{f^*r}"', two heads, from=2-1, to=3-1]
	\arrow["\lrcorner"{anchor=center, pos=0.125}, draw=none, from=2-1, to=3-3]
	\arrow["r", two heads, from=2-3, to=3-3]
	\arrow["f"', from=3-1, to=3-3]
\end{tikzcd}\]
we have that $f^*r \circ(q(f,B)^*p)=f^*(r\circ p)$ and $q(q(f,B),A)=q(f,A)$.
	\end{enumerate}
\end{definition}

\begin{remark}
	We use the term ``display map'' in a rather different way to Cartmell. For us, a display map can have any height, and it is only bounded by the regular cardinal $ \kappa $. 
\end{remark}

We have already seen one example of such a category.

\begin{corollary} \label{syntactic:contextual}
	For any generalized $\kappa$-algebraic theory $ T $ the syntactic category $ \C_T $ is a $ \kappa $-contextual category.
\end{corollary}
\begin{proof}
	This is done throughout \cref{syntacticcat}.
\end{proof}

\begin{remark} \label{convention1}
	It follows from \cref{contextualcat} that for any object $ B\in \catC $ the map $ B \display 1 $ can be decomposed as a transfinite composition of display maps
	\[ B_\lambda \display \dots \display B_1 \display 1.
	\]
	The length of the decomposition above is given by the degree of $ B $. This is what \cite{cartmell1978} calls the tree structure of the category. Whenever we refer to objects in a $ \kappa $-contextual category as above, we will emphasize its height by writing $ B_\lambda $.
	Likewise, we will denote the display maps as $ p_\alpha: B_\lambda \display B_\alpha $ for each $ \alpha<\lambda $.
\end{remark}

The following lemma is a consequence of \cref{contextualcat} and \cref{convention1}.

\begin{lemma}
	Let $ B\in Ob_\lambda(\catC) $ such that $ \lambda $ is a limit ordinal. Then $ B $ itself is a limit object in
	$\catC$. 
\end{lemma}
\begin{proof}
	From \cref{remarkdisplay} we obtain a sequence
	\[\begin{tikzcd}
		  \cdots & {B_{3}} & {B_{2}} & {B_{1}} & {1.}
		\arrow[two heads, from=1-3, to=1-4]
		\arrow[two heads, from=1-4, to=1-5]
		\arrow[two heads, from=1-2, to=1-3]
		\arrow[two heads, from=1-1, to=1-2]
	\end{tikzcd}\]
	It follows from Axiom 4 of \cref{contextualcat} that $B$ must be the limit of
	the sequence. Finally, we use that the inclusion $ Dis(\catC) \to \catC $ preserve limits.
\end{proof}

\begin{definition}
	Let $ \catC, \, \catD $ contextual categories. A functor $ F: \catC \to \catD $ it is called a \emph{contextual functor} if it satisfies the following conditions:
	\begin{enumerate}
		\item $ F(Ob_\lambda(\catC))\subseteq Ob_\lambda(\catD) $ for all $ \lambda<\kappa $,
		\item $ F $ restricts to a functor $ Dis(\catC)\to Dis(\catD) $,
		\item $ F $ preserves canonical pullbacks up to equality, meaning that for any square in $ \catC $
		\[\begin{tikzcd}
			{f^*A} & A \\
			C & B
			\arrow["f"', from=2-1, to=2-2]
			\arrow["p", two heads, from=1-2, to=2-2]
			\arrow["{f^*p}"', two heads, from=1-1, to=2-1]
			\arrow["{q(f,A)}", from=1-1, to=1-2]
			\arrow["\lrcorner"{description, pos=0}, draw=none, from=1-1, to=2-2]
		\end{tikzcd}\]
		we have $ F(f^*A)=(Ff)^*(FA) $ and $ F(q(f,A))=q(Ff,FA) $.
		
	\end{enumerate}
\end{definition}

Since the degree of each object is preserved by a $ \kappa $-contextual functor, it makes sense to denote $ F(A_\lambda):=F(A)_\lambda $ for $ A_\lambda\in \catC $. Another piece of notation we can introduce is from the functor $ F:Dis(\catC) \to Dis(\catD) $. Since any display map $ p_\alpha:A_\lambda \display A_\alpha $ is sent to a display map  $ F(p_\alpha):F(A)_\lambda \display F(A)_\alpha $, and the degrees are preserved, we agree to omit $ F $ on these maps.

Contextual functors are the morphisms of the category of $ \kappa$-contextual categories, which we will denote it as $ \kcon $.

\subsection{Interlude: categorical facts} \label{interlude}

We collect and recall some categorical facts about general $ \kappa $-contextual categories.

\begin{proposition}[The slice $ \kappa $-contextual category]
	Let $ \catC $ be a $ \kappa $-contextual category. For any object $ B\in Ob_\mu(\catC) $ there is a $ \kappa $-contextual category which is a full subcategory of the slice $ \catC_{/B} $ which has objects display maps $ A \display B $ where $ A \in Ob_\lambda(\catC) $ with $ \lambda \geq \mu $. 
\end{proposition}
Since we will rarely use categories other than $ \kappa $-contextual categories, we will employ the slice notation $ \catC_{/B} $ for the category from the previous proposition. 

\begin{proof}
	The proof is completely formal. The important fact to remember is that the pullback of a display map is also a display map.
\end{proof}

It is a well known fact that the pasting of two pullbacks give us a pullback, in our case consider the following diagram:
% https://q.uiver.app/#q=WzAsOCxbMiwzLCJCX3tcXG51fSJdLFsyLDIsIkJfe1xcbnUrMX0iXSxbMiwxLCJCX3tcXG51KzJ9Il0sWzIsMCwiQl97XFxtdX0iXSxbMCwzLCJBX1xcbGFtYmRhIl0sWzAsMiwiZl4qQl97XFxudSsxfSJdLFswLDEsInEoZixCX3tcXG51KzF9KV4qQl97XFxudSsyfSJdLFswLDAsImZeKkJfXFxtdSJdLFs0LDAsImYiLDJdLFsxLDAsIiIsMCx7InN0eWxlIjp7ImhlYWQiOnsibmFtZSI6ImVwaSJ9fX1dLFsyLDEsIiIsMCx7InN0eWxlIjp7ImhlYWQiOnsibmFtZSI6ImVwaSJ9fX1dLFszLDIsIlxcdmRvdHMiLDEseyJzdHlsZSI6eyJib2R5Ijp7Im5hbWUiOiJub25lIn0sImhlYWQiOnsibmFtZSI6Im5vbmUifX19XSxbNSwxLCJxKGYsQl97XFxudSsxfSkiXSxbNSw0LCIiLDIseyJzdHlsZSI6eyJoZWFkIjp7Im5hbWUiOiJlcGkifX19XSxbNiw1LCIiLDAseyJzdHlsZSI6eyJoZWFkIjp7Im5hbWUiOiJlcGkifX19XSxbNiwyLCJxKHEoZixCX3tcXG51KzF9KSxCX3tcXG51KzJ9KSJdLFs3LDMsInEoZixCX1xcbXUpIl0sWzcsNiwiXFx2ZG90cyIsMSx7InN0eWxlIjp7ImJvZHkiOnsibmFtZSI6Im5vbmUifSwiaGVhZCI6eyJuYW1lIjoibm9uZSJ9fX1dLFs1LDAsIlxcbHJjb3JuZXIiLDEseyJsYWJlbF9wb3NpdGlvbiI6MCwic3R5bGUiOnsiYm9keSI6eyJuYW1lIjoibm9uZSJ9LCJoZWFkIjp7Im5hbWUiOiJub25lIn19fV0sWzYsMSwiXFxscmNvcm5lciIsMSx7ImxhYmVsX3Bvc2l0aW9uIjowLCJzdHlsZSI6eyJib2R5Ijp7Im5hbWUiOiJub25lIn0sImhlYWQiOnsibmFtZSI6Im5vbmUifX19XSxbNywwLCJcXGxyY29ybmVyIiwxLHsibGFiZWxfcG9zaXRpb24iOjAsInN0eWxlIjp7ImJvZHkiOnsibmFtZSI6Im5vbmUifSwiaGVhZCI6eyJuYW1lIjoibm9uZSJ9fX1dXQ==
\[\begin{tikzcd}
	{f^*B_\mu} && {B_{\mu}} \\
	{q(f,B_{\nu+1})^*B_{\nu+2}} && {B_{\nu+2}} \\
	{f^*B_{\nu+1}} && {B_{\nu+1}} \\
	{A_\lambda} && {B_{\nu}}
	\arrow["f"', from=4-1, to=4-3]
	\arrow[two heads, from=3-3, to=4-3]
	\arrow[two heads, from=2-3, to=3-3]
	\arrow["\vdots"{description}, draw=none, from=1-3, to=2-3]
	\arrow["{q(f,B_{\nu+1})}", from=3-1, to=3-3]
	\arrow[two heads, from=3-1, to=4-1]
	\arrow[two heads, from=2-1, to=3-1]
	\arrow["{q(q(f,B_{\nu+1}),B_{\nu+2})}", from=2-1, to=2-3]
	\arrow["{q(f,B_\mu)}", from=1-1, to=1-3]
	\arrow["\vdots"{description}, draw=none, from=1-1, to=2-1]
	\arrow["\lrcorner"{description, pos=0}, draw=none, from=3-1, to=4-3]
	\arrow["\lrcorner"{description, pos=0}, draw=none, from=2-1, to=3-3]
	\arrow["\lrcorner"{description, pos=0}, draw=none, from=1-1, to=4-3]
\end{tikzcd}\]
Then if $ \mu $ is a limit ordinal, the object $ B_\mu $ is the limit of the sequence on the right-hand side. Thus, $ f^*B_\mu $ is the limit of the sequence on the left-hand side. Note that pairwise we have $ q(f,B_{\nu+1})^*B_{\nu+2}=f^*B_{\nu+2} $ and $ q(f,B_{\mu+2})=q(q(f,B_{\mu+1}),B_{\mu+2}).$ \\
If $ f:A_\lambda \to B_\nu $ and $ p_\nu:B_\mu \display B_\nu $ is a display map with $ \mu=\nu+1 $, using the universal property of the pullback, we can construct the following diagram:
% https://q.uiver.app/#q=WzAsNSxbMiwxLCJCX3tcXG11fSJdLFsyLDIsIkJfXFxudSJdLFsxLDIsIkFfXFxsYW1iZGEiXSxbMSwxLCIocF9cXG51IGYpXipCX3tcXG11fSJdLFswLDAsIkFfXFxsYW1iZGEiXSxbMiwxLCJwX1xcbnUgZiIsMl0sWzAsMSwicF9cXG51IiwwLHsic3R5bGUiOnsiaGVhZCI6eyJuYW1lIjoiZXBpIn19fV0sWzMsMiwiKHBfXFxudSBmKV4qcF9cXG51IiwxLHsic3R5bGUiOnsiaGVhZCI6eyJuYW1lIjoiZXBpIn19fV0sWzMsMF0sWzMsMSwiXFxscmNvcm5lciIsMSx7ImxhYmVsX3Bvc2l0aW9uIjowLCJzdHlsZSI6eyJib2R5Ijp7Im5hbWUiOiJub25lIn0sImhlYWQiOnsibmFtZSI6Im5vbmUifX19XSxbNCwwLCJmIiwxXSxbNCwyLCJJZF97QV9cXGxhbWJkYX0iLDFdLFs0LDMsIlxcZGVsdGFfZiIsMSx7InN0eWxlIjp7ImJvZHkiOnsibmFtZSI6ImRhc2hlZCJ9fX1dXQ==
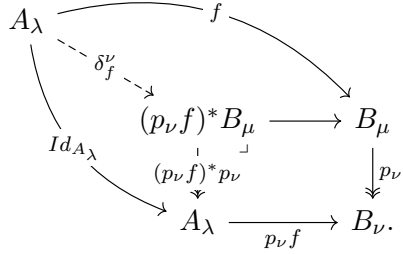
\[\begin{tikzcd}
	{A_\lambda} \\
	& {(p_\nu f)^*B_{\mu}} & {B_{\mu}} \\
	& {A_\lambda} & {B_\nu.}
	\arrow["{p_\nu f}"', from=3-2, to=3-3]
	\arrow["{p_\nu}", two heads, from=2-3, to=3-3]
	\arrow["{(p_\nu f)^*p_\nu}"{description}, two heads, from=2-2, to=3-2]
	\arrow[from=2-2, to=2-3]
	\arrow["\lrcorner"{description, pos=0}, draw=none, from=2-2, to=3-3]
	\arrow["f"{description}, from=1-1, to=2-3,bend left]
	\arrow["{Id_{A_\lambda}}"{description}, from=1-1, to=3-2,bend right]
	\arrow["{\delta_{f}^\nu}"{description}, dashed, from=1-1, to=2-2]
\end{tikzcd}\]
The map $ \delta_{f}^\nu $ makes both triangles commutative. We will focus on the fact that $ ((f_\nu)^*p_\nu)\delta_{f}^\nu=Id_{A_\lambda}$, where $ f_\nu=p_\nu f $.
Assume that we have a map $ p: B_\mu \display B_\nu $ with $ \mu $ a limit ordinal, in particular the length of $ p $ is a limit ordinal. Then a map $ f: A_\lambda \to B_\mu $ is determinate by a family of maps $ \{f_\gamma: A_\lambda \to B_\gamma\} $. Then we obtain:
% https://q.uiver.app/#q=WzAsOSxbMyw0LCJCX3tcXG51fSJdLFszLDMsIkJfe1xcbnUrMX0iXSxbMywyLCJCX3tcXG51KzJ9Il0sWzMsMSwiQl97XFxtdX0iXSxbMSw0LCJBX1xcbGFtYmRhIl0sWzEsMywiZl4qQl97XFxudSsxfSJdLFsxLDIsInEoZixCX3tcXG51KzF9KV4qQl97XFxudSsyfSJdLFsxLDEsImZeKkJfXFxtdSJdLFswLDAsIkFfXFxsYW1iZGEiXSxbNCwwLCJmX1xcbnUiLDJdLFsxLDAsIiIsMCx7InN0eWxlIjp7ImhlYWQiOnsibmFtZSI6ImVwaSJ9fX1dLFsyLDEsIiIsMCx7InN0eWxlIjp7ImhlYWQiOnsibmFtZSI6ImVwaSJ9fX1dLFszLDIsIlxcdmRvdHMiLDEseyJzdHlsZSI6eyJib2R5Ijp7Im5hbWUiOiJub25lIn0sImhlYWQiOnsibmFtZSI6Im5vbmUifX19XSxbNSwxLCJxKGYsQl97XFxudSsxfSkiXSxbNSw0LCIiLDIseyJzdHlsZSI6eyJoZWFkIjp7Im5hbWUiOiJlcGkifX19XSxbNiw1LCIiLDAseyJzdHlsZSI6eyJoZWFkIjp7Im5hbWUiOiJlcGkifX19XSxbNiwyLCJxKHEoZixCX3tcXG51KzF9KSxCX3tcXG51KzJ9KSJdLFs3LDMsInEoZixCX1xcbXUpIl0sWzcsNiwiXFx2ZG90cyIsMSx7InN0eWxlIjp7ImJvZHkiOnsibmFtZSI6Im5vbmUifSwiaGVhZCI6eyJuYW1lIjoibm9uZSJ9fX1dLFs1LDAsIlxcbHJjb3JuZXIiLDEseyJsYWJlbF9wb3NpdGlvbiI6MCwic3R5bGUiOnsiYm9keSI6eyJuYW1lIjoibm9uZSJ9LCJoZWFkIjp7Im5hbWUiOiJub25lIn19fV0sWzYsMSwiXFxscmNvcm5lciIsMSx7ImxhYmVsX3Bvc2l0aW9uIjowLCJzdHlsZSI6eyJib2R5Ijp7Im5hbWUiOiJub25lIn0sImhlYWQiOnsibmFtZSI6Im5vbmUifX19XSxbNywwLCJcXGxyY29ybmVyIiwxLHsibGFiZWxfcG9zaXRpb24iOjAsInN0eWxlIjp7ImJvZHkiOnsibmFtZSI6Im5vbmUifSwiaGVhZCI6eyJuYW1lIjoibm9uZSJ9fX1dLFs4LDQsIklkIiwyLHsiY3VydmUiOjN9XSxbOCwzLCJmIiwwLHsiY3VydmUiOi0zfV0sWzgsNywiXFxkZWx0YV9mIiwxLHsic3R5bGUiOnsiYm9keSI6eyJuYW1lIjoiZGFzaGVkIn19fV0sWzMsMCwicCIsMCx7ImN1cnZlIjotM31dXQ==
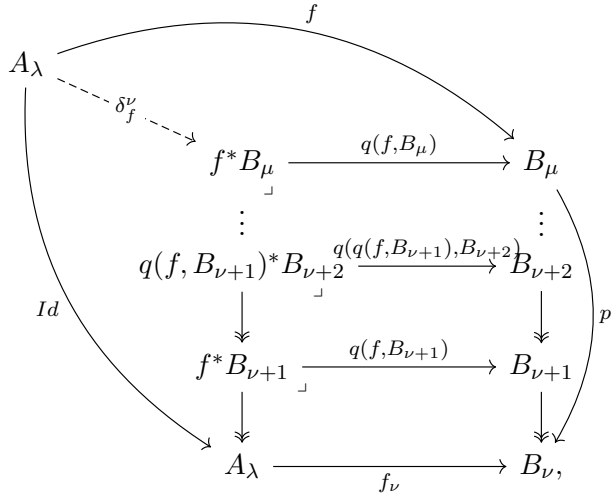
\[\begin{tikzcd}
	{A_\lambda} \\
	& {f^*B_\mu} && {B_{\mu}} \\
	& {q(f,B_{\nu+1})^*B_{\nu+2}} && {B_{\nu+2}} \\
	& {f^*B_{\nu+1}} && {B_{\nu+1}} \\
	& {A_\lambda} && {B_{\nu},}
	\arrow["{f_\nu}"', from=5-2, to=5-4]
	\arrow[two heads, from=4-4, to=5-4]
	\arrow[two heads, from=3-4, to=4-4]
	\arrow["\vdots"{description}, draw=none, from=2-4, to=3-4]
	\arrow["{q(f,B_{\nu+1})}", from=4-2, to=4-4]
	\arrow[two heads, from=4-2, to=5-2]
	\arrow[two heads, from=3-2, to=4-2]
	\arrow["{q(q(f,B_{\nu+1}),B_{\nu+2})}", from=3-2, to=3-4]
	\arrow["{q(f,B_\mu)}", from=2-2, to=2-4]
	\arrow["\vdots"{description}, draw=none, from=2-2, to=3-2]
	\arrow["\lrcorner"{description, pos=0}, draw=none, from=4-2, to=5-4]
	\arrow["\lrcorner"{description, pos=0}, draw=none, from=3-2, to=4-4]
	\arrow["\lrcorner"{description, pos=0}, draw=none, from=2-2, to=5-4]
	\arrow["Id"', bend right, from=1-1, to=5-2]
	\arrow["f", bend left, from=1-1, to=2-4]
	\arrow["{\delta_f^\nu}"{description}, dashed, from=1-1, to=2-2]
	\arrow["p", bend left, from=2-4, to=5-4]
\end{tikzcd}\]
where the map $ \delta_f^\nu $ is given as the family of maps $ (\delta_f^\nu)_\gamma $, each given by an intermediate pullback square in the diagram above.

\begin{notation}
	If the situation above, for $ f:A_\lambda \to B_\mu $ we denote
	\[\Gamma(B_\nu^\mu):=\{h: A_\lambda \to (p_\nu f)^*B_\mu \mid ((p_\nu f)^*p_\nu)h=Id_{A_\lambda} \}. \]
	We can consider a more general case, if $ A_\lambda \in Ob_\lambda(\catC) $ and $ B_\mu \in Ob_\mu(\catC)$ with $ \lambda<\mu $, then there is a unique display map $ p:B_\mu \display A_\lambda $. We set
	\[
	\Gamma(B_\lambda^\mu):=\{s: A_\lambda \to B_\mu \mid ps=Id_{A_\lambda} \}
	\]
	for this situation as well, since the object $ A_\lambda $ will be inferred from the context.
\end{notation}

If the contextual category is $ \C_T $, then recalling \cref{pullbacksyntactic}, we can give an explicit description of the map $ \delta_f^\nu $.

\begin{lemma} \label{descriptiondelta}
	Assume that $ f:=[\langle t_\beta \rangle_{\beta<\nu}]: [\{x_\alpha:A_\alpha\}_{\alpha<\lambda}] \to [\{x_\beta:B_\beta\}_{\beta<\nu}] $ and there is a display map $ p:[\{x_\beta:B_\beta\}_{\beta<\mu}] \display [\{x_\beta:B_\beta\}_{\beta<\nu}] $, then $ \delta_f^\nu= [\langle x_\alpha, t_\beta \rangle_{\substack{\alpha<\lambda \\ \nu<\beta<\mu }}]. $
\end{lemma}
\begin{proof}
	This follows by induction on $ \mu $ and the explicit construction of pullbacks from \cref{pullbacksyntactic}.
\end{proof}

In certain situations, the property above characterizes the map $ \delta_f^\nu $.

\begin{lemma}
	If $ [\{x_\beta:B_\beta\}_{\beta<\mu}] $ is an object of $ \C_T $ and $ \nu<\mu $ then $ f\in \Gamma(B_\nu^\mu) $ if and only if $ f=[\langle x_\beta, t_\gamma \rangle_{\substack{\beta<\nu<\gamma<\mu}}] $, where for all $ \nu<\gamma<\mu$, the rule $ \{x_\beta:B_\beta\}_{\beta<\nu}, \{t_{\gamma'}:B_{\gamma'}\}_{\gamma'<\gamma}\vdash t_\gamma:B_\gamma $ is a derived rule.
\end{lemma}

The next result follows from the previous lemma, and is used in \cref{definitionfunctorepsilon}.

\begin{lemma} \label{lemma3220}
	Let $ A_\lambda $, $ B_\mu $ objects of $ \catC $ and for each $ \beta<\mu$ we have maps $ r_{\beta+1} \in \Gamma( r_\beta^* \cdots r_1^*p^* B_{\beta+1} ) $ then there exists a unique sequence of maps $ \{g_\beta:A_\lambda \to B_\beta \}_{\beta<\mu} $ such that for all $ \beta<\mu $ we have $ p_\beta g_{\beta+1}= g_\beta $ and $ \delta_{g_\beta}=r_\beta $.
\end{lemma}

Some words about the previous lemma are in order. The expression $ r_\beta^* \cdots r_1^*p^* B_{\beta+1}$ can be illustrated by the first two steps:

% https://q.uiver.app/#q=WzAsMTAsWzEsMiwiMSJdLFsxLDEsIkJfMSJdLFswLDIsIkFfXFxsYW1iZGEiXSxbMCwxLCJwXipCXzEiXSxbMSwwLCJCXzIiXSxbMCwwLCJwXipCXzIiXSxbMywyLCJwXipCXzEiXSxbMiwyLCJBX1xcbGFtYmRhIl0sWzMsMSwicF4qQl8yIl0sWzIsMSwicl8xXipwXipCXzIiXSxbMSwwLCIiLDAseyJzdHlsZSI6eyJoZWFkIjp7Im5hbWUiOiJlcGkifX19XSxbMiwwLCJwIiwyXSxbMywyXSxbMywxXSxbNCwxLCIiLDAseyJzdHlsZSI6eyJoZWFkIjp7Im5hbWUiOiJlcGkifX19XSxbNSwzXSxbNSw0XSxbNSwxLCJcXGxyY29ybmVyIiwxLHsibGFiZWxfcG9zaXRpb24iOjAsInN0eWxlIjp7ImJvZHkiOnsibmFtZSI6Im5vbmUifSwiaGVhZCI6eyJuYW1lIjoibm9uZSJ9fX1dLFszLDAsIlxcbHJjb3JuZXIiLDEseyJsYWJlbF9wb3NpdGlvbiI6MCwic3R5bGUiOnsiYm9keSI6eyJuYW1lIjoibm9uZSJ9LCJoZWFkIjp7Im5hbWUiOiJub25lIn19fV0sWzIsMywicl8xIiwwLHsiY3VydmUiOi0yfV0sWzcsNiwicl8xIiwyXSxbOCw2XSxbOSw3XSxbOSw4XSxbOSw2LCJcXGxyY29ybmVyIiwxLHsibGFiZWxfcG9zaXRpb24iOjAsInN0eWxlIjp7ImJvZHkiOnsibmFtZSI6Im5vbmUifSwiaGVhZCI6eyJuYW1lIjoibm9uZSJ9fX1dLFs3LDksInJfMiIsMCx7ImN1cnZlIjotMn1dXQ==
\[\begin{tikzcd}
	{p^*B_2} & {B_2} \\
	{p^*B_1} & {B_1} & {r_1^*p^*B_2} & {p^*B_2} \\
	{A_\lambda} & 1 & {A_\lambda} & {p^*B_1}
	\arrow[two heads, from=2-2, to=3-2]
	\arrow["p"', from=3-1, to=3-2]
	\arrow[from=2-1, to=3-1]
	\arrow[from=2-1, to=2-2]
	\arrow[two heads, from=1-2, to=2-2]
	\arrow[from=1-1, to=2-1]
	\arrow[from=1-1, to=1-2]
	\arrow["\lrcorner"{description, pos=0}, draw=none, from=1-1, to=2-2]
	\arrow["\lrcorner"{description, pos=0}, draw=none, from=2-1, to=3-2]
	\arrow["{r_1}", bend left, from=3-1, to=2-1]
	\arrow["{r_1}"', from=3-3, to=3-4]
	\arrow[from=2-4, to=3-4]
	\arrow[from=2-3, to=3-3]
	\arrow[from=2-3, to=2-4]
	\arrow["\lrcorner"{description, pos=0}, draw=none, from=2-3, to=3-4]
	\arrow["{r_2}", bend left, from=3-3, to=2-3]
\end{tikzcd}\]

\subsection{The equivalence between \texorpdfstring{$\kgat$}{Lg} and \texorpdfstring{$\kcon$}{Lg}}

\subsubsection{The functor \texorpdfstring{$ \C: \kgat \to \kcon $}{Lg}}

To establish this equivalence of categories, we first define a functor $ \C: \kgat \to \kcon $ using the construction of \cref{syntacticcat}. The proof again comes from \cite[Section~2.4.1]{cartmell1978}. We register all preliminary results needed to define this functor, however again we omit the proofs since they are similar to the original ones given by Cartmell.

On objects $ \C: \kgat \to \kcon $ is defined as $ \C_T $ for $ T $ a generalized $\kappa$-algebraic theory. For a map $ [I]:T \to T' $ between theories, we need a functor $ \C(I): \C_T \to \C_{T'} $:
\begin{enumerate}
	\item On objects; $ \C(I)([\{x_\alpha:\Delta_\alpha\}_{\alpha<\lambda}]):=[\{x_\alpha:\widetilde{I}( \Delta_\alpha)\}_{\alpha<\lambda}] $,
	\item On morphisms: If $ [\langle t_\beta \rangle_{\beta<\mu}]: [\{x_\alpha:\Delta_\alpha\}_{\alpha<\lambda}] \to  [\{x_\beta:\Delta_\beta\}_{\beta<\mu}]$ then $ \C(I)([\langle t_\beta \rangle_{\beta<\mu}]):= [\langle \widetilde{I} (\langle t_\beta \rangle_{\beta<\mu})] $.
\end{enumerate}

If there is an interpretation $ J $ in the equivalence class $ [I] $, then by \cref{lemma11114} any rule $r$ of $ T $ we get  $ \widehat{I}(r)\approx \widehat{J}(r) $. Therefore, it follows that the definition of $ \C(I) $ does not depend on the representative of $ [I] $.

It remains to verify that $ \C(I) $ is indeed a contextual functor. Firstly, it is essential to verify that it is well-defined.

\begin{lemma}
	Let $ [I]: T \to T' $ be a map in $ \kgat $ then the following hold:
	\begin{enumerate}
		\item The interpretation $ I $ preserves contexts: If $ \{x_\alpha:\Delta_\alpha \}_{\alpha<\lambda} $ is a context in the theory $ T $ then $ \{x_\alpha:\widetilde{I}(\Delta_\alpha)\}_{\alpha<\lambda} $ is a context in the theory $ T' $.
		
		\item The interpretation $ I $ preserves the equivalence relation $ \approx $ between contexts: If $ \{x_\alpha:\Delta_\alpha \}_{\alpha<\lambda} $ and $ \{x_\alpha:\Omega_\alpha \}_{\alpha<\lambda} $ are contexts in the theory $ U $ with $ \{x_\alpha:\Delta_\alpha \}_{\alpha<\lambda}\approx \{x_\alpha:\Omega_\alpha \}_{\alpha<\lambda} $ then $ \{x_\alpha:\widetilde{I}(\Delta_\alpha) \}_{\alpha<\lambda}\approx \{x_\alpha:\widetilde{I}(\Omega_\alpha) \}_{\alpha<\lambda} $.
		
		\item The interpretation $ I $ preserves morphisms between contexts: If $ \langle t_\beta \rangle_{\beta<\mu}\} : \{x_\alpha:\Delta_\alpha \}_{\alpha<\lambda} \to \{x_\beta:\Omega_\beta \}_{\beta<\mu} $ is a morphism between contexts in the theory $ T $ then $ \langle \widetilde{I}( t_\beta) \rangle_{\beta<\mu}\} : \{x_\alpha:\widetilde{I}(\Delta_\alpha) \}_{\alpha<\lambda} \to \{x_\beta:\widetilde{I}(\Omega_\beta) \}_{\beta<\mu}  $ is a morphism between contexts in the theory $ T' $.
		
		\item \label{itemiv} The interpretation $ I $ preserves the equivalence relation $ \approx $ between morphisms of contexts: If $\langle s_\beta\rangle_{\beta<\mu}, \,\langle t_\beta \rangle_{\beta<\mu} : \{x_\alpha:\Delta_\alpha \}_{\alpha<\lambda} \to \{x_\beta:\Omega_\beta \}_{\beta<\mu} $ are morphisms between contexts in the theory $ T $ with $ \langle s_\beta\rangle_{\beta<\mu} \approx \langle t_\beta \rangle_{\beta<\mu} $ then $ \langle \widetilde{I}(s_\beta)\rangle_{\beta<\mu} \approx \langle \widetilde{I}(t_\beta) \rangle_{\beta<\mu} $.		
	\end{enumerate}
\end{lemma}
\begin{proof}
	The proof of each statement is consequence of \cref{lemma1111} or \cref{lemma1112}. Our enumeration of variables give us a notation simplification of the proof given by \cite{cartmell1978}.\\
	For example, to prove \ref{itemiv}; we have by assumption that $ \{x_\alpha:\Delta_\alpha\}_{\alpha<\lambda} \vdash t_\gamma \equiv_{\Omega_\gamma[t_\beta\mid x_\beta]_{\beta<\gamma} } s_\gamma $ for all $ 0<\gamma\leq \mu. $ Therefore, since the interpretation preserves this rule o$ T $ we get that $ \{x_\alpha:\widetilde{I}(\Delta_\alpha)\}_{\alpha<\lambda} \vdash \widetilde{I}(t_\gamma) \equiv_{\widetilde{I}(\Omega_\gamma)[\widetilde{I}(t_\beta)\mid x_\beta]_{\beta<\gamma} } \widetilde{I}(s_\gamma) $ for all $ 0<\gamma\leq \mu$. This exactly establishes $ \langle \widetilde{I}(s_\beta)\rangle_{\beta<\mu} \approx \langle \widetilde{I}(t_\beta) \rangle_{\beta<\mu} $.
\end{proof}

We have seen that the definition of $ \C(I) $ give us the correct objects and morphisms. Now we show that it is indeed a contextual functor.

\begin{lemma}
	Let $ I:T\to T' $ be a morphism in $ \kgat $. Then the map $ \C(I):\C_T \to \C_{T'} $ is a contextual functor.
\end{lemma}
\begin{proof}
	The map is a functor trivially. That it preserves the grading and restricts to a functor between the display subcategories $ Dis(\C_T)$ and $Dis(\C_{T'}) $ is also immediate. To prove it preserves canonical pullbacks, consider the following pullback square in the category $ \C_T $:
	\[% https://q.uiver.app/#q=WzAsNCxbMywwLCJbXFx7eF9cXGJldGE6XFxPbWVnYV9cXGJldGFcXH1fe1xcYmV0YTxcXG11K1xcdmFyZXBzaWxvbn1dIl0sWzMsMSwiW1xce3hfXFxiZXRhOlxcT21lZ2FfXFxiZXRhXFx9X3tcXGJldGE8XFxtdX1dIl0sWzAsMSwiW1xce3hfXFxhbHBoYTpcXERlbHRhX1xcYWxwaGFcXH1fe1xcYWxwaGE8XFxrYXBwYX1dIl0sWzAsMCwiW1xce3hfXFxhbHBoYTpcXERlbHRhX1xcYWxwaGEsXFwsIHhfXFxnYW1tYTpcXE9tZWdhX1xcZ2FtbWFbdF9cXGJldGFcXG1pZCB4X1xcYmV0YSBdX3tcXGJldGE8XFxtdX1cXH1fe1xcc3Vic3RhY2t7XFxhbHBoYTxcXGthcHBhLFxcXFwgXFxtdSBcXGxlcSBcXGdhbW1hPFxcbXUrXFx2YXJlcHNpbG9ufX1dIl0sWzAsMSwiW1xcbGFuZ2xlIHhfXFxiZXRhIFxccmFuZ2xlX3tcXGJldGE8XFxtdX1dIiwwLHsic3R5bGUiOnsiaGVhZCI6eyJuYW1lIjoiZXBpIn19fV0sWzIsMSwiW1xcbGFuZ2xlIHRfXFxiZXRhXFxyYW5nbGVfe1xcYmV0YTxcXG11fV0iLDJdLFszLDIsIltcXGxhbmdsZSB4X1xcYWxwaGEgXFxyYW5nbGVfe1xcYWxwaGE8XFxrYXBwYX1dIiwyXSxbMywwLCJbXFxsYW5nbGUgdF9cXGJldGEsXFwsIHhfXFxnYW1tYSBcXHJhbmdsZV97XFxzdWJzdGFja3tcXGJldGE8XFxtdSwgXFxcXCBcXG11XFxsZXFcXGdhbW1hPFxcbXUrXFx2YXJlcHNpbG9ufX1dIl1d
	\begin{tikzcd}
		{[\{x_\alpha:\Delta_\alpha,\, x_\gamma:\Omega_\gamma[t_\beta\mid x_\beta ]_{\beta<\mu}\}_{\substack{\alpha<\kappa,\\ \mu \leq \gamma<\mu+\varepsilon}}]} &&& {[\{x_\beta:\Omega_\beta\}_{\beta<\mu+\varepsilon}]} \\
		{[\{x_\alpha:\Delta_\alpha\}_{\alpha<\kappa}]} &&& {[\{x_\beta:\Omega_\beta\}_{\beta<\mu}]}
		\arrow["{[\langle x_\beta \rangle_{\beta<\mu}]}", two heads, from=1-4, to=2-4]
		\arrow["{[\langle t_\beta\rangle_{\beta<\mu}]}"', from=2-1, to=2-4]
		\arrow["{[\langle x_\alpha \rangle_{\alpha<\kappa}]}"', from=1-1, to=2-1]
		\arrow["{[\langle t_\beta,\, x_\gamma \rangle_{\substack{\beta<\mu, \\ \mu\leq\gamma<\mu+\varepsilon}}]}", from=1-1, to=1-4]
	\end{tikzcd}
	\]
	Then a straightforward computation, using the definition of $ \C(I)$, shows that this is sent to a pullback square in the category $ \C_{T'}.$
\end{proof}

\begin{corollary} \label{C-functor-kcon-kgat:corollary}
	There is a functor $ \C:\kgat\to \kcon $.
\end{corollary}

\subsubsection{The functor \texorpdfstring{$ U : \kcon \to \kgat $}{Lg}} \label{functor-contextual-to-gat}

We now turn to construct a functor that associates a generalized $\kappa$-algebraic theory $ U(\catC) $ to each $ \kappa $-contextual category $ \catC $. This is part of \cite[Section 2.4]{cartmell1978}. We will use the notation introduced in \cref{convention1}. This means we identify each object by its height, say $ B_\lambda $, and write display maps as $ p_\alpha:B_\lambda \display B_\alpha $ if $ \lambda>0 $ and $ \alpha<\lambda $. If $ \alpha=0 $ then $ B_0=1 $ the terminal object. A morphism $ f:A_\lambda \to B_\mu $ is trivial when $ B_\mu $ is trivial, \ie $ \mu=0 $.

\begin{definition} \label{definitiongatfromcat}
	We define $ U(\catC) \in \kappa $-GAT as:
	\begin{enumerate}
		\item For each non-trivial object $ B_{\mu} $ with $ \mu=\lambda+1 $, there is a type symbol $ \overline{B_\mu} $ with the introductory rule:
		$ \{x_\beta: \overline{B_\beta}\}_{\beta<\mu} \vdash \overline{B_\mu}(x_\beta)_{\beta<\mu} \, \type$. The notation emphasizes the fact that $ \overline{B_\mu} $ depends on the indicated variables.
		
		\item If $ f: A_\lambda \to B_{\mu} $ is morphism of $ \catC $  with $ \mu=\nu+1 $, we get an operator symbol $ \overline{f} $. It has the introductory rule:
		
		\begin{itemize}
			\item If $ f: A_\lambda \to B_{\mu+1}$, let $ \rho_\mu:B_{\mu+1} \display B_\mu $ be the display map. Then the operator symbol has introductory rule:
			\[ \{x_\alpha:\overline{A}_\alpha\}_{\alpha<\lambda} \vdash \overline{f}(x_\alpha)_{\alpha<\lambda} :  \overline{(\rho_\mu f)^* B_{\mu+1}}(x_\alpha)_{\alpha<\lambda}. \]
		\end{itemize}
		This does not clash with the notation from the previous point since it always refer to an object of $ \catC $ and in this case refers to a map.
		
	\end{enumerate}
	
	Subject to the following axioms in $ U(\catC) $:
	
	\begin{enumerate}
		\item Let $ A_\lambda,\, B_\mu,\, C_{\nu+1} $ be objects of $ \catC $ and maps $ f: A_\lambda \to B_\mu, \, g: B_\mu \to C_{\nu+1} $:
		\[
		\{x_\alpha:\overline{A}_\alpha\}_{\alpha<\lambda} \vdash \overline{g f}(x_\alpha)_{\alpha<\lambda} \equiv_{\overline{(p_\nu g f)^*C_{\nu+1}}(x_\alpha)_{\alpha<\lambda}} \overline{g} \big( \overline{p_\beta f}(x_\alpha)_{\alpha<\lambda} \big)_{\beta<\mu}.
		\]
		
		\item Let $ B_\mu $ be a non-trivial object of $ \catC $. For each $ \delta<\mu $ we have
		\[
		\{x_\beta:\overline{B}_\beta\}_{\beta<\mu} \vdash \overline{p_\delta}(x_\beta)_{\beta<\mu}\equiv_{\overline{B}_\delta(x_\beta)_{\beta<\delta}} x_\delta.
		\]
		
		\item Let $ A_\lambda,\, B_{\mu+1}$ objects of $ \catC $ and a map $ f: A_\lambda \to B_{\mu}$ then
		\[
		\{x_\alpha:\overline{A}_\alpha\}_{\alpha<\lambda} \vdash \overline{f^*B_{\mu+1}}(x_\alpha)_{\alpha<\lambda} \equiv \overline{B_{\mu+1}}\big( \overline{p_\beta f}(x_\alpha)_{\alpha<\lambda}\big)_{\beta<\mu}
		\]
		and
		\[
		\{x_\alpha:\overline{A}_\alpha,\,x_\delta:\overline{f^*B_{\mu+1}}(x_\alpha)_{\alpha<\lambda} \}_{\alpha<\lambda} \vdash \overline{q(f,B_{\mu+1})}(x_\alpha,x_\delta)_{\alpha<\lambda} \equiv_{\overline{f^*B_\mu}(x_\alpha)_{\alpha<\lambda}} x_\delta.
		\]
	\end{enumerate}
\end{definition}

\begin{observation}
	It is immediate to observe that $ U(\catC) $ as defined is a $ \kappa $-pretheory. We have type and operator symbols introduced by the type and type element judgments respectively. Note that the list of axioms we provided are well-formed rules. This is because the premise of each axiom is by definition a context.
\end{observation}

\begin{remark}
	If $ f: A_\lambda \to B_\mu $ is a map in $ \catC $, where $ \mu $ is a limit ordinal, \ie $ B_\mu $ is a limit object, then we get a family of maps $ \{f_\nu: A_\lambda \to B_\nu \}_{\nu<\mu} $. Therefore, the associated operator $ \overline{f} $ is uniquely determined by the family of operators $ \overline{f_\nu}, $ for which in this case we can assume that $ \nu $ is a successor ordinal. 
\end{remark}

If $ F:\catC \to \catD $ is a functor between $ \kappa $-contextual categories, then we need an interpretation $ U(F):U(\catC) \to U(\catD) $;
\begin{enumerate}
	\item For an object $ A_\lambda $, the interpretation is defined as \[ U(F)(\overline{A_\lambda}):= \overline{FA_\lambda}(x_\alpha)_{\alpha<\lambda} .\]
	\item For a morphism $ f:A_\lambda \to B_{\mu+1} $, the operator $ \overline{f} $ is interpreted as \[ U(F)(\overline{f}):= \overline{F(f)}(x_\alpha)_{\alpha<\lambda}. \]
        \end{enumerate}

The next step is to prove that this is indeed a map between the generalized $\kappa$-algebraic theories, this is done in \cite[pp 2.29]{cartmell1978}. For this, it is enough to show that rules and axioms of $ U(\catC) $ are sent to rules of $ U(\catD) $. The functoriality of $ U : \kappa\text{-CON} \to \kappa\text{-GAT} $ is also immediate from its definition. This is tested on each type and operator symbol. It is then enough to take the equivalence class $ [U(F)] $.

\subsubsection{The natural isomorphism \texorpdfstring{$ U\circ \C \cong Id_{\kappa\text{-}GAT} $}{Lg}}

For each $T \in \kappa $-GAT we want to define an interpretation $ [\varphi_T]: T \to U(\C_T)$, we do this by defining a preinterpretation $\varphi_T: Exp(T) \to Exp(U(\C_T)) $:

\begin{enumerate}
	\item If $ \Delta $ is a type symbol of $ T $ with introduction rule
	\[ \{x_\alpha: \Delta_\beta\}_{\beta<\mu} \vdash \Delta(x_\beta)_{\beta<\mu} \,\type\]
	then \[ \varphi_T(\Delta):= \overline{[\{x_\beta: \Delta_\beta,\, x_\delta:\Delta(x_\beta)_{\beta<\mu}\}_{\beta<\mu}]}(x_\beta)_{\beta<\mu} \]
	\item If $ f $ is an operator symbol with introductory rule
	\[ \{x_\alpha: \Delta_\beta\}_{\beta<\mu} \vdash f(x_\beta)_{
		\beta<\mu}:\Delta, \]
	then
	\[
	\varphi_T(f):=  \overline{ [ \langle x_\beta, f(x_\beta)_{\beta<\mu}\rangle_{\beta<\mu} ] }(x_\beta)_{\beta<\mu},
	\]
	where $  \langle x_\beta, f(x_\beta)_{\beta<\mu}\rangle_{\beta<\mu} $ is the morphism $ \{x_\alpha: \Delta_\beta\}_{\beta<\mu} \to \{x_\alpha: \Delta_\beta, x_\delta:\Delta \}_{\beta<\mu}$.
\end{enumerate}

We proceed to verify that as defined $ \varphi_T:T\to U(\C_T) $ is an interpretation as defined. This is a crucial point in the proof, so we spell out some details in \cref{isinterpreation}. The results below are the technical steps towards it.

\begin{lemma}\label{lemma1231}
	If $ \catC $ is a contextual category, objects $ A_\lambda,\, B_\mu$ and $ f:  A_\lambda \to B_\mu$ is map with $ \mu=\nu+1 $ (in particular it is non-trivial), then the rule
	\[
	\{x_\alpha:\overline{A}_\alpha(x_\gamma)_{\gamma<\alpha} \}_{\alpha<\lambda} \vdash \overline{f}(x_\alpha)_{\alpha<\lambda} : \overline{B_\mu}\big( \overline{p_\beta\circ f}(x_\alpha)_{\alpha< \lambda }\big)_{\beta<\mu}
	\]
	is a derived rule of $ U(\catC) $.
\end{lemma}
\begin{proof}
	We have the axiom
	\[
	\{x_\alpha:\overline{A}_\alpha\}_{\alpha<\lambda} \vdash \overline{f^*B_\mu}(x_\alpha)_{\alpha<\lambda} \equiv \overline{B_\mu}\big( \overline{p_\beta\circ f}(x_\alpha)_{\alpha<\lambda}\big)_{\beta<\mu}
	\]
	for $ U(\catC) $ and the derivation rule for $ \kappa $-GAT
	\[\inferrule{\Gamma \vdash A_1\equiv A_2 \\  t:A_1}{\Gamma \vdash t:A_2}.\]
	These put together give us the result.
\end{proof}

\begin{lemma}
	Let $ \catC $ a $ \kappa $-contextual category, objects $ \{A_\alpha\}_{\alpha<\lambda},$ $ \{B_\beta\}_{\beta<\mu+1},$ $ \{C_\gamma\}_{\gamma < \varepsilon} $ and a commutative diagram
	% https://q.uiver.app/#q=WzAsNCxbMCwwLCJDX1xcbnUiXSxbMCwxLCJBX1xcbGFtYmRhIl0sWzEsMSwiQl9cXGJldGEiXSxbMSwwLCJCX1xcbXUiXSxbMSwyLCJmX1xcYmV0YSIsMl0sWzMsMiwicF9cXGJldGEiLDAseyJzdHlsZSI6eyJoZWFkIjp7Im5hbWUiOiJlcGkifX19XSxbMCwzLCJsX1xcYmV0YSJdLFswLDEsImtfXFxiZXRhIiwyXV0=
	\[\begin{tikzcd}
		{C_\varepsilon} & {B_{\mu+1}} \\
		{A_\lambda} & {B_\mu.}
		\arrow["{f}"', from=2-1, to=2-2]
		\arrow["{p}", two heads, from=1-2, to=2-2]
		\arrow["{l}", from=1-1, to=1-2]
		\arrow["{k}"', from=1-1, to=2-1]
	\end{tikzcd}\]
	If $ h:C_\varepsilon \to f^*B_{\mu+1} $ is the unique map given by the pullback, then the rule
	\[
	\{x_\gamma:\overline{C_\gamma}(x_\delta)_{\delta<\gamma}\}_{\gamma<\varepsilon} \vdash \overline{h}(x_\gamma)_{\gamma<\varepsilon} \equiv_{\overline{(f k)^*B_{\mu+1}}(x_\gamma)_{\gamma<\varepsilon}} \overline{l}(x_\gamma)_{\gamma<\varepsilon}
	\]
	is a derived rule of $ U(\catC). $
\end{lemma}
\begin{proof}
	The proof is the same as \cite[Lemma 2~pp. 2.32]{cartmell1978} using \cref{lemma1231}.
\end{proof}

\begin{lemma}\label{lemma2232}
	Let $ \catC $ a $ \kappa $-contextual category, objects $ \{A_\alpha\}_{\alpha<\lambda},$ $ \{B_\beta\}_{\beta<\mu},$ $ \{C_\gamma\}_{\gamma < \varepsilon} $ and for $ 0<\nu<\mu $ a commutative diagram
	% https://q.uiver.app/#q=WzAsNCxbMCwwLCJDX1xcbnUiXSxbMCwxLCJBX1xcbGFtYmRhIl0sWzEsMSwiQl9cXGJldGEiXSxbMSwwLCJCX1xcbXUiXSxbMSwyLCJmX1xcYmV0YSIsMl0sWzMsMiwicF9cXGJldGEiLDAseyJzdHlsZSI6eyJoZWFkIjp7Im5hbWUiOiJlcGkifX19XSxbMCwzLCJsX1xcYmV0YSJdLFswLDEsImtfXFxiZXRhIiwyXV0=
	\[\begin{tikzcd}
		{C_\varepsilon} & {B_{\mu}} \\
		{A_\lambda} & {B_\nu.}
		\arrow["{f}"', from=2-1, to=2-2]
		\arrow["{p_\nu}", two heads, from=1-2, to=2-2]
		\arrow["{l_\nu}", from=1-1, to=1-2]
		\arrow["{k_\nu}"', from=1-1, to=2-1]
	\end{tikzcd}\]
	If $ h_\nu:C_\varepsilon \to f^*B_{\mu} $ is the unique map given by the pullback, then the rule
	\[
	\{x_\gamma:\overline{C_\gamma}(x_\delta)_{\delta<\gamma}\}_{\gamma<\varepsilon} \vdash \overline{h_\nu}(x_\gamma)_{\gamma<\varepsilon} \equiv_{\overline{(f k_\nu)^*B_{\mu}}(x_\gamma)_{\gamma<\varepsilon}} \overline{l_\nu}(x_\gamma)_{\gamma<\varepsilon}
	\]
	is a derived rule of $ U(\catC). $
\end{lemma}
\begin{proof}
	This by induction on the height of $ p_\nu $. When it is a successor ordinal, this is the previous lemma. When it is a limit ordinal $ B_\mu $ is a limit object, therefore the result reduces to the inductive hypothesis, which is the successor case again.
\end{proof}

Recall from \cref{interlude} we defined the set of maps $ \Gamma(B) $. It follows from the previous result that

\begin{corollary} \label{corollary3234}
	If $ \catC $ is a $ \kappa $-contextual category and $ f:A_\lambda \to B_\mu $ is a map in $ \catC $, then for all $ \nu<\mu $
	\[
	\{x_\alpha:A_\alpha(x_\delta)_{\delta<\alpha}\}_{\alpha<\lambda} \vdash \overline{\delta_f^\nu}(x_\alpha)_{\alpha<\lambda}\equiv \overline{f}(x_\alpha)_{\alpha<\lambda}.
	\]
	is a derived rule of $ U(\catC). $
\end{corollary}

If we specialize \cref{corollary3234} to the syntactic $ \kappa $-contextual category of a generalized $\kappa$-algebraic theory $ T $, then

\begin{corollary} \label{corollary4234}
	Assume that $ \{x_\beta:B_\beta\}_{\beta<\mu} $ is a context, $ \nu<\mu $ and \[f_\nu:=[\langle t_\beta \rangle_{\beta<\nu}]:[\{x_\alpha:A_\alpha\}_{\alpha<\lambda}] \to  [\{x_\beta:B_\beta\}_{\beta<\nu}] \] a map in $ \C_T $ then
	\[
	\{x_\alpha: \overline{A_\alpha}(x_\gamma)_{\gamma<\alpha} \}_{\alpha<\lambda} \vdash \overline{[\langle x_\alpha, t_\varepsilon \rangle_{\substack{\alpha<\lambda \\ \nu \leq \varepsilon <\mu}}]} \equiv \overline{[\langle t_\beta, t_\varepsilon \rangle_{\beta<\nu\leq\varepsilon<\mu}]}
	\]
	is a derived rule of $ U(\C_T) $.
\end{corollary}
\begin{proof}
	This follows from \cref{corollary3234} and the explicit description of $ \delta_{f_\nu}^\nu $ given in \cref{descriptiondelta}.
\end{proof}

\begin{lemma} \label{lemma5234}
	If $ A_\lambda, B_\mu $ are objects and $ f_\nu:A_\lambda \to B_\nu$, with $ \nu<\mu $, is a map in a $ \kappa $-contextual category $ \catC $, then:
	\begin{enumerate}
		\item The rule \[
		\{x_\alpha:\overline{A_\alpha}(x_\delta)_{\delta<\alpha}\}_{\alpha<\lambda} \vdash \overline{f_\nu^*B_\mu}(x_\alpha)_{\alpha<\lambda} \equiv \overline{B} (\delta_{(p_\gamma f)}^\gamma(x_\alpha)_{\alpha<\lambda} )_{\gamma<\nu}
		\]
		is a derived rule of $ U(\catC) $.
		\item If $ g:\Gamma(B_\nu^\mu) $ then the rule \[
		\{x_\alpha:\overline{A_\alpha}(x_\delta)_{\delta<\alpha}\}_{\alpha<\lambda} \vdash
		\overline{\delta_{gf}^\nu}(x_\alpha)_{\alpha<\lambda} \equiv \overline{\delta_g^\nu} ( \overline{\delta_{p_\gamma f}^\gamma}(x_\alpha)_{\alpha<\lambda} )_{\gamma<\nu}
		\]
		is a derived rule of $ U(\catC) $.
	\end{enumerate}
\end{lemma}

\begin{corollary} \label{corollary6235}
	If $ T $ is a generalized $\kappa$-algebraic theory, $ \{x_\beta:B_\beta\}_{\beta<\mu} $ is a context, $ \nu<\mu $ and \[f_\nu:=[\langle t_\beta \rangle_{\beta<\nu}]:[\{x_\alpha:A_\alpha\}_{\alpha<\lambda}] \to  [\{x_\beta:B_\beta\}_{\beta<\nu}] \] is a map in $ \C_T $
	then;
	\begin{enumerate}
		\item \[
		\inferrule{\{x_\alpha:\overline{A_\alpha}(x_\delta)_{\delta<\alpha}\}_{\alpha<\lambda} }{\overline{[\{x_\alpha, x_\gamma:B_\gamma[t_\delta|x_\delta]_{\delta<\gamma} \}_{\substack{\alpha<\lambda \\ \nu \leq \gamma<\mu}} ]}(x_\alpha)_{\alpha<\lambda} \equiv \overline{[\{x_\beta:B_\beta\}_{\beta<\nu}]}(\overline{g_\beta}(x_\alpha)_{\alpha<\lambda} )_{\beta<\nu}} 
		\]
		where for each $ \beta<\nu $ the map $ g_\beta:=[\langle x_\alpha,t_\beta \rangle_{\alpha<\lambda} ]. $
		\item If for all $ \gamma $, with $ \nu<\gamma<\mu$, the rule 
		\[ \{x_\beta:B_\beta\}_{\beta<\nu}, \{t_{\gamma'}:B_{\gamma'}\}_{\gamma'<\gamma}\vdash t_\gamma:B_\gamma
		\]
		is a derived rule then
		\[
		\{x_\alpha:\overline{A_\alpha}(x_\delta)_{\delta<\alpha}\}_{\alpha<\lambda} \vdash \overline{[\langle x_\alpha, t_\gamma[t_{\gamma'}\mid x_{\gamma'}]_{\gamma'<\gamma} \rangle_{\substack{\alpha<\lambda \\ \nu<\gamma<\mu }}]}  \equiv \overline{h}(\overline{g_\beta}(x_\alpha)_{\alpha<\lambda})_{\beta<\nu}
		\]
		where $ g_\beta $ is defined as in the previous point and $ h:=[\langle x_\beta, t_\gamma \rangle_{\substack{\beta<\nu \\ \nu<\gamma<\mu}} ]. $
	\end{enumerate}
\end{corollary}
\begin{proof}
	This is a direct application of \cref{lemma5234}. We remark that the assumption of point (2) simply gives us an element of $ \Gamma (B_\nu^\mu) $ and the map on the left depends on variables that, according to our convention, we leave implicit.
\end{proof}

The following lemma is key to prove that we have an interpretation $ \varphi_T: T \to U(\C_T) $, the results above are used to prove:

\begin{lemma} \label{lemma7236}
	If $ T $ is a generalized $\kappa$-algebraic theory then:
	\begin{enumerate}
		\item If $ \{x_\alpha:\Delta_\alpha\}_{\alpha<\lambda} \vdash \Delta \,\type $ is a type judgment of $ T $, then the rule \[ \{x_\alpha: \overline{A_\alpha}(x_\delta)_{\delta<\alpha}\}_{\alpha<\lambda} \vdash \overline{A}(x_\alpha)_{\alpha<\lambda+1} \equiv \widetilde{\varphi_T}(\Delta) \] is a derived rule of $ U(\C_T) $
		where $ A:=\{x_\alpha:\Delta_\alpha\}_{\alpha<\lambda+1} $ and $ A_\alpha:=\{x_\delta:\Delta_\delta\}_{\delta\leq\alpha} $.
		\item If $ \{x_\alpha:\Delta_\alpha\}_{\alpha<\lambda} \vdash t:\Delta$ is a type element judgment of $ T $, then the rule \[ \{x_\alpha: \overline{A_\alpha}(x_\delta)_{\delta<\alpha}\}_{\alpha<\lambda} \vdash \overline{\langle x_\alpha, t \rangle_{\substack{\alpha<\lambda}}}(x_\alpha)_{\alpha<\lambda+1} \equiv_{\overline{A}(x_\alpha)_{\alpha<\lambda}} \widetilde{\varphi_T}(t) \] is a derived rule of $ U(\C_T) $.
	\end{enumerate}
\end{lemma}
\begin{proof}
	The proof is by induction on the derivations, by showing that rule derivation preserves the properties above.
\end{proof}

The important result of this section is the following.

\begin{corollary} \label{isinterpreation}
	For every generalized $\kappa$-algebraic theory $ T $, the map $ \varphi_T: T \to U(\C_T) $ is an interpretation.
\end{corollary}
\begin{proof}
	We see that the function $ \widehat{\varphi_T}:Rul(T) \to Rul(U(\C_T)) $ is well-defined. We start with a rule $ \mathcal{J} $ of $ T $ and show that $ \widehat{\varphi_T}(\mathcal{J}) $ is a rule of $ U(\C_T) $
	\begin{enumerate}
		\item Type judgment: Assume that $ \mathcal{J}:=\{x_\alpha:\Delta_\alpha\}_{\alpha<\lambda} \vdash \Delta\, \type $ is a rule of $ T $, from \cref{interpretation} it follows that \[\widehat{\varphi_T}(\mathcal{J})= \{x_\alpha:\widetilde{\varphi}(\Delta_\alpha)\}_{\alpha<\lambda} \vdash \widetilde{\varphi_T}(\Delta )\,\type.\]
		From \cref{lemma7236} we have for any $ \gamma<\lambda+1 $, the rule
		\[ \{x_\alpha: \overline{\Delta_\alpha}(x_\delta)_{\delta<\alpha}\}_{\alpha<\lambda} \vdash \overline{A_{\gamma+1}}(x_\alpha)_{\alpha<\gamma+1} \equiv \widetilde{\varphi_T}(\Delta_\gamma) \]
		is a derived rule of $ U(\C_T) $. Thus, the following is also a derived rule
		\[ \{x_\alpha: \widetilde{\varphi_T}(\Delta_\alpha)(x_\delta)_{\delta<\alpha}\}_{\alpha<\lambda} \vdash \overline{A_{\gamma+1}}(x_\alpha)_{\alpha<\lambda+1} \equiv \widetilde{\varphi_T}(\Delta).\]
		Then it must be the case that
		$ \{x_\alpha:\widetilde{\varphi}(\Delta_\alpha)\}_{\alpha<\lambda} \vdash \widetilde{\varphi_T}(\Delta )\,\type $ is a rule of $ U(\C_T) .$
		
		\item Element judgment: $\Gamma \vdash t:\Delta.$ This very similar to the previous rule.
		\item Type equality judgment: $\Gamma \vdash \Delta\equiv \Delta'.$ Also follows from \cref{lemma7236}.
		\item Term equality judgment: $\Gamma \vdash t\equiv_\Delta t'. $ The same argument works.
	\end{enumerate}
\end{proof}

\begin{corollary}
	For every generalized $\kappa$-algebraic theory $ T $, the map $ [\varphi_T]: T \to U(\C_T) $ is morphism in the category $ \kgat $.
\end{corollary}

Next, we will now show that $ [\varphi_{-}]:Id_{\kgat} \Rightarrow U\circ \C $ is a natural transformation.

\begin{lemma}
	Let $ T, \, T' $ two generalized $\kappa$-algebraic theories and $ I: T \to T' $ an interpretation between them. Then, we have  a commutative diagram
	% https://q.uiver.app/#q=WzAsNCxbMCwwLCJUIl0sWzAsMSwiVCciXSxbMSwwLCJVKFxcQ19UKSJdLFsxLDEsIlUoXFxDX3tUJ30pIl0sWzAsMSwiW0ldIiwyXSxbMCwyLCJbXFx2YXJwaGlfVF0iXSxbMiwzLCJVKFxcQyhJKSkiXSxbMSwzLCJbXFx2YXJwaGlfe1QnfV0iLDJdXQ==
	\[\begin{tikzcd}
		T & {U(\C_T)} \\
		{T'} & {U(\C_{T'})}.
		\arrow["{[I]}"', from=1-1, to=2-1]
		\arrow["{[\varphi_T]}", from=1-1, to=1-2]
		\arrow["{U(\C(I))}", from=1-2, to=2-2]
		\arrow["{[\varphi_{T'}]}"', from=2-1, to=2-2]
	\end{tikzcd}\]
\end{lemma}
\begin{proof}
	We use \cref{corollary2114}. It will therefore be enough to test the commutativity of the diagram on type element judgments. Let $ \{x_\alpha:\Delta_\alpha\}_{\alpha<\lambda} \vdash t:\Delta_\lambda $ a type element judgment of $ T $. For any $ \alpha\leq\lambda $ we denote $ A_\alpha:=[\{x_\delta:\Delta_\delta\}_{\delta\leq\alpha}] $. It follows from \cref{lemma7236} that
	\[\widehat{\varphi_T}\left(\inferrule{\{x_\alpha:\Delta_\alpha\}_{\alpha<\lambda}}{t:\Delta_\lambda}\right) \approx \inferrule{\{ x_\alpha:\overline{A_\alpha} \}_{\alpha<\lambda}}{\overline{ [\langle x_\alpha,t \rangle_{\alpha<\lambda}] }: \overline{A_\lambda}(x_\alpha)_{\alpha<\lambda}}. \]
	We conclude that
	\[U(\C(I))\left(\widehat{\varphi_T}\left(\inferrule{\{x_\alpha:\Delta_\alpha\}_{\alpha<\lambda}}{t:\Delta_\lambda}\right)\right) \approx \inferrule{\{ x_\alpha:\overline{\C(I)(A_\alpha)} \}_{\alpha<\lambda}}{\overline{ \C(I)([\langle x_\alpha,t \rangle_{\alpha<\lambda}]) }: \overline{\C(I)(A_\lambda)}(x_\alpha)_{\alpha<\lambda}}. \]
	Looking at the other composition: we get
	\[
	\widehat{I}\left(\inferrule{\{x_\alpha:\Delta_\alpha\}_{\alpha<\lambda}}{t:\Delta_\lambda}\right) = \inferrule{\{x_\alpha:\widetilde{I}(\Delta_\alpha)\}_{\alpha<\lambda}}{\widetilde{I}(t):\widetilde{I}(\Delta_\lambda)}.
	\]
	A second use of \cref{lemma7236} give us that
	\[
	\widehat{\varphi_{T'}}\left(\widehat{I}\left(\inferrule{\{x_\alpha:\Delta_\alpha\}_{\alpha<\lambda}}{t:\Delta_\lambda}\right)\right) \approx\inferrule{\{ x_\alpha:\overline{B_\alpha} \}_{\alpha<\lambda}}{\overline{ [\langle x_\alpha,\widetilde{I}(t) \rangle_{\alpha<\lambda}] }: \overline{B_\lambda}(x_\alpha)_{\alpha<\lambda}},
	\]
	where for $ \alpha\leq\lambda $, $ B_\alpha:=[\{x_\delta:\widetilde{I}(\Delta_\delta)\}_{\delta\leq\alpha}] $. However, by definition we have $ \C(I)(A_\alpha)=B_\alpha $ for $ \alpha\leq\lambda $. This completes our verification.
\end{proof}

Remains to show that $ [\varphi_T] $ is an isomorphism, and natural in $ T $. We proceed to give an inverse $ \psi_T:U(\C_T) \to T $. Recall that a type symbol of $ U(\C_T) $ is of the form $ \overline{A_\lambda}= \overline{[\{x_\alpha:\Delta_\alpha\}_{\alpha<\lambda}]} $. If $ \lambda=\nu+1 $, then by choosing a representative of this equivalence class of the context we can define $ \psi_T(\overline{A_\lambda}):= \Delta_\nu $.

If $ \lambda $ is a limit ordinal, once we chose a representative, $ \Delta_\lambda=\{ x_\alpha:\Delta_\alpha\}_{\alpha<\lambda} $. Then we know that $ [ \Delta_\lambda]=\lim_{\alpha<\lambda} [\Delta_\alpha]  $ in $ \C_T $, and this limit is unique. In this case, the value of $ \psi_T $ is determined by non-limit ordinals $ \alpha<\lambda $, which are $ \psi_T(\overline{\Delta_\alpha})=\Delta_\alpha $. Therefore, we define $ \psi_T(\overline{[ \Delta_\lambda]}):=\Delta_\lambda $ for some choice of a representative of the equivalence class. However, note that the successor case determinate the limit case.

Operator symbols of $ U(\C_T) $ come from morphisms of $ \C_T $. Therefore, for a morphism $ \overline{f}:=[\langle t_\beta \rangle_{\beta<\mu}]:[\{x_\alpha:\Delta_\alpha\}_{\alpha<\lambda}] \to [\{x_\beta:\Omega_\beta\}_{\beta<\mu}] $ in order to define $ \psi_T $ on the associated operator, it is enough to assume that $ \mu $ is a successor ordinal. Firstly, we need to make choices for the contexts and morphism. However, the definition does not depend on these choices because of (1) from \cref{lemma4}. This allows us to define $ \psi_T $ as
\[
\psi_T(\overline{f}):=t_\mu
\]
where $ t_\mu : \Omega_\mu[t_\beta|x_\beta]_{\beta<\mu}$.

\begin{lemma} \label{lemma10244}
	The function $ \psi_T $ is an interpretation from $ U(\C_T) \to T $.
\end{lemma}
\begin{proof}
	We need to check that rules and axioms are preserved by $ \psi_T $.
	It will be enough to deal with the case where $ \lambda=\nu+1 $. Suppose that $ \overline{A_\lambda} $ has
	\[
	\inferrule{\{x_\alpha:\overline{A_\alpha}(x_\delta)_{\delta<\alpha} \}_{\alpha<\nu}}{\overline{A_\nu}(x_\alpha)_{\alpha<\nu} \, \type }
	\]
	
	Furthermore, we assume that $ \{x_\alpha:\Delta_\alpha \}_{\alpha<\lambda} $ is such that $ A_\lambda=[\{x_\alpha:\Delta_\alpha \}_{\alpha<\lambda} ] $. By definition,
	\[
	\widehat{\psi_T}\left(\inferrule{\{x_\alpha:\overline{A_\alpha}(x_\delta)_{\delta<\alpha} \}_{\alpha<\nu}}{\overline{A_\lambda}(x_\alpha)_{\alpha<\lambda} \, \type }\right)=\inferrule{\{x_\alpha:\Delta_\alpha\}_{\alpha<\nu} }{\Delta_\nu \, \type}.
	\]
	This is obviously a derived rule of $ T $. Preservation of the rule for operator symbols is straightforward.
\end{proof}

\begin{lemma}
	For any generalized $\kappa$-algebraic theory $ T $ we have $ \psi_T\circ \varphi_T\approx Id_T $.
\end{lemma}
\begin{proof}
	From \cref{corollary2114} it is enough to verify the statement on type element judgments. Let $ \{x_\alpha:\Delta_\alpha\}_{\alpha<\lambda} \vdash t:\Delta_\lambda $ a type element judgment. For any $ \alpha\leq\lambda $ we denote $ A_\alpha:=[\{x_\delta:\Delta_\delta\}_{\delta\leq\alpha}] $. It follows from \cref{lemma7236} that
	\[\widehat{\varphi_T}\left(\inferrule{\{x_\alpha:\Delta_\alpha\}_{\alpha<\lambda}}{t:\Delta_\lambda}\right) \approx \inferrule{\{ x_\alpha:\overline{A_\alpha} \}_{\alpha<\lambda}}{\overline{ [\langle x_\alpha,t \rangle_{\alpha<\lambda}] }: \overline{A_\lambda}(x_\alpha)_{\alpha<\lambda}}. \]
	Hence \[\widehat{\psi_T}\left(\widehat{\varphi_T}\left(\inferrule{\{x_\alpha:\Delta_\alpha\}_{\alpha<\lambda}}{t:\Delta_\lambda}\right)\right) \approx \widehat{\psi_T}\left(\inferrule{\{ x_\alpha:\overline{A_\alpha} \}_{\alpha<\lambda}}{\overline{ [\langle x_\alpha,t \rangle_{\alpha<\lambda} ] }: \overline{A_\lambda}(x_\alpha)_{\alpha<\lambda}}\right) = \inferrule{\{x_\alpha:\Delta_\alpha\}_{\alpha<\lambda}}{t:\Delta_\lambda}. \]
\end{proof}

\begin{lemma}
	For any generalized $\kappa$-algebraic theory $ T $ we have $ \psi_T\circ \varphi_T\approx Id_T $.
\end{lemma}
\begin{proof}
	The proof is similar to the previous lemma. All the definitions and technical results have been established, especially \cref{lemma7236}.
\end{proof}

\begin{corollary}
	There is a natural isomorphism $Id_{\kgat} \Rightarrow U\circ \C $.
\end{corollary}
\begin{proof}
	We have constructed $ [\varphi_{-}]:Id_{\kgat} \Rightarrow U\circ \C $.
      \end{proof}

\subsubsection{The natural isomorphism \texorpdfstring{$ \C \circ U \cong Id_{\kappa\text{-}CON} $}{Lg}}

In this section we aim to construct a natural isomorphism $ \eta: Id_{\kcon} \Rightarrow \C \circ U $. Let $ \catC $ be a $ \kappa $-contextual category. For this, we first construct a $ \kappa $-contextual functor $ \eta_\catC:\catC \to \C_{U(\catC)} $. Recall that if $ A_\lambda $ is an object in $ \catC $ then for any $ \alpha\leq \lambda $, we denote $ p_\alpha: A_\lambda \display A_\alpha $ as the canonical display map that exists. Then we can make the following definition:

\begin{enumerate}
	\item For $ \eta_\catC(1):=1 $.
	\item If $ A_\mu $ is an object with $ \mu=\lambda+1 $, then \[ \eta_\catC(A_\mu):=[ \{x_\alpha: \overline{A_\alpha}(x_\delta)_{\delta<\alpha} \}_{\alpha\leq \mu} ].\]
	\item For an object $ A_\lambda $, we define $ \eta_\catC(p_0):= \eta_\catC(p)_0 $ where $ \eta_\catC(p)_0: \eta_\catC(A) \display 1 $.
	\item If $ A_\lambda, B_\mu $ are non-trivial objects, with $ \mu $ a successor ordinal, and $ f: A_\lambda \to B_\mu $ is a morphism in $ \catC $, then
	\[ \eta_\catC(f):= [\langle \overline{p_\beta f}(x_\alpha)_{\alpha<\lambda} \rangle_{\beta\leq\mu}].\]
\end{enumerate}

We observe that if $ \mu $ is a limit ordinal, then any map $ f:A_\lambda \to B_\mu $ is determined by a family of maps $ \{f_\nu:A_\lambda \to B_\nu\}_{\nu<\mu} $. Thus, in order to define $ \eta $ on such map, it is enough to do it on ordinals $ \nu<\mu $ which we can assume to be successor ordinals. The map $ \eta(f) $ is the map induced by the family of maps $ \{\eta(f_\nu):\eta(A_\lambda) \to \eta(B_\nu)\}_{\nu<\mu} $. In conclusion, we simply need to prove properties of $ \eta $ for successor ordinals; the property for limit ordinals follows using the universal property of the limit object.

\begin{lemma} \label{etaisfunctor}
	For any $ \catC $, $ \eta_\catC:\catC \to \C_{U(\catC)} $ is a $\kappa$-contextual functor.
\end{lemma}
\begin{proof}
	First we verify that it is a functor. Since for any $ \alpha<\lambda $ we have $ \overline{p_\alpha}(x_\alpha)_{\alpha<\lambda}=x_\alpha $, then it is immediate to see that $ \eta_\catC $ preserves the identities.\\
	Assume we have non-trivial morphisms $ f: A_\lambda \to B_\mu $ and $ g:B_\mu \to C_\nu $, then
	\[
	\eta_\catC(gf)= [\langle \overline{p_\gamma g f}(x_\alpha)_{\alpha<\lambda} \rangle_{\beta\leq\nu}]
	\]
	From the first axiom in \cref{definitiongatfromcat} of $ U(\catC) $, it follows that the above must be $ \eta_\catC(g)\eta_\catC(f) $ whenever $ \mu $ and $ \nu $ are successor ordinals. When we have limits, it follows by the universal property. \\
	Now we must verify that it preserves display maps and canonical pullbacks. Both statements are direct consequences of the definitions. Furthermore, the proof from \cite{cartmell1978} works without mayor changes.\\
	For the preservation of pullbacks: We let $ f:A_\lambda \to B_{\mu+1} $ then
	\begin{align*}
		\eta_\catC(f^*B) &= [\langle x_\alpha: \overline{A_\delta}( x_\gamma )_{\gamma<\alpha}, \, x_\epsilon : \overline{f^*B_{\mu+1}}(x_\alpha)_{\alpha<\lambda} \rangle_{\alpha<\lambda} ] \\
		&= [\langle x_\alpha: \overline{A_\delta}( x_\gamma )_{\gamma<\alpha}, \, x_\epsilon : \overline{B_{\mu+1}}\big( \overline{p_\beta f}(x_\alpha)_{\alpha<\lambda}\big)_{\beta<\mu} \rangle_{\alpha<\lambda} ] \\
		&= [ \langle \overline{p_\beta f}(x_\alpha)_{\alpha<\lambda} \rangle_{\beta\leq \mu} ]^*[\langle x_\beta:\overline{B_\beta}(x_\gamma)_{\gamma<\beta} \rangle_{\beta\leq\mu}] \\
		&= \eta_\catC(f)^*\eta_\catC(B).
	\end{align*}
	For a display map of $ p_\nu: B_\mu \display B_\nu $ with height a successor ordinal, the same argument shows that the pullback along $ f_\nu:A_\lambda \to B_\nu  $ is preserved. When the height is a limit ordinal, we combine the previous case and the fact that in any $ \kappa $-contextual category canonical pullbacks are unique.
\end{proof}

\begin{lemma}
	Let $ \catC, \, \catC' $ be $ \kappa $-contextual categories and a contextual functor $ F:\catC \to \catC' $. Then the following diagram is commutative:
	% https://q.uiver.app/#q=WzAsNCxbMCwwLCJcXGNhdEMiXSxbMCwxLCJcXGNhdEMnIl0sWzEsMCwiXFxDX3tVKFxcY2F0Qyl9Il0sWzEsMSwiXFxDX3tVKFxcY2F0QycpfSJdLFsyLDMsIlxcQyhVKEYpKSJdLFswLDIsIlxcZXRhX3tcXGNhdEN9Il0sWzAsMSwiRiIsMl0sWzEsMywiXFxldGFfe1xcY2F0Qyd9IiwyXV0=
	\[\begin{tikzcd}
		\catC & {\C_{U(\catC)}} \\
		{\catC'} & {\C_{U(\catC')}}.
		\arrow["{\C(U(F))}", from=1-2, to=2-2]
		\arrow["{\eta_{\catC}}", from=1-1, to=1-2]
		\arrow["F"', from=1-1, to=2-1]
		\arrow["{\eta_{\catC'}}"', from=2-1, to=2-2]
	\end{tikzcd}\]	
\end{lemma}
\begin{proof}
	If $ f: A_\lambda \to B_\mu $ is a map in $ \catC $ then
	\begin{align*}
		\C(U(F))(\eta_{\catC}(f)) &= \C(U(F))([\langle \overline{p_\beta f}(x_\alpha)_{\alpha<\lambda} \rangle_{\beta\leq \mu} ]) \\
		&= [\langle \overline{F(p_\beta f)}(x_\alpha)_{\alpha<\lambda} \rangle_{\beta\leq \mu} ] \\
		&= [\langle \overline{p_\beta F(f)}(x_\alpha)_{\alpha<\lambda} \rangle_{\beta\leq \mu} ] \\
		&= \eta_{\catC'}(Ff).
	\end{align*}
\end{proof}

\begin{corollary}
	There is a natural transformation $ Id_{\kcon} \Rightarrow \C \circ U $.
\end{corollary}

It remains to show that this natural transformation is an isomorphism. For each $ \kappa $-contextual category $ \catC $ we construct a $ \kappa $-contextual functor
\[ \xi_{\catC}: \C_{U(\catC)} \to \catC \]
which is a two-sided inverse to $ \eta_{\catC} $.
From \cref{derivationlemma} we see that:
\begin{enumerate}
	\item Every derived type judgment of $ U(\catC) $ is of the form
	$$ \{x_\beta:\Omega_\beta\}_{\beta<\mu} \vdash \overline{A_\lambda} (t_\alpha)_{\alpha<\lambda} \,\type $$
	for some object $ A_\lambda $ of $ \catC $ where for $ \alpha\leq\lambda $ the rule $$ \{ x_\beta: \Omega_\beta \}_{\beta<\mu} \vdash t_\alpha:\overline{A_\alpha}[t_\delta\mid x_\delta ]_{\delta<\alpha}$$
	is a derived rule of $ U(\catC). $
	
	\item Every type element judgment of $U(\catC)$ is of the form
	\[ \{x_\beta:\Omega_\beta\}_{\beta<\mu} \vdash x_\beta: \Omega_\beta \]
	for some $ \beta<\mu $, or is of the form
	\[ \{x_\beta:\Omega_\beta\}_{\beta<\mu} \vdash \overline{f}(t_\alpha)_{\alpha<\lambda} : \Omega \] for some map $ f: A_\lambda \to B_\mu $ of $ \catC $ such that for each $ \alpha<\lambda $ the rules
	\[ \{x_\beta:\Omega_\beta\}_{\beta<\mu} \vdash t_\alpha: \overline{A_\alpha}[t_\delta\mid x_\delta]_{\delta<\alpha}\]
	and
	\[ \{x_\beta:\Omega_\beta\}_{\beta<\mu} \vdash \overline{B_\mu}(t_\beta)_{\beta<\mu} \equiv \Omega \]
	are derived rules of $ U(\catC) $.
\end{enumerate}

We may assume that $ \mu=\nu+1 $, the limit case will follow by induction. Let  $ \mathcal{R}_\catC $ be the set of type and element type judgments of $ U(\catC) $. Next, we define $ \mathcal{J}: \mathcal{R}_\catC \to  \catC $ inductively. First we get maps:

\begin{enumerate}
		\item A rule $ r_{\Omega_\mu}:= \{x_\beta:\Omega_\beta\}_{\beta<\mu} \vdash \Omega_\mu $ is sent an object $ \mathcal{J}(r_{\Omega_\mu}) \in\catC $.
	
	\item For any $ \alpha<\lambda $ the judgment $ r_{t_{\alpha}}:= \{x_\beta:\Omega_\beta\}_{\beta<\mu} \vdash t_{\alpha} : \overline{A_{\alpha}}[t_\delta|x_\delta]_{\delta< \alpha} $ is sent to a map $ \mathcal{J}(r_{t_\alpha}) $.
	
\end{enumerate}

The we can make the following definitions:

\begin{enumerate}
	\item $ \mathcal{J}(r_{A_\mu}):= ( \mathcal{J}(t_\alpha)_{\alpha<\lambda} )^*A_\mu $, \\
	where $ \mathcal{J}(t_\alpha)_{\alpha<\lambda} $ denotes the pullbacks as in \cref{lemma3220}.
	\item $ \mathcal{J}( \{x_\beta:\Omega_\beta\}_{\beta<\mu} \vdash \overline{f}(t_\alpha)_{\alpha<\lambda} : \Omega):= ( \mathcal{J}(t_\alpha)_{\alpha<\lambda} )^*\delta_f^\nu. $	
	\item $ \mathcal{J}(\{x_\beta:\Omega_\beta\}_{\beta<\mu} \vdash x_\beta: \Omega_\beta):= \delta_{p_\beta}^\beta$ where $ p_\beta:\mathcal{J}(r_{\Omega_\mu}) \to \mathcal{J}(r_{\Omega_\beta}) $.
\end{enumerate}

The burden of the proof falls into showing that the function $ \mathcal{J} $ is well-defined. The proof is by induction on the derived rules of $ U(\catC) $. We will focus on writing down the inductive hypothesis $ H $ as in \cite{cartmell1978} for this induction.
\begin{itemize}
	\item For rules $ r_{\Omega_\mu} $ of the form $ \{x_\beta:\Omega_\beta\}_{\beta<\mu} \vdash \Omega_\mu \, \type $ then $ H(r_{\Omega_\mu}) $ is either:
	
	\begin{enumerate}
		\item If the premise of $ r_{\Omega_\mu} $ is a non-empty context then $ H(r_{\Omega_\beta}) $ for all $ \beta<\mu $.
		\item If $ r_{\Omega_\mu} $ is the rule $ \vdash \Delta \, \type $ then $ ht(\mathcal{J}(r_{\Omega_\mu}))=1 $. Otherwise, for all $ \beta<\mu $ we have $ ht(\mathcal{J}(r_{\Omega_\beta}))<ht(\mathcal{J}(r_{\Omega_\mu})) $.
		\item For a map $ \langle t_\beta \rangle_{\beta<\mu}: \{x_\alpha:\Delta_\alpha\}_{\alpha<\lambda} \to \{x_\beta:\Omega_\beta\}_{\beta<\mu} $. If for each $ \beta+1<\mu $ we have $ \mathcal{J}(r_{t_\beta+1}) \in \Gamma (\mathcal{J}(r_{\Omega_{\beta+1}[t_\gamma|x_\gamma]_{\gamma\leq\beta}})) $ where
		$ r_{\Omega_{\beta+1}[t_\gamma|x_\gamma]_{\gamma\leq\beta}} $ is the rule $ \{x_\alpha : \Delta_\alpha \}_{\alpha<\lambda} \vdash \Omega_{\beta+1}[t_\gamma|x_\gamma]_{\gamma\leq\beta} \,\type $ then
		\[ \mathcal{J}(r_{\Omega_{\mu}[t_\beta|x_\beta]_{\beta<\mu}}) = (\mathcal{J}(t_\beta)_{\beta<\mu})^* \mathcal{J}(r_{\Omega_\mu}) \]
	\end{enumerate}
	
	\item For rules $ r_{t_\mu} $ of the form $ \{x_\beta:\Omega_\beta\}_{\beta<\mu} \vdash t_\mu:\Omega_\mu $ then $ H(r_{t_\mu}) $ is either:
	\begin{enumerate}
		
		\item $ H(r_{\Omega_\mu}) $.
		
		\item $ \mathcal{J}(r_{t_\mu}) \in \Gamma(\mathcal{J}(r_{\Omega_\mu})). $
		
		\item For a map $ \langle t_\beta \rangle_{\beta<\mu}: \{x_\alpha:\Delta_\alpha\}_{\alpha<\lambda} \to \{x_\beta:\Omega_\beta\}_{\beta<\mu} $. If for each $ \beta+1<\mu $ we have $ \mathcal{J}(r_{t_\beta+1}) \in \Gamma (\mathcal{J}(r_{\Omega_{\beta+1}[t_\gamma|x_\gamma]_{\gamma\leq\beta}})) $ then
		\[ \mathcal{J}(r_{t_{\mu}[t_\beta|x_\beta]_{\beta<\mu}}) = (\mathcal{J}(t_\beta)_{\beta<\mu})^* \mathcal{J}(r_{t_\mu}) \]
		where $ r_{t_{\mu}[t_\beta|x_\beta]_{\beta<\mu}} $ is the rule
		$ \{ x_\alpha:\Delta_\alpha\}_{\alpha<\lambda} \vdash t_\mu[t_\beta|x_\beta]_{\beta<\mu} : \Omega_\mu[t_\beta|x_\beta]_{\beta<\mu}. $
	\end{enumerate}
	
	\item For rules $ r_{\equiv} $ or of the form $ \{x_\alpha: \Delta_\alpha\}_{\alpha<\lambda} \vdash \Delta\equiv \Delta' $, the hypothesis $ H(r_\equiv) $ is either:
	\begin{enumerate}
		\item $ H(r_{\Delta'}) $ and $ \mathcal{J}(r_{\Delta})=\mathcal{J}(r_{\Delta'}) $.
		\item $ H(r_{\Delta}) $ and $ \mathcal{J}(r_{\Delta})=\mathcal{J}(r_{\Delta'}) $.
	\end{enumerate}
	
	\item For rules $ r_{\epsilon} $ or of the form $ \{x_\alpha: \Delta_\alpha\}_{\alpha<\lambda} \vdash t \equiv_\Delta t' $, the hypothesis $ H(r_\epsilon) $ is either:
	\begin{enumerate}
		\item $ H(r_t) $ and $ \mathcal{J}(r_{t})=\mathcal{J}(r_{t'}) $.
		
		\item $ H(r_{t'}) $ and $ \mathcal{J}(r_{t})=\mathcal{J}(r_{t'}) $.
	\end{enumerate}
	
\end{itemize}

\begin{lemma} \label{lemma11256}
	Let $ \{x_\beta:\Omega_\beta\}_{\beta<\mu} \vdash \Omega $ a rule such that $ H $ is satisfied. If $ \langle t_\beta \rangle_{\beta<\mu}: \{x_\alpha:\Delta_\alpha\}_{\alpha<\lambda} \to \{x_\beta:\Omega_\beta\}_{\beta<\mu} $ is a map such that $ H(r_{t_\beta}) $ for all $ \beta<\mu $ then $  H(\{x_\beta:\Omega_\beta\}_{\beta<\mu} \vdash \Omega[t_\beta|x_\beta]_{\beta<\mu}) $
\end{lemma}
\begin{proof}
	By induction on $ \mu $ and treating all the different cases for $ H $. The proof in \cite[Lemma 11 pp.2.56]{cartmell1978} works here too.
\end{proof}

\begin{lemma} \label{lemma12263}
	\begin{enumerate}
		\item For any object $ A_\lambda \in \catC $, we have:
		\begin{enumerate}
			\item  $ A\lambda=\mathcal{J}( \{x_\alpha: \overline{A_\alpha}(x_\gamma)_{\gamma<\alpha} \}_{\alpha<\lambda} \vdash \overline{A_\lambda}(x_\alpha)_{\alpha<\lambda} \, \type ) $ .
			
			\item For all $ \alpha<\lambda $, $ \delta_{p_\alpha^\lambda}= \mathcal{J}( \{x_\alpha: \overline{A_\alpha}(x_\gamma)_{\gamma<\alpha} \}_{\alpha<\lambda} \vdash x_\alpha:\overline{A_\alpha}(x_\gamma)_{\gamma<\alpha} )$ where $ p_\alpha^\lambda:A_\lambda \display A_\alpha $.
		\end{enumerate}
		
		\item For any non-trivial object $ A_\lambda $ and $ f:A_\lambda \to B_{\mu+1}  $,
		$ \delta_f=\mathcal{J}( \{x_\alpha: \overline{A_\alpha}(x_\gamma)_{\gamma<\alpha} \}_{\alpha<\lambda} \vdash \overline{f}(x_\alpha)_{\alpha<\lambda} \overline{(p_\mu f)^*B}(x_\alpha)_{\alpha<\lambda} ) $ where $ p_\mu:B_{\mu+1} \display B_\mu. $
	\end{enumerate}
\end{lemma}
\begin{proof}
	This is \cite[Lemma 12 pp.263]{cartmell1978}.
\end{proof}

\begin{lemma}\label{lemma265}
	Every derived rule of $ U(\catC) $ satisfies the hypothesis $ H $.
\end{lemma}
\begin{proof}
	This is by induction on derived rules of $ U(\catC) $. Indeed, \cite[Lemma pp.2.65]{cartmell1978} shows that every derivation from \cref{derivedrules} preserves $ H $.
\end{proof}

\begin{corollary}
	\begin{enumerate}
		\item For any type symbol $ \overline{A_\lambda} $ of the theory $ U(\catC) $ we have $ H(\{ x_\alpha: \overline{A_\alpha}(x_\gamma)_{\gamma<\alpha} \}_{\alpha<\lambda} \vdash \overline{A_\lambda}(x_\alpha)_{\alpha<\lambda} \, \type). $
		
		\item For every operator symbol $ \overline{f} $ in $ U(\catC) $ where $ f:A_\lambda \to B_{\mu+1} $ we have $ H(\{x_\alpha: \overline{A_\alpha}(x_\gamma)_{\gamma<\alpha}\}_{\alpha<\lambda} \vdash \overline{f}(x_\alpha)_{\alpha<\lambda} \overline{(p_\mu f)^*B}(x_\alpha)_{\alpha<\lambda}). $
	\end{enumerate}
\end{corollary}

The foremost important result, which summarizes the above, is:
\begin{corollary} \label{corollary14272}
	\begin{enumerate}
		\item If $ \{x_\alpha: \Delta_\alpha\}_{\alpha<\lambda} $ is a context of the theory then for any $ \alpha<\delta < \lambda $ we have $ ht(r_{\Delta_\alpha})<ht(r_{\Delta_\beta}) $.
		
		\item If there is a map $ \langle t_\beta \rangle_{\beta<\mu}: \{x_\alpha:\Delta_\alpha\}_{\alpha<\lambda} \to \{x_\beta:\Omega_\beta\}_{\beta<\mu} $ then for each $ \beta<\mu $ we have $ \mathcal{J}(r_{t_\beta}) \in \Gamma (\mathcal{J}(r_{\Omega_{\beta}[t_\gamma|x_\gamma]_{\gamma<\beta}})) $ where
		$ r_{\Omega_{\beta}[t_\gamma|x_\gamma]_{\gamma<\beta}} $ is the rule $ \{x_\alpha : \Delta_\alpha \}_{\alpha<\lambda} \vdash \Omega_{\beta}[t_\gamma|x_\gamma]_{\gamma<\beta} \,\type $.
		
		\item If $ \{ x_\alpha:\Delta_\alpha\}_{\alpha<\lambda}\equiv \{ x_\alpha:\Delta_\alpha'\}_{\alpha<\lambda} $ then $ \mathcal{J}(r_{\Delta_\lambda})=\mathcal{J}(r_{\Delta_\lambda'}) $.
		
		\item If $ \langle t_\alpha \rangle_{\alpha<\lambda} \equiv \langle t_\alpha' \rangle_{\alpha<\lambda} $ then for each $ \beta<\mu $, $ \mathcal{J}(r_{t_\beta})= \mathcal{J}(r_{t_\beta'}) $.
	\end{enumerate}
\end{corollary}

We are almost ready to define a contextual functor $ \xi_\catC: \catC_{U(\catC)} \to \catC $. We only need the next:
\begin{observation} \label{definitionfunctorepsilon}
	Let $ \langle t_\beta \rangle_{\beta<\mu}: \{x_\alpha:\Delta_\alpha\}_{\alpha<\lambda} \to \{x_\beta:\Omega_\beta\}_{\beta<\mu} $ be a map, then there are maps $ \{g_\beta: \mathcal{J}(r_{\Delta_\lambda}) \to \mathcal{J}(r_{\Omega_\beta}) \}_{\beta<\mu} $ with $ \delta_{g_\beta}=\mathcal{J}(r_{t_beta}) $ and $ p g_{\beta+1}=g_\beta $. This is a consequence of \cref{corollary14272} and \cref{lemma3220}. Therefore, there exists a unique $ g: \mathcal{J}(r_{\Delta_\lambda}) \to \mathcal{J}(r_{\Omega_\mu}) $ such that for all $ \beta<\mu $ we have $ \delta_{p g}=\mathcal{J}(r_{t_\beta}) $ where $ p: \mathcal{J}(r_{\Delta_\lambda}) \to \mathcal{J}(r_{\Omega_\beta}). $
\end{observation}

\begin{definition}
	We define a function
	\[
	\xi_\catC: \catC_{U(\catC)} \to \catC
	\]
	by:
	\begin{enumerate}
		\item For an object $ [\{x_\alpha:\Delta_\alpha\}_{\alpha<\lambda}] \in 	   \catC_{U(\catC)}  $,
		\[\xi([\{x_\alpha:\Delta_\alpha\}_{\alpha<\lambda}]):=\mathcal{J}(r_{\Delta_\lambda}).\]
		
		\item For a morphism $ [\langle t_\beta \rangle_{\beta<\mu}]: [\{x_\alpha:\Delta_\alpha\}_{\alpha<\lambda}] \to [\{x_\beta:\Omega_\beta\}_{\beta<\mu}] $
		\[ \xi([\langle t_\beta \rangle_{\beta<\mu}]):=g,\]
		where $ g: \mathcal{J}(r_{\Delta_\lambda}) \to \mathcal{J}(r_{\Omega_\mu}) $ is the unique map from \cref{definitionfunctorepsilon}.
	\end{enumerate}
\end{definition}

\begin{lemma} \label{lemma15274}
	\begin{enumerate}
		\item If $ \{x_\alpha:\Delta_\alpha\}_{\alpha<\lambda} \vdash \Delta_\lambda \, \type $ is a derived rule of $ U(\catC) $ then for all $ \alpha\leq \lambda $, $ \{x_\gamma:\Delta_\gamma\}_{\gamma<\lambda} \vdash \Delta_\alpha \equiv \mathcal{J}(r_{\Delta_{\alpha}})(x_\gamma)_{\gamma<\alpha} $ is a derived rule of $ U(\catC) $.
		
		\item If $ \{x_\alpha:\Delta_\alpha\}_{\alpha<\lambda} \vdash t_\lambda: \Delta_\lambda $ is a derived rule of $ U(\catC) $ then $ \{x_\gamma:\Delta_\gamma\}_{\gamma<\lambda} \vdash t \equiv \mathcal{J}(r_{t_\lambda})(x_\alpha)_{\alpha<\lambda} $ is a derived rule of $ U(\catC) $.
	\end{enumerate}
\end{lemma}
\begin{proof}
	See \cite[Lemma 15 pp. 2.74]{cartmell1978}.
\end{proof}

\begin{corollary}
	As functions, we have that $ \eta_\catC \xi_\catC= id_{\catC_{U(\catC)}} $ and $ \xi_\catC \eta_\catC=Id_\catC $ 
\end{corollary}

The results needed for this have been introduced throughout the section. Using that we have a bijection and that $ \eta_{\catC} $ is already a functor, it follows:

\begin{corollary}
	The function $ \xi_\catC: \catC_{U(\catC)} \to \catC $ is a contextual functor.
\end{corollary}

The main result that is of our interest is:

\begin{theorem} \label{synactic-forgetful-identiy}
	There is a natural isomorphism $ \C_{\_ } \circ U \cong Id_{\kcon} $.
\end{theorem}

Finally,
\begin{corollary}
	The categories $ \kcon $ of $ \kappa $-contextual categories and $ \kgat $ of $ \kappa $-algebraic theories are equivalent. 
\end{corollary}

\begin{remark} \label{map-syntactic-to-concat}
        In \cref{isinterpreation} we defined a map $T \to U(\C_T)$ simply by interpreting the axioms of the theory $T$ \ie by defining an interpretation sending axioms of the theory $T$ to derived rules of $U(\C_T)$. In the same way, given a $\kappa$-contextual category $\catC$, we can define a map $T \to U(\catC)$ by sending axioms of $T$ to derived rules in $U(\catC)$. It follows that we have a $\kappa$-contextual functor $\C_T \to \catC$.
      \end{remark}

\subsection{Models of a generalized Cartmell theory}
\label{appendix-b4-models}

In this section, we aim to make precise what we mean by a model of a generalized $\kappa$-algebraic theory $T$. Furthermore, if we were to prove a theorem in the same spirit of Lawvere's Functorial semantics, we would prove that there is an equivalence of categories \[T\text{-}\mathbf{Alg_\kappa} \cong [\C_T, \cfam_\kappa]\] where $T\text{-}\mathbf{Alg_\kappa}$ is the category of models of the theory $T$ and $\cfam_\kappa$ is a certain $\kappa$-contextual category of ``sets'' or rather families of sets, and $[\C_T, \cfam_\kappa]$ is the category of $\kappa$-contextual functors between these two $\kappa$-contextual categories. Since in the paper we do not use the category $T\text{-}\mathbf{Alg_\kappa}$, we are simply interested in constructing the (large) $\kappa$-contextual category $\cfam_\kappa$. Then we can define a model of the theory $T$ simply as a $\kappa$-contextual functor $M: \C_T \to \cfam_\kappa$. Once more, this is a straightforward generalization of Cartmell's construction of the contextual category $\cfam$ \cite[Section 2.2~pag. 2.9]{cartmell1978}.

\bigskip

We fix a set of sets $\catU$, which will play the role of the set of all sets. Ideally, $\catU$ is a Grothendieck universe and in some places we will assume this, though this is technically not needed for the definition to make sense.

An object $X$ of $\cfam_\kappa$ of height $\alpha$ is a functor $X:(\alpha+1)^\op \to \catU$, such that:
\begin{itemize}
\item $X_0 = 1$,
\item  For each $\beta < \alpha$ there is map $f:X_\beta \to \catU$ such that \[X_{\beta+1} = \coprod_{x \in X_\beta} f(x)\] where the map $X_{\beta+1} \to X_\beta$ is the canonical map $\coprod_{x \in X_\beta} f(x) \to X_\beta$,
\item For each limit ordinal $\beta$, $X_\beta = \lim_{\gamma<\beta} X_\gamma$.
\end{itemize}
 Note that in the definitions above, we do mean equality of sets. Alternatively, we can give a more categorical definition by asking for some compatible isomorphisms and identify objects that have isomorphism compatible to the map to $\catU$, or we can give an inductive presentation of the notion, but this makes the exposition slightly more complicated.

 Morphisms in $\cfam_\kappa$ between two objects $X$ and $Y$ of height $\alpha$ and $\beta$, respectively, are just functions $X_\alpha \to Y_\beta$. We call $X_\alpha$ and the underlying set of $X$: by construction, this underlying set gives us a functor $\cfam_\kappa \to \set$, which is an equivalence of categories (or at last a fully faithful functor depending on $\catU$). Display maps are functions from $X$ to the restriction of $X$ to an ordinal $\beta \leqslant \alpha$ given by the obvious map $X_\alpha \to X_\beta$.

Given a map $v:X_\alpha \to Y_\beta$ and a display map $Y_{\beta+\lambda} \to Y_\beta$, we can extend $X$ from $X_\alpha$ to $X_{\alpha+\lambda}$ with pullback squares
\[\begin{tikzcd}
X_{\alpha+\lambda} \ar[d] \ar[r] & Y_{\beta+\lambda} \ar[d] \\
  X_\alpha \ar[r,"v"] & Y_\beta
\end{tikzcd}
\]
where at each successor stage, we condition that the composite function $X_{\alpha+\lambda} \to Y_{\beta + \lambda} \to \catU$ to define $X_{\alpha+\lambda+1}$, and at a limit stage we just define $X$ to be the limit. 

One can easily check that $\cfam_\kappa$ and the datum specified above, constitute a $\kappa$-contextual category.

% In order to construct pullbacks of d thatisplay maps, it will be enough to give the pullback for projection operators. For $\kappa$-trees $A_0 \triangleleft A_1 \triangleleft \dots \triangleleft A_{\alpha}$, $B_0 \triangleleft B_1 \triangleleft \dots \triangleleft B_{\beta} \triangleleft B_{\beta+1}$ and an operator $\langle f_\sigma \rangle_{\sigma<\beta+1}: A_\alpha \to B_{\beta+1}$ we define the $\kappa$-tree $A_0 \triangleleft A_1 \triangleleft \dots A_\alpha \triangleleft f^*B$ where $f^*B$ is the $\alpha$-indexed family given by $B_{\beta+1}(f_{\sigma}(a_\delta)_{\delta<\alpha})_{\sigma<\beta+1}$ for each sequence $(a_\delta)_{\delta<\alpha}$ with $a_{\delta+1} \in A_{\delta+1}(a_{\delta'})_{\delta'<\delta+1}$. This produces a diagram
% \[\begin{tikzcd}
%     A_0 \triangleleft A_1 \triangleleft \dots A_\alpha \triangleleft f^*B \ar[r] \ar[d] & B_0 \triangleleft B_1 \triangleleft \dots \triangleleft B_{\beta} \triangleleft B_{\beta+1} \ar[d] \\
%   A_0 \triangleleft A_1 \triangleleft \dots \triangleleft A_{\alpha} \ar[r] & B_0 \triangleleft B_1 \triangleleft \dots \triangleleft B_{\beta}
% \end{tikzcd}
% \]
% One verifies that this is indeed a pullback square in a similar manner as for $\C_T$. We have set up the data for $\cfam_\kappa$ so that it is a $\kappa$-contextual category: the objects are $\kappa$-trees and the morphisms are the operators between $\kappa$-trees. Also, we have specified the display maps and how to construct pullbacks. As we anticipated at the beginning of the section, we make the following:

\begin{definition}
  Let $T$ be a generalized $\kappa$-algebraic theory. A \emph{model} for $T$ is a $\kappa$-contextual functor $M: \C_T \to \cfam_\kappa$.
\end{definition}

\begin{remark} \label{rk:models}
  Our definition of model might seem ad hoc; however, thanks to \cref{map-syntactic-to-concat}, in order to specify such a model we just need to specify how the axioms of $T$ are interpreted in $\cfam_\kappa$, and this corresponds to the naive notion of model---a structure where types are interpreted as sets, terms as functions and all equation axioms are valid. In other words, a model for a theory $T$ is really an interpretation of its axioms into the contextual category $\cfam_\kappa$.
\end{remark}

Recall that a context $\Gamma \in \C_T$ has an associated length or height. If $\Gamma$ is a context of height $\alpha$, then we extend it by adding a fresh variable to obtain a context of height $\alpha+1$. Moreover, we saw that a context whose height is a limit ordinal is obtained as a limit of generalized display maps. Throughout \cref{fol-gat}, and particularly in \cref{cstr:validity_of_formula_syntactic}, we use the notion of model of a generalized $\kappa$-algebraic theory. We take the time explain the notation used there.

\begin{remark} \label{rk:notation-models}
  Recall that we have an ``underlying set'' functor $\cfam_\kappa \to \set$. So given any model $X:\C_T \to \cfam_\kappa$, we get a composite functor $\C_T \to \cfam_\kappa \to \set$, so that each model of $\C_T$ provides a functor from $\C_T$ to $\set$. We will denote also this functor $X$, so that given a model $X$ and a context $\Gamma$, we can form the set $X(\Gamma)$, which is just $X(\Gamma)_\alpha$ where $\alpha$ is the height of the context $\Gamma$.
\end{remark}

\begin{remark}\label{rk:ModelsOfT_vs_ModelsOfClans}
  One could also define models more naively as functors $\C_T \to \set$ that preserve the pullbacks of display maps, the terminal object and limits of $\kappa$-small tower of display maps (in the usual up-to-isomorphisms sense).  We call this alternative notion of models the \emph{models of of the underlying $\kappa$-clan of $\C_T$}. There is an obvious forgetful functor from the category of models of $T$ to the category of models of the underlying $\kappa$-clan of $\C_T$, using \cref{rk:notation-models}. This functor $\cfam_\kappa \to \set$ is fully faithful by definition of morphisms in $\cfam_\kappa$, and this allows us to show the forgetful functor from models of $T$ to models of $\C_T$ as a $\kappa$-clan is also fully faithful.

If the theory $T$ has no type equality axiom, then it is also easy to show using \cref{rk:models} that this forgetful functor is essentially surjective, \ie that every model of the underlying $\kappa$-clan of $\C_T$ is isomorphic to a model of $T$. But if $T$ has type equality axioms this is no longer always possible, see \cref{rk:Models_of_clam_mismatch} below.
\end{remark}

\begin{construction}\label{cstr:representable_models} If $\Gamma \in \C_T$ is a context of $T$, then, assuming $\catU$ is a universe and the theory is (locally) $\catU$-small, the corresponding representable functor $\C_T \to \set$ can be promoted to a model $\Gamma^*$ of $T$. Indeed, to any context $\Delta = (x_i : X_i)_{i <\alpha}$, we can associate the tower of sets $\Hom(\Gamma,\Delta_\gamma)$ where $\Delta_\gamma = (x_i : X_i)_{i <\gamma}$, is the subcontext of $\Gamma$ containing the first $\gamma$-variables. Given any morphism $f:\Gamma \to \Delta_\gamma$, a lift of this as a morphism $\Gamma \to \Delta_{\gamma+1}$ is the same as a term $\Gamma \vdash t : f^*(X_\gamma)$. We can therefore iteratively replace the set $\Hom(\Gamma,\Delta_\gamma)$ such that these identifications become equalities. One can then check that this does provide a morphism of contextual categories. Note that, as morphisms of models are just natural transformations, the Yoneda Lemma applies here and for any model $M$ of $T$ we have that $\Hom(\Gamma^*,M) \simeq M(\Gamma)$.

  This defines a functor from $\C_T^\op$ to the category of models of $T$.
\end{construction}

\begin{remark} \label{rk:Models_of_clam_mismatch}
  In \cite{frey2025duality}, J.~Frey has given a characterization of the categories (with their weak factorization systems as discussed in \cref{sec:models_of_clans}) that arise as categories of models of an $\omega$-clan.

  Consider the theory $T$ with two type axioms:
  \[ \vdash X \, \type \qquad x : X \vdash O(x) \, \type \]
  One term axiom
  \[x:X \vdash s(x) :X\]
  and one type equality axiom
  \[ x: X \vdash O(x) = O(s(x))\]
  Models of $T$ are given by a set $X$, with a function $s:X \to X$ together with a collection of set indexed by the quotient $X/s$. It is then possible to prove ( we omit the details here) that:

  \begin{itemize}
  \item The category of models of $T$, equiped with its weak factorization system as defined in \cref{sec:models_of_clans}, does not satisfy J.~Frey's characterization, hence is not the category of model of a clan.
  \item The category of models of the underlying clan of $\C_T$ is equivalent to the category of models of the theory $T'$, similar to $T$ but where the type equality axiom is replaced by the existence of a bijection between $O(x)$ and $O(s(x))$.
  \end{itemize}

\end{remark}

\subsection{Coclans and contextual categories}\label{appendix-c}

In this section, we prove that every $ \kappa $-contextual category can be obtained by strictification of a $ \kappa $-clan. Clans were introduced in \cite{joyal2017clans}, a related definition appears in \cite{henry20weak} under the name category with fibrations.

\begin{definition}
	We say that a category $ \catC $ is a $ \kappa $-\emph{coclan} if it has a collection of maps $ \cof(\catC) $ satisfying the following conditions:
	\begin{enumerate}
		\item $ \catC $ has initial object $ 0 $.
		\item For any $ X \in \catC $, the map $ 0 \to X$ is an element in $ \cof(\catC)$.
		\item Any isomorphism is an element of $ \cof(\catC) $.
		\item $ \cof(\catC) $ is closed under compositions.
		\item $ \cof(\catC) $ is closed under pushouts: If $ f:A \to C $ is a morphism in $ \catC $ and $ A \to B \in \cof(\catC)$, then the map $ C \to C \coprod_A B$ is an element in $\cof(\catC) $.
		\item $ \cof(\catC) $ is closed under transfinite compositions: for any $ \lambda<\kappa $ and any $ \lambda $-diagram of maps in $ \cof(\catC) $
		% https://q.uiver.app/#q=WzAsNCxbMCwwLCJBXzEiXSxbMSwwLCJBXzIiXSxbMiwwLCJBXzMiXSxbMywwLCJcXGNkb3RzIl0sWzAsMV0sWzEsMl0sWzIsM11d
		\[\begin{tikzcd}
			{A_0} & {A_1} & {A_2} & \cdots
			\arrow[from=1-1, to=1-2]
			\arrow[from=1-2, to=1-3]
			\arrow[from=1-3, to=1-4]
		\end{tikzcd}\]
		$ \Colim_\lambda A_\alpha $ exists and the map $ A_0 \to \Colim_\lambda A_\alpha $ belongs to $ \cof(\catC) $. 
		
	\end{enumerate}
\end{definition}

As is usual, maps in $ \cof(\catC) $ are called \emph{cofibrations} and they are indicated by arrows $ `` \rightarrowtail " $.

Dually, a category $ \catC $ is $ \kappa $-\emph{clan} if $ \catC^{op} $ is a $ \kappa $-coclan. The distinguished maps are called \emph{fibrations} and they are denoted by $ \fib(\catC) $. The fibrations are indicated by arrows $`` \display"$. When working with $\kappa$-clans we keep the terminology ``transfinite compositions" from $\kappa$-coclans as there is no risk of confusion.

\begin{observation} \label{contextual:clan}
  The $\kappa$-contextual category $\C_T$ associated to a generalized $\kappa$-algebraic theory $T$ has a natural $\kappa$-clan structure. Indeed, we can take $\fib(\C_T)$ as the set of display maps. All the axioms are easily verified. Moreover, this is true for any $\kappa$-contextual category not only for $\C_T$.
\end{observation}

Recall that a \emph{comprehension category} consists of a category $ \catC $, a fibration $ p : \catE \to \catC $ and a functor $ F:\catE \to \catC^{\rightarrow} $ such that:
\begin{enumerate}
	\item $ \partial_0 F=p$.
	\item If $ f $ is a cartesian arrow in $ \catE $, then $ Ff $ is a pullback in $ \catC $; equivalently, $ Ff $ is a cartesian arrow with respect to the codomain functor $ \partial_0:\catC^{\rightarrow} \to \catC$.
\end{enumerate}

The fibration $ p $ is \emph{cloven} if it comes with a choice of cartesian lifts. The comprehension category is said to be \emph{split} is $ p $ is a split fibration. We also say that is \emph{full} if $ F $ is fully faithful, we use the notation $ (\catC,\catE,p,F) $ for a comprehension category.

The following example appears in \cite[Example 4.5]{jacobs1993}, we rewrite it in our setting of $\kappa$-clans. Let us fix a $ \kappa $-clan $ \catC $, then the inclusion functor $ \iota:\fib(\catC) \hookrightarrow \catC^{\rightarrow} $ and $ P = \partial_0 \iota $ form a full comprehension category. More precisely: $ \fib(\catC) $ has objects fibrations in $ \catC $ and arrows between two fibrations $ \alpha: f \to g $ are commutative squares of the form
% https://q.uiver.app/#q=WzAsNCxbMCwwLCJBIl0sWzEsMCwiQiJdLFswLDEsIlgiXSxbMSwxLCJZIl0sWzAsMiwiZiIsMix7InN0eWxlIjp7ImhlYWQiOnsibmFtZSI6ImVwaSJ9fX1dLFsxLDMsImciLDAseyJzdHlsZSI6eyJoZWFkIjp7Im5hbWUiOiJlcGkifX19XSxbMCwxLCJrIl0sWzIsMywibCIsMl1d
\[\begin{tikzcd}
	A & B \\
	\Delta & \Gamma.
	\arrow["f"', two heads, from=1-1, to=2-1]
	\arrow["g", two heads, from=1-2, to=2-2]
	\arrow["k", from=1-1, to=1-2]
	\arrow["l"', from=2-1, to=2-2]
\end{tikzcd}\]
Hence, an object in $ \fib(\catC)_\Gamma $ over $ \Gamma \in \catC $ is a fibration $ A \display \Gamma $. Observe that an arrow $ \alpha: f \to g $ as above is cartesian if and only if it is a pullback square in $ \catC $. In conclusion, for an arrow $ l:\Delta \to \Gamma $ and $ g:B \display \Gamma \in \fib(\catC)_\Gamma $, a cartesian lift in $ \fib(\catC) $ is a pullback square
% https://q.uiver.app/#q=WzAsNCxbMCwwLCJBIl0sWzEsMCwiQiJdLFswLDEsIlxcRGVsdGEiXSxbMSwxLCJcXEdhbW1hIl0sWzAsMiwiZiIsMix7InN0eWxlIjp7ImhlYWQiOnsibmFtZSI6ImVwaSJ9fX1dLFsxLDMsImciLDAseyJzdHlsZSI6eyJoZWFkIjp7Im5hbWUiOiJlcGkifX19XSxbMCwxLCJrIl0sWzIsMywibCIsMl0sWzAsMywiXFxscmNvcm5lciIsMSx7ImxhYmVsX3Bvc2l0aW9uIjowLCJzdHlsZSI6eyJib2R5Ijp7Im5hbWUiOiJub25lIn0sImhlYWQiOnsibmFtZSI6Im5vbmUifX19XV0=
\[\begin{tikzcd}
	A & B \\
	\Delta & \Gamma.
	\arrow["f"', two heads, from=1-1, to=2-1]
	\arrow["g", two heads, from=1-2, to=2-2]
	\arrow["k", from=1-1, to=1-2]
	\arrow["l"', from=2-1, to=2-2]
	\arrow["\lrcorner"{description, pos=0}, draw=none, from=1-1, to=2-2]
\end{tikzcd}\]

This comprehension category is not necessarily split, reflecting the fact that taking pullbacks is not strictly functorial. Nevertheless, we can replace it by a split one via the functor
\[
(-)_!:\mathbf{CompCat}(\catC) \to \mathbf{SplCompCat}(\catC)
\]
from the category of comprehension categories over $ \catC $ to the category of split comprehension categories over $ \catC $, the description of this functor appears in \cite[3.1]{lumsdaine2015local} which we now recall. This produces a split comprehension category $ (\catC_!,\fib(\catC)_!,p_!,F_!) $ which is equivalent to the one we started with. Unfolding the result, we take the $ \catC_!$ to be simply $ \catC $.

The category $ \fib(\catC)_! $ has:
\begin{itemize}
	\item Objects: for each $ \Gamma\in\catC $ an object is a tuple $ A \coloneqq (V_A,E_A,f_A) $ where $ V_A\in \catC $, $ E_A \display V_A \in \fib(\catC)_{V_A} $ and $ f_A:\Gamma \to V_A \in \catC$. We also employ the notation $ [A] \coloneqq f_A^*E_A $ given by taking the pullback of $ E_A \display V_A $ along $ f_A $, so we get a fibration $ [A] \display \Gamma  $. In addition, we write $ (E_A)_{f_A}$ for the arrow $[A] \to E_A $. Thus, an object over $ \Gamma $ is a diagram in $ \catC $ of the form
	% https://q.uiver.app/#q=WzAsMyxbMSwwLCJFX0EiXSxbMSwxLCJWX0EiXSxbMCwxLCJcXEdhbW1hIl0sWzIsMSwiZl9BIiwyXSxbMCwxLCIiLDIseyJzdHlsZSI6eyJoZWFkIjp7Im5hbWUiOiJlcGkifX19XV0=
	\[\begin{tikzcd}
		& {E_A} \\
		\Gamma & {V_A}.
		\arrow["{f_A}"', from=2-1, to=2-2]
		\arrow[two heads, from=1-2, to=2-2]
	\end{tikzcd}\]
    
    \item Morphisms: A map between $ (V_B,E_B,f_B) \to (V_A,E_A,f_A) $ over $ \sigma: \Delta \to \Gamma $ is a map in $ \catE $ between $ [B] \display \Delta  $ and $ [A] \display \Gamma $, \ie a commutative square
    % https://q.uiver.app/#q=WzAsNCxbMSwxLCJcXEdhbW1hIl0sWzEsMCwiW0FdIl0sWzAsMSwiXFxEZWx0YSJdLFswLDAsIltCXSJdLFsyLDAsIlxcc2lnbWEiLDJdLFsxLDAsIiIsMCx7InN0eWxlIjp7ImhlYWQiOnsibmFtZSI6ImVwaSJ9fX1dLFszLDIsIiIsMix7InN0eWxlIjp7ImhlYWQiOnsibmFtZSI6ImVwaSJ9fX1dLFszLDFdXQ==
    \[\begin{tikzcd}
    	{[B]} & {[A]} \\
    	\Delta & \Gamma.
    	\arrow["\sigma"', from=2-1, to=2-2]
    	\arrow[two heads, from=1-2, to=2-2]
    	\arrow[two heads, from=1-1, to=2-1]
    	\arrow[from=1-1, to=1-2]
    \end{tikzcd}\]
    
    \item Composition is induced by the composition in $ \catE $, consequently, given by pasting commutative squares.
    
    \item The identity for $ (V_A,E_A,f_A) $ is the identity of $ [A] \display \Gamma $ as an object in $ \catC^{\to} $.
\end{itemize}  

We now unpack the cartesian lifts for the induced functor $ p_!: \fib(\catC)_! \to \catC_! $. Let $ \sigma:\Delta \to \Gamma $ and $ (V_A,E_A,f_A) \in \fib(\catC)_! $ over $ \Gamma $. Set $A[\sigma] \coloneqq (V_A,E_A,f_A\sigma) $, pulling back along $ f_A\sigma $, we obtain the commutative outer rectangle below
% https://q.uiver.app/#q=WzAsNixbMCwxLCJcXERlbHRhIl0sWzEsMSwiXFxHYW1tYSJdLFsyLDEsIlZfQSJdLFsyLDAsIkVfQSJdLFsxLDAsIltBXSJdLFswLDAsIihmX0FcXHNpZ21hKV4qRV9BIl0sWzAsMSwiXFxzaWdtYSIsMl0sWzEsMiwiZl9BIiwyXSxbMywyLCIiLDAseyJzdHlsZSI6eyJoZWFkIjp7Im5hbWUiOiJlcGkifX19XSxbNCwxLCIiLDAseyJzdHlsZSI6eyJoZWFkIjp7Im5hbWUiOiJlcGkifX19XSxbNCwzLCIoRV9BKV97Zl9BfSJdLFs0LDIsIlxcbHJjb3JuZXIiLDEseyJsYWJlbF9wb3NpdGlvbiI6MCwic3R5bGUiOnsiYm9keSI6eyJuYW1lIjoibm9uZSJ9LCJoZWFkIjp7Im5hbWUiOiJub25lIn19fV0sWzUsMCwiIiwyLHsic3R5bGUiOnsiaGVhZCI6eyJuYW1lIjoiZXBpIn19fV0sWzUsMywiIiwyLHsiY3VydmUiOi0zfV0sWzUsNCwiIiwwLHsic3R5bGUiOnsiYm9keSI6eyJuYW1lIjoiZGFzaGVkIn19fV1d
\[\begin{tikzcd}
	{[A[\sigma]]} & {[A]} & {E_A} \\
	\Delta & \Gamma & {V_A.}
	\arrow["\sigma"', from=2-1, to=2-2]
	\arrow["{f_A}"', from=2-2, to=2-3]
	\arrow[two heads, from=1-3, to=2-3]
	\arrow[two heads, from=1-2, to=2-2]
	\arrow[ from=1-2, to=1-3]
	\arrow["\lrcorner"{description, pos=0}, draw=none, from=1-2, to=2-3]
	\arrow[two heads, from=1-1, to=2-1]
	\arrow[bend left, from=1-1, to=1-3]
	\arrow[dashed, from=1-1, to=1-2]
\end{tikzcd}\]
The universal property of the pullback on the right give us the unique map $ A_\sigma: [A[\sigma]] \to [A] $. Therefore, a lift for $ \sigma $ is given by the evident map $ A_\sigma:(V_A,E_A,f_A\sigma) \to (V_A,E_A,f_A) $. From the definition of $ A_\sigma $ the square
% https://q.uiver.app/#q=WzAsNCxbMSwxLCJcXEdhbW1hIl0sWzEsMCwiW0FdIl0sWzAsMSwiXFxEZWx0YSJdLFswLDAsIihmX0FcXHNpZ21hKV4qRV9BIl0sWzIsMCwiXFxzaWdtYSIsMl0sWzEsMCwiIiwwLHsic3R5bGUiOnsiaGVhZCI6eyJuYW1lIjoiZXBpIn19fV0sWzMsMiwiIiwyLHsic3R5bGUiOnsiaGVhZCI6eyJuYW1lIjoiZXBpIn19fV0sWzMsMSwiQV9cXHNpZ21hIl0sWzMsMCwiXFxscmNvcm5lciIsMSx7ImxhYmVsX3Bvc2l0aW9uIjowLCJzdHlsZSI6eyJib2R5Ijp7Im5hbWUiOiJub25lIn0sImhlYWQiOnsibmFtZSI6Im5vbmUifX19XV0=
\[\begin{tikzcd}
	{[A[\sigma]]} & {[A]} \\
	\Delta & \Gamma
	\arrow["\sigma"', from=2-1, to=2-2]
	\arrow[two heads, from=1-2, to=2-2]
	\arrow[two heads, from=1-1, to=2-1]
	\arrow["{A_\sigma}", from=1-1, to=1-2]
	\arrow["\lrcorner"{description, pos=0}, draw=none, from=1-1, to=2-2]
\end{tikzcd}\]
is a pullback, this implies that the square as a map in $ \fib(\catC)_! $ is a cartesian lift of $ \sigma $ for $ p_! $. Most importantly, this lift is uniquely determined by the composition $ f_A \sigma  $. Note that the transfinite composition of fibrations play no role in the construction. We summarize the discussion above in the following:

\begin{theorem}\label{split-comprehension-clan}
  For any $ \kappa $-clan $ \catC $ there exist a full split comprehension category $ (\catC', \catE,p_!,\iota_!) $ equivalent to $ (\catC,\fib(\catC),p,\iota). $
\end{theorem}
\begin{proof}
  We apply the previous construction, this give us $ (\catC_!,\fib(\catC)_!,p_!) $. Since the putative cartesian map is uniquely determined by the composition $ f_A\sigma  $, we can use a slight abuse of notation and write $ A_{\sigma} \coloneqq f_A\sigma $. Thus, if $ \chi:\Xi \to \Delta $ is another map then $ f(\sigma \chi)= (f\sigma) \chi$. This shows that the fibration $ p_!:\fib(\catC)_! \to \catC_! $ is indeed split.
  The functor $ \iota_!:\fib(\catC)_! \to \catC^{\to}  $ is defined as $ \iota_!(V_A,E_A,f_A) \coloneqq \iota ([A] \display \Gamma)= [A] \display \Gamma $; similarly for arrows. The comprehension category $ (\catC_!,\fib(\catC)_!,p_!,\iota_!) $ is full, since $ (\catC,\fib(\catC),p,\iota) $ is full.
\end{proof}

A \emph{category with attributes} is a comprehension category $ (\catC,\catE,p,F) $ such that $ p $ is a discrete fibration. Equivalently, a category with attributes can be defined by the following data:
\begin{enumerate}
\item A category $\catC$ with a terminal object $1$.
\item A presheaf $\ty: \catC^{op} \to \set$.
\item A function that assigns to each object $A \in \ty(\Gamma)$, an object $\Gamma. A \in \catC$, together with a map $\Gamma.A \to \Gamma$.
\item For each $A \in \ty(\Gamma)$ and $\sigma:\Delta \to \Gamma$, a pullback square
  % https://q.uiver.app/#q=WzAsNCxbMSwwLCJcXEdhbW1hLkEiXSxbMSwxLCJcXEdhbW1hIl0sWzAsMSwiXFxEZWx0YSJdLFswLDAsIlxcc2lnbWFeKlxcR2FtbWEuQSJdLFswLDFdLFsyLDEsIlxcc2lnbWEiLDJdLFszLDJdLFszLDBdLFszLDEsIlxcbHJjb3JuZXIiLDEseyJsYWJlbF9wb3NpdGlvbiI6MCwic3R5bGUiOnsiYm9keSI6eyJuYW1lIjoibm9uZSJ9LCJoZWFkIjp7Im5hbWUiOiJub25lIn19fV1d
\[\begin{tikzcd}
	{\sigma^*\Gamma.A} & {\Gamma.A} \\
	\Delta & \Gamma
	\arrow[from=1-2, to=2-2]
	\arrow["\sigma"', from=2-1, to=2-2]
	\arrow[from=1-1, to=2-1]
	\arrow[from=1-1, to=1-2]
	\arrow["\lrcorner"{description, pos=0}, draw=none, from=1-1, to=2-2]
\end{tikzcd}\]
\end{enumerate}

\begin{corollary}\label{category-attributes-clan}
	For any $ \kappa $-clan $ \catC $ there exist a category with attributes equivalent to $\catC$.
\end{corollary}
\begin{proof}
  \Cref{split-comprehension-clan} give us a full split comprehension category $ (\catC_!,\fib(\catC)_!,p_!,\iota_!) $. We take the category to be $\catC_!=\catC$. The additional data is given in the obvious way. Defining $ \ty(\Gamma) \coloneqq (\fib(\catC)_!)_\Gamma$, for each $A \in \ty(\Gamma)$, we get $[A] \display \Gamma$ as described above. The required pullbacks are given by the cartesian lifts of $p_!$. Furthermore, these pullbacks are computed strictly along compositions, since $p_!$ is a split fibration.
\end{proof}

Our next goal is to define a $\kappa$-contextual category equivalent to $\catC$ from the category with attributes given by \cref{category-attributes-clan}. In particular, for each object $\Gamma \in \catC$, we get a $\kappa$-contextual category $\catC(\Gamma) $. We start with the following:

\begin{definition} \label{category-structure-c-gamma}

  The category structure is given by the following data:

\begin{itemize}
\item \textbf{Objects}: For each ordinal $\mu<\kappa$, we define the set $Ob_{\mu}(\catC (\Gamma))$ inductively over $\mu$;
  \begin{itemize}
  \item If $ \mu=\lambda+1 $, then we define $ Ob_\mu(\catC (\Gamma)) \coloneqq \ty( [A_\lambda])$. More explicitly, an object $A_\mu\in Ob_\mu(\catC(\Gamma)) $ can be represented as the sequence
    \[ A_\mu\display A_\lambda \display \cdots \display \Gamma \] and comes with a fibration $A_\mu \display \Gamma $.
  \item If $\mu$ is a limit ordinal, then $Ob_{\mu}(\catC (\Gamma))$ is the collection of objects of the form $A_\mu := \Lim_{\lambda<\mu}A_\lambda $ obtained as the transfinite composition of a sequence
\[ \cdots \display A_\lambda \display \cdots \display \Gamma. \]
Each object comes with a fibration $A_\mu \display \Gamma$. This is given by the transfinite composition axiom of $\catC$.
  \end{itemize}

\item \textbf{Morphisms}: For ordinals $\mu \leq \lambda<\kappa$ and objects $ B_\lambda \in Ob_\lambda(\catC (\Gamma)), A_\mu\in Ob_\mu(\catC(\Gamma))$, we set
  \[ \Hom_{\catC(\Gamma)}(B_\lambda,A_{\mu}) \coloneqq \Hom_{\catC/\Gamma}(B_\lambda,A_\mu). \]
\item The rest of the structure of $\catC(\Gamma)$ is induced by $\catC/\Gamma$, in particular, the transfinite composition is that of $\catC/\Gamma$.
\end{itemize}
\end{definition}

Before proving that this gives us a $\kappa$-contextual category, let us explain the objects of this category. Recall that for $ A \in \ty(\Gamma) $ means we have a diagram of the form
% https://q.uiver.app/#q=WzAsMyxbMSwwLCJFX0EiXSxbMSwxLCJWX0EiXSxbMCwxLCJcXEdhbW1hIl0sWzIsMSwiZl9BIiwyXSxbMCwxLCIiLDIseyJzdHlsZSI6eyJoZWFkIjp7Im5hbWUiOiJlcGkifX19XV0=
\[\begin{tikzcd}
	& {E_A} \\
	\Gamma & {V_A}.
	\arrow["{f_A}"', from=2-1, to=2-2]
	\arrow[two heads, from=1-2, to=2-2]
\end{tikzcd}\]
When we identify this object with $ [A]$, then $ \ty([A]) $ is the set of objects of the form
% https://q.uiver.app/#q=WzAsMyxbMSwwLCJFX0IiXSxbMSwxLCJFX0EiXSxbMCwxLCJbQV0iXSxbMiwxLCIoRV9BKV97Zl9BfSIsMl0sWzAsMSwiIiwyLHsic3R5bGUiOnsiaGVhZCI6eyJuYW1lIjoiZXBpIn19fV1d
\[\begin{tikzcd}
	& {E_B} \\
	{[A]} & {E_A.}
	\arrow["{(E_A)_{f_A}}"', from=2-1, to=2-2]
	\arrow[two heads, from=1-2, to=2-2]
\end{tikzcd}\]

Each of such objects gives $ (V_A,f_A,E_B) \in \ty(\Gamma)$, where $ E_B \display V_A $ is the composition $ E_B \display E_A \display V_A $. Equivalently, this is the composition $[B] \display [A] \display \Gamma$. Furthermore, if we write $\Gamma.A \coloneqq [A]$, then we can rewrite this in a more familiar fashion $\Gamma.A.B \display \Gamma.A \display \Gamma $. This illustrates the general procedure for successor ordinals. A related construction appears in \cite[Definition 4.3]{kapulkinlumsdaine2018}.

 \begin{lemma}\label{contextual-clan}
   For any $ \kappa $-clan $\catC$ and any $\Gamma \in \catC$, the category $\catC(\Gamma)$ is a $\kappa$-contextual category.
 \end{lemma}
Each axiom can be verified more or less immediately. We start with the category with attributes in \cref{category-attributes-clan} and the construction from \cref{category-structure-c-gamma}.
\begin{proof}
  \begin{enumerate}
		\item The objects of $ \catC(\Gamma) $ have grading $ Ob(\catC(\Gamma))=\coprod_{\mu<\kappa}Ob_\mu (\catC(\Gamma)) $ as in \cref{category-structure-c-gamma}. This grading determines the height of each object.
		\item The terminal object is $\Gamma$.
		\item Given ordinals $\mu \leq \lambda<\kappa$ and objects $A_\lambda, A_\mu \in \catC(\Gamma)$, the display maps between them are the maps in $Hom_{\catC(\Gamma)}(A_\lambda,A_\mu)$ which are also fibrations of $\catC$. We group these maps and objects in $Dis(\catC(\Gamma))$, which is easily seen to be a subcategory.
		\item $ Dis(\catC(\Gamma)) $ is closed under transfinite compositions, since $\catC$ is itself closed under such compositions.

		\item The inclusion functor $i:Dis(\catC(\Gamma)) \hookrightarrow \catC(\Gamma)$ preserve transfinite compositions.
		
		\item  If $ A \display B $ is an arrow in $ Dis(\catC(\Gamma)) $, then $ B \in Ob_\mu(\catC(\Gamma)) $ and $ A \in Ob_\lambda(\catC(\Gamma)) $ for some ordinals $ \lambda,\, \mu $ with $ \mu\leq\lambda $: This follows directly by the definition of the objects of $\catC(\Gamma)$
		
		\item For any object $ A \in Ob_\lambda(\catC(\Gamma)) $ and any $ \mu\leq\lambda $, there exists a unique object $ B \in Ob_\mu(\catC(\Gamma)) $ and a unique display map $ A \display B $: We can easily obtain this by induction on $\lambda$ and verify that the map has the correct length.
		
		\item Canonical pullbacks: This is given by the category with attributes structure on $\catC$, as explained in \cref{category-attributes-clan}.
		
              \item Canonical pullbacks are strictly functorial: This is exactly what \cref{category-attributes-clan} achieves.
              \item It follows from the description of objects given above.
	\end{enumerate}
 \end{proof}

 Before we can state our main result, we first need to state the appropriate notion of equivalence between $ \kappa$-clans. We borrow the definitions from \cite{joyal2017clans} adapted to our setting. Let $\catC$ and $\catE$ be two $\kappa$-coclans. We say that a functor $ F: \catC \to \catE $ is a \emph{morphism of $\kappa$-coclans} if
 
 \begin{enumerate}
   \item sends initial objects to initial objects,
   \item preserves cofibrations,
   \item preserves pushouts of cofibrations along any map
   \item preserves transfinite compositions.
 \end{enumerate}

 Furthermore, a morphism between $\kappa$-coclans $F: \catC \to \catE$ is an \emph{equivalence of $\kappa$-coclans} if there exists another morphism of $\kappa$-coclans $ G: \catE \to \catC $ and natural isomorphisms $GF\cong Id_\catC$ and $FG\cong Id_\catE$.

 Similarly, $F: \catC \to \catE$ is a \emph{morphism of $\kappa$-clans} simply if $F^{op}:\catC^{op} \to \catE^{op} $ morphism of $\kappa$-coclans, and an \emph{equivalence of $\kappa$-clans} if $F^{op}:\catC^{op} \to \catE^{op} $ is an equivalence $\kappa$-coclans. 
 
 \begin{proposition}
   A morphism of clans $ F:\catC \to \catE$ is an equivalence of clans if and only if $F$ reflects fibrations and transfinite compositions in $ Dis(\catE) $; that is, if $ F(\Lim_\lambda A_\alpha) \display F(A_0) $ is the transfinite composition of the sequence
   \[
    F(\Lim_\lambda A_\alpha) \cdots \display FA_2 \display FA_1 \display FA_0
  \]
  then $ \Lim_\lambda A_\alpha \display A_0$ is the transfinite composition of the sequence 
  \[
     \cdots \display A_2 \display A_1 \display A_0.
  \]
\end{proposition}

The equivalence of \cref{split-comprehension-clan} give us an equivalence between clans.
 \begin{corollary} \label{clan-equivalent-contextual}
   For any $ \kappa $-coclan $\catC$ there exists a $\kappa$-contextual category equivalent to it.
 \end{corollary}
 \begin{proof}
   Let us take the $\kappa$-clan given by $\catD:= \catC^{op}$. We can then observe that $\catD\cong \catD(1)$, where $\catD(1)$ is the $\kappa$-contextual category obtained from \cref{contextual-clan}. We can take the opposites again to get $\catC$.
 \end{proof}

%%%%%%%%%%%%%%%%%%%%%%%%%%%%%%%% 

%%% Local Variables:
%%% mode: latex
%%% TeX-master: "main"
%%% End:

        % Appendix C: weak model categories
        
        \section{Weak model categories}\label{appendix-c}

The most general setting in which we will show good homotopy-theoretic properties of the language introduced in \cref{fol-gat} is the framework of weak model categories introduced in \cite{henry20weak}, which we will briefly recall here. In practice this extra generality compared to a Quillen model structure is not extremely useful --- all the examples we will consider in \cref{sec:examples} are Quillen model structures --- so it would not be unreasonable to skip the present subsection. There are two reasons why we need weak model categories:
\begin{itemize}
\item A key construction toward the proof of the third invariance theorem in \cref{sec:invariance} is in general only a weak model structure, and we need to use its language as an intermediate tool.

\item Future applications to left and right semi-model structures --- actual weak model structure that are not left or right semi-model structures --- are fairly uncommon, but the weak model categories which include both left and right semi-model structure at the same time, are considerably more common.

\end{itemize}

\subsection{Review} \label{sec:wms_intro}

\begin{definition}\label{def:wms}
  A \emph{weak model category} is a category $\catM$ with three classes of maps: \emph{cofibrations}, \emph{fibrations} and \emph{weak equivalences} satisfying the following conditions:
  \begin{enumerate}
  \item \label{def:wms:initial-terminal} $\catM$ has an initial object $0$ and a terminal object $1$, the identity of $0$ is a cofibration, the identity of $1$ is a fibration.
  \item \label{def:wms:composite-fib-cofib} A composite of cofibrations with cofibrant domain is a cofibration. A composite of fibrations with fibrant codomain is a fibration.
  \item \label{def:wms:2-out-of-3} Given two composable arrows $X \overset{f}\to Y \overset{g}\to Z$ where each of $X$,$Y$ and $Z$ are fibrant or cofibrant, if two of $f$, $g$, $g \circ f$ are weak equivalences, then the third is also a weak equivalence.
     \item \label{def:wms:iso-are-we} Every isomorphism between objects that are either fibrant or cofibrant is a weak equivalence.
  \item \label{def:wms:po-cof} Given a solid diagram:
    \[\begin{tikzcd}
        A \ar[r] \ar[d,>->,"i"] \ar[dr,phantom,"\ulcorner"{description,very near end}] & B \ar[d,>->,dotted,"j"] \\
        C \ar[r,dotted] & D
      \end{tikzcd}\]
    Where $i$ is a cofibration and $A$ and $B$ are cofibrant, then the pushout $j$ exists and is a cofibration.
  \item \label{def:wms:pb-fib} The dual of condition \ref{def:wms:po-cof} holds for fibrations between fibrant objects.
  \item \label{def:wms:iso-closed} Every arrow isomorphic to a fibration, cofibration, or weak equivalence is also one.
  \item \label{def:wms:cof-trivfib-facto} Every arrow from a cofibrant to a fibrant object can be factored as a cofibration followed by a trivial fibration.
  \item \label{def:wms:trivcof-fib-facto} Every arrow from a cofibrant to a fibrant object can be factored as a trivial cofibration followed by a fibration.
  \item \label{def:wms:lifting-prop} Given a solid square:
\[ \begin{tikzcd}
    A \ar[r] \ar[d,>->,"i"] & X \ar[d,->>,"p"] \\
    B \ar[r] \ar[ur,dotted] & Y
    \end{tikzcd}\]
  Where $A$ and $B$ are cofibrant, $i$ is a cofibration, $X$ and $Y$ are fibrant, $p$ is a fibration and either $p$ or $i$ is a weak equivalence, then there exists a dotted map that makes the diagram to commute.
  \end{enumerate}

\end{definition}

\begin{remark}
  In \cref{def:wms} we use the usual conventions: a \emph{cofibrant object} is an object such that the unique map $0 \to X$ is a cofibration, and a \emph{fibrant object} is an object such that the unique map $X \to 1$ is a fibration. A trivial (co)fibration is a map which is both an equivalence and a (co)fibration.
 We will also use the term \emph{core cofibrations} to mean ``cofibration between cofibrant objects'' and \emph{core fibrations} to mean ``fibration between fibrant objects''.
\end{remark}

\begin{remark}\label{rk:only_core_matter}
  It is crucial to observe that \cref{def:wms} only involve the core cofibrations, core fibrations and weak equivalences between objects that are either fibrant or cofibrant. By that we mean that if given $\catM$ a category with these three classes of maps, then ($\catM$, cofibrations, fibrations, weak equivalences) is a weak model structure if and only if ($\catM$, core cofibrations, core fibrations, weak equivalences between objects that are either fibrant or cofibrant) is a model structure.

  For this reason, we generally consider that only core cofibrations, core fibrations and weak equivalence between objects that are either fibrant or cofibrant are to be treated as relevant notions. Nothing we will do here depends on the three class of maps outside these restrictions. In \cite{henry20weak} it was even considered that the words cofibrations, fibrations and weak equivalences to mean ``core cofibrations'', ``core fibrations'' and ``weak equivalences between fibrant or cofibrant objects''.
\end{remark}

\begin{remark}
  The definition of weak model structure in \cite{henry20weak} is different from \cref{def:wms}, but it is equivalent. It is stated without reference to the class of weak equivalence, and using the notion of (weak relative) path object and cylinder object. It is easy to show that a weak model structure in the sense of \cref{def:wms} is a weak model structure in the sense of \cite{henry20weak} by constructing the cylinder and path objects as factorization of the codiagonal and diagonal maps (see \ref{cstr:path_and_cylinder} below). Conversely, it is shown in \cite{henry20weak} that given a weak model structure, it admits a (unique\footnote{Keeping in mind \cref{rk:only_core_matter}. Only the class of weak equivalence between fibrant or cofibrant objects is uniquely defined, outside of this, there are no restriction whatsoever on weak equivalence from \cref{def:wms}.}) class of weak equivalences such that all conditions of \cref{def:wms} are satisfied.
\end{remark}

It is shown in \cite{henry20weak} that most of the basic theory of Quillen model categories carries over to weak model categories, with only some additional care taken - mostly replacing objects by fibrant and cofibrant replacement of objects before applying the usual construction. The main significant difference is that the homotopy category (defined in terms of homotopy class of maps between bifibrant objects as we will recall below) is no longer equivalent to $\catM[W^{-1}]$ - the localization of $\catM$ at weak equivalence, but only to $\catM^{\text{cof} \vee \text{fib}}[W^{-1}]$ the localization the full subcategory of objects that are either fibrant or cofibrant at the weak equivalences. The problem is that the axioms of a weak model category allows us to take a fibrant replacement of a cofibrant object $C$ as a (trivial cofibration/fibration) factorization of $C \to 1$. Similarly we can take a cofibrant replacement of a fibrant objects, but there is no way to do similar replacement with an object which is neither fibrant nor cofibrant.

We now quickly go over some aspects of the construction of the homotopy category of a weak model category, the results mentioned below are all proven in section 2.1 and 2.2 of \cite{henry20weak}.

\begin{construction}\label{cstr:path_and_cylinder}
If $X$ is a bifibrant object (i.e. fibrant and cofibrant), we can form a \emph{cylinder objects} $IX$ for $X$ as a (cofibration, trivial fibration) factorization:
 \[X \coprod X \cofibration IX \trivialfib X\]
and a path objects for $X$ as a (trivial cofibration, fibration) factorization
\[ X \trivialcof PX \fibration X \times X. \]

Given a pair of maps $f,g : X \rightrightarrows Y$ between bifibrant objects, we say they are homotopic if there is a dotted map $h$ making the diagram below commutative:
\[\begin{tikzcd}
  X \ar[d] \ar[rd,"f"] & \\
  IX \ar[r,"h"description,dotted] & Y \\
  X  \ar[u] \ar[ru,"g"swap]
\end{tikzcd}\]
or equivalently a map $h$
\[\begin{tikzcd}
   & Y \\
  X \ar[dr,"f"swap] \ar[ur,"g"] \ar[r,"h"description,dotted] & PY \ar[d] \ar[u]  \\
  &Y. 
\end{tikzcd}\]
This is an equivalence relation, and the homotopy category $\text{Ho}(\catM)$ of $\catM$ can be defined as the category of bifibrant objects with homotopy class of maps between them. Moreover, this category is equivalent to the formal localization $\catM^{\text{cof} \vee \text{fib}}[W^{-1}]$.
\end{construction}

\begin{construction}\label{cstr:weak_path_and_cylinder} Note that if an object $C \in \catM$ is only cofibrant and not fibrant we cannot define a cylinder object in the same way as above since the factorization axiom does not allow us to factor the maps $X \coprod X \to X$ if $X$ is not fibrant. In place of this, we can consider a fibrant replacement $X \trivialcof X^\fib \fibration 1$, and then form a factorization:
  \[\begin{tikzcd}
  X \coprod X \ar[d,"\nabla"] \ar[r,hook] & IX \ar[d,->>,"\sim"] \\    
X \ar[r,hook,"\sim"] & X^\fib.
\end{tikzcd}\]
This object $IX$, and more generally any object fitting into a diagram:
  \[\begin{tikzcd}
  X \coprod X \ar[d,"\nabla"] \ar[r,hook] & IX \ar[d,"\sim"] \\    
X \ar[r,hook,"\sim"] & DX   
\end{tikzcd}\]
is called a weak cylinder object. Dually, if $Y$ is fibrant we define a weak path object of $Y$ as any object $PY$ that fits into a diagram:
\[\begin{tikzcd}
    TX \ar[r,"\sim"] \ar[d,->>,"\sim"] & PX \ar[d,->>] \\
    X \ar[r,"\Delta"] & X \times X \\
  \end{tikzcd}\]

We can then show that for a pair of maps $X \rightrightarrows Y$ from a cofibrant object $X$ to a fibrant object $Y$ the following are equivalent:

\begin{itemize}
\item $f$ is homotopic to $g$ in terms of a weak cylinder object for $X$.
\item $f$ is homotopic to $g$ in terms of a weak path object for $Y$.
\item $f$ and $g$ are equal in the localization $\catM^{\text{cof} \vee \text{fib}}[W^{-1}]$.
\end{itemize}

Moreover, any arrow $X \to Y$ in the localization $\catM^{\text{cof} \vee \text{fib}}[W^{-1}]$ comes from an arrow $X \to Y$ in $\catM$.

\end{construction}

  \subsection{Weak Reedy model structure}

  Before doing all the constructions, we need to set up the formalism needed for them. In this section, we study Reedy weak model categories. These are, as the name suggests, the counterpart of Reedy model categories. Most of the proofs are straightforward adaptation of the classical ones, so they are omitted.

  \begin{definition}
    A \emph{Reedy category} is a category $R$ together with two wide subcategories $R_+$ and $R_-$ and a functor $deg: R \to \alpha $, where $\alpha$ is an ordinal, such that:
    \begin{enumerate}
    \item For every non-identity arrow $a \to b \in R_+$, $\deg(a) < \deg(b)$.
    \item For every $a \to b \in R_-$ a non-identity arrow, $\deg(b) < \deg(a)$.
    \item Every arrow in $R$ factors uniquely as an arrow in $R_-$ followed by an arrow in $R_+$.
    \end{enumerate}
    When the subcategory $R_-$ consists of identity arrows only, then $R$ is called a \emph{direct category}. Similarly, when the subcategory $R_+$ consists of identity arrows only, then $R$ is called an \emph{inverse category}. 
  \end{definition}

  Let $R$ be a Reedy category and $\catM$ be a weak model category. Consider $\catM^R$ the category of $R$-shaped diagram in $\catM$.
  Given $X: R \to \catM$ such a diagram and $r \in R$ any object. The \emph{latching object} at $r$ is the colimit (if it exists)
  \[L_rX \coloneqq \Colim_{s\in (R_+/r)-\{Id_r\}} X_s.\]
  Dually, the \emph{matching object} at $r$ is the limit (if it exists)
  \[M_rX \coloneqq \Lim_{s\in (r/R_-)-\{Id_r\}} X_s.\]
  
  \begin{definition}\label{def:ReedyMap}
    A map $f: X \to Y$ in $\catM^R$ is said to be a \emph{(trivial) Reedy cofibration} at  $r \in R$ if the colimit $L_rY \sqcup_{L_rX} X_r$ exists and the induced dotted map in the diagram below
% https://q.uiver.app/#q=WzAsNSxbMCwwLCJMX3JYIl0sWzAsMSwiTF9yWSJdLFsxLDAsIlhfciJdLFsxLDEsIkxfcllcXHNxY3VwX3tMX3JYfVhfciJdLFsyLDIsIllfciJdLFsyLDNdLFswLDJdLFswLDFdLFsxLDNdLFsyLDRdLFsxLDRdLFszLDQsIiIsMix7InN0eWxlIjp7ImJvZHkiOnsibmFtZSI6ImRhc2hlZCJ9fX1dLFszLDAsIiIsMix7InN0eWxlIjp7Im5hbWUiOiJjb3JuZXIifX1dXQ==
\[\begin{tikzcd}
	{L_rX} & {X_r} \\
	{L_rY} & {L_rY\sqcup_{L_rX}X_r} \\
	&& {Y_r}
	\arrow[from=1-1, to=1-2]
	\arrow[from=1-1, to=2-1]
	\arrow[from=1-2, to=2-2]
	\arrow[from=1-2, to=3-3,bend left]
	\arrow[from=2-1, to=2-2]
	\arrow[from=2-1, to=3-3,bend right]
	\arrow["\lrcorner"{anchor=center, pos=0.125, rotate=180}, draw=none, from=2-2, to=1-1]
	\arrow[dashed, from=2-2, to=3-3]
      \end{tikzcd}\]
 is a (trivial) cofibration in $\catM$. 
    
    Dually, $f: X \to Y$ in $\catM^R$ is said to be a \emph{(trivial) Reedy fibration} at $r \in R$ if the limit $M_rX\times_{M_rY} Y_r$ exists and the induced dotted map in the diagram below
    % https://q.uiver.app/#q=WzAsNSxbMSwyLCJNX3JYIl0sWzIsMSwiWV9yIl0sWzIsMiwiTV9yWSJdLFsxLDEsIk1fclhcXHRpbWVzX3tNX3JZfSBZX3IiXSxbMCwwLCJYX3IiXSxbMCwyXSxbMSwyXSxbMywwXSxbMywxXSxbMywyLCIiLDEseyJzdHlsZSI6eyJuYW1lIjoiY29ybmVyIn19XSxbNCwwXSxbNCwxXSxbNCwzLCIiLDEseyJzdHlsZSI6eyJib2R5Ijp7Im5hbWUiOiJkYXNoZWQifX19XV0=
\[\begin{tikzcd}
	{X_r} \\
	& {M_rX\times_{M_rY} Y_r} & {Y_r} \\
	& {M_rX} & {M_rY}
	\arrow[dashed, from=1-1, to=2-2]
	\arrow[from=1-1, to=2-3,bend left]
	\arrow[from=1-1, to=3-2,bend right]
	\arrow[from=2-2, to=2-3]
	\arrow[from=2-2, to=3-2]
	\arrow["\lrcorner"{anchor=center, pos=0.125}, draw=none, from=2-2, to=3-3]
	\arrow[from=2-3, to=3-3]
	\arrow[from=3-2, to=3-3]
      \end{tikzcd}\]
   exists and is a (trivial) fibration in $\catM$.

   A map is said to be a (trivial) Reedy (co)fibration if it is one at each $r \in R$.
 \end{definition}

 \begin{remark}
   We want to clarify that in \cref{def:ReedyMap} the colimit $L_rY \sqcup_{L_rX} X_r$ is considered as a single colimit and not as a pushout using the objects $L_r X$ and $L_r Y$. It is possible that $L_rY \sqcup_{L_rX} X_r$ exists without the colimit $L_r Y$ or $L_r X$ existing. Explicitly, it is the colimits of all the $X_{i}$ for $i \in R^+/r$ and of the $Y_i$ for $i \in R^+/r - \{id_r\}$. with all the maps coming from the functoriality in $i$ and the natural map $X_i \to Y_i$. We apply the same logic to the limit $M_rX\times_{M_rY} Y_r$. 
 \end{remark}

 \begin{definition}
   A Reedy category is said to be \emph{locally finite} if for any object $X \in R$ the categories $(R_+/X)$ and $(R_-/X)$ are finite.
 \end{definition}

  It is a classical result that for any Quillen model category $\catM$ and a Reedy category $R$ that the category of functors $\catM^R$ carries a model structure in which the weak equivalences are the level-wise weak equivalences, the (trivial) (co)fibrations are precisely the Reedy (trivial) (co)fibrations. The same result can be obtained if we simply assume that the base category carries a weak model structure.

  \begin{theorem} \label{reedy-model:theorem}
    Assume that $\catM$ is a weak model category and that $R$ is a locally finite Reedy category. Then there is a weak model structure on $\catM^R$ such that a map $f:X \to Y$ is:
    \begin{enumerate}
    \item A \textit{weak equivalence} if and only if $f_r:X_r \to Y_r$ is a weak equivalence for all $r \in R$.
    \item An \textit{(trivial) cofibration} if it is a (trivial) Reedy cofibration.
    \item An \textit{(trivial) fibration} if it is a (trivial) Reedy fibration.
    \end{enumerate}
  \end{theorem}

  \begin{remark}\label{Rk:latching_map_as_single_colimits}
    When the Reedy category is directed, this model structure coincides with the projective weak model structure. It is straightforward to define this last weak model category. In this weak model, the weak equivalences and the fibrations are the level-wise weak equivalences and fibrations respectively. Similarly, when the Reedy category is an inverse category, then the Reedy weak model structure is Quillen equivalent to the injective model structure. In this other case, weak equivalences and cofibrations are given level-wise.
  \end{remark}

We now prove the theorem:

\begin{lemma}\label{lem:Direct_colimit}
   Let $I$ be a direct category and $X :I \to \catM$ be a diagram. Let $U \subset V \subset I$ be two sieves\footnote{That is subcategories with the property that if there is an arrow $x \to x'$ and $x' \in V$ then $x \in V$.} of $I$, such that $V - U$ has a finite number of objects. Assume that the colimit
   \[ X(U) \coloneq \Colim_{u \in U} X(u) \]
   exists and is cofibrant, and that for each $v \in V-U$, the latching object $L_v X$ exists and is cofibrant, and the map $L_v X \to X(v)$ is a cofibration. Then $X(V)$ exists and the comparison map $X(U) \to X(V)$ is a cofibration. If $L_v X \to X(v)$ is actually a trivial cofibration for every $v \in V-U$, then $X(U) \to X(V)$ is a trivial cofibration.
 \end{lemma}

 \begin{proof}
   This is immediate by induction on the number of objects of $V - U$. If it only has one object, then $X(U) \to X(V)$ can be seen to be a pushout of the core cofibration $L_v X \to X_v$ to the cofibrant object $X(U)$. If $V-U$ has several objects, we iterate this process once for each object of $V-U$.  
 \end{proof}

 \begin{corollary}\label{cor:Latching_are_cofibrant}
   Let $R$ be a locally finite Reedy category, $X: R \to \catM$ be a diagram and let $k \in R$ an object. Assume that $X$ is Reedy cofibrant at every $r$ such that $\deg(r) < \deg(k)$, then the latching object $L_k(X)$ exists and is cofibrant.
 \end{corollary}
 
 \begin{proof}
   Using a proof by induction on $\deg(x)$, we can freely assume that all the latching object $L_r(X)$ are cofibrant for all $r$ such that $\deg(r) < \deg(x)$. We can then just apply the \cref{lem:Direct_colimit} to the finite direct category $I = R^+/x$ and $U = \varnothing$, $V = I$.
 \end{proof}

 \begin{corollary}\label{cor:Direct_colimit} Let $I$ be a finite direct category, and let $X:I \to \catM$ be a Reedy cofibrant diagram and $U \subset I$ be a sieve, then $\Colim_{I} X$ and $\Colim_U X$ exist, are cofibrant and the obvious comparison map  $\Colim_U X \to \Colim_{I} X$ is a cofibration.

   If furthermore the latching map $L_r X \to X(r)$ is a trivial cofibration for each $r \in I-U$, then the map $\Colim_U X \to \Colim_{I} X$ is a trivial cofibration.

\end{corollary}

\begin{proof}
   By \cref{cor:Latching_are_cofibrant} all the latching objects of $X$ are cofibrant, so we can simply apply  \cref{lem:Direct_colimit} and conclude.
 \end{proof}

 \begin{corollary}\label{cor:core_cof_are_levelwise} Let $R$ be a locally finite Reedy category.

   \begin{itemize}
   \item Any core (trivial) Reedy cofibration $X \to Y$ in $\catM^R$ is in particular a levelwise (trivial) cofibration. That is, the map $X(r) \to Y(r)$ are (trivial) cofibrations for any $r \in R$.
   \item A map $X \to Y$ in $\catM^R$ which is both a core Reedy cofibration and a level-wise weak equivalence is a trivial Reedy cofibration.
   \end{itemize}
   
 \end{corollary}

Dually, the same is true for fibrations and trivial fibrations.
 
 \begin{proof} As both statement only depends on the restriction to the subcategory $R^+$, we can freely assume that $R$ is a (locally finite) direct category. In both cases, we consider the natural transformation $X \to Y$ as a diagram $T: R \times \{0 <1 \} \to \catM$. We then observe that the latching map of $T$ at an object $(r,0)$ is just $L_r X \to X$, and the latching map of $T$ at $(r,1)$ is
   \[L_rY \sqcup_{L_r X} X(r) \to Y(r) \]
   Hence the assumption that $X \to Y$ is a core Reedy cofibration translates into the fact that $T$ is Reedy cofibrant. For any object $r \in R$, the composite $R\times \{0 < 1 \} /(r,1) \to R \times \{0 <1\} \to \catM$ is immediately seen to be Reedy cofibrant as well, and we can then apply \cref{cor:Direct_colimit} to the sieve $U = R/r \times \{0\}$ to conclude that $X(r) \to Y(r)$ is a cofibration.

 If $X \to Y$ is further assumed to be trivial, then the latching map of $T$ at all objects of the form $(r,1)$ are trivial, and hence using the ``trivial'' case of \cref{cor:Direct_colimit}, we conclude that $X(r) \to Y(r)$ is trivial. 

 If instead we assume that $X(r) \to Y(r)$ is a weak equivalence for all $r$, then we proceed by strong induction on $\deg(r)$. Assume that we already know that at all $k$ such that $\deg(k)< \deg(r)$.

 If $\deg(r) = 0$, then the latching map is just $X(r) \to Y(r)$ itself, so it is a trivial cofibration as it is a cofibration and a weak equivalence. Assume now that we already know that all the latching maps
 \[L_rY \sqcup_{L_r X} X(r) \to Y(r) \]
 are trivial cofibrations for any $r$ such that $\deg(r) < \deg(k)$. We can then deduce by the same argument as above that the map $L_k(X) \to L_k(Y)$ is a core trivial cofibration, which shows that the map $X(r) \to L_rY \sqcup_{L_r X} X(r)$ is a trivial cofibration, hence an equivalence, and hence by $2$-out-of-$3$ for equivalences, the map $L_rY \sqcup_{L_r X} X(r) \to Y(r)$, is both an equivalence and a core cofibration, so it is a (core) trivial cofibration.
\end{proof}

Note that we have also proved that:

\begin{lemma} \label{lem:LatchingDomain_are_cof}
 Let $R$ be a locally finite Reedy category, and  $i:X \to Y$ be a core Reedy cofibration in $\catM^R$. Then the domain of the latching map $L_rY \sqcup_{L_r X} X(r)$ is cofibrant.
\end{lemma}

\begin{proof}
 At the beginning of the proof of \cref{cor:core_cof_are_levelwise} we observed that it could be written as a latching object $L_{(r,1)} T$ of a cofibrant Reedy diagram $T$. Hence, the result follows from \cref{cor:Latching_are_cofibrant}. \end{proof}

 \begin{proposition}\label{prop:composite_fib}
 For any locally finite Reedy category $R$, in $\catM^R$, the composite of two Reedy core cofibrations is a Reedy core cofibrations.
 \end{proposition}

 \begin{proof}
 We use a strategy very similar to the proof of \cref{cor:core_cof_are_levelwise}. Here again, the result only depends on the restriction to $R^+$ so we can freely assume that $R$ is a direct category. Let $X \to Y \to Z$ be two composable Reedy core cofibrations in $\catM^R$. We consider this as a diagram $T:R \times \{0 < 1 < 2\} \to \catM$. As in the proof of \cref{cor:core_cof_are_levelwise}. We observe that the latching map at an element of the form $(r,0)$ is the latching map $L_r X \to X$ of $X$ hence is a cofibration as $X$ is Reedy cofibrant. The latching map at an element $(r,1)$ is the map
   \[L_rY \sqcup_{L_r X} X(r) \to Y(r) \]
   which is a cofibration as $X \to Y$ is assumed to be a Reedy cofibration. And finally, the latching map at $(r,2)$ is the map
   \[L_rZ \sqcup_{L_r Y} Y(r) \to Z(r) \]
   which is also a cofibration. So this diagram $R \times \{0 < 1< 2 \} \to \catM$ is Reedy cofibrant. It immediately follows that, for any $r \in R$ the composite $R \times \{0 <1 <2 \} / (r,2) \to R^- \times \{0 < 1< 2 \} \to \catM$ is a Reedy cofibrant diagram. Hence, applying \cref{cor:Direct_colimit}, we can deduce that the map
   \[\Colim_U T \to Z(r) \]
   is a cofibration, where $U \subset R \times \{0 <1 <2 \} / (r,2)$ is the sieve containing all the objects except $(r,1)$ and $(r,2)$. But this map can be seen to be exactly
   \[L_rZ \sqcup_{L_r X} X(r) \to Z(r) \]
   by \cref{Rk:latching_map_as_single_colimits}. This concludes the proof, as this can be applied to any object $r \in R$.
 \end{proof}

 \begin{proposition}\label{prop:po-Reedy-cofibration} Consider a cospan $Y \leftarrow X \to Z$ of diagram $R \to \catM$, such that $X,Y,Z$ are all Reedy cofibrant and the arrow $X \to Y$ is a Reedy cofibration. Then the (level-wise) pushout $Y \sqcup_X Z$ exists in $\catM^R$ and the natural transformation $Z \to Y \sqcup_X Z$ is a Reedy cofibration.
 \end{proposition}

 \begin{proof}
   It follows from \cref{cor:core_cof_are_levelwise} that for each $r \in R$ the three objects in the diagram $Y(r) \leftarrow X(r) \to Z(r)$ are cofibrant and the map $X(r) \to Y(r)$ is a cofibration, so the levelwise pushout $Y(r) \sqcup_{X(r)} Z(r)$ exists and by general category-theoretic results is functorial in $r$ and is a pushout in the category of diagrams $\catM^R$. We only need to check that the map $Z(r) \to Y(r) \sqcup_{X(r)} Z(r)$ is a Reedy cofibration. For this observe that as colimits commute with colimits we have:
   \[ L_r( Y \sqcup_X Z) = \Colim_{r' \to r \in R^+} Y(r') \sqcup_{X(r')} Z(r') = L_r Y \sqcup_{L_r X} L_r Z  \]
   So that in the latching map
   \[ L_r(Y \sqcup_X Z) \sqcup_{L_r Z} Z \to Y \sqcup_X Z \]
   the domain can be identified with
   \[ \left( L_rY \sqcup_{L_rX} L_r Z \right) \sqcup_{L_r Z} Z = L_r Y \sqcup_{L_r X} Z = (L_r Y \sqcup_{L_r X} X ) \sqcup_X Z \]
   so the latching map is
   \[(L_r Y \sqcup_{L_r X} X ) \sqcup_X Z \to Y \sqcup_X Z\]
 which is a pushout of the latching map $L_r Y \sqcup_{L_r X} X \to Y$. The latter map is itself a core cofibration since $X \to Y$ is a core Reedy cofibration. Hence, this concludes the proof.
\end{proof}

We are now ready to prove \cref{reedy-model:theorem}:

\begin{proof}
  We go over all the conditions of \cref{def:wms}. The validity of conditions \ref{def:wms:initial-terminal}, \ref{def:wms:2-out-of-3}, \ref{def:wms:iso-closed} and \ref{def:wms:iso-are-we} is trivial.
  Condition \ref{def:wms:composite-fib-cofib} is \cref{prop:composite_fib} together with its dual.
  Condition \ref{def:wms:po-cof} is \cref{prop:po-Reedy-cofibration}, and condition \ref{def:wms:pb-fib} is the dual statement.

  The proof of conditions \ref{def:wms:lifting-prop} is essentially the same as the proof for ordinary model categories, as for example in Chapter 15 of \cite{hirschhorn2003model} or in Chapter 5.2 of \cite{hovey1999models}. The key step in the proof is that in order to construct a diagonal lift in a square:
\[ \begin{tikzcd}
    A \ar[r] \ar[d,>->,"i"] & X \ar[d,->>,"p"] \\
    B \ar[r] \ar[ur,dotted] & Y
  \end{tikzcd}\]
where say $i$ is a core cofibration and $p$ is a core fibration, one of them being a (level-wise) weak equivalence. Then we proceed by induction as in the usual proof, at each step we need to produce a diagonal lift in a square of the form
\[ \begin{tikzcd}
    A(r) \sqcup_{L_r A} L_r(B) \ar[r] \ar[d,>->] & X(r) \ar[d,->>] \\
    B(r) \ar[r] \ar[ur,dotted] & Y(r) \times_{M_r Y} M_r X
  \end{tikzcd}\]
Now, by \cref{lem:LatchingDomain_are_cof} (and its dual) the object $A(r) \sqcup_{L_r A} L_r(B)$ is cofibrant and $Y(r) \times_{M_r Y} M_r X$ is fibrant. By definition of Reedy cofibration and fibration, the left vertical map is a cofibration and the right vertical is a fibration, and if one of $i$ or $p$ (say $i$) is a weak equivalence. Then the second point of \cref{cor:core_cof_are_levelwise} shows that the left vertical map is a trivial cofibration, hence the square admits a diagonal lift, which concludes the proof.

The proof of  condition \ref{def:wms:cof-trivfib-facto} and (dually of condition \ref{def:wms:trivcof-fib-facto}), also follows very closely the classical proof, as in Chapter 15 of \cite{hirschhorn2003model} or in Chapter 5.2 of \cite{hovey1999models}. Given $A \to X$ a map from a Reedy cofibrant diagram to a Reedy fibrant diagram that we want to factor as a core trivial Reedy cofibration followed by a core Reedy fibration, $ A \to B \to X$. We proceed by induction to construct the diagram, the object $B(r)$, and the maps $A(r) \to B(r) \to X(r)$ gradually by induction on the degree of $r$. Following the classical proof, at each stage, we need to construct a factorization of a map in $\catM$:
\[ A(r) \sqcup_{L_r A} L_r B \to X(r) \times_{M_r X} M_r B \]
as a trivial cofibration followed by a fibration. But as observed above, the domain is cofibrant and the target is fibrant, so this is indeed possible in $\catM$. The case of condition \ref{def:wms:trivcof-fib-facto} is done in the exact same way, but factoring the map above as a cofibration followed by a trivial fibration

\end{proof}

\bibliography{homotopy-languages}
\bibliographystyle{alpha}

\end{document}